\documentclass[11pt]{article}
\usepackage{latexsym}
\usepackage{amsmath}
\usepackage{amssymb}
\usepackage{amsthm}
\usepackage{epsfig}
\usepackage{xcolor}
\usepackage{graphicx}
\usepackage{bm}
\usepackage{enumitem}
\usepackage{mathtools}

\usepackage{graphicx}
\usepackage[toc,page]{appendix}
\usepackage{graphicx}
\usepackage{bbm}
\usepackage{ytableau}
\usepackage[linesnumbered, ruled, vlined]{algorithm2e}
\usepackage[top=1in, bottom=1in, left=1in, right=1in]{geometry}

\usepackage{array}
\usepackage{multirow}

\usepackage{cite}

\usepackage{hyperref}
\hypersetup{
    colorlinks,
    linkcolor={blue!80!black},
    citecolor={green!50!black},
}
\colorlet{linkequation}{blue}

\usepackage{mathtools}
\mathtoolsset{showonlyrefs}

\usepackage{authblk}

\renewcommand{\P}{\mathbb{P}}
\newcommand{\E}{\mathbb{E}}
\newcommand{\Var}{{\rm Var}}
\newcommand{\Cov}{\text{Cov}}
\newcommand{\cN}{\mathcal{N}}

\newcommand{\R}{\mathbb{R}}

\newcommand{\N}{\mathbb{N}}
\renewcommand{\S}{\mathbb{S}}

\newcommand{\Mat}{\mathrm{Mat}}
\newcommand{\eps}{\varepsilon} 

\def\id{{\rm id}}

\renewcommand{\d}{\textup{d}}

\newcommand{\<}{\langle}
\renewcommand{\>}{\rangle}
\newcommand{\sign}{\text{sign}}

\newcommand{\op}{{\rm op}}

\def\sT{{\mathsf T}}
\def\bzero{{\boldsymbol 0}}

\DeclareMathOperator*{\argmin}{arg\,min}

\newtheoremstyle{customplain}
  {\topsep}        % space above
  {\topsep}        % space below
  {\itshape}       % body font -> italic
  {}               % indent
  {\bfseries}      % header font -> bold
  {.}              % punctuation after header
  {.5em}           % space after header
  {}   

\theoremstyle{customplain}
\newtheorem{theorem}{Theorem}
\newtheorem*{theorem*}{Theorem}
\newtheorem{lemma}{Lemma}

\newtheorem{assumption}{Assumption}
\newtheorem{definition}{Definition}

\newtheorem{proposition}{Proposition}

\newtheorem{example}{Example}

\newtheoremstyle{uprightremark}      % Custom style
  {\topsep}                          % Space above
  {\topsep}                          % Space below
  {\normalfont}                      % Body font: upright
  {}                                 % Indent amount
  {\bfseries}                        % Theorem head font
  {.}                                % Punctuation after head
  {.5em}                             % Space after head
  {}                                 % Theorem head spec

\theoremstyle{uprightremark}
\newtheorem{remark}{Remark}[section]

\DeclareSymbolFont{rsfs}{U}{rsfs}{m}{n}
\DeclareSymbolFontAlphabet{\mathscrsfs}{rsfs}

\def\bA{{\boldsymbol A}}
\def\bB{{\boldsymbol B}}
\def\bC{{\boldsymbol C}}
\def\bD{{\boldsymbol D}}
\def\bE{{\boldsymbol E}}

\def\bG{{\boldsymbol G}}

\def\bI{\mathbf{I}}

\def\bM{{\boldsymbol M}}

\def\bP{{\boldsymbol P}}
\def\bQ{{\boldsymbol Q}}

\def\bS{{\boldsymbol S}}
\def\bT{{\boldsymbol T}}
\def\bU{{\boldsymbol U}}
\def\bV{{\boldsymbol V}}
\def\bW{{\boldsymbol W}}
\def\bX{{\boldsymbol X}}
\def\bY{{\boldsymbol Y}}
\def\bZ{{\boldsymbol Z}}

\def\ba{{\boldsymbol a}}
\def\bb{{\boldsymbol b}}
\def\be{{\boldsymbol e}}

\def\bh{{\boldsymbol h}}

\def\bt{{\boldsymbol t}}
\def\bu{{\boldsymbol u}}
\def\bv{{\boldsymbol v}}
\def\bw{{\boldsymbol w}}
\def\bx{{\boldsymbol x}}
\def\by{{\boldsymbol y}}
\def\bz{{\boldsymbol z}}

\def\blambda{{\boldsymbol \lambda}}
\def\bpsi{{\boldsymbol \psi}}
\def\bphi{{\boldsymbol \phi}}
\def\btheta{{\boldsymbol \theta}}

\def\bxi{{\boldsymbol \xi}}

\def\bDelta{{\boldsymbol \Delta}}

\def\bTheta{{\boldsymbol \Theta}}

\def\bPi{{\boldsymbol \Pi}}

\def\cR{\mathcal{R}}

\def\de{{\rm d}}
\def\Tr{{\rm Tr}}

\def\de{{\rm d}}
\def\Unif{{\rm Unif}}

\def\cV{{\mathcal V}}

\def\cO{{\mathcal O}}
\def\cP{{\mathcal P}}

\def\cT{{\mathcal T}}

\def\cL{{\mathcal L}}

\def\cE{{\mathcal E}}
\def\cS{{\mathcal S}}

\def\cV{{\mathcal V}}

\def\cO{{\mathcal O}}
\def\cP{{\mathcal P}}

\def\cT{{\mathcal T}}
\def\cH{{\mathcal H}}
\def\cA{{\mathcal A}}

\def\Unif{{\sf Unif}}

\def\proj{{\mathsf P}}

\def\naturals{{\mathbb N}}

\def\proj{{\mathsf P}}

\def\Unif{{\sf Unif}}

\def\proj{{\mathsf P}}

\def\naturals{{\mathbb N}}
\def\proj{{\mathsf P}}

\def\cE{{\mathcal E}}
\def\cD{{\mathcal D}}

\def\cS{{\mathcal S}}

\def\He{{\rm He}}

\def\de{{\rm d}}
\def\Unif{{\rm Unif}}

\def\cE{{\mathcal E}}

\def\bt{{\boldsymbol t}}

\def\bDelta{{\boldsymbol \Delta}}

\def\bA{{\boldsymbol A}}
\def\btheta{{\boldsymbol \theta}}
\def\bTheta{{\boldsymbol \Theta}}

\def\blambda{{\boldsymbol \lambda}}

\def\cM{{\mathcal M}}

\def\cT{{\mathcal T}}
\def\cV{{\mathcal V}}
\def\bP{{\boldsymbol P}}

\def\bS{{\boldsymbol S}}

\def\bD{{\boldsymbol D}}

\def\bb{{\boldsymbol b}}

\def\bpsi{{\boldsymbol \psi}}

\def\bC{{\boldsymbol C}}

\def\bzeta{{\boldsymbol \zeta}}

\def\ind{\mathbbm{1}}

\def\balpha{\boldsymbol{\alpha}}

\def\cY{\mathcal{Y}}
\def\cZ{\mathcal{Z}}

\def\br{{\boldsymbol r}}
\def\bGamma{{\boldsymbol \Gamma}}

\def\balpha{\boldsymbol{\alpha}}

\def\Sym{\mathfrak{S}}

\def\polylog{\text{polylog}}

\def\sA{{\sf A}}

\def\SQ{{\sf SQ}}

\def\Cov{{\sf Cov}}

\def\dist{{\rm dist}}

\def\sm{\mathsf{m}}

\def\Stf{{\rm Stf}}
\def\s{\mathsf{s}}

\def\cSO{\cS\cO}
\def\Lop{\mathbbm{L}}

\def\qLeap{\mathsf{qLeap}}
\def\sLeap{\mathsf{sLeap}}

\def\qAlign{\mathsf{qAlign}}
\def\sAlign{\mathsf{sAlign}}

\def\bUpsilon{{\boldsymbol \Upsilon}}

\newcommand{\frob}{\mathsf{F}}

\newcommand{\perm}{\mathfrak{S}}
\let\Sym\relax
\DeclareMathOperator{\Sym}{Sym} % e.g. \Sym_\ell(\R^d) or \Sym_\ell^W(\R^d)
\DeclareMathOperator{\TSym}{TSym} % e.g. \TSym_\ell(\R^d) or \TSym_\ell^W(\R^d)

\newcommand{\psym}{\proj_{\mathsf{sym}}}
\newcommand{\ptf}{\proj_{\mathsf{tf}}}

\usepackage{mathrsfs}
\newcommand{\sh}{\mathscr{S}} % Spherical Harmonics, e.g. \sh_{d,\ell} for degree \ell in d dimensions or maybe \sh_\ell(\R^d)
\newcommand{\Span}{\mathrm{span}}

\newcommand{\rnk}{\mathsf{r}}

\newcommand{\m}{\mathsf{m}}

\newcommand{\thr}{\mathsf{t}}

\DeclareMathOperator{\rank}{\mathrm{rank}}

\newcommand{\rt}{\texttt{T}}

\newcommand{\tol}{\tau_{\mathrm{SQ}}}

\title{Statistical-Computational Trade-offs in Learning Multi-Index Models via Harmonic Analysis}

\author{Hugo Latourelle-Vigeant}
\author{Theodor Misiakiewicz}

\affil{\small Department of Statistics and Data Science, Yale University}

\date{\today}

\allowdisplaybreaks

\begin{document}

\maketitle

\begin{abstract}
    We study the problem of learning \textit{multi-index models} (MIMs), where the label depends on the input $\boldsymbol{x} \in \mathbb{R}^d$ only through an unknown $\mathsf{s}$-dimensional projection $\boldsymbol{W}_*^\mathsf{T} \boldsymbol{x} \in \mathbb{R}^\mathsf{s}$. Exploiting the equivariance of this problem under the orthogonal group $\mathcal{O}_d$, we obtain a sharp harmonic-analytic characterization of the learning complexity for MIMs with spherically symmetric inputs---which refines and generalizes previous Gaussian-specific analyses. Specifically, we derive statistical and computational complexity lower bounds within the  Statistical Query (SQ) and Low-Degree Polynomial (LDP) frameworks. These bounds decompose naturally across spherical harmonic subspaces. Guided by this decomposition, we construct a family of spectral algorithms based on \textit{harmonic tensor unfolding} that sequentially recover the latent directions and (nearly) achieve these SQ and LDP lower bounds. Depending on the choice of harmonic degree sequence, these estimators can realize a broad range of  trade-offs between sample and runtime complexity. 
    %Our lower bounds suggest that, in general, no algorithm can achieve both optimal sample and runtime complexity (among polytime algorithms) simultaneously when learning spherical MIMs.
    From a technical standpoint, our results build on the semisimple decomposition of the $\mathcal{O}_d$-action on $L^2 (\mathbb{S}^{d-1})$ and the intertwining isomorphism between spherical harmonics and traceless symmetric tensors. 
    
    %This symmetry-based analysis extends naturally to a broader class of equivariant learning problems, a direction we develop in a companion paper.
\end{abstract}

\tableofcontents

\clearpage

\section{Introduction}

Over the past decades, a major focus in statistics and learning theory has been to understand \textit{computational bottlenecks} in high-dimensional learning---that is, when a task can be solved computationally efficiently, not just statistically efficiently. Indeed, in many settings, computational tractability is dramatically more restrictive than statistical feasibility: there exist broad parameter regimes where learning is information-theoretically possible, but no polynomial-time algorithm is known to succeed---a so-called \textit{computational-statistical gap} \cite{kunisky2019notes,bandeira2022franz,wein2025computational}. More generally, one observes computational-statistical trade-offs, where additional computational power can compensate for fewer samples, and vice versa.  Understanding when such trade-offs arise and how to quantify them have become major goals in the learning theory community.

Two main lines of work have approached these questions: (1) \textit{In high-dimensional inference}, including planted clique \cite{alon1998finding,barak2019nearly}, sparse PCA \cite{berthet2013complexity,brennan2019optimal}, tensor PCA \cite{montanari2014statistical,hopkins2017power} and community detection \cite{hajek2015computational,abbe2018community}.  These works identify signal-to-noise thresholds separating regimes where no polynomial-time algorithms succeed (conjecturally) from regimes where efficient methods exist. (2) \textit{In learning theory}, with an emphasis on how structural assumptions enable computational tractability of learning function classes---e.g., parities, sparse functions, juntas, or decision trees \cite{mansour1994learning,kearns1998efficient,feldman2007attribute,shalev2014understanding,feldman2017statistical,daniely2021local}. In particular, recent works \cite{abbe2021staircase,abbe2022merged,abbe2023sgd,joshi2024complexity,bietti2023learning,troiani2024fundamental} have highlighted the role of staircase-type structure in the function spectral decomposition (e.g., Fourier or Hermite) that algorithms can exploit to efficiently learn components in increasing order of difficulty.
%with simpler components facilitating the recovery of more complex ones.

The goal of this paper is to bring these two approaches together. We focus on the classical problem of learning \textit{multi-index models} in high dimensions, which has seen a resurgence of interest in recent years, in part due to its connections to neural networks; see, e.g., \cite{dudeja2018learning,chen2020learning,arous2021online,chen2022learning,damian2022neural,abbe2023sgd,gollakota2023agnostically,troiani2024fundamental,diakonikolas2025robust,diakonikolas2025algorithms,damian2025generative} and references therein. We characterize the computational-statistical trade-offs in these models, and show that the Pareto frontier (achievable trade-offs) can be  highly non-uniform and discontinuous.
% and study computational-statistical trade-offs in heterogeneous, multi-component models. Prior work suggests that in such settings, the landscape of achievable trade-offs (Pareto frontier) can be highly non-uniform and discontinuous: different components may exhibit very different computational-statistical trade-offs, and the overall complexity depends on the particular recovery sequence.
 Our main contributions are two-fold: 
\begin{itemize}
    \item[(1)]  We establish \textit{query complexity} lower bound (proxy for runtime) within the Statistical Query (SQ) framework and \textit{sample complexity} lower bound (smallest sample size below which no polynomial-time algorithms succeed) within the Low-Degree Polynomial (LDP) framework, for multi-index models with \textit{arbitrary spherically symmetric} input distribution. Prior work has largely focused on Gaussian inputs, with analyses relying heavily on Gaussian-specific properties.

    \item[(2)] We introduce a family of \textit{iterative harmonic tensor unfolding} algorithms that sequentially recover the support and generalize recent algorithms from \cite{joshi2025learning,damian2025generative}.
\end{itemize}
 We show that these estimators match the optimal sample complexity and (nearly) the optimal query complexity within LDP and SQ. Moreover, by choosing different sequence of harmonic degrees, these estimators can realize intermediate trade-offs between sample complexity and runtime. 
 
 %Our lower bounds suggest that, in general, \textit{no algorithm can achieve both optimal sample and query complexity simultaneously when learning  general MIMs}.

\subsection{Learning Multi-Index Models}

A \textit{multi-index model} (MIM) is a joint distribution on $(y,\bx) \in \cY \times \R^d$ of the form
\begin{equation}\label{eq:MIMs_intro}
   (y,\bx) \sim \P^{\bW_*}_{\rho} :\qquad \bx \sim \mu, \qquad \quad y | \bx \sim \rho ( \cdot | \bW_*^\sT \bx ),
\end{equation}
where $\bW_* \in \R^{d \times \s}$ is the (unknown) rank-$\s$ support and $\rho (\cdot | \bt) \in \cP(\cY)$, $\bt \in \R^{\s}$, is the link function. Thus the response $y$ depends on covariate $\bx$ only through its $\s$-dimensional projection $\bW_*^\sT \bx$. The dimension $\s$ of the hidden subspace is assumed fixed, much smaller than the ambient dimension $d$. The case $\s = 1$ is also known as \textit{single-index models} (SIMs) or \textit{generalized linear models}. 

Due to their simplicity and flexibility, multi-index models have played a central role in statistics and learning theory for several decades \cite{Nelder_Wedderburn_1972,McCullagh_1984,hardle1989investigating,li1991sliced,dalalyan2008new,kalai2009isotron,kakade2011efficient,bruna2025survey}.  
The problem of learning \eqref{eq:MIMs_intro} from samples has a long and rich history, and we refer to the recent survey \cite{bruna2025survey} for an overview of this literature. Below, we briefly summarize some key prior work relevant to the present paper (see also Section \ref{sec:related-work}). 

%, encompassing many influential models that include intersection of half-spaces \cite{baum1990learning,vempala1997random,klivans2009cryptographic}, sum of ridge functions \cite{friedman1981projection}, or narrow multi-layer neural networks.
%\begin{itemize}
%     \item Intersection of half-spaces \cite{baum1990learning,vempala1997random,klivans2009cryptographic}: $y|\bx = \prod_{j = 1}^\s \ind [ \<\bw_j ,\bx \> > a_j]$. 

%     \item Sum of ridge functions \cite{friedman1981projection}: $y| \bx = \sum_{j = 1}^{\s} f_j ( \<\bw_j ,\bx\>)$.

%     \item Narrow multi-layer neural networks: 
%     \[
%     y | \bx = \bw_L^\sT \sigma ( \bW_{L-1} \sigma ( \cdots \sigma ( \bW_1 \bx) \cdots),
%     \]
%     with $\bw_L \in \R^\s$, $\bW_1 \in \R^{\s \times d}$, and $\bW_\ell \in \R^{\s \times \s}$, $2\leq \ell \leq L-1$.
% \end{itemize}

 At a high-level, learning MIMs can be viewed as a two-step task: (i) recover the subspace $\bW_*$ (a high-dimensional problem), and (ii) estimate the link function $\rho$ on this low-dimensional subspace. Under mild regularity conditions, this problem is \emph{information-theoretically easy}: one can recover $\bW_*$ to accuracy $\eps$ using $O_d (d/\eps^2)$ samples via exhaustive search over an $\eps$-net \cite{damian2024computational,bruna2025survey}. Such a procedure is of course computationally intractable in high dimensions. Several efficient procedures have been proposed over the years, including linear \cite{brillinger1982generalized} and moment-based estimators \cite{dennis2000save,li2007directional,klock2021estimating}, principal Hessian directions \cite{li1992principal,lu2020phase,mondelli2018fundamental}, and gradient outer-product span \cite{samarov1993exploring,hristache2001direct,mukherjee2006learning,trivedi2014consistent}. A prototypical example (e.g., see \cite{chen2020learning}) estimates the support by taking the span of the top eigenvectors of 
\begin{equation}\label{eq:chen-meka}
  \widehat{\bM} :=  \frac{1}{n} \sum_{i = 1}^n \cT(y_i) (\bx_i \bx_i^\sT - \bI_d).
\end{equation}
However, this estimator---and others of similar flavor---succeeds only under restrictive conditions on the link function, and can fail dramatically when these assumptions are violated.

A recent line of work has sought to characterize the precise limits of learning MIMs with polynomial-time algorithms under Gaussian inputs $\mu = \cN(0,\bI_d)$ \cite{barbier2019optimal,lu2020phase,mondelli2018fundamental,damian2024computational,troiani2024fundamental,kovavcevic2025spectral,diakonikolas2025robust,diakonikolas2025algorithms,damian2025generative}. Two key insights have emerged from these studies: 
\begin{itemize}
    \item[(a)] \emph{Complexity is governed by the Hermite expansion of the link function} \cite{dudeja2018learning,barbier2019optimal,mondelli2018fundamental,arous2021online,damian2024computational}: Polynomial-time algorithms require $\Theta_d(d^{\max(1,k_*/2)})$ samples to succeed (under SQ and LDP frameworks), where $k_*$ is the order of the first non-zero Hermite coefficient, the so-called \textit{generative (leap) exponent} \cite{damian2024computational,damian2025generative,diakonikolas2025algorithms,diakonikolas2025robust}.

    \item[(b)] \emph{Optimal recovery requires a multi-step procedure}: One-step estimators---such as \eqref{eq:chen-meka}---may be provably suboptimal (e.g., see \cite{diakonikolas2025robust}). A simple illustrative example is
\begin{equation}\label{eq:multi-step-process-intro}
y = z_1 +z_1z_2z_3, \qquad z_i := \sign (\< \bw_{*,i}, \bx\> ), \qquad \<\bw_{*,i},\bw_{*,j} \> = \delta_{ij}.
\end{equation}
A one-step method must fit all three directions simultaneously using the cubic term, and require $\Theta_d(d^{3/2})$ samples.  
In contrast, a two-step procedure first estimates $z_1$ from the linear term, and then $z_2 z_3$ from the cubic term, requiring only $\Theta_d(d)$ samples overall.  
The complexity of such procedures is captured by the \emph{leap complexity}: the cost of the hardest stage in the optimal multi-step recovery process \cite{abbe2022merged,abbe2023sgd,bietti2023learning,joshi2024complexity,damian2025generative,diakonikolas2025algorithms}.
\end{itemize}

In this paper, we revisit the problem of learning multi-index models in high dimensions, and consider a \textit{general spherically-invariant} \textit{input distribution}\footnote{By a Hunt-Stein type argument, such distributions are \textit{least favorable} for equivariant estimation.} $\mu$. In this setting, the model \eqref{eq:MIMs_intro} is equivariant with respect to the orthogonal group $\cO_d$, i.e., $(y,\bx) \sim \P_{\rho}^{\bW_*}$ implies $(y,g^{-1} \cdot \bx) \sim \P_{\rho}^{g \cdot \bW_*}$ for all $g \in \cO_d$, %and optimal estimators are $\cO_d$-equivariant.
% \begin{equation}
%  \widehat \bW \big( \{ y_i,g^{-1} \cdot \bx_i \}_{i \leq n} \big)  = g \cdot\widehat \bW \big( \{ y_i, \bx_i \}_{i \leq n} \big), \qquad \forall g \in \cO_d.
% \end{equation}
and the difficulty in recovering $\bW_*$ arises from the need to \textit{break this rotational symmetry}. Only those components of the link function that transform non-trivially under $\cO_d$ carry information about the latent subspace. 

This perspective naturally leads to analyzing the action of $\cO_d$ on the model---specifically, its decomposition into \textit{spherical harmonics}, which arise as irreducible representations of $\cO_d$. Intuitively, the complexity of the estimation problem will be governed by the first symmetry-breaking harmonic components. We argue that this equivariant viewpoint provides a particularly natural approach to studying the complexity of learning MIMs, even in the Gaussian setting. A recent paper by the second author \cite{joshi2025learning} developed this perspective for single-index models ($\s = 1$), showing that expanding in the spherical harmonic basis, rather than the Hermite basis, leads to a more principled derivation of optimal algorithms, while clarifying and revealing a number of new phenomena (see additional discussion in Section \ref{sec:related-work}).

We develop an harmonic-analytic characterization of the learning complexity for MIMs under spherically symmetric inputs. In particular, our work extends recent work on learning Gaussian MIMs \cite{troiani2024fundamental,diakonikolas2025algorithms,diakonikolas2025robust,damian2025generative} in two directions: 
\begin{description}
    \item[Arbitrary spherically-invariant input distribution.] Our analysis leverages group-theoretic properties of $\cO_d$, which lead to natural derivations of upper and lower bounds. Our results specialize cleanly to the Gaussian case, while capturing behavior that arises beyond the Gaussian setting (see also discussions in \cite{joshi2025learning}).
    
    \item[Dissociate statistical and runtime complexity.] Prior works have focused on either the optimal \emph{sample complexity} for polynomial-time algorithms (via LDP) \cite{damian2024computational,diakonikolas2025robust,damian2025generative} or the optimal \emph{query complexity} (via SQ) \cite{abbe2023sgd,joshi2024complexity,bietti2023learning}, resulting in different definitions of leap complexity across papers.  We show that this discrepancy is intrinsic: the two frameworks capture fundamentally different barriers and lead to two distinct quantities---a \emph{sample-leap} and a \emph{query-leap} complexity---which describe two procedures, optimal in their respective resource. In general these do not coincide, and our lower bounds suggest that \textit{no algorithm can achieve both optimal sample complexity} (among polynomial time algorithms) \emph{and optimal runtime complexity simultaneously when learning MIMs}. 
    %Prior work in the Gaussian setting only focused on matching the sample complexity lower bound \cite{diakonikolas2025robust,damian2025generative}.
\end{description}

From a technical standpoint, our results exploit the semisimple decomposition of the $\cO_d$-action on $L^2 (\S^{d-1})$ into irreducible subspaces of spherical harmonics, and the intertwining isomorphism between spherical harmonics and traceless symmetric tensors. 
This symmetry-based analysis naturally extends beyond MIMs and $\cO_d$ to more general equivariant learning problems under the action of compact groups. We develop this direction further in a follow-up paper \cite{joshi2026equivariance}. 

%A brief summary of this abstract setting is provided in Section \ref{sec:abstract-equivariance-problem}.

\subsection{Summary of main results}
\label{sec:summary_main_results}

Let $\S^{d-1}$ denote the unit sphere in $\R^d$ and $\tau_d := \Unif(\S^{d-1})$ the uniform measure on the sphere.  
For $\s \in [d]$, we write $\tilde \tau_{d,\s}$ for the marginal distribution of the first $\s$ coordinates of $\bz \sim \tau_d$,  
and   $\Stf_\s (\R^d) = \{ \bM \in \R^{d \times \s} : \bM^\sT \bM = \bI_\s \}$ for the Stiefel manifold of orthonormal $\s$-frames in $\R^d$. 

Throughout the paper, we consider the following class of \emph{spherical multi-index models}.

\begin{definition}[Spherical Multi-Index Models]\label{def:spherical_MIM}
    A \emph{spherical multi-index model} of index \(\s\in \N\) is a joint distribution \(\P_{\nu_d}^{\bW_*}\) on \((\by,\bz) \in \cY \times \S^{d-1}\) specified by a Markov kernel \(\nu_d(\cdot | \bt) \in \cP(\cY)\), \(\bt \in \R^{\s}\), and an orthonormal \(\s\)-frame \(\bW_* \in \Stf_{\s}(\R^d)\) such that
    \begin{equation}
        (\by,\bz) \sim \P_{\nu_d}^{\bW_*}: \;\;\;\;\;\; \bz \sim \tau_{d} \quad \text{ and }\quad \by | \bz \sim \nu_d ( \cdot | \bW_*^\sT \bz).
    \end{equation}
\end{definition}

We allow the response $\by$ to take values in an arbitrary measurable space $\cY$, and write $\nu_d^Y$ for the marginal distribution of $\by$ under $\P_{\nu_d}^{\bW_*}$ (independent of $\bW_*$). Note that $(\by,g^{-1} \cdot \bz) \sim \P^{g \cdot \bW_*}_{\nu_d}$ for all $g \in \cO_d$ and, for simplicity, we will often suppress the superscript and write $\P_{\nu_d}$.

%The distribution $\P^{\bW_*}_{\nu_d}$ is equivariant under rotations, i.e., $(\by,g^{-1} \cdot \bz) \sim \P^{g \cdot \bW_*}_{\nu_d}$ for all $g \in \cO_d$, and  

Definition~\ref{def:spherical_MIM} covers the classical MIM formulation~\eqref{eq:MIMs_intro} with spherically-invariant inputs:

\begin{example}[Spherically-invariant input distribution]\label{ex:spherically_invariant_MIM}
Let $\mu \in \cP(\R^d)$ be invariant under orthogonal transformations, i.e., $g_\# \mu = \mu$ for all $g \in \cO_d$. Such distributions admit the polar decomposition $\bx= r \bz$, where $r = \| \bx\|_2 \sim \mu_r$ is independent of $\bz = \bx/\|\bx\|_2 \sim \tau_d$. Then the MIM \eqref{eq:MIMs_intro} can be rewritten as a spherical MIM by defining $\by  := (y,r)$ and
\[
\bz \sim \tau_d, \qquad \by= (y,r) | \bz \sim \nu_d (\de \by | \bW_*^\sT \bz) := \rho (\de y | r \bW_*^\sT \bz ) \mu_r (\de r). 
\]
Gaussian MIMs correspond to setting $\mu_r := \chi_d$.
\end{example}

Given i.i.d.\ samples \(\{(\by_i,\bz_i)\}_{i\le n}\) drawn from a spherical multi-index model \(\P_{\nu_d}^{\bW_*}\), we consider the problem of recovering the latent subspace $\Span (\bW_*)$.  We focus on the high-dimensional regime where the ambient dimension $d$ is large while the index $\s$ remains fixed (or grows slowly\footnote{Our guarantees will hold non-asymptotically, for fixed $\nu_d$, which is allowed itself to depend on $d$.} with $d$). Our aim is to characterize both the sample size $n$ and runtime $\rt$ required by algorithms for this recovery task. Specifically, we study:
\begin{itemize}
    \item[(i)] The optimal \textit{sample complexity} achievable by polynomial-time algorithms (in the sense of the conjectured LDP lower bounds). With a slight abuse of terminology, we refer to this as the \emph{sample-optimal} complexity.  
    By contrast, the purely \textit{information-theoretic} sample-optimal complexity---for unrestricted algorithms---is typically \(\Theta_d(d)\).

    \item[(ii)] The optimal \textit{runtime complexity}, which we heuristically capture through the query complexity within the SQ framework.
\end{itemize}
While these lower bounds are (necessarily) conjectural (as is standard in the SQ/LDP literature), we present a family of iterative algorithms that achieve matching upper bounds.

\begin{remark}[Do we know $\nu_d$?] For simplicity, we assume that $\nu_d$ is known and fixed; this corresponds to a fully Bayesian setting in which the model is fixed and $\bW_*$ is drawn uniformly from $\Stf_\s(\R^d)$.  
Our lower bounds evidently hold when $\nu_d$ is unknown, and our algorithms---which depend on general transformations of the data---could be extended to that setting. Because of space constraints, we do not pursue this direction here.  
After recovering $\bW_*$, one can fit $\nu_d$ by a piecewise linear function or other nonparametric procedures (e.g., see \cite{diakonikolas2025algorithms,diakonikolas2025robust}).
\end{remark}

\subsubsection{Lower bounds on learning MIMs: the Leap complexity}

We begin by establishing lower bounds on the sample and runtime complexity for learning spherical MIMs with polynomial time algorithms, within the low-degree polynomial \cite{kunisky2019notes,wein2025computational} and statistical query framework \cite{kearns1998efficient,reyzin2020statistical} (see Section \ref{sec:LDP-SQ-background} for a brief overview).

\paragraph*{Harmonic decomposition.} Our lower bounds are expressed in terms of the $\cO_d$-semisimple decomposition of $L^2(\S^{d-1})$ into spherical harmonic subspaces:
\begin{equation}
    L^2 ( \S^{d-1} ) = \bigoplus_{\ell = 0}^\infty \sh_{d,\ell}, \qquad\quad N_{d,\ell} := \dim (\sh_{d,\ell})= \Theta_d (d^\ell),
\end{equation}
where $\sh_{d,\ell}$ denotes the irreducible subspace of degree-$\ell$ spherical harmonics. For each $\ell \geq 0$, the subspace $\sh_{d,\ell}$ can be identified (by unitary equivalence) to the space of traceless symmetric tensors of order $\ell$, denoted $\TSym_\ell(\R^d)$. More precisely, there exists a degree-$\ell$ \textit{harmonic tensor} $\cH_{d,\ell} : \S^{d-1} \to \TSym_{\ell} (\R^d)$, whose entries are degree-$\ell$ spherical harmonics, such that the mapping
\begin{equation}
   \Phi_{d,\ell} : \TSym_\ell (\R^d) \to \sh_{d,\ell}, \qquad \Phi_{d,\ell} (\bA) (\bz) = \< \bA, \cH_{d,\ell} (\bz)\>_\frob 
\end{equation}
is an intertwining isometric isomorphism: for all $\bA,\bB \in \TSym_\ell (\R^d)$, $\bz \in \S^{d-1}$, and $g \in \cO_d$,
\begin{equation}
    \< \Phi_{d,\ell} (\bA),\Phi_{d,\ell} (\bB)\>_{L^2} = \<\bA,\bB\>_\frob , \qquad \Phi_{d,\ell} (\bA) (g^{-1} \cdot \bz) =  \Phi_{d,\ell} (g \cdot \bA) (\bz).
\end{equation}

\paragraph*{Lower bounds on weak recovery.} We first consider the task of \emph{weak recovery} of the signal subspace $\bW_*$, namely, achieving better  performance than random guessing. For each $\ell \geq 1$, define the $\ell$-th harmonic coefficient of $\nu_d$ by
\begin{equation}
    \bxi_{\emptyset,\ell} (\by) := \E_{\P_{\nu_d}} \left[ \cH_{d,\ell} (\bz) \big| \by\right] \in \TSym_\ell (\R^d),
\end{equation}
and write $\|  \bxi_{\emptyset,\ell}  \|_{L^2}^2 := \E_{\by \sim \nu_d^Y} [ \|  \bxi_{\emptyset,\ell} (\by)  \|_\frob^2 ]$. We establish the following lower bounds on the sample complexity $n$ (within LDP) and runtime $\rt$ (within SQ) required for weak recovery (Theorem \ref{thm:LB-weak-recovery}):
\begin{equation}\label{eq:intro_LB_weak_recovery}
    n \;\gtrsim \;\inf_{\ell \geq 1} \; \frac{d^{\ell/2}}{\| \bxi_{\emptyset,\ell} \|_{L^2}^2}, \qquad\qquad \rt\; \gtrsim \;\inf_{\ell \geq 1} \; \frac{d^{\ell}}{\| \bxi_{\emptyset,\ell} \|_{L^2}^2}.
\end{equation}
These bounds decompose naturally across irreducible subspaces. Each term in the infinum represents a lower bound for algorithms that are limited to using degree-$\ell$ spherical harmonics, and these bounds are essentially tight: we design spectral estimators, based on tensor unfolding of $\cH_{d,\ell}(\bz)$, that nearly achieve these lower bounds for every $\ell$. Note that $\| \bxi_{\emptyset,\ell} \|_{L^2}^2 \lesssim 1$ and can vanish with $d$, so these lower bounds capture the competition between the dimension $N_{d,\ell} = \Theta_d (d^\ell)$ of the harmonic subspace and the signal strength $\| \bxi_{\emptyset,\ell} \|_{L^2}^2 $ it carries about $\P_{\nu_d}$. 

From \eqref{eq:intro_LB_weak_recovery}, a natural weak-learning strategy is to choose the degree that minimizes the associated lower bound: identify the sample- or runtime-optimal degree
\begin{equation}\label{eq:intro_sample_runtime_optimal}
    \ell_\star^{(s)}
    \in \argmin_{\ell\ge 1}
        \frac{d^{\ell/2}}{\|\bxi_{\emptyset,\ell}\|_{L^2}^2},
    \qquad\qquad
    \ell_\star^{(q)}
    \in \argmin_{\ell\ge 1}
        \frac{d^{\ell}}{\|\bxi_{\emptyset,\ell}\|_{L^2}^2},
\end{equation}
and apply the associated tensor unfolding algorithm. We always have $\ell_\star^{(s)} \ge \ell_\star^{(q)}$. If $\ell_\star^{(s)}=\ell_\star^{(q)}$, then the tensor unfolding estimator attains both the conjectured sample-optimal and (almost) runtime-optimal complexity. In contrast, when $\ell_\star^{(s)} > \ell_\star^{(q)}$, our lower bounds suggest that no single algorithm can simultaneously be optimal in terms of both sample complexity and runtime: \textit{one must decide which of these resources to prioritize}. Finally, one might select intermediate degrees $  \ell \not\in \{ \ell_\star^{(s)},\ell_\star^{(q)}\}$ to achieve intermediate trade-offs between sample complexity and runtime.

\paragraph*{Lower bounds on strong recovery.} We now turn to \emph{strong recovery} of $\bW_*$, namely, recovering the whole span of $\bW_*$ with arbitrarily good accuracy. As illustrated in \eqref{eq:multi-step-process-intro}, a single step of the tensor unfolding estimator may only recover a subspace of $\bW_*$, and one need to iterate the procedure.

Suppose that at some stage we have recovered a strict subset of directions  
$\bU \subsetneq \bW_*$, with $\bU \in \Stf_r (\R^d)$. Then, by conditioning on $\bU^\sT \bz$, we can reparametrize the model $(\by,\bz) \sim \P_{\nu_d}$ as a \textit{reduced} spherical MIM $(\by_\bU,\bz_{\bU}) \sim \P_{\nu_{d,\bU}}$, now in dimension $d - r$ with $\s-r$ indices. Specifically, let $\bU_\perp \in \R^{d \times (d-r)}$ be an orthonormal complement of $\bU$ and decompose the input $\bz \in \S^{d-1}$ as
\begin{equation}
\bz = \bU \br_\bU + \sqrt{1 - \| \br_\bU\|_2^2} \bU_{\perp} \bz_\bU, \qquad \br_\bU := \bU^\sT \bz \in \R^r, \qquad \bz_\bU = \frac{\bU_\perp^\sT \bz}{\|\bU_\perp^\sT \bz \|_2} \in \S^{d-r-1},
\end{equation}
so that $\br_\bU \sim \tilde \tau_{d,r}$ is independent of $\bz_\bU \sim \tau_{d-r}$. Set $\by_\bU := (\by,\br_\bU)$. Then, under $(\by,\bz) \sim \P^{\bW_*}_{\nu_d}$, the pair $(\by_\bU,\bz_\bU)$ is also a spherical MIM of ambient dimension \(d-r\), with \((\s-r)\)-dimensional signal subspace spanned by $\bU_\perp^\sT \bW_* $ and link function $\nu_{d,\bU}$ given by
\begin{equation}\label{eq:reduced-MIM-intro}
\by_{\bU} | \bz_\bU \sim \nu_{d,\bU} \big(\de \by_\bU \big| \bW_*^\sT \bU_\perp \bz_\bU\big) := \nu_d \big( \de \by \big| \bW_*^\sT \bz \big) \tilde \tau_{d,r} (\de \br_\bU).
\end{equation}

Thus, we can reduce strong recovery to weak recovery of a sequence of such reduced spherical MIMs. For each $\ell\ge 1$, define the corresponding harmonic coefficient
\begin{equation}
    \bxi_{\bU,\ell} (\by_\bU) := \E_{\P_{\nu_d}} \left[ \cH_{d - r,\ell} (\bz_\bU) \big| \by_\bU\right] \in \TSym_\ell (\R^{d - r}).
\end{equation}
Applying the weak-recovery lower bounds \eqref{eq:intro_LB_weak_recovery} to the reduced model yields bounds depending on $d-r$ and $\bxi_{\bU,\ell}$ (associated to the subgroup $\cO_{d-r}$ acting on the subspace $\bU_\perp$ of the input).  

Repeating this argument along any sequence of intermediate subspaces produces the following worst-case complexity measures, which we call \emph{sample-leap} and \emph{query-leap} complexities:\footnote{Sample-leap and query-leap complexities are well-defined only if the intrinsic dimension is exactly $\s$. If the residual $\bz_\bU$ is independent of $y$ for any strict subframe $\bU \subsetneq \bW_*$, the leap complexity becomes infinite, implying that learning stalls after recovering $\bU$.}
\begin{equation}\label{eq:q-and-s-leap-complexities}
      \sLeap (\nu_d) =  \sup_{\bU \in {\rm Sub} (\bW_*)} \;\inf_{\ell \geq 1}\; \frac{d^{\ell/2}}{\| \bxi_{\bU,\ell} \|_{L^2}^2}, \qquad \quad  \qLeap (\nu_d) = \sup_{\bU \in {\rm Sub} (\bW_*)} \;\inf_{\ell \geq 1}\; \frac{d^{\ell}}{\| \bxi_{\bU,\ell} \|_{L^2}^2},
\end{equation}
where ${\rm Sub}(\bW_*)$ denote the set of all strict subframes of $\bW_*$ (here, $\bW_*$ is fixed arbitrarily). 
These capture the \emph{hardest} intermediate subproblem one must solve in order to recover all of $\bW_*$. We then obtain the strong-recovery lower bounds (Theorem \ref{thm:LB-strong-recovery}):
\begin{equation}\label{eq:LB-strong-recovery-intro}
    n \;\gtrsim  \;\sLeap (\nu_d) , \qquad \qquad \rt \;\gtrsim \;\qLeap (\nu_d),
\end{equation}
within the LDP (sample) and SQ (runtime) frameworks respectively. %Informally, the leap complexity is the cost of making progress on the \emph{most difficult} remaining component of the signal.

\paragraph*{Multi-step recovery algorithm.}
The leap complexity naturally suggests a family of multi-step algorithms.  
Fix a sequence of harmonic degrees $\ell_1,\ell_2,\ldots \geq 1$.  
After $t$ steps, suppose we have recovered $\bU_{\le t} \subsetneq \bW_*$.  
At step $t+1$, apply the degree-$\ell_{t+1}$ harmonic tensor unfolding algorithm to the reduced MIM $\nu_{d,\bU_{\le t}}$ to extract new directions  
$\bU_{t+1} \subseteq \bU_{\le t,\perp}^\sT \bW_*$. Then, set\footnote{Here and in what follows, the notation \(\bU_{\le t}\oplus\bU_{t+1}\) is used in an extended sense: it denotes an orthonormal frame whose span equals \(\Span(\bU_{\le t})\oplus \Span(\bU_{\le t,\perp}\bU_{t+1})\),
i.e., we view \(\bU_{t+1}\in\Stf(\R^{d-\s_{\leq t}})\) as a frame in \(\R^d\) via the identification with \(\Span(\bU_{\le t})^\perp\).}
\[
\bU_{\le t+1} := \bU_{\le t} \oplus \bU_{t+1},\]
and repeat until $\bU_{\le T} = \bW_*$.

To achieve sample-optimal (resp.\ runtime-optimal) performance, we choose at each step the sample-optimal (resp.\ runtime-optimal) harmonic degree \eqref{eq:intro_sample_runtime_optimal}.  If the two degree sequences coincide, then the same procedure is simultaneously sample- and runtime-optimal. However, in general, the harmonic degree sequence may differ significantly, with very different subspace recovery sequence $\{ \bU_{\leq t}\}$ (see examples in Section \ref{sec:examples} below). In this case, \textit{one must choose whether to be data or compute efficient.} Finally, selecting intermediate degree sequences allows to achieve intermediate trade-offs between the two resources.

\subsubsection{Learning MIMs via iterative harmonic tensor unfolding}

For each harmonic degree \(\ell \geq 1\), we propose a polynomial-time algorithm that recovers a subset of the signal directions \(\bW_*\) and (nearly) matches the lower bounds in~\eqref{eq:intro_LB_weak_recovery}.
The algorithm exploits the isomorphism between spherical harmonics and traceless symmetric tensors, together with a tensor unfolding operation that maps higher-order tensors to matrices, in the spirit of the seminal work of~\cite{montanari2014statistical} on Tensor PCA.

The algorithm depends on a choice of unfolding shape parameters \((a,b) \in \N^2\).
For \(\ell=1\), we take \((a,b)=(1,0)\), while for \(\ell \geq 2\) we choose integers \(1 \leq a \leq b\) satisfying \(a+b=\ell\).
For the degree-\(\ell\) harmonic tensor \(\cH_{d,\ell}(\bz) \in (\R^d)^{\otimes \ell}\),
we denote by \(\Mat_{a,b}(\cH_{d,\ell}(\bz))\) its unfolding into a \(d^a \times d^b\) matrix.
Given \(n\) samples \(\{(\by_i,\bz_i)\}_{i \in [n]}\) from a spherical MIM and a positive semidefinite kernel
\(K : \cY \times \cY \to \R\),
the algorithm computes the leading eigenvectors of the empirical matrix
\begin{equation}\label{eq:hatM-intro-def}
\widehat{\bM} = \frac{1}{n^2} \sum_{i,j= 1}^n (1 -\delta_{a\neq b} \delta_{i=j}) K(\by_i,\by_j) \Mat_{a,b} (\cH_{d,\ell} (\bz_i)) \Mat_{a,b} (\cH_{d,\ell} (\bz_j))^\sT \in \R^{d^a \times d^a},
\end{equation}
The resulting eigenvectors (which lie in \(\R^{d^a}\)) are then contracted with themselves to form a \(d \times d\) matrix, and the top eigenvectors of this matrix define an estimate \(\widehat{\bU}_0 \in \Stf_{\s_0} (\R^d)\).
A complete description of the procedure is given in Algorithm~\ref{alg:tensor_unfold_one_step} (Section~\ref{sec:one-step-HTU}).

While the sample complexity of the method is independent of the unfolding shape for \(a \leq b\), the runtime depends on this choice:
to minimize the computational cost, the unfolding should be as close to square as possible. We take \(a=b=\ell/2\) when \(\ell\) is even, and \((a,b)=((\ell-1)/2,(\ell+1)/2)\) when \(\ell\geq 3\) is odd.
With this choice, we obtain the following guarantees for a suitable choice of kernel (Theorem~\ref{thm:tensor_unfolding_one_step} and Proposition~\ref{prop:runtime_tensor_unfolding_one_step}).
For \(\ell \geq 2\), the estimated subspace \(\widehat{\bU}_0\) is a good approximation of a subspace of \(\bW_*\) as soon as
\begin{equation}\label{eq:intro-guarantees-intro-greater-2}
n \asymp  \frac{d^{\ell/2}}{\| \bxi_{\emptyset,\ell} \|_{L^2}^2} , \qquad\qquad \rt \asymp  \frac{d^{\ell+ \frac{1}{2}\delta_{\ell}}}{\| \bxi_{\emptyset,\ell} \|_{L^2}^2} \log(d), %\begin{cases}
    % \frac{d^{\ell}}{\| \bxi_{\emptyset,\ell} \|_{L^2}^2} \log(d) & \text{if $\ell$ is even}, \\
    % \frac{d^{\ell+ 1/2}}{\| \bxi_{\emptyset,\ell} \|_{L^2}^2} \log(d) & \text{if $\ell$ is odd}.
% \end{cases} 
\end{equation}
where $\delta_\ell = 0$ for even $\ell$ and $\delta_\ell = 1$ for odd $\ell$.
Thus, for even \(\ell\), the algorithm matches the conjectured LDP and SQ lower bounds (up to a logarithmic factor in runtime). For odd \(\ell\), the runtime is worse by an additional factor of \(\sqrt{d}\), and we leave open whether this factor can be removed.
In the special case of single-index models (\(\s=1\)),~\cite{joshi2025learning} showed that a simple online SGD algorithm achieves the optimal runtime
\(\rt \asymp d^\ell / \| \bxi_{\emptyset,\ell} \|_{L^2}^2\) (without logarithmic factors) for all \(\ell \geq 3\), albeit at the cost of a substantially worse sample complexity
\(n \asymp d^{\ell-1} / \| \bxi_{\emptyset,\ell} \|_{L^2}^2\).

The case \(\ell=1\) is more delicate.
Here, the lower bound \eqref{eq:intro_LB_weak_recovery} applies to detection and is not always tight for recovery: a detection--recovery gap can appear in this setting.
Our tensor unfolding estimator recovers a subset of signal directions with
\begin{equation}\label{eq:intro-guarantees-intro-1}
n \asymp \frac{d}{\| \bxi_{\emptyset, 1} \|_{L^2}^2}, \qquad\qquad \rt \asymp \frac{d^2}{\| \bxi_{\emptyset, 1} \|_{L^2}^2} \log (d).
\end{equation}
When \(\| \bxi_{\emptyset,1} \|_{L^2}^2 \asymp 1\), the sample complexity is tight and matches the information-theoretic lower bound for recovery.
When \(\| \bxi_{\emptyset,1} \|_{L^2}^2 \ll 1\), the above bounds can be improved using a modified algorithm under additional assumptions on \(\nu_d\); see the discussion in Section~\ref{sec:harmonic-one-step-guarantees}.

To recover the entire support $\bW_*$, we extend this approach to a multi-step procedure. Fixing a sequence of degrees \(\{\ell_t\}_{t \geq 1}\), we iteratively apply harmonic tensor unfolding to reduced MIMs obtained by conditioning on the previously recovered subspace \(\widehat{\bU}_{\leq t}^\sT\bz\). Under a stability assumption on the conditional distributions of the MIM, we show that the complexity of each step matches the corresponding single-step guarantees in~\eqref{eq:intro-guarantees-intro-greater-2} and~\eqref{eq:intro-guarantees-intro-1}.
As a result, when choosing the sample-optimal or runtime-optimal degree sequence (i.e., choosing the optimal degree \eqref{eq:intro_sample_runtime_optimal} at each step), this multi-step procedure recovers the whole support with complexity matching the leap complexity lower bounds \eqref{eq:LB-strong-recovery-intro} up to an \(O(\sqrt{d})\) factor in sample complexity (when the hardest step occurs at \(\ell=1\)), and up to a \(\widetilde{O}(\sqrt{d})\) factor in runtime (when the hardest step occurs at odd \(\ell\)).

\subsubsection{Examples}
\label{sec:examples}

We next present several examples to illustrate how iterative harmonic tensor unfolding learns multi-index models. Further discussions can be found in Section \ref{sec:applications}.

\paragraph*{Example 1: Gaussian single-index models.} This example was studied extensively in~\cite{joshi2025learning}, and we briefly recall some of its properties.
Consider a Gaussian SIM (\(\s=1\)) with
\begin{equation}\label{eq:Gaussian-SIM-example}
\bx \sim \cN(0,\bI_d), \qquad\qquad  y | \bx \sim \rho (\cdot | \<\bw_*,\bx\>),
\end{equation}
where the link function \(\rho\) has generative exponent \(k_* \geq 1\)~\cite{damian2024computational}
(see Section~\ref{sec:applications-gaussian-mims}).
By Example~\ref{ex:spherically_invariant_MIM}, this model can be rewritten as a spherical SIM and we can apply our results.

For every degree \(\ell \leq k_*\) such that \(\ell\) and \(k_*\) have the same parity, the harmonic coefficients satisfy $\| \bxi_{\emptyset, \ell} \|_{L^2}^2 \asymp d^{-(k_* - \ell)/2}$. Substituting these estimates into~\eqref{eq:intro_LB_weak_recovery}, we find that the sample complexity lower bound is \(n \asymp d^{k_*/2}\), and this bound is attained at every degree \(\ell \leq k_*\) with $\ell \equiv k_* [2]$.
In contrast, the runtime scales as \(d^{(k_*+\ell)/2}\) at degree \(\ell\).
Consequently, to achieve both sample-optimal and runtime-optimal performance, one should choose \(\ell \in \{1,2\}\) with the same parity as \(k_*\).
Harmonic tensor unfolding at this degree (with a modification for \(\ell=1\); see~\cite{joshi2025learning}) achieves
\[
n \asymp d^{k_*/2},
\qquad
\rt \asymp d^{k_*/2+1}\log d,
\]
which are both optimal. In Section~\ref{sec:applications-gaussian-mims}, we extend this analysis to Gaussian MIMs with \(\s>1\), where  no single algorithm can be simultaneously sample- and runtime-optimal in general.

An interesting variation to \eqref{eq:Gaussian-SIM-example} discussed in~\cite{joshi2025learning} arises when one observes only the normalized inputs
\(\bz := \bx/\|\bx\|_2\), a common preprocessing step in statistics and machine learning.
The resulting model is no longer a Gaussian SIM but remains a spherical SIM, so our theory still applies.
In this case, the harmonic coefficients satisfy $\| \bxi_{\emptyset,\ell} \|_{L^2}^2 \asymp d^{-(k_*-\ell)}$.  The runtime is now \(d^{k_*}\) for all $\ell \leq k_*$ with same parity, while the sample complexity scales as \(d^{k_*-\ell/2}\).
Thus, both the sample-optimal and runtime-optimal choice becomes \(\ell=k_*\), and tensor unfolding at $k_*$ achieves optimal $n \asymp d^{k_*/2}$ and $\rt \asymp d^{k_* + \frac{1}{2}\delta_{k_*}} \log d$. In particular, normalizing the inputs does not change the sample complexity but the runtime becomes quadratically worse.
This phenomenon has important implications for gradient-based algorithms; see~\cite{joshi2025learning}.

\paragraph*{Example 2: Parity functions.} 
Let \(\bz \sim \tau_d\), and let
\(\bW_* = [\bw_{*,1},\ldots,\bw_{*,\s}] \in \Stf_\s(\R^d)\) denote the unknown signal subspace. Write \(z_i := \langle \bw_{*,i}, \bz \rangle\), \(i \in [\s]\)
for the corresponding projections onto these signal directions. Consider the noisy parity model with $\s \geq 2$, defined by
\begin{equation}\label{eq:parity_function_ex_intro}
y = \sign ( z_1 z_2 \cdots z_\s  )  + \eps,
\end{equation}
with $\eps \sim \cN (0, \omega^2)$ independent additive noise\footnote{This additive-noise model is chosen for simplicity; the same conclusions hold for any link function of the form $y | \bz \sim \rho (y |  \sign ( z_1 z_2 \cdots z_\s  )).$}. This is a spherical MIM, and one can verify\footnote{For example, by adapting the proof of~\cite[Proposition~4]{damian2025generative} from Gaussian to spherical data.} that $\| \bxi_{\emptyset,\ell} \|_{L^2}^2 = 0$ for all $\ell < \s$, and $\| \bxi_{\emptyset,\s} \|_{L^2}^2 \asymp 1$. Therefore, both the sample-optimal and runtime-optimal degrees coincide at $\ell^{(s)}_\star = \ell^{(q)}_\star = \s$. Applying harmonic tensor unfolding at degree \(\ell=\s\) recovers the entire support \(\bW_*\) in a single step, with
\begin{equation}\label{eq:complexity-learning-staircases}
n \asymp d^{\s /2}, \qquad \qquad \rt \asymp d^{\s + \frac{1}{2} \delta_\s} \log(d).
\end{equation}
This achieves both optimal sample and near-optimal runtime complexity within LDP and SQ.  

\paragraph*{Example 3: Mixture of parities.} Let \(p \in (0,1/2)\), and let \(k_0 < k_1 < k_2\) be integers which we take even for simplicity.
Consider the following response model:
\[
y = \eta \cdot \sign (z_1 z_2 \cdots z_{k_1} ) + (1 - \eta) \cdot \sign ( z_{k_1 - k_0 +1} z_{k_1 - a+2} \cdots z_{k_2 + k_1 - k_0} ) + \eps,
\]
where \(\eta \sim \mathrm{Ber}(p)\) and \(\eps \sim \cN(0,\omega^2)\).
Equivalently, \(y\) is a noisy mixture of two parity functions: one of size \(k_1\) (with probability \(p\)) and one of size \(k_2\) (with probability \(1-p\)), sharing \(k_0\) common signal directions.

One can verify that $\| \bxi_{\emptyset,\ell} \|_{L^2}^2 = 0$ for $\ell < k_1$, $\| \bxi_{\emptyset,\ell} \|_{L^2}^2 \asymp p^2$ if $k_1 \leq \ell < k_2 $, and $\| \bxi_{\emptyset,\ell} \|_{L^2}^2 \asymp 1$ if $\ell \geq k_2$ ($\ell$ even). Consequently, the optimal iterative procedure is achieved at one of the following two degree sequences
\((\ell_1,\ell_2)=(k_1,k_2-k_0)\) or \((k_2,k_1-k_0)\):
\begin{itemize}
    \item \textbf{Degree sequence $(k_1,k_2-k_0)$:} Applying tensor unfolding at degree \(k_1\) first recovers the \(k_1\) directions corresponding to the smaller parity.
Conditioning on these directions, the reduced model becomes a parity function over the remaining \(k_2-k_0\) directions, which can then be recovered by tensor unfolding at degree \(k_2-k_0\).
The total complexity of these two steps is
    \[
    n \asymp d^{k_1/2}/p^2 + d^{(k_2 - k_0)/2}, \qquad \rt \asymp (d^{k_1}/p^2 + d^{k_2 - k_0}) \log (d).
    \]

    \item \textbf{Degree sequence $(k_2,k_1 - k_0)$:} We first recover the \(k_2\) directions associated with the larger parity by applying tensor unfolding at degree \(k_2\). Conditioning on these directions, we recover the remaining \(k_1-k_0\) directions via the degree \(k_1-k_0\). This yields
    \[
    n \asymp d^{k_2/2} + d^{(k_1 - k_0)/2}/p^2, \qquad \rt \asymp (d^{k_2} + d^{k_1 - k_0}/p^2) \log (d).
    \]
\end{itemize}

As a concrete illustration, fix \(q \in \naturals\) and set
\(k_0=2q\), \(k_1=4q\), \(k_2=8q\), and \(p=d^{-3q/2}\).
Then the two strategies yield
\[
\begin{aligned}
\textbf{Sample-optimal algorithm at $(k_2,k_1 - k_0)$:}& \qquad  n \asymp d^{4q}, \qquad  \rt \asymp d^{8q}\log (d), \\
\textbf{Runtime-optimal algorithm at $(k_1,k_2 - k_0)$:}& \qquad  n \asymp d^{5q}, \qquad  \rt \asymp d^{7q}\log (d).
\end{aligned}
\]
These complexities match the sample-leap and query-leap lower bounds
in~\eqref{eq:q-and-s-leap-complexities}, respectively. In particular, the runtime-optimal strategy reduces the runtime by a factor of \(d^q\) at the cost of a factor \(d^q\) increase in sample size.
Our lower bounds indicate that no single algorithm can simultaneously achieve optimal scaling for both sample complexity and runtime in this model.

\subsubsection{Organization of the paper}

The remainder of the paper is organized as follows. In Section~\ref{sec:tech-bg}, we introduce the technical background used throughout the paper, including standard tools from tensor algebra, spherical harmonics, relevant group actions, and the intertwining operators that relate these objects. 
In Section~\ref{sec:lower-bounds-leap}, we briefly review the SQ and LDP frameworks and then establish the lower bounds for weak recovery (alignment complexity) and strong recovery (leap complexity) for spherical MIMs.
Section~\ref{sec:tensor-unfolding} presents the family of iterative harmonic tensor unfolding algorithms. We first describe a one-step unfolding procedure and provide guarantees on its sample and runtime complexity. We then extend this approach to a multi-step algorithm by iterating the one-step procedure and establish learning guarantees for strong recovery of spherical MIMs. Section~\ref{sec:applications} applies our general results to more concrete model classes, including Gaussian and directional multi-index models. We conclude in Section~\ref{sec:discussion} and discuss directions for future work.

\subsection{Additional related work}\label{sec:related-work}

Below, we discuss further related work on multi-index models, beyond the references already cited in the introduction. Many classical problems, such as learning halfspaces~\cite{vempala1997random,baum1990learning,klivans2009cryptographic,klivans2024learning,vempala2010learning}, one-bit compressed sensing~\cite{Plan_Vershynin_2013,ai2014one,gopi2013one}, phase retrieval~\cite{candes2015phase,candes2013phaselift,netrapalli2013phase,mondelli2018fundamental}, and learning fixed width neural networks~\cite{ge2017learning,bakshi2019learning,chen2023learning}, arise as special cases of MIMs with particular choices of the link function. These connections have motivated a broad and diverse literature, which we do not attempt to survey exhaustively here.

In the case of single index, early work showed that SIMs with monotone link functions admit efficient learning with linear sample complexity~\cite{kalai2009isotron,kakade2011efficient}. However, statistical-computational gaps emerge even for simple non-monotonic examples like noisy phase retrieval~\cite{barbier2019optimal,maillard2020phase,mondelli2018fundamental}. For general Gaussian SIMs, \cite{damian2024computational} provided a sharp characterization (in terms of the polynomial exponent in $d$) of the statistical and computational complexity in terms of a \textit{generative exponent}, with optimal algorithm based on partial trace of an Hermite tensor.  
 This sharp characterization was extended to spherical SIMs (case $\s = 1$ in Definition \ref{def:spherical_MIM}) in \cite{joshi2025learning}, which derived SQ and low-degree lower bounds (analogous to Theorem \ref{thm:LB-weak-recovery}), and presented a harmonic tensor-unfolding\footnote{For rank-one kernels, the estimator \eqref{eq:hatM-intro-def} recovers the tensor unfolding estimator in~\cite{joshi2025learning} for spherical SIMs.} and online SGD algorithm that achieve optimal sample complexity and runtime, respectively. \cite{joshi2025learning} further illustrates the advantage of this harmonic-analytic perspective by clarifying a number of phenomena: why landscape smoothing achieves optimal complexity, why SGD is suboptimal, how exploiting the input norm is necessary to achieve optimal runtime complexity, and  how additional statistical-computational trade-offs arise beyond the Gaussian setting.

The landscape of learning MIMs is considerably richer and requires adaptive, multi-phase procedures~\cite{abbe2022merged,abbe2023sgd,bietti2023learning,damian2025generative,diakonikolas2025algorithms,diakonikolas2025robust}. Here, the complexity of learning Gaussian MIMs is captured by a \emph{(generative) leap complexity} \cite{abbe2023sgd,bietti2023learning,joshi2024complexity,damian2025generative,diakonikolas2025robust}. The weak recovery of Gaussian MIMs with leap exponent two was studied in the (optimal) proportional scaling regime $n\asymp d$ using spectral methods and approximate message passing (AMP) \cite{troiani2024fundamental,defilippis2025optimal,kovavcevic2025spectral}. Beyond the proportional regime,~\cite{damian2025generative} characterized sample complexity via the \emph{generative leap exponent}, proposing a Hermite tensor unfolding estimator with shape \((a,b)=(1,\ell-1)\) that iteratively conditions on the recovered subspace to attain the low-degree lower bound. Our estimator recovers these guarantees when instantiated with this configuration. Crucially, however, we show that adopting a ``squarer'' unfolding shape \((a,b)=(\lfloor \ell/2 \rfloor, \lceil\ell/2 \rceil)\) and selecting a different degree sequence may improves runtime complexity. This modification yields improved runtime (and (nearly) match the SQ lower bounds) while preserving optimal sample complexity for Gaussian MIMs, though our lower bounds suggest that in broader regimes, a sample-runtime trade-off is unavoidable. More broadly, our results extend the analysis of~\cite{damian2025generative} to spherical MIMs, provide explicit query complexity (runtime) lower bounds, and characterize sample-runtime trade-offs not addressed there. In another contemporary work,~\cite{diakonikolas2025algorithms,diakonikolas2025robust} developed a related learning procedure based on iterative subspace conditioning for Gaussian MIMs. While these works do not characterize the sharp polynomial dependency in \(d\), they address agnostic learning of MIMs and track the dependence on additional problem parameters. We leave extending our analysis to the agnostic setting, and obtaining similarly refined bounds to future work.

 Beyond these works, alternative estimation algorithms were studied, including semiparametric maximum likelihood~\cite{Horowitz_2009,dudeja2018learning}, sliced inverse regression~\cite{li1991sliced,Babichev_Bach_2018}, and gradient-based methods \cite{arous2021online,arnaboldi2023high,zweig2023single,arnaboldi2024online,damian2024smoothing,dandi2024benefits,lee2024neural,arnaboldi2024repetita,montanari2026phase}; see~\cite{chen2020learning,dudeja2018learning,bruna2025survey} for overviews. Recent work also investigates robust learning~\cite{diakonikolas2022learning,li2024learning,zarifis2024robustly,wang2024sample} and complementary hardness results, including cryptographic lower bounds for agnostic learning~\cite{diakonikolas2023near} and distribution-free PAC hardness~\cite{diakonikolas2022hardness}.

\subsection{Notation}

We use boldface letters (e.g., $\bx$, $\bA$) to denote vectors, matrices, and higher-order tensors, while regular letters (e.g., $x$, $A$) denote scalars. For a positive integer $d$, we write $[d] = \{1,2,\ldots,d\}$. The unit sphere in $\R^d$ is denoted by \(\S^{d-1}\). For a finite set $S$, $|S|$ denotes its cardinality. We denote by $\langle \cdot, \cdot \rangle$ and $\|\cdot\|_2$ the Euclidean inner product and norm on $\R^d$, and by $\langle \cdot, \cdot \rangle_{\frob}$ and $\|\cdot\|_{\frob}$ the Frobenius inner-product and norm for higher-order tensors $(\R^d)^{\otimes \ell}$.  For a measurable space \(\cY\), \(\cP(\cY)\) denote the collection of probability measures on \(\cY\). For a measurable function $f$, $\|f\|_{L^p(\mu)}$ denotes the $L^p$ norm with respect to a measure $\mu$. When $\mu$ is clear from context, we simply write $\|f\|_{L^p}$. In particular, for a tensor-valued random variable $\bxi$, we write \(\|\bxi\|_{L^2} = \|\|\bxi\|_{\frob}\|_{L^2}\). We will further denote $\| \cdot \|_\op$ the standard operator norm for linear operators.

For a matrix $\bA$,  $\Span(\bA)$ denotes the linear subspace spanned by its columns (more generally, $\Span(\{\bv_i\})$ denotes the linear span of a collection of vectors). For $1 \leq s \leq d$, $\Stf_{s}(\R^d)$ denotes the Stiefel manifold of orthonormal $s$-frames in $\R^d$.
For \(s, k\in [d]\) with \(s+k\leq d\), \(\bU\in \Stf_s(\R^d)\) and \(\bV \in \Stf_k(\R^d)\), we write \(\bU \oplus \bV\) for the orthonormal \((s+k)\)-frame obtained by concatenating the columns of \(\bU\) and \(\bV\) whenever \(\Span(\bU) \cap \Span(\bV) = \{0\}\).
Also, we say that \(\bU\) is a subframe of \(\bV\), denoted \(\bU\subseteq\bV\), if \(\Span(\bU)\subseteq\Span(\bV)\); it is strict when the inclusion is proper.
%, and we write \({\rm Sub}(\bV)\) for the set of all strict subframes of \(\bV\).
Given two orthonormal \(s\)-frames $\bU,\bU' \in \Stf_{s}(\R^d)$, we denote by $\dist(\bU,\bU') = \|\bU\bU^\sT - \bU'(\bU')^\sT\|_{\op}$ the distance, in operator norm, between the orthogonal projections onto their column subspaces.

Unless otherwise specified, $c',c,C',C$ denote positive constants that are independent of the dimension parameters $(n,d)$ but may depend on fixed model parameters (e.g., $\s, \ell$, and constants in the assumptions). Typically, $c,c'$ denote sufficiently small constants and $C,C'$ sufficiently large constants. Their values may change from line to line.
For functions $f,g : \N \to \R_{\ge 0}$, we write $f(d)=O(g(d))$ (resp.\ $f(d)=\Omega(g(d))$) if there exists a constant $C>0$ such that $f(d)\le C g(d)$ (resp.\ $f(d)\ge C g(d)$) for all sufficiently large $d$. We write $f(d)=\Theta(g(d))$ if both $f(d)=O(g(d))$ and $g(d)=O(f(d))$. The notation $f(d)=\widetilde{O}(g(d))$ indicates that $f(d)\le C g(d)\log^C(d)$ for some constant $C>0$ and all sufficiently large $d$. Unless otherwise specified, the constant \(C\) may depend on fixed model parameters and constants in the assumptions, but not on the dimension parameter \((n,d)\).  We use $f(d)\lesssim g(d)$, $f(d)\gtrsim g(d)$, and $f(d)\asymp g(d)$ as shorthands for $f(d)=O(g(d))$, $f(d)=\Omega(g(d))$, and $f(d)=\Theta(g(d))$, respectively. $f(d)=o(g(d))$ (or $f(d)\ll g(d)$) means that $\lim_{d\to\infty} f(d)/g(d)=0$, while $f(d)=\omega(g(d))$ (or $f(d)\gg g(g)$) means that $\lim_{d\to\infty} f(d)/g(d)=\infty$. When \(f,g\) are functions of multiple variables, we sometimes write \(f = O_d(g)\) to indicate that the asymptotic notation applies only to the variable \(d\), with other variables held fixed, and similarly for \(\Omega, \Theta, \widetilde{O},\omega, o\).

\section{Technical background}\label{sec:tech-bg}

In this section, we review technical background on traceless symmetric tensors, spherical harmonics, and their relation to representations of the orthogonal group $\cO_d$.  Most of these properties are classical, though a few are less standard and included here for completeness. For a comprehensive treatment of spherical harmonics, we refer to \cite{szeg1939orthogonal,chihara2011introduction}, and for the representation theory of $\cO_d$ and its relation to spaces of traceless symmetric tensors, see \cite{goodman2000representations,fulton2013representation}. We defer detailed statements, proofs, and additional properties to Appendix~\ref{app:tech-bg}.

\subsection{Traceless symmetric tensors, function spaces, and group action}

\paragraph*{Traceless symmetric tensors.}
Let \((\R^d)^{\otimes \ell}\) be the space of order-\(\ell\) tensors over \(\R^d\), equipped with the Frobenius inner product
\[
    \<\bA,\bB\>_\frob  = \sum_{i_1,\ldots, i_\ell=1}^{d} A_{i_1,\ldots,i_\ell} B_{i_1,\ldots,i_\ell}, \qquad \bA,\bB\in (\R^d)^{\otimes \ell}.
\]
For integers $p,q,r \geq 0$ with $r \leq \min(p,q)$, the tensor contraction $\otimes_r : (\R^d)^{\otimes p} \times (\R^d)^{\otimes q} \to (\R^d)^{\otimes p+q-2r}$ is defined by
\[
(\bA \otimes_{r} \bB)_{i_1,\ldots,i_{p-r},j_{1},\ldots,j_{q-r}} = \sum_{s_1, \ldots , s_r =1}^d A_{i_1,\ldots,i_{p-r},s_1,\ldots,s_r} B_{s_1,\ldots,s_r,j_1,\ldots,j_{q-r}}.
\]
In particular $\otimes := \otimes_0$ corresponds to the standard tensor product. For $\bM \in \R^{m\times d}$ and $\bA \in (\R^d)^{\otimes \ell}$, we write $\bM^{\otimes \ell}\bA \in (\R^m)^{\otimes \ell}$ for the tensor obtained by applying $\bM$ to each index:
\[
    \big(\bM^{\otimes \ell}\bA\big)_{i_1,\ldots, i_\ell} = \sum_{j_1,\ldots,j_\ell=1}^d M_{i_1,j_1}\cdots M_{i_\ell,j_\ell} A_{j_1,\ldots,j_\ell}.
\]

We denote by \(\Sym_\ell(\R^d)\subseteq (\R^d)^{\otimes \ell}\) the subspace of \textit{symmetric tensors}, that is tensors \(\bA\) such that for any permutation \(\sigma\in \perm_\ell\), \(A_{i_1,\ldots,i_\ell} = A_{i_{\sigma(1)},\ldots,i_{\sigma(\ell)}}\). We define the orthogonal projection onto this subspace by
\begin{equation}\label{eq:projection_symmetric_tensors}
    \psym(\bA) = \frac{1}{\ell!}\sum_{\sigma\in \perm_\ell} \bA^\sigma,
\end{equation}
where \(\bA^\sigma\) is the tensor defined by \(A^\sigma_{i_1,\ldots,i_\ell} = A_{i_{\sigma(1)},\ldots,i_{\sigma(\ell)}}\). 

For \(\ell\geq 2\), we define the partial trace operator \(\tau: \Sym_\ell(\R^d) \to \Sym_{\ell-2}(\R^d)\) by
\begin{equation}\label{eq:partial_trace_operator}
    \tau(\bA)_{i_1,\ldots,i_{\ell-2}} = \sum_{j=1}^d A_{j,j,i_1,\ldots,i_{\ell-2}}.
\end{equation}
For symmetric tensors, this definition does not depend on the choice of indices to trace out. The subspace of \textit{traceless symmetric tensors} is then
\[
    \TSym_\ell(\R^d) = \{\bA\in \Sym_\ell(\R^d): \tau(\bA) = 0\}.
\]
For \(\ell=0,1\), we set \(\TSym_0(\R^d) = \Sym_0(\R^d) = \R\) and \(\TSym_1(\R^d) = \Sym_1(\R^d) = \R^d\). The orthogonal projection onto $\TSym_\ell(\R^d)$ is denoted $\ptf:\Sym_\ell(\R^d)\to \TSym_\ell(\R^d)$; an explicit expression is given in Appendix~\ref{app:symmetric-traceless-tensor}. Finally, denote by $\Lop_{d,\ell}$ the space of linear operators on $\TSym_\ell(\R^d)$, which can be identified with $\TSym_\ell(\R^d)\otimes \TSym_\ell(\R^d)$. The operator norm, trace, and Frobenius norm on $\Lop_{d,\ell}$ are defined in the usual way.

\paragraph*{Function spaces on the sphere.} Let $\S^{d-1} = \{\bx \in \R^d : \|\bx\|_2=1\}$ denote the unit sphere and $\tau_d$ the uniform probability measure on $\S^{d-1}$. We write $L^2(\S^{d-1}) := L^2(\S^{d-1},\tau_d)$ for the space of square-integrable functions on the sphere, endowed with inner product and norm
\begin{equation}\label{eq:inner-prod-L2-sphere}
    \langle f,g \rangle_{L^2(\tau_d)} = \int_{\S^{d-1}} f(\bx) g(\bx) \tau_d(\d \bx),
\qquad \|f\|_{L^2(\tau_d)}^2 = \langle f,f\rangle_{L^2(\tau_d)}.
\end{equation}
We will also consider product spaces
\begin{equation}\label{eq:product-spaces-background}
L^2(\cV\times\S^{d-1},\nu\otimes\tau_d)
= L^2(\cV,\nu)\otimes L^2(\S^{d-1},\tau_d),
\end{equation}
with the corresponding inner product denoted $\langle \cdot,\cdot\rangle_{L^2(\nu\otimes\tau_d)}$. For simplicity, we often write $L^2(\cV)=L^2(\cV,\nu)$ and $\langle \cdot,\cdot\rangle_{L^2}$ when the measure is clear from context.

\paragraph*{Orthogonal group.} Let $\cO_d$ denote the orthogonal group in $\R^d$, identified with the set of orthogonal matrices $\bQ\in\R^{d\times d}$ satisfying $\bQ\bQ^\sT=\bI_d$. We denote by $\pi_d$ the Haar probability measure on $\cO_d$. The group $\cO_d$ acts naturally on $\R^d$ and on $(\R^d)^{\otimes\ell}$ via
\[
g\cdot \bx = \bQ\bx, \qquad g\cdot \bA = \bQ^{\otimes\ell}\bA.
\]

For any $0 \le \s \le d$, let $\Stf_\s(\R^d) = \{\bW\in\R^{d\times\s}:\bW^\sT\bW=\bI_\s\}$ be the Stiefel manifold of orthonormal $\s$-frames in $\R^d$, endowed with its uniform measure $\pi_{d,\s}$. For $\bW\in\Stf_\s(\R^d)$, the stabilizer subgroup of $\bW$ under the left action of $\cO_d$ is
\begin{equation}\label{eq:stabilizer_subgroup}
    \cO_d^{\bW} = \{ g\in\cO_d : g\cdot \bW = \bW \}.
\end{equation}
The group $\cO_d^{\bW}$ acts on $\R^d$ and on tensors as the restriction of the natural $\cO_d$-action.

Finally, we consider the unitary representations of $\cO_d$ acting on $L^2(\S^{d-1})$ and $L^2(\cV \times \S^{d-1})$ via
\begin{equation}\label{eq:left-action-fcts}
[\rho(g)\cdot f](\bz) = f(g^{-1}\cdot\bz),
\qquad
[\rho(g)\cdot f](\bv,\bz) = f(\bv,g^{-1}\cdot\bz),
\end{equation}
which define the standard left action representations of $\cO_d$ on these $L^2$ spaces.

\subsection{Spherical harmonics and intertwining operator}

\paragraph*{Spherical harmonics.}  Spherical harmonics of degree $\ell$ are defined as degree-$\ell$ homogeneous harmonic polynomials restricted to $\S^{d-1}$, that is, polynomials $P$ such that $P(t\bz) = t^\ell P(\bz)$ and $\Delta P=0$, where $\Delta$ is the Laplace operator. We denote by $\sh_{d,\ell}$ the space of such functions and $N_{d,\ell} := \dim (\sh_{d,\ell})$, where
\begin{equation*}
N_{d,0} = 1, \qquad N_{d,1} = d, \qquad N_{d,\ell} = \frac{d+2\ell-2}{\ell} \binom{d+\ell-3}{\ell-1} \quad \text{for $\ell \geq 2$.}
\end{equation*}
In particular, $|N_{d,\ell}/\binom{d}{\ell} - 1| \leq C_\ell /d$, where $C_\ell>0$ is some constant that only depend on $\ell$. 

The family $\{\sh_{d,\ell}\}_{\ell \geq 0}$ forms a collection of mutually orthogonal subspaces in $L^2(\S^{d-1})$ with respect to the inner-product \eqref{eq:inner-prod-L2-sphere}. This yields the orthogonal decomposition 
\begin{equation}\label{eq:l2-decomposition-sphere}
    L^2(\S^{d-1})
    = \bigoplus_{\ell=0}^{\infty} \sh_{d,\ell}.
\end{equation}
Each subspace $\sh_{d,\ell}$ is an irreducible subspace under the left action \eqref{eq:left-action-fcts} of the orthogonal group $\cO_d$. Hence, \eqref{eq:l2-decomposition-sphere} is exactly the semisimple (Peter-Weyl) decomposition of $L^2 (\S^{d-1})$ into irreducible representations.

\paragraph*{Isomorphism with traceless symmetric tensors.} There exists an $\cO_d$-equivariant isometry between $\TSym_{\ell}(\R^d)$  and $\sh_{d,\ell}$ with their respective inner product. For any $\ell,d \in \naturals$, define the \textit{harmonic tensor} $\cH_{d,\ell} : \S^{d-1} \to \TSym_{\ell} (\R^d)$ as\footnote{Note that this is a different but more natural normalization than \cite{joshi2025learning} where $ \cH_{d,\ell} ( \bz) = \kappa_{d,\ell}^2 \sqrt{N_{d,\ell}} \ptf (\bz^{\otimes \ell})$.} 
\begin{equation}\label{eq:harmonic_tensor_definition}
    \cH_{d,\ell} ( \bz) = \kappa_{d,\ell} \sqrt{N_{d,\ell}} \ptf (\bz^{\otimes \ell}), \qquad  \kappa_{d,\ell} = \sqrt{2^\ell \frac{(d/2 - 1)_\ell}{(d-2)_\ell}}
\end{equation}
where \(\ptf\) is the projection onto traceless symmetric tensors (see the explicit expression~\eqref{eq:projection_traceless_symmetric_tensors} in Appendix \ref{app:symmetric-traceless-tensor}) and $(a)_\ell = a (a+1) \cdots (a+\ell-1)$ is the (rising) Pochhammer symbol. The constant $\kappa_{d,\ell}$ is chosen such that $\kappa_{d,\ell} \| \ptf (\bw^{\otimes \ell}) \|_\frob = 1$ for all $\| \bw \|_2 = 1$ and satisfy $\kappa_{d,\ell} = 1 + \Theta_d(d^{-1/2})$. Then, 
\begin{equation}\label{eq:isometry_spherical_harmonics}
    \Phi_{d,\ell} : \TSym_{\ell} (\R^d) \to \sh_{d,\ell}, \qquad  \Phi_{d,\ell} (\bA) (\bz) = \< \bA, \cH_{d,\ell} (\bz) \>_\frob 
\end{equation}
is an isometry (see Lemma \ref{lem:isometry_spherical_harmonics} in Appendix \ref{app:isomorphism}), that is,
\begin{equation}\label{eq:Phi-d-ell-isometry}
    \< \Phi_{d,\ell} (\bA) , \Phi_{d,\ell} (\bB) \>_{L^2}  = \< \bA,\bB\>_\frob , \qquad \text{for all }\bA,\bB \in \TSym_{\ell} (\R^d).
\end{equation}
Furthermore, this mapping intertwines the tensor and function representations of $\cO_d$
\begin{equation}\label{eq:equiv-Phi-d-ell}
    \Phi_{d,\ell}(g\cdot\bA)=\rho(g)\cdot\Phi_{d,\ell}(\bA), \qquad \text{for all }\bA \in \TSym_\ell (\R^d) \text{ and }g \in \cO_d.
\end{equation}
Note that all the entries of $\cH_{d,\ell}(\bz)$ are in $\sh_{d,\ell}$ and \eqref{eq:Phi-d-ell-isometry} implies that 
\[
\E_{\tau_d}[\cH_{d,\ell} (\bz) \otimes \cH_{d,\ell} (\bz)] =\id_{\TSym_{\ell}(\R^d)} \in \Lop_{d,\ell}.
\]

\paragraph*{Harmonic decomposition.} Combining \eqref{eq:l2-decomposition-sphere} and the fact that \(\Phi_{d,\ell}\) is an isometric isomorphism (see Lemma \ref{lem:isometry_spherical_harmonics}), any $f \in L^2 (\S^{d-1})$ admits an expansion in terms of harmonic tensors (equality in $L^2$)
\begin{equation}\label{eq:harmonic_expansion_invariant_functions}
    f(\bz) = \sum_{\ell=0}^{\infty} \<\bA_{d,\ell}, \cH_{d,\ell}(\bz)\>_\frob  ,\quad \text{ where }\quad \bA_{d,\ell} = \E_{\bz\sim \tau_d}[f(\bz)\cH_{d,\ell}(\bz)] \in \TSym_\ell(\R^d).
\end{equation}
If, in addition, $f$ is $\cO_d^\bW$-invariant for some $\bW \in \Stf_\s (\R^d)$, that is $f$ only depends on $\bW^\sT \bz$, then one may write $\bA_{d,\ell} = \ptf (\bW^{\otimes \ell} \bB_{\s,\ell})$ for some $\bB_{\s,\ell} \in \Sym_\ell (\R^s)$ (see Lemma \ref{lem:harmonic_expansion_invariant_functions} in Appendix \ref{app:isomorphism}). An important example corresponds to $\s = 1$ and \textit{zonal harmonics}
\[
\< \bA_{d,\ell}, \cH_{d,\ell} (\bz) \>_\frob  =  b_{1,\ell} \< \ptf (\bw^{\otimes \ell} ), \cH_{d,\ell} (\bz) \>_\frob  = b_{1,\ell} \kappa_{d,\ell} Q_{\ell}^{(d)} (\<\bw,\bz\>),
\]
where $Q_\ell^{(d)}:[-1,1]\to 1$ are the (normalized) Gegenbauer polynomials with $\E_{\tilde{\tau}_{d,1}}[ Q^{(d)}_{\ell} (z_1) Q^{(d)}_{\ell} (z_1) ] = \delta_{\ell \ell'}$, where $z_1 \sim \tilde{\tau}_{d,1}$ is the marginal distribution of $\< \be_1,\bz\>$ under $\bz \sim \tau_d$.

In the case of $f \in L^2 (\cV \otimes \S^{d-1}) = \bigoplus_{\ell \geq 0} L^2 (\cV) \otimes \sh_{d,\ell}$, we decompose
\begin{equation}\label{eq:L2-decompo-product-space}
    f(\bv,\bz) = \sum_{\ell  =0}^\infty 
    \<\bA_{d,\ell} (\bv) , \cH_{d,\ell} (\bz)\>_\frob , \qquad \bA_{d,\ell} (\bv) = \E_{\bz \sim \tau_d} [ f(\bv,\bz) \cH_{d,\ell} (\bz) | \bv] \in \TSym_{\ell} (\R^d),
\end{equation}
where $\E_{\bv \sim \nu} [ \| \bA_{d,\ell} (\bv) \|_\frob^2] < \infty$.

\paragraph*{Tensor product representation.} We will also consider the tensor representation $\TSym_{p} (\R^d) \otimes \TSym_{q} (\R^d)$ with action $g\cdot (\bA\otimes \bB) = (g \cdot \bA)\otimes (g\cdot \bB)$, which admits the semisimple decomposition
\begin{equation}\label{eq:semisimple-decomposition-tensor-representation}
    \TSym_{p} (\R^d) \otimes \TSym_{q} (\R^d) \cong \bigoplus_{j = 0}^{\min(p,q)} \TSym_{p+q - 2j} (\R^d).
\end{equation}
In particular, we have the explicit harmonic decomposition of $\sh_{d,p}\otimes \sh_{d,q}$: for any $\bC \in   \TSym_{p} (\R^d) \otimes \TSym_{q} (\R^d)$,
\[
    \< \bC, \cH_{d,p} (\bz) \otimes \cH_{d,q} (\bz)\>_\frob  = \sum_{j = 0}^{\min(p,q)} b^{(d)}_{p,q,j}\< \proj_{\mathsf{tf},j}^{(p,q)} (\bC),\cH_{d,p+q - 2j} (\bz)\>_\frob , 
\]
where $b^{(d)}_{p,q,j} = \Theta_d (1)$ are related to the Clebsh-Gordan coefficients and $\proj_{\mathsf{tf},j}^{(p,q)}: \TSym_{p}(\R^d) \otimes \TSym_{q}(\R^d) \to \TSym_{p+q-2j} (\R^d)$ is the linear operator defined for all $\bA\in \TSym_{p} (\R^d)$, $\bB\in \TSym_{q} (\R^d)$ by
\begin{equation}\label{eq:definition_diamond_operator}
    \proj_{\mathsf{tf},j}^{(p,q)} (\bA \otimes \bB) = \bA \diamond_j \bB := \ptf ( \psym( \bA \otimes_j \bB))
\end{equation}
for $0 \leq j \leq \min(p,q)$ and \(\proj_{\mathsf{tf},j}^{(p,q)} (\bA \otimes \bB) = 0\) otherwise.

\paragraph*{Schur orthogonality relations and hypercontractivity.} Each $\sh_{d,\ell}$ (equivalently $\TSym_\ell(\R^d)$) forms an irreducible unitary representation of $\cO_d$. The functions
\begin{equation}\label{eq:matrix-coefficients}
g\mapsto \langle \rho(g)\cdot f,h\rangle_{L^2(\tau_d)}, \qquad f,h\in\sh_{d,\ell},
\end{equation}
are the \emph{matrix coefficients} of this representation, and satisfy the Schur orthogonality relations: for any $f,h \in \sh_{d,\ell}$ and $f',h' \in \sh_{d,k}$,
\begin{equation}\label{eq:schur-orthogonality}
\int_{\cO_d}
\langle \rho(g)\cdot f,h\rangle_{L^2}\,
\langle \rho(g)\cdot f',h'\rangle_{L^2}
\,\pi_d(\de g)
=
\frac{\delta_{\ell k}}{N_{d,\ell}}\,
\langle f,f'\rangle_{L^2}\,
\langle h,h'\rangle_{L^2}.
\end{equation}  
By unitary equivalence, the same relation holds when $f,h$ are replaced by traceless symmetric tensors in $\TSym_\ell(\R^d)$ (we have equality $\<\rho(g) \cdot \Phi_{d,\ell}(\bA), \Phi_{d,\ell} (\bB)\>_{L^2} = \<g \cdot \bA, \bB\>_\frob $).

Denote $\cM_{d,\ell} \subset L^2 (\cO_d)$ the span of all matrix coefficients \eqref{eq:matrix-coefficients} with $f,h \in \sh_{d,\ell}$. Using Gross' theorem, the heat semigroup on the connected component $\cSO_d$, and the Casimir eigenvalue associated to $\sh_{d,\ell}$, the subspace $\cM_{d,\ell}$ satisfies an hypercontractivity property (see Appendix \ref{app:hypercontractivity}): for any $d\geq 3$, any $F \in \cM_{d,\ell}$, and any $1 < p \leq q < \infty $, 
\begin{equation}\label{eq:hypercontractivity-ineq-O-d}
    \| F \|_{L^q (\cO_d)} \leq 2^{\frac{1}{p} - \frac{1}{q}} \left( \frac{q-1}{p-1}\right)^{\gamma_{d} (\ell)} \| F\|_{L^p (\cO_d)}, \qquad \gamma_d(\ell) = \frac{2 \ell (\ell + d - 2)}{d-2}.
\end{equation}

We further recall the following celebrated result on the hypercontractivity of polynomials on the sphere \cite{beckner1992sobolev}:
\begin{equation}\label{eq:hypercontractivity-ineq-sphere}
\| f\|_{L^q (\tau_d)} \leq \left(\frac{q-1}{p-1}\right)^{\ell/2} \| f \|_{L^p (\tau_d)},
\end{equation}
for all polynomials $f$ of degree at most $\ell$.

\section{Lower bounds and leap complexity}\label{sec:lower-bounds-leap}

We begin by establishing computational lower bounds for learning spherical multi-index models. As noted in the introduction, under mild assumptions these models are information-theoretically learnable with $\Theta_d( d)$ samples; see, e.g., \cite[Appendix H]{joshi2025learning}. However, it is conjectured that no polynomial-time algorithms can recover the latent subspace with this many samples for general link functions $\nu_d$ \cite{barbier2019optimal,damian2024computational,damian2025generative,diakonikolas2025algorithms,joshi2025learning}. That is, these models exhibit a so-called \textit{Statistical-Computational gap}. Ruling out all polynomial-time algorithms would require resolving $\texttt{NP} \neq \texttt{P}$, so the standard approach is to prove hardness within restricted, yet powerful, computational models. In this paper, we use the popular \textit{Statistical Query} (SQ) and \textit{Low-Degree Polynomial} (LDP) framework to derive such computational lower bounds for spherical MIMs.

\subsection{Statistical Query and Low-Degree frameworks}
\label{sec:LDP-SQ-background}

We briefly review the SQ and LDP frameworks and refer to \cite{feldman2017general,reyzin2020statistical,wein2025computational,misiakiewicz2025short} for a more comprehensive introduction. Both frameworks reduce estimation to a \textit{detection} problem.

Suppose we observe data generated under a family of distributions $\{ \P_{\btheta} : \btheta \in \bTheta\}$. We lower bound the complexity of estimating $\btheta$ (from $n$ samples) by the  complexity of the simpler task of distinguishing 
\begin{equation}\label{eq:detection-general}
        H_0: (\by_i,\bz_i)_{i\in[n]}\overset{\text{i.i.d.}}{\sim} \P_{0}
    \qquad \text{versus} \qquad
    H_1: (\by_i,\bz_i)_{i\in[n]} \overset{\text{i.i.d.}}{\sim} \P_\btheta \;\; \text{for some}\;\; \btheta \in \bTheta.
\end{equation}
for some null $\P_0$. %Since successful estimation implies successful detection (with the same sample size), any lower bound for detection immediately yields a lower bound for estimation. 
Although estimation may be strictly harder than detection, detection lower bounds provide a robust baseline.

\paragraph*{Statistical Query algorithms.} The SQ framework, introduced by Kearns \cite{kearns1998efficient}, models algorithms that interact with the data only through noisy estimates of expectations of query functions, rather than through direct sample access. A broad class of algorithms---including gradient-based methods, SVMs, and MCMC---can be implemented in this model. The SQ framework is generally viewed as capturing the power of noise-robust algorithms.

For a number of queries $q$ and tolerance $\tol >0$, an SQ algorithm $\cA \in \SQ(q,\tol)$ for the detection task \eqref{eq:detection-general} takes an input distribution $\P \in \{\P_0\} \cup \{ \P_{\btheta} : \btheta \in \bTheta\} $ and proceeds in $q$ rounds. At each round $t \in \{1,\ldots, q\}$, it issues a query $\phi_t : \cY \times \cZ \to [-1,1]$ and receives a response $v_t$ satisfying
\begin{equation}\label{eq:tolerance-SQ}
    | v_t - \E_{(\by,\bz) \sim \P} [ \phi(\by,\bz)] | \leq \tol. 
\end{equation}
The choice of $\phi_t$ can depend on previous responses $v_1,\ldots,v_{t-1}$. After $q$ rounds, the algorithm outputs either $0$ or $1$. We say that $\cA$ succeeds at distinguishing $\P_0$ from $\{ \P_{\btheta} : \btheta \in \bTheta\}$ if, for any valid responses $v_t$, $\cA$ outputs $0$ if $\P = \P_0$, and $1$ if $\P = \P_\btheta$ for some $\btheta \in \Omega$.

The oracle response \eqref{eq:tolerance-SQ} can be implemented using $n \asymp 1/\sqrt{\tol}$ samples, via an empirical average. Consequently, the query complexity $q/\tol^2$ is often used as a proxy lower bound for the runtime of the algorithm, under the heuristic that $\Omega (1/\tol^2)$ computation is required to simulate each query.

\paragraph*{Low-Degree Polynomial algorithms.} The LDP framework is a complementary approach for studying statistical–computational trade-offs in high-dimensional inference. Conceptually, it captures algorithms whose test statistics can be expressed as multivariate polynomials of bounded degree in the observed data.

Let $\btheta \sim \pi$ be a prior on the parameter space and consider the likelihood ratio
\begin{equation}
    \cR((\by_i,\bz_i)_{i \in [n]}) = \E_{\btheta \sim \pi} \left[ \prod_{i \in [n]} \frac{\de \P_\btheta}{\de \P_0} (\by_i,\bz_i) \right].
\end{equation}
Let $\proj_{\leq D}$ denote the orthogonal projection in $L^2 ( \P^{\otimes n}_0)$ onto the subspace of polynomials of degree at most $D$ in the covariates $\{\bz_i\}_{i \in [n]}$. The low-degree likelihood ratio is defined as
\begin{equation}
        \cR_{\leq D}((\by_i,\bz_i)_{i \in [n]}) = \proj_{\leq D} \cR((\by_i,\bz_i)_{i \in [n]}). 
\end{equation}
The degree $D$ is interpreted as a proxy for computational complexity: evaluating degree-$D$ polynomials in $O(n)$ variables takes time at most $n^{O(D)}$.  

The \textit{low-degree conjecture} \cite{hopkins2018statistical} asserts that if for $D = \omega_n (\log n)$, $\| \cR_{\leq D} \|_{L^2}^2 = 1 + o_n(1)$, then  no polynomial time algorithm can achieve weak detection (have non-vanishing advantage compared to random guessing). In other words, if all  degree-$\omega_n (\log n)$ polynomial tests fail, this is considered strong evidence that all polynomial-time algorithms fail for the problem \cite{wein2025computational}.

\subsection{Lower bound on weak recovery: Alignment complexity}

We begin by establishing computational lower bounds for the task of \emph{weak recovery}, namely estimating $\bW_*$ with accuracy better than random guessing. To capture this notion, we consider the detection problem
\begin{equation}\label{eq:hypothesis-testing-weak-recovery}
     H_0: (\by_i,\bz_i)_{i\in[n]}\overset{\text{i.i.d.}}{\sim} \P_{\nu_d,\emptyset}
    \quad \text{v.s.} \quad
    H_1: (\by_i,\bz_i)_{i\in[n]} \overset{\text{i.i.d.}}{\sim} \P_{\nu_d}^{\bW} \;\; \text{for some}\;\; \bW \in \Stf_\s (\R^d),
\end{equation}
where the null distribution $\P_{\nu_d,\emptyset} = \nu_d^Y \otimes \tau_d$ corresponds to the model in which the response $\by$ is independent of the input $\bz$. We also write $\overline{\nu}_{d,\emptyset} := \nu_d^Y \otimes \tilde{\tau}_{d,\s}$.  

Averaging over $\bW\sim\pi_{d,\s}$ and using the rotational invariance of $\tau_d$, we obtain 
\begin{equation}\label{eq:expectation-mixture-null}
    \P_{\nu_d,\emptyset} = \E_{\bW \sim \pi_{d,\s}}\left[\P_{\nu_d}^{\bW} \right] = \E_{g \sim \pi_d} \left[\P_{\nu_d}^{g \cdot \bW_*} \right] = \E_{g \sim \pi_d} \left[\rho(g) \cdot \P_{\nu_d}^{\bW_*} \right],
\end{equation}
and $\P_{\nu_d,\emptyset}$ is precisely the mixture of the planted model under a uniform prior on~$\bW$.

To apply the LDP framework, we assume that the model has finite chi-squared divergence with respect to the null (see Remark~\ref{rmk:chi-squared-bound}):
\begin{assumption}\label{ass:chi-squared-weak}
We have $\nu_d \ll \overline{\nu}_{d,\emptyset}$ and the Radon-Nikodym derivative satisfies $\frac{\de \nu_d}{\de \overline{\nu}_{d,\emptyset}} \in L^2 ( \overline{\nu}_{d,\emptyset})$.
\end{assumption}

Equivalently, the likelihood ratio
\begin{equation}
\frac{\de \P_{\nu_d}^\bW}{\de \P_{\nu_d,\emptyset}} \in L^2 ( \P_{\nu_d,\emptyset})
\end{equation}
has finite $L^2$-norm. This norm is independent of the particular choice of $\bW\in\Stf_{\s}(\R^d)$.  
Using the decomposition \eqref{eq:L2-decompo-product-space}, we expand the likelihood ratio as
\begin{equation}\label{eq:decompo-L2-likelihood-weak}
    \frac{\de \P_{\nu_d}^\bW}{\de \P_{\nu_d,\emptyset}}  (\by,\bz) = 1+ \sum_{\ell =1}^\infty \< \bxi_{\emptyset,\ell}^\bW (\by), \cH_{d,\ell} (\bz) \>_\frob , \qquad \bxi_{\emptyset,\ell}^\bW (\by) := \E_{\P_{\nu_d}^\bW} \left[ \cH_{d,\ell} (\bz) \mid \by\right].
\end{equation}
For each degree $\ell\ge 1$, define the second-moment operator
\begin{equation}
    \bGamma^\bW_{\emptyset,\ell} := \E_{\nu_{d,\emptyset}} [\bxi_{\emptyset,\ell}^\bW (\by) \otimes \bxi_{\emptyset,\ell}^\bW (\by) ] \in \Lop_{d,\ell}.
\end{equation}
We will often omit the superscript $\bW$ and simply write 
$\bxi_{\emptyset,\ell}$ and $\bGamma_{\emptyset,\ell}$. 
In particular, $\Tr(\bGamma_{\emptyset,\ell})$, 
$\|\bGamma_{\emptyset,\ell}\|_\op$, and $\|\bGamma_{\emptyset,\ell}\|_\frob$ do not depend on~$\bW$.

For weak recovery, we establish two complementary lower bounds:  
(i)~a query-complexity lower bound in the SQ framework (a proxy for algorithmic runtime), and  
(ii)~a sample-complexity lower bound in the LDP framework (a barrier for all polynomial-time algorithms).  
They are captured by the \emph{query-alignment} and the \emph{sample-alignment} complexity respectively:
\begin{equation}\label{eq:alignment-complexities}
    \qAlign (\nu_d \| \overline{\nu}_{d,\emptyset}) = \inf_{\ell \geq 1} \; \frac{N_{d,\ell}}{\| \bGamma_{\emptyset,\ell} \|_\op}, \qquad\quad
    \sAlign (\nu_d \| \overline{\nu}_{d,\emptyset}) = \inf_{\ell \geq 1} \; \frac{\sqrt{N_{d,\ell}}}{\| \bGamma_{\emptyset,\ell} \|_\frob}. 
\end{equation}
Further note that $\de \P_{\nu_d}^\bW/\de \P_{\nu_d,\emptyset}$ is $\cO_d^\bW$-invariant. Thus, by Lemma~\ref{lem:harmonic_expansion_invariant_functions} (Appendix \ref{app:tech-bg}), we can write the coefficients as $\bxi_{\emptyset,\ell} (\by) = \ptf (\bW^{\otimes \ell} \bzeta_{\emptyset,\ell} (\by))$. In particular, $\bxi_{\emptyset,\ell}$ has rank at most $\s^\ell$, and 
\[
\| \bGamma_{\emptyset,\ell} \|_\op =  \Theta_d ( \| \bxi_{\emptyset,\ell} \|_{L^2}^2 ) , \quad \|\bGamma_{\emptyset,\ell}\|_\frob = \Theta_d ( \| \bxi_{\emptyset,\ell} \|_{L^2}^2 ), \quad \text{where $\| \bxi_{\emptyset,\ell} \|_{L^2}^2 := \Tr( \bGamma_{\emptyset,\ell}) = \E [\| \bxi_{\emptyset,\ell} (\by)\|^2_\frob] $.}
\]
We always have $\| \bGamma_{\emptyset,\ell} \|_\op \leq 1$ and $\| \bGamma_{\emptyset,\ell} \|_\frob \leq \s^{\ell/2}$.

\begin{theorem}[Lower bounds on weak recovery of spherical MIMs]\label{thm:LB-weak-recovery}$\;$
    \begin{itemize}
        \item[(a)] \emph{(Query lower bound.)} Let $\nu_d$ be a spherical MIM satisfying Assumption \ref{ass:chi-squared-weak}. If an algorithm $\cA \in \SQ(q,\tol) $ succeeds at distinguishing $\{ \P_{\nu_d}^{\bW} : \bW \in \Stf_\s (\R^d) \}$ from $\P_{\nu_d,\emptyset}$, then 
        \begin{equation}\label{eq:sq-LB-weak}
            q/\tol^2 \geq \qAlign (\nu_d \| \overline{\nu}_{d,\emptyset}) .
        \end{equation}

        \item[(b)] \emph{(Sample lower bound.)} Let $\{\nu_d\}_{d \geq 1}$ be a sequence of spherical MIMs satisfying Assumption \ref{ass:chi-squared-weak}. Assume there exists $p \in \naturals$ such that $\sAlign (\nu_d \| \overline{\nu}_{d,\emptyset}) = O_d (d^{p/2})$. If for $D = o_d (d^{1/(p+2)})$, 
        \begin{equation}\label{eq:LDP-LB-weak}
            n = o_d \left( \frac{\sAlign (\nu_d \| \overline{\nu}_{d,\emptyset})}{D^{9p/2 - 1}}
            \right),
        \end{equation}
        then $\| \cR_{\leq D} \|_{L^2} = 1 + o_d(1)$ and under the low-degree conjecture, no polynomial time algorithm can achieve weak detection between $\{ \P_{\nu_d}^{\bW} : \bW \in \Stf_\s (\R^d) \}$ and $\P_{\nu_d,\emptyset}$.
    \end{itemize}
\end{theorem}

These lower bounds are obtained by applying the Schur orthogonality relations \eqref{eq:schur-orthogonality}, a standard second-moment argument (for SQ) and hypercontractivity of matrix coefficients \eqref{eq:hypercontractivity-ineq-O-d} (for LDP). The detailed proof of Theorem~\ref{thm:LB-weak-recovery} can be found in Appendix~\ref{app:proof-lower-bounds}. 

The lower bounds in Theorem~\ref{thm:LB-weak-recovery} decouple across the irreducible subspaces $\sh_{d,\ell}$ and admit a simple heuristic interpretation. Fix $\ell\ge 1$ and consider estimators whose statistics depend only on degree-$\ell$ spherical harmonics of $\bz$. For such estimators, the lower bounds reduce to $q/\tol^2 \gtrsim N_{d,\ell}/\|\bGamma_{\emptyset,\ell}\|_\op$ and $n \gtrsim \sqrt{N_{d,\ell}}/\|\bGamma_{\emptyset,\ell}\|_\frob$ (see \cite[Appendix~C]{joshi2025learning} for details).
 The global bounds \eqref{eq:alignment-complexities} are then obtained by taking the infimum (the best complexity) over $\ell\ge 1$. Thus, the lower bounds decompose the problem into separate detection subproblems, one for each harmonic subspace $\sh_{d,\ell}$, and each term in \eqref{eq:alignment-complexities} corresponds to the complexity of algorithms restricted to degree-$\ell$ spherical harmonics.

In Section~\ref{sec:tensor-unfolding}, we will construct algorithms that (nearly) attain these lower bounds for each subspace $\sh_{d,\ell}$. This suggests choosing the degree $\ell$ that achieves the minimum in $\qAlign$ (for runtime-optimal procedures) or in $\sAlign$ (for sample-optimal procedures); see equation \eqref{eq:intro_sample_runtime_optimal}.

\begin{remark}[Detection--recovery gap]\label{rmk:detection-recovery-gap}
The bounds in Theorem~\ref{thm:LB-weak-recovery} concern the detection task \eqref{eq:hypothesis-testing-weak-recovery}, rather than estimation directly. A detection–recovery gap arises in regimes where the infimum in \eqref{eq:alignment-complexities} is achieved at $\ell=1$ and $\sqrt{d}/\|\bGamma_{\emptyset,1}\|_\frob \ll d$.  
In this case the bounds are tight for detection, but estimation still requires $n = \Omega_d(d)$ samples information-theoretically. Moreover, any algorithm must incur runtime of at least $\Omega(dn)$ simply to process the $n$ samples.
\end{remark}

\begin{remark}[Finite chi-squared divergence]\label{rmk:chi-squared-bound}
Assumption~\ref{ass:chi-squared-weak} requires the likelihood ratio to be squared-integrable. This is essential for applying the LDP framework, and can be interpreted as assuming a sufficient amount of noise in the label. In practice, it can be enforced by adding a small Gaussian noise to the response. For SQ lower bounds, this assumption can be relaxed via a similar argument as in \cite{diakonikolas2025algorithms}. More broadly, Assumption~\ref{ass:chi-squared-weak} rules out non-robust algorithms that outperform the bounds \eqref{eq:alignment-complexities} in the noise-free setting (see, e.g., \cite{song2021cryptographic}). Further note that $\bGamma_{\emptyset,\ell}$ is always defined (with $\bxi_{\emptyset,\ell} (\by) := \E_{\P_{\nu_d}} [ \cH_{d,\ell}(\bz) |\by]$) even for infinite chi-squared divergence, and so are the alignment (and leap) complexities. 
\end{remark}

\begin{remark}[Additional trade-offs]\label{rmk:additional-trade-offs}
The LDP lower bound \eqref{eq:LDP-LB-weak} leaves room for potential trade-offs of order $D^{-\Theta(1)}$ in sample complexity by considering degree-$D$ algorithms (with runtime $d^{\widetilde{\Theta}(D)}$). Such trade-offs are known to be tight (under the low-degree conjecture) in tensor PCA \cite{wein2019kikuchi,kothari2025smooth,li2025smooth}. Exploring analogous trade-offs in the present setting is an interesting direction which we leave to future work.
\end{remark}

\subsection{Lower bound on strong recovery: Leap complexity}\label{sec:leap_complexity}

We now turn to the task of \emph{strong recovery}, namely recovering the full subspace $\bW_*$. As discussed in the introduction, optimal procedures will proceed in multiple steps and gradually recover $\bW_*$. We capture the complexity of this process by considering a family of \emph{partial} detection problems, corresponding to intermediate stages where only a strict subspace of $\bW_*$ has been recovered.

For simplicity and without loss of generality, we fix a reference frame $\bW_* \in \Stf_\s (\R^d)$. Define ${\rm Sub}(\bW_*)$ the set of all strict subframes of $\bW_*$, that is, all $\bU \in \Stf_r (\R^d)$ with $0 \leq r < \s$ and $\bU \subset \bW_*$ (i.e., $\Span (\bU) \subset \Span (\bW_*)$). Write $s_\bU := \mathrm{rank}(\bU)$ and $d_\bU := d - s_\bU$, and let $\bU_\perp \in \Stf_{d_\bU}(\R^{d})$ be an orthogonal complement of $\bU$ in $\R^d$.

% Define
% \[
% \Stf_{<} (\R^\s) = \bigcup_{r= 0}^{s-1} \Stf_r (\R^s),
% \]
% the set of strict subspaces of $\R^\s$. Fix a reference frame $\bW_*$ and, with a slight abuse of notation, identify $\bU \in \Stf_{<}(\R^\s)$ with its image $\bU := \bW_* \bU \in \Stf_{s_\bU}(\R^d)$ where $s_\bU := \mathrm{rank}(\bU)$. Write $d_\bU := d - s_\bU$ and let $\bU_\perp \in \Stf_{d_\bU}(\R^{d})$ be an orthogonal complement of $\bU$ in $\R^d$.

Consider data $(\by,\bz) \sim \P_{\nu_d}^{\bW_*}$, and suppose we have already recovered a subspace $\bU \in {\rm Sub}(\bW_*)$. We decompose the input as
\begin{equation}\label{eq:decomposition-input}
    \bz = \bU \br_\bU + \sqrt{1 - \| \br_\bU\|_2^2} \; \bU_\perp \bz_\bU, \qquad \br_\bU := \bU^\sT \bz \in \R^{s_\bU} , \qquad \bz_\bU = \frac{\bU_\perp^\sT \bz}{\sqrt{1 - \| \br_\bU\|_2^2}} \in \S^{d_\bU - 1},
\end{equation}
so that $(\br_\bU,\bz_\bU) \sim \tilde{\tau}_{d,s_\bU} \otimes \tau_{d_\bU}$. We then define a new spherical MIM $(\by_\bU,\bz_\bU) \sim \P_{\nu_{d,\bU}}$ by setting the new response $\by_\bU := (\by,\br_\bU)$ with conditional law
\begin{equation}\label{eq:new-MIM-when-conditioned-U}
\by_\bU | \bz_\bU \sim \nu_{d,\bU} ( \de \by_\bU | \bW_*^\sT \bU_\perp \bz_\bU) = \nu_d \big( \de \by \big| \bW_*^\sT\bz \big) \Tilde \tau_{d,s_{\bU}} (\de \br_{\bU}).
\end{equation}
We define the corresponding \emph{partial null} as 
\begin{equation}
    \P_{\nu_d,\bU} = \nu_{d,\bU}^Y \otimes \tau_{d_\bU} = \E_{g \sim \Unif (\cO^\bU_d)} \left[ \P_{\nu_d}^{g \cdot \bW_*}\right],
\end{equation}
where we recall that $\cO_d^\bU$ corresponds to the subgroup of $\cO_d$ that leaves $\bU$ invariant. We also write $\overline{\nu}_{d,\bU} := \nu_{d,\bU}^{Y} \otimes \tilde{\tau}_{d,\s-s_\bU}$, where $\nu_{d,\bU}^{Y}$ is the marginal distribution of $(\by,\bU^\sT \bz)$. 

We then consider the \emph{partial} detection problem
\begin{equation}\label{eq:hypothesis-testing-strong-recovery}
     H_0: (\by_{\bU,i},\bz_{\bU,i})_{i\in[n]}\overset{\text{i.i.d.}}{\sim} \P_{\nu_d,\bU}
    \quad \text{v.s.} \quad
    H_1: (\by_{\bU,i},\bz_{\bU,i})_{i\in[n]} \overset{\text{i.i.d.}}{\sim} \P_{\nu_d}^{g \cdot \bW} \;\; \text{for some}\;\; g \in \cO^{\bU}_d.
\end{equation}
In our lower bound, we require the algorithm to succeed for \emph{all} such partial problems indexed by $\bU\in {\rm Sub}(\bW_*)$. Equivalently, we consider the worst-case $\bU$ for which one must still decide whether $\by_\bU$ and $\bz_\bU$ are independent.

As before, we impose a finite chi-squared condition, now with respect to all partial nulls:
\begin{assumption}\label{ass:chi-squared-strong}
    For every $\bU\in {\rm Sub}(\bW_*)$, the Radon-Nikodym derivative satisfies $\frac{\de \nu_d}{\de \overline{\nu}_{d,\bU}} \in L^2 (\overline{\nu}_{d,\bU})$.
\end{assumption}

For $\bU \subset \bW$, we can expand the likelihood ratio as before
\begin{equation}
    \frac{\de \P_{\nu_d}^\bW}{\de \P_{\nu_d,\bU}}  (\by_\bU,\bz_{\bU}) = 1+ \sum_{\ell =1}^\infty \< \bxi_{\bU,\ell}^\bW (\by), \cH_{d_\bU,\ell} (\bz_\bU) \>_\frob , \qquad \bxi_{\bU,\ell}^\bW (\by) := \E_{\P_{\nu_d}^\bW} \left[ \cH_{d_\bU,\ell} (\bz_\bU) \mid \by_\bU\right],
\end{equation}
and define the second-moment operators 
\begin{equation}
    \bGamma^\bW_{\bU,\ell} := \E_{\nu_{d,\bU}} [\bxi_{\bU,\ell}^\bW (\by_\bU) \otimes \bxi_{\bU,\ell}^\bW (\by_\bU) ] \in \Lop_{d_{\bU},\ell}.
\end{equation}
We again omit the superscript $\bW$ and write $\bxi_{\bU,\ell}$ and $\bGamma_{\bU,\ell}$. In particular, $\Tr(\bGamma_{\bU,\ell})$, $\|\bGamma_{\bU,\ell}\|_\op$, and $\|\bGamma_{\bU,\ell}\|_\frob$ do not depend on the choice of the complement $\bW\bU_\perp$.

For strong recovery, we take the worst (i.e., most difficult) weak-recovery instance over all partial problems \eqref{eq:hypothesis-testing-strong-recovery} indexed by $\bU\in{\rm Sub}(\bW_*)$. Since each such instance is itself a weak-recovery problem for the reduced spherical MIM \eqref{eq:new-MIM-when-conditioned-U}, we may apply Theorem~\ref{thm:LB-weak-recovery} to each $\bU$ and then take a supremum. This leads to the \emph{query-leap} and \emph{sample-leap} complexities:
\begin{align}
    \qLeap (\nu_d) =&~ \sup_{\bU \in {\rm Sub}(\bW_*)} \;\qAlign (\nu_d \| \overline{\nu}_{d,\bU}) = \sup_{\bU \in {\rm Sub}(\bW_*)} \; \inf_{\ell \geq 1} \; \frac{N_{d_\bU,\ell}}{\|\bGamma_{\bU,\ell}\|_\op},\\
    \sLeap (\nu_d) =&~ \sup_{\bU \in {\rm Sub}(\bW_*)} \;\sAlign (\nu_d \| \overline{\nu}_{d,\bU}) = \sup_{\bU \in {\rm Sub}(\bW_*)} \; \inf_{\ell \geq 1} \; \frac{\sqrt{N_{d_\bU,\ell}}}{\|\bGamma_{\bU,\ell}\|_\frob}.
\end{align}
The following theorem is then a direct consequence of Theorem~\ref{thm:LB-weak-recovery}.

\begin{theorem}[Lower bounds on strong recovery of spherical MIMs]\label{thm:LB-strong-recovery}$\;$
    \begin{itemize}
        \item[(a)] \emph{(Query lower bound.)} Let $\nu_d$ be a spherical MIM satisfying Assumption \ref{ass:chi-squared-strong}. If an algorithm $\cA \in \SQ(q,\tol) $ succeeds at the strong recovery task \eqref{eq:hypothesis-testing-strong-recovery}, then 
        \begin{equation}\label{eq:sq-LB-strong}
            q/\tol^2 \geq \qLeap (\nu_d ) .
        \end{equation}

        \item[(b)] \emph{(Sample lower bound.)} Let $\{\nu_d\}_{d \geq 1}$ be a sequence of spherical MIMs satisfying Assumption \ref{ass:chi-squared-strong}. Assume there exists $p \in \naturals$ such that $\sLeap (\nu_d ) = O_d (d^{p/2})$. If for $D = o_d (d^{1/(p+2)})$, 
        \begin{equation}\label{eq:LDP-LB-strong}
            n = o_d \left( \frac{\sLeap (\nu_d )}{D^{9p/2 - 1}}
            \right),
        \end{equation}
        then $\| \cR_{\leq D} \|_{L^2} = 1 + o_d(1)$ and under the low-degree conjecture, no polynomial time algorithm can succeed at the strong recovery task \eqref{eq:hypothesis-testing-strong-recovery}.
    \end{itemize}
\end{theorem}

The leap complexities above motivate a natural \emph{sequential} recovery process.  Fix a sequence of harmonic degrees $\ell_1, \ell_2, \ldots \in \naturals$.  
At the first step, we apply an estimator based on degree-$\ell_1$ spherical harmonics.  
Suppose this recovers an $s_1$-dimensional subspace $  \bU_1 \in \Stf_{s_1}(\R^d)$ with $\bU_1 \subset \bW_*$. At the second step, we switch to a degree-$\ell_2$ harmonic estimator, now on input $\bz_{\bU_1}$ within the orthogonal complement $\bU_{1,\perp}$.  
This yields an additional block of new directions $\bU_2 \subset \bU_{1,\perp}$ (with $\bU_2 \in \Stf_{s_2}(\R^{d-s_1})$) corresponding to new components of $\bW_*$ orthogonal to $\bU_1$.
Iterating this procedure, we obtain a sequence of nested subspaces
\[
    \bU_{\le t} := \bU_1 \oplus \bU_2 \oplus \cdots \oplus \bU_t \in \Stf_{s_{\leq t} } (\R^d) \subseteq \bW_*,
    \qquad
    s_{\le t} := s_1 + \cdots + s_t,
\]
with the convention $s_0:=0$ and $\bU_0=\emptyset$. Since at least one new direction must be recovered at each stage, the process terminates after some $T \le \s$ steps when $\bU_{\le T}=\bW_*$.

The total cost of this sequential procedure is governed by the \emph{hardest} intermediate step.  
At step $t$, the relevant SQ and LDP complexities are $d^{\ell_t} / \|\bGamma_{\bU_{\le t},\ell_t}\|_\op$ and $d^{\ell_t/2}/\|\bGamma_{\bU_{\le t},\ell_t}\|_\frob$. Hence the total sequential complexity is
\begin{equation}\label{eq:cost-sequential-procedure}
    \max_{t\in[T-1]}
    \frac{d^{\ell_t}}{\|\bGamma_{\bU_{\le t},\ell_t}\|_\op},
    \qquad\qquad
    \max_{t\in[T-1]}
    \frac{d^{\ell_t/2}}{\|\bGamma_{\bU_{\le t},\ell_t}\|_\frob}
\end{equation}
for the SQ and LDP settings respectively.

The \emph{runtime-optimal} and \emph{sample-optimal} strategies choose harmonic degrees inductively by minimizing the corresponding intermediate costs:
\begin{equation}\label{eq:optimal-sequences}
\ell_{t+1}^{(q)} \in \argmin_{\ell \geq 1}    \frac{d^{\ell}}{\|\bGamma_{\bU^{(q)}_{\le t},\ell}\|_\op},\qquad \quad
    \ell_{t+1}^{(s)} 
    \in \argmin_{\ell \ge 1}
        \frac{d^{\ell/2}}{\|\bGamma_{\bU^{(s)}_{\le t},\ell}\|_\frob},
\end{equation}
where $\bU^{(q)}_{\le t}$ and $\bU^{(s)}_{\le t}$ denote the subspaces recovered along each procedure. By construction,
\[
    \max_{t\in[T^{(q)}-1]}
    \frac{d^{\ell^{(q)}_t}}{\|\bGamma_{\bU^{(q)}_{\le t},\ell^{(q)}_t}\|_\op} \lesssim \qLeap (\nu_d), 
    \qquad\qquad
    \max_{t\in[T^{(s)}-1]}
    \frac{d^{\ell^{(s)}_t/2}}{\|\bGamma_{\bU^{(s)}_{\le t},\ell_t}\|_\frob} \lesssim \sLeap (\nu_d)
\]
If the complexity of recovering $\bU^{(q)}_{t+1}$ or $\bU^{(s)}_{t+1}$ matches the SQ and LDP lower bound at each stage, then these sequential procedures achieve the corresponding leap complexities.

Many different degree sequences may achieve the same overall complexity.  
However, as illustrated in the examples (Section~\ref{sec:examples}), the runtime-optimal and sample-optimal sequences can lead to very different recovery trajectories.  
In such cases, one must choose whether to optimize for runtime or for sample size:  
our lower bounds indicate that \emph{no single sequential procedure is simultaneously optimal for both}.  
Intermediate trade-offs are also possible by selecting a degree sequence that interpolates between the two extremes.

\section{Iterative harmonic tensor unfolding algorithms}\label{sec:tensor-unfolding}

In this section, we present a family of sequential spectral algorithms for recovering the latent subspace $\bW_*$ using i.i.d.\ samples from a spherical multi-index model. These methods are based on iteratively applying a \emph{harmonic tensor unfolding} algorithm---with specified harmonic degree and kernel function---to the reduced spherical MIM obtained by conditioning on the signal subspace already recovered. The sequence of harmonic degrees used in the process can be chosen arbitrarily, or by following an \emph{optimal leap sequence}, based on either the sample-leap ($\sLeap$) or query-leap ($\qLeap$) complexities, as presented in Section \ref{sec:leap_complexity}.

We first describe and analyze a single step of harmonic tensor unfolding in Section~\ref{sec:one-step-HTU}.  
We then show in Section~\ref{sec:multi-step-HTU} how iterating this procedure leads to full recovery of the latent subspace $\bW_*$.

\subsection{One-step harmonic tensor unfolding}
\label{sec:one-step-HTU}

Let \(\{(\by_i,\bz_i)\}_{i=1}^n\) be i.i.d.\ samples drawn from a spherical multi-index model \(\P_{\nu_d}^{\bW}\) with link function $\nu_d$ and unknown support \(\bW\in \Stf_{\s}(\R^d)\). In this section, we describe a single step of the \emph{harmonic tensor unfolding} procedure at a fixed harmonic degree \(\ell \geq 1\), and establish guarantees for recovering a subspace \(\bU_0 \subseteq \bW\) of the latent signal directions. Throughout, we assume without loss of generality that $\| \bxi^\bW_{\emptyset,\ell} \|_{L^2} >0$; otherwise, no signal is present at harmonic degree \(\ell\) and the procedure does not identify any latent direction. For notational simplicity, we suppress the dependence on \(\bW\) in the remainder of this section and write \(\P_{\nu_d} := \P^{\bW}_{\nu_d}\) and \(\bxi_{\emptyset,\ell} := \bxi^\bW_{\emptyset,\ell}\).

For a tensor \(\bA \in (\R^d)^{\otimes k}\), we define its \emph{\((a,b)\)-unfolding}, for integers \(a,b \geq 0\) with \(a+b = k\), as the \(d^a \times d^b\) matrix \(\Mat_{a,b}(\bA)\) whose entries are given by
\begin{equation}
    (\Mat_{a,b}(\bA))_{(i_1,\ldots,i_a),(j_1,\ldots,j_b)} = A_{i_1,\ldots,i_a,j_1,\ldots,j_b},
\end{equation}
for all multi-indices \((i_1,\ldots,i_a) \in [d]^a\) and \((j_1,\ldots,j_b) \in [d]^b\). We also introduce a standard notion of distance between orthonormal frames \(\bU \in \R^{d \times k}\) and \(\bU' \in \R^{d \times k}\), defined as
\begin{equation}
    \dist(\bU,\bU') := \| \bPi_{\bU} - \bPi_{\bU'} \|_{\op},
\end{equation}
where \(\bPi_{\bU} = \bU \bU^\sT\) denotes the orthogonal projector onto \(\Span(\bU)\). Finally, for any matrix \(\bA = [\ba_1,\ldots,\ba_q] \in \R^{p \times q}\), we write $\Span(\bA) := \Span(\ba_1,\ldots,\ba_q) \subseteq \R^p$ for its column span.

\subsubsection{Preliminaries}
\label{sec:preliminaries_one_step}

We begin by characterizing the subspace \(\bU_0 \subseteq \bW\) of signal directions that can be recovered by a single step of the harmonic tensor unfolding algorithm. 

Recall that the \(\ell\)-th harmonic coefficient of the link function \(\nu_d\) is defined as
\[
\bxi_{\emptyset,\ell} (\by) := \E_{\P_{\nu_d}} \left[  \cH_\ell (\bz) | \by \right] \in \TSym_\ell (\R^d).
\]
By invariance of \(\P_{\nu_d}\) under orthogonal transformations that leave \(\bW\) fixed, the tensor \(\bxi_{\emptyset,\ell}(\by)\) is itself \(\cO_d^{\bW}\)-invariant for all \(\by \in \cY\). As a consequence, by Lemma~\ref{lem:harmonic_expansion_invariant_functions}, there exists a tensor-valued function \(\bzeta_{\emptyset,\ell}(\by) \in \Sym_\ell(\R^\s)\) such that
\begin{equation}
    \bxi_{\emptyset,\ell} (\by) = \ptf(\bW^{\otimes \ell}\bzeta_{\emptyset,\ell}(\by)).
\end{equation}
Let \(D_{\s,\ell} := \dim\!\big(\Sym_\ell(\R^\s)\big)\), and define
\begin{equation}\label{eq:eigendecomposition-Upsilon-ell}
\bUpsilon_{\emptyset, \ell} := \E[\bzeta_{\emptyset,\ell}(\by) \otimes \bzeta_{\emptyset,\ell}(\by)] \in \Sym_\ell (\R^\s) \otimes \Sym_\ell (\R^\s).
\end{equation}
This tensor can be naturally identified with a self-adjoint, positive semidefinite linear operator on \(\Sym_\ell(\R^\s)\). By the spectral theorem, there exists an orthonormal basis \(\{\bV_j\}_{j \in [D_{\s,\ell}]}\) of \(\Sym_\ell(\R^\s)\) and non-negative eigenvalues \(\mu_1 \geq \mu_2 \geq \cdots \geq 0\) such that
\begin{equation}\label{eq:spectral_decomposition_zeta}
\bUpsilon_{\emptyset,\ell}
=
\sum_{j=1}^{D_{\s,\ell}} \mu_j \, \bV_j \otimes \bV_j .
\end{equation}
For any $k \in [D_{\s,\ell}]$, define the signal tensor
\begin{equation}\label{eq:definition_signal_tensor}
    \bZ^{(k)} := \sum_{j =1}^k (\bW^{\otimes \ell}\bV_j) \otimes (\bW^{\otimes \ell}\bV_j) \in (\R^d)^{\otimes 2\ell}.
\end{equation}
Let \(\bU^{(k)} \in \R^{d \times \s^{(k)}}\) denote the matrix of left singular vectors of the unfolded matrix
\[
    \Mat_{1,2\ell-1}(\bZ^{(k)}) = \bW \Mat_{1,2\ell-1}\left(\sum_{j=1}^{k}\bV_j\otimes \bV_j\right)[\bW^{\otimes (2\ell-1)}]^\sT,
\]
so that $\bU^{(k)} \subseteq \bW$. Here, with a slight abuse of notation, $\bW^{\otimes (2\ell-1)}$ denotes the $d^{2\ell - 1} \times \s^{2\ell - 1}$ matrix with entries, indexed by multi-indices $(i_1,\ldots,i_{2\ell - 1}) \in [d]^{2\ell -1}$ and $(s_1,\ldots, s_{2\ell - 1}) \in [\s]^{2\ell -1}$, given by
\begin{equation}\label{eq:matrix_W_tensorialize_a_times}
(\bW^{\otimes (2\ell-1)})_{(i_1,\ldots,i_{2\ell - 1}),(s_1,\ldots, s_{2\ell - 1})} = W_{i_1s_1}W_{i_2 s_2} \ldots W_{i_{2\ell - 1} s_{2\ell - 1}}. 
\end{equation}
In particular, since $\bW^\sT \bW = \bI_\s$, we have $[\bW^{\otimes (2\ell-1)}]^\sT\bW^{\otimes (2\ell-1)} = \bI_{\s^\ell}$.

Our results show that when \(\mu_{k+1} \ll \mu_k\), the one-step estimator with sample size $n \asymp d^{\ell/2} / \mu_k$ recovers the subspace \(\bU^{(k)}\) of signal directions. Since our goal is weak recovery, we select \(k = \rnk\) as the smallest index for which a spectral gap occurs, so that $\mu_{\rnk} \asymp \mu_1 \asymp \| \bxi_{\emptyset,\ell} \|_{L^2}^2$ and $\mu_{\rnk+1} \ll \mu_{\rnk}$, and the corresponding sample complexity
\(n \asymp d^{\ell/2} / \| \bxi_{\emptyset,\ell} \|_{L^2}^2\)
matches the LDP lower bound. To simplify the exposition, we assume the following condition, which holds without loss of generality.

\begin{assumption}[Spectral gap $\gamma > 1$]\label{ass:spectral-gap}
There exist an integer \(\rnk \in[ D_{\s,\ell}]\) and a constant \(c_\mu \in (0,1)\) such that the eigenvalues \(\{\mu_j\}_{j \in [D_{\s,\ell}]}\) of the operator \(\bUpsilon_{\emptyset, \ell} \) satisfies (with convention $\mu_{D_{\s,\ell}+1} = 0$)
\begin{equation}\label{eq:spectral_gap}
    \mu_{\rnk} \geq c_\mu \mu_{1} \quad \text{ and }\quad \mu_{\rnk+1}\leq \frac{c_\mu \mu_{1}}{\gamma}.
\end{equation}
\end{assumption} 

Note that we could restate our sample guarantee (Theorem \ref{thm:tensor_unfolding_one_step}) without this assumption by replacing $\|\bxi_{\emptyset,\ell} \|_{L^2}^2$ and $\gamma$ in \eqref{eq:dist_guarantee_one_step} by $\mu_{\rnk}$ and $\mu_{\rnk}/\mu_{\rnk+1}$ respectively. For the choice of rank $\rnk$ in Assumption \ref{ass:spectral-gap}, we denote
\begin{equation}\label{eq:def_Z_0_U_0_s_0}
    \bZ_0 := \bZ^{(\rnk)}, \qquad \bU_0 := \bU^{(\rnk)}, \qquad \s_0 = {\rm rank} (\bU_0).
\end{equation}
In the remainder of this section, we describe and establish guarantees for recovering \(\bU_0\) via a single step of harmonic tensor unfolding.

%Note that we will always have $\s_0 >0$ if $\| \bxi_{\emptyset,\ell} \|_{L^2} >0$.

\subsubsection{Harmonic tensor unfolding algorithm}
\label{sec:harmonic-one-step-guarantees}

The one-step harmonic tensor unfolding algorithm depends on a choice of \emph{shape parameters} \((a,b)\) satisfying \(a+b=\ell\). These parameters determine how the \(\ell\)-th order harmonic tensor \(\cH_{d,\ell}(\bz)\) is unfolded into a matrix. Given a positive semidefinite kernel \(K : \cY \times \cY \to \R\), we form an empirical matrix \(\widehat{\bM} \in \R^{d^a \times d^a}\) as follows:
\begin{itemize}
    \item If $\ell$ is even and $a = b = \ell/2$,
    \[
    \widehat \bM = \frac{1}{n^2} \sum_{1 \leq i,j \leq n} K(\by_i,\by_j) \Mat_{a,b} (\cH_\ell (\bz_i)) \Mat_{a,b} (\cH_\ell (\bz_j))^\sT.
    \]
    \item If $\ell = 1$ and $(a,b) = (1,0)$, or $\ell \geq 3$ and $1 \leq a < b$ with $a+b=\ell$,
        \[
    \widehat \bM = \frac{1}{n(n-1)} \sum_{1 \leq i\neq j \leq n} K(\by_i,\by_j) \Mat_{a,b} (\cH_\ell (\bz_i)) \Mat_{a,b} (\cH_\ell (\bz_j))^\sT.
    \]
\end{itemize}

We now informally describe how a subspace
\(\widehat{\bU}_0 \in \R^{d \times \s_0}\)
can be estimated from the leading eigenvectors of \(\widehat{\bM}\), so that
\(\dist(\widehat{\bU}_0,\bU_0)=o(1)\),
where \(\bU_0 \subseteq \bW\) is the signal subspace introduced in
Section~\ref{sec:preliminaries_one_step}. For a suitable choice of kernel \(K\) (see Assumption~\ref{ass:kernel} below), one can show that
\begin{equation}\label{eq:approx_bMhat}
\widehat \bM \approx \| \bxi_{\emptyset,\ell} \|_{L^2}^2 \Mat_{a,2\ell - 1} (\bZ_0) \Mat_{a,2\ell - 1} (\bZ_0)^\sT + \bDelta,
\end{equation}
where the error term satisfies with high probability 
\[
\| \bDelta \|_\op \lesssim \frac{d^{\ell/2}}{n} + \frac{d^{\ell/4} \| \bxi_{\emptyset,\ell} \|_{L^2}}{\sqrt{n}}.
\]

Let $\thr = {\rm rank} ( \Mat_{a,2\ell - a} (\bZ_0))$, and let $\bV = [\bv_1, \ldots , \bv_{\thr} ] \in \R^{d^a \times \thr}$ denote the top $\thr$ left eigenvectors of $\Mat_{a,2\ell - 1} (\bZ_0) $. By \eqref{eq:approx_bMhat}, when $n \gtrsim d^{\ell/2} / \| \bxi_{\emptyset,\ell} \|_{L^2}^2$, computing 
\[
\widehat{\bV}
=
[\widehat{\bv}_1,\ldots,\widehat{\bv}_{\thr}]
\in \R^{d^a \times \thr}
\;\gets\;
\text{top \(\thr\) eigenvectors of \(\widehat{\bM}\)}
\]
yields ${\rm span} (\widehat{\bV}) \approx {\rm span} (\bV )$. To extract a subspace in $\R^d$, we further contract these eigenvectors and compute
\[
\widehat\bU_0 \gets \text{top $\s_0$ eigenvectors of } \sum_{s = 1}^{\thr} \Mat_{1,a-1} ( \widehat \bv_s)  \Mat_{1,a-1} ( \widehat \bv_s)^\sT.
\]
The resulting subspace $\Span (\widehat\bU_0)$ is close to the span of $\sum_{s = 1}^{\thr} \Mat_{1,a-1} ( \bv_s)  \Mat_{1,a-1} ( \bv_s)^\sT$ which coincides exactly with \(\Span(\bU_0)\).
The full procedure is summarized in
Algorithm~\ref{alg:tensor_unfold_one_step}.

\begin{algorithm}[t]
\caption{One-step Harmonic Tensor Unfolding with split $(a,b),\; a+b = \ell$}
\label{alg:tensor_unfold_one_step}

\SetKwFunction{TUstep}{TensorUnfoldOneStep}

\SetKwProg{Fn}{Function}{}{}
\Fn{\TUstep{\(\{(\by_i,\bz_i)\}_{i=1}^n\), degree \(\ell\), kernel \(K\), rank \(\thr\), rank $\s_0$}}{
    \tcp{Compute the empirical matrix $\widehat{\bM} \in \R^{d^a \times d^a}$}
    \eIf{$a=b$}{
    \(
    \widehat{\bM} \gets \frac{1}{n^2} \sum_{1 \leq i,j \leq n} K(\by_i,\by_j) \Mat_{a,a} (\cH_\ell (\bz_i)) \Mat_{a,a} (\cH_\ell (\bz_j))^\sT 
    \)
    }{
    \(
    \widehat{\bM} \gets \frac{1}{n(n-1)} \sum_{1 \leq i\neq j \leq n} K(\by_i,\by_j) \Mat_{a,b} (\cH_\ell (\bz_i)) \Mat_{a,b} (\cH_\ell (\bz_j))^\sT 
    \)
    }

    \tcp{Extract the top \(\thr\) eigenvectors}
    \(\displaystyle \widehat{\bV} = [\widehat \bv_1,\ldots, \widehat \bv_\thr] \gets \) top \(\thr\) eigenvectors of \(\widehat{\bM}\)
    \tcp*[r]{\(\widehat{\bV} \in \R^{d^{a}\times \thr}\)}

    \tcp{Contract to a \(d\)-dimensional subspace}
    \(\widehat{\bU}_0 \gets\) top $\s_0$ eigenvectors of \(\sum_{s=1}^{\thr}\Mat_{1,a-1}(\widehat \bv_s)\Mat_{1,a-1}(\widehat \bv_s)^\sT \) 
    \tcp*[r]{\(\widehat{\bU}_0 \in \R^{d\times \s_0}\)}

    \KwRet{\(\widehat{\bU}_0\)}
}
\end{algorithm}

\paragraph*{Sample complexity of the algorithm.} We assume that the kernel satisfies the following conditions.

\begin{assumption}[Admissible kernels]\label{ass:kernel}
A PSD kernel \(K:\cY\times \cY\to \R\) is admissible if it satisfies:
\begin{enumerate}[label=(K\arabic*), ref=(K\arabic*)]
    \item \label{ass:kernel-boundedness} \emph{(Boundedness)} There exists a constant \(B_K>0\) such that \(\| K\|_\infty \leq B_K\).
    
    \item \label{ass:kernel-positive-definite} \emph{(Spectral lower bound on signal directions)} There exists a constant \(c_K>0\) such that, 
    \begin{equation}\label{eq:kernel-positive-ass}
        \E\left[K(\by,\by')\< \bV,\bzeta_{\emptyset,\ell}(\by)\>_\frob \<  \bzeta_{\emptyset,\ell}(\by'), \bV\>_\frob  \right] 
         \geq c_K \E\left[\< \bV,\bzeta_{\emptyset,\ell}(\by)\>_\frob \<  \bzeta_{\emptyset,\ell}(\by), \bV\>_\frob  \right],
    \end{equation}
    for all $\bV \in \TSym_\ell (\R^s)$, where \((\by,\by')\sim \nu_d^Y \otimes \nu_d^Y\).

    \item \label{ass:finite_rank} \emph{(Finite rank)} There exists $\m \in \naturals$ such that $K$ is of rank $\m$, that is, there exist $\m$ functions $\cT_s : \cY \to \R$ such that
    \[
    K(\by,\by') = \sum_{s = 1}^\sm \cT_s (\by) \cT_s (\by').
    \]
\end{enumerate}
\end{assumption}

Assumption~\ref{ass:kernel}.\ref{ass:kernel-boundedness} is used to control the concentration of the empirical matrix \(\widehat{\bM}\) around its expectation.
Assumption~\ref{ass:kernel}.\ref{ass:kernel-positive-definite} ensures that the expectation \(\E[\widehat{\bM}]\) is positive-definite along the signal directions, allowing the algorithm to recover the signal subspace.
Finally, Assumption~\ref{ass:kernel}.\ref{ass:finite_rank} allows to match the runtime predicted by the statistical query framework.

\begin{remark}[Oracle kernel] 
Ideally, one would like to use the following oracle kernel defined as
\begin{equation}\label{eq:ideal-kernel}
K (\by,\by') := \blambda(\by)^\sT \E[\blambda(\by)\blambda(\by')^\sT ]^\dagger \blambda(\by'), \qquad \text{where}\;\;\; \blambda (\by) := \Mat_{\ell,0} (\bzeta_{\emptyset,\ell}) \in \R^{\s^\ell}
\end{equation}
With this choice, \eqref{eq:kernel-positive-ass} holds with equality with \(c_K = 1\), while the rank of \(K\) is at most \(\s^\ell\).
However, this kernel may not satisfy the boundedness condition
(Assumption~\ref{ass:kernel}.\ref{ass:kernel-boundedness}), which is the most restrictive requirement in our analysis. This condition can be relaxed to a suitable moment bound on \(K\), at the cost of an additional \(\polylog(d)\) factor in the sample complexity.
Moreover, in settings where \(\nu_d\) is not (or only partially) known, it may be preferable to consider dense kernels with \(\sm = \infty\).
Our analysis can be extended to this case with minor modifications (see Appendix \ref{app:proof-concentration-asym-one-step}).
We further comment on the setting where $\nu_d$ is unknown in Remark \ref{rmk:unknow-nu-d}.
\end{remark}

Under Assumption~\ref{ass:kernel}, we obtain the following guarantee for recovering the signal subspace \(\bU_0\) from i.i.d.\ samples drawn from \(\P_{\nu_d}\).

\begin{theorem}[One-step harmonic tensor unfolding]
\label{thm:tensor_unfolding_one_step}
    Let \((\nu_d,\bW)\) be a spherical MIM, and $\ell \geq 1$ be an integer with $\| \bxi_{ \emptyset,\ell} \|_{L^2} >0$. Under Assumption \ref{ass:spectral-gap} with gap $\gamma >1$ and Assumption \ref{ass:kernel}, there exist constants $C,C',c,c'>0$ that only depend on $\ell,\, \s$, and constants in these assumptions, such that the following holds. Let $\bU_0 \subseteq \bW$ be the signal subspace defined in \eqref{eq:def_Z_0_U_0_s_0}, and let $\widehat{\bU}_0 \in \Stf_{\s_0} (\R^d)$ be the output of Algorithm \ref{alg:tensor_unfold_one_step}. Then, for any 
    %$C' \leq d$, 
    $ C' \leq \gamma \leq \sqrt{d}$, and $n \leq \exp (d^{c'})$,
    \begin{equation}\label{eq:dist_guarantee_one_step}
        \dist (\widehat{\bU}_0, \bU_0) \leq C \left[ \sqrt{\frac{d^{\ell/2 \vee 1}}{n \| \bxi_{\emptyset,\ell} \|_{L^2}^2} } + \frac{1}{\sqrt{\gamma}}\right],
    \end{equation}
    with probability at least $1 - \exp (-d^{c})$.
\end{theorem}

The proof of Theorem~\ref{thm:tensor_unfolding_one_step} is given in
Appendix~\ref{app:asymmetric_one_step} for the asymmetric unfolding case (\(a \neq b\)),
and in Appendix~\ref{app:symmetric_one_step} for the symmetric case (\(a=b\)).
Taking \(\gamma\) sufficiently large, the bound~\eqref{eq:dist_guarantee_one_step} implies that a good approximation of \(\bU_0\) can be recovered with sample size
\[
n \asymp \frac{d^{\ell/2 \vee 1}}{\| \bxi_{\emptyset,\ell} \|_{L^2}^2}.
\]
For \(\ell \geq 2\), this matches the LDP lower bound~\eqref{eq:LDP-LB-weak}
for any choice of unfolding \(1 \leq a \leq b\).
For \(\ell=1\), the bound is worse by a factor \(\sqrt{d}\) compared to the lower bound for detection in Theorem~\ref{thm:LB-weak-recovery}(a).
However, as discussed in Remark~\ref{rmk:detection-recovery-gap}, this reflects an intrinsic detection-recovery gap in this regime.
When \(\| \bxi_{\emptyset,1} \|_{L^2}^2 \asymp 1\), the sample complexity \(n \asymp d\) is in fact information-theoretically optimal.
When \(\| \bxi_{\emptyset,1} \|_{L^2}^2 \ll 1\), following~\cite{joshi2025learning}, one can show that the algorithm recovers a subset of directions with overlap
\(\widetilde{\Theta}_d(d^{-1/4})\),
which can be boosted to \(\Theta_d(1)\) using higher-order harmonics under additional assumptions, thereby matching the detection lower bound.
Without such assumptions, determining the optimal sample complexity for recovery at \(\ell=1\) would require establishing LDP lower bounds for estimation (e.g., along the lines of~\cite{schramm2022computational,carpentier2025low}).
We leave this question to future work.
Finally, note that even if Assumption \ref{ass:spectral-gap} holds with $\gamma = \infty$, the guarantee \eqref{eq:dist_guarantee_one_step} in Theorem \ref{thm:tensor_unfolding_one_step} only yields $\dist (\widehat{\bU}_0,\bU_0) \lesssim d^{-1/4}$ as $n \to \infty$. To achieve arbitrary accuracy requires a second refinement step with a different algorithm.

\paragraph{Runtime of the algorithm.}
The dominant computational cost of Algorithm~\ref{alg:tensor_unfold_one_step} lies in computing the leading eigenvectors of the empirical matrix \(\widehat{\bM}\).
This is performed using subspace power iteration, and the total runtime is governed by the cost of multiplying vectors by \(\widehat{\bM}\).
Under the finite-rank assumption
(Assumption~\ref{ass:kernel}.\ref{ass:finite_rank}),
the matrix \(\widehat{\bM}\) admits the decomposition
\begin{equation}\label{eq:decompo_M_runtime}
\widehat{\bM} = \left(1 + \frac{\delta_{a \neq b}}{n-1} \right)\sum_{r = 1}^{\m} \left[  \bS_r \bS_r^\sT - \delta_{a \neq b} \bD_r \right],
\end{equation}
where 
\[
\bS_r := \frac{1}{n} \sum_{i \in [n]} \cT_r (\by_i) \Mat_{a,b} ( \cH_{d,\ell}(\bz_i)), \qquad \bD_r := \frac{1}{n^2} \sum_{i \in [n]} \cT_r (\by_i)^2 \Mat_{a,b} ( \cH_{d,\ell}(\bz_i))\Mat_{a,b} ( \cH_{d,\ell}(\bz_i))^\sT.
\]
Thus, multiplying a vector by \(\widehat{\bM}\) reduces to computing matrix-vector products involving \(\bS_r\) and \(\bD_r\), each of which requires summing \(n\) terms, instead of $n^2$.
Furthermore, by exploiting the structure of the harmonic tensor, these matrix-vector products can be computed in time
\(O(d^a + d^b)\),
without explicitly forming the full matrix \(\Mat_{a,b}(\cH_{d,\ell}(\bz))\), which would otherwise require \(O(d^{a+b})\) operations.
This leads to the following runtime guarantee.

\begin{proposition}[Runtime of harmonic tensor unfolding]\label{prop:runtime_tensor_unfolding_one_step}
There exists a constant $C>0$ that only depends on $\ell$ and $\s$ such that for all $\eps >0$, one can approximate the output $\widehat{\bU_0}$ of Algorithm \ref{alg:tensor_unfold_one_step} within precision $\eps>0$ (with respect to $\dist$) in time
\begin{equation}
    \rt \leq C  \frac{\m \thr }{g} n (\rt_K + d^a + d^b) \log (d/\eps),
\end{equation}
where $\rt_K$ denotes the cost of evaluating $\{ \cT_r \}_{r \in [\m]}$, and $g = 1- \hat \lambda_{\thr+1}/\hat \lambda_{\thr}$ with $\hat \lambda_j $ being the $j$-th eigenvalue of $\widehat{\bM}$. In particular, under the setting of Theorem \ref{thm:tensor_unfolding_one_step}, $g \geq 1 - c/\gamma$ with high probability. 
\end{proposition}

Further details of the runtime analysis are provided in Appendix~\ref{app:runtime-tensor-unfolding}.
While the sample complexity is independent of the unfolding shape for \(a \leq b\),
the runtime depends on this choice:
the cost is minimized when the unfolding is as close to square as possible.
Accordingly, we take \(a = \lfloor \ell/2 \rfloor\) and \(b = \lceil \ell/2 \rceil\), which corresponds to \(a=b=\ell/2\)
for even \(\ell\),
and to \((a,b)=((\ell-1)/2,(\ell+1)/2)\) for odd \(\ell\).

Combining Theorem~\ref{thm:tensor_unfolding_one_step} with
Proposition~\ref{prop:runtime_tensor_unfolding_one_step},
the signal subspace \(\bU_0\) can be recovered in time
\[
\rt \asymp \begin{cases}
d^2 \log (d)/ \| \bxi_{\emptyset,1} \|_{L^2}^2 & \text{if $\ell =1$};\\
d^\ell \log (d) / \| \bxi_{\emptyset,\ell} \|_{L^2}^2 & \text{if $\ell$ is even};\\
d^{\ell + \frac{1}{2}} \log (d) / \| \bxi_{\emptyset,\ell} \|_{L^2}^2& \text{if $\ell$ is odd}.
\end{cases}
\]
This matches the SQ runtime lower bound for even \(\ell\) (up to a logarithmic factor),
but is worse by a factor \(\sqrt{d}\) for odd\footnote{For $\ell = 1$, it is worse by a factor $d$. In the case $\| \bxi_{\emptyset,1} \|_{L^2} \asymp d$, this is optimal: it corresponds to the runtime to read $\Theta_d(d)$ samples in $\R^d$, which is information theoretically necessary.} \(\ell \geq 3\).
Whether this gap can be removed remains open.
In the special case of single-index models (\(\s=1\)),
\cite{joshi2025learning} showed that online SGD achieves runtime
\(d^\ell / \| \bxi_{\emptyset,\ell} \|_{L^2}^2\)
(without logarithmic factors) for all \(\ell \geq 3\),
albeit with a larger sample complexity
\(d^{\ell-1} / \| \bxi_{\emptyset,\ell} \|_{L^2}^2\).
We leave the question of achieving the optimal SQ runtime for odd \(\ell\) to future work.

\begin{remark}[Unknown $\nu_d$] \label{rmk:unknow-nu-d}
Algorithm~\ref{alg:tensor_unfold_one_step} relies on knowledge of \(\nu_d\) through the choice of the kernel \(K\), as well as the ranks \(\thr\) and \(\s_0\).
When \(\nu_d\) is only partially known, one may instead select a dense, bounded kernel adapted to a suitable class of link functions, so that
Assumption~\ref{ass:kernel}.\ref{ass:kernel-positive-definite} holds.
%(the sample complexity depends as $1/c_K^2$)
A natural example is a positive kernel such as a Gaussian kernel, as considered in~\cite{damian2025generative}.
We note, however, that the setting of~\cite{damian2025generative} differs from ours in that the labels do not depend on the ambient dimension, whereas in the present work the dependence on \(d\) introduces additional challenges.
We do not pursue this direction further here.
Such dense kernels generally do not admit a finite-rank representation, which precludes the efficient decomposition \eqref{eq:decompo_M_runtime} and do not satisfy our definition of admissibility (see Assumption \ref{ass:kernel}.\ref{ass:finite_rank}). Nonetheless, one may employ a random feature approximation to recover the computational efficiency nearly matching that of our SQ lower bounds.
Finally, when \(\thr\) and \(\s_0\) are unknown, these quantities can be selected adaptively, for instance by retaining only those eigenvalues that are sufficiently separated from the bulk.
\end{remark}

\subsection{Multi-step procedure}
\label{sec:multi-step-HTU}

We now show how iteratively applying the one-step harmonic tensor unfolding algorithm to a sequence of reduced multi-index models allows us to recover the entire support \(\bW\).

We first fix a sequence of harmonic degrees \(\ell_1,\ell_2,\ell_3,\ldots \geq 1\), and describe the algorithm associated with this choice.
We then discuss how to select the degree sequence so as to achieve optimal sample and runtime complexity (within the LDP and SQ frameworks).
The algorithm sequentially estimates a collection of subspaces
\(\{\widehat{\bU}_t\}_{t=1}^{T}\).
Writing \(\s_0=\s_{\leq 0}=0\) and \(d_0:=d\), and defining for \(t\geq 1\)
\[
\s_t := {\rm rank} (\widehat{\bU}_t), \qquad \s_{\leq t} := \s_1 + \ldots + \s_t, \qquad d_t := d - \s_{\leq t - 1},
\]
each subspace $\widehat{\bU}_t \in \Stf_{\s_t} ( \R^{d_t})$ is viewed as living in the orthogonal complement of the previously recovered directions $\widehat{\bU}_{\leq t-1}$ given by
\[
\widehat{\bU}_{\leq t-1} := \widehat{\bU}_1 \oplus \ldots \oplus \widehat{\bU}_{t-1} \in \Stf_{\s_{\leq t-1}} (\R^d). 
\]

To present the algorithm, we first define an idealized \emph{population} recovery sequence \(\{\bU_t\}_{t=1}^{T}\), and then explain how this sequence can be approximately recovered from data using harmonic tensor unfolding.

\paragraph*{Population recovery sequence.} Fix a spectral gap parameter \(\gamma>0\), and set \(\bU_0 := \emptyset\). We construct the sequence \(\{\bU_t\}_{t=1}^{T}\) recursively as follows.
Suppose that \(\{\bU_k\}_{k=0}^{t}\) have already been defined, with
\(\s_k := \rank(\bU_k)\). Define 
\[
\bU_{\leq t} := \bU_0 \oplus \ldots \oplus \bU_t \in \Stf_{\s_{\leq t}} (\R^d), \qquad \s_{\leq t} := \s_0 + \ldots + \s_t, \qquad d_{t+1} := d - \s_{\leq t}.
\]
We consider the reduced multi-index model obtained by conditioning on
\(\bU_{\leq t}^\sT \bz\), as described in~\eqref{eq:reduced-MIM-intro}.
Let \(\bU_{\leq t,\perp} \in \R^{d \times d_{t+1}}\) be an orthonormal basis for the orthogonal complement of \(\bU_{\leq t}\),
and decompose the data \((\by,\bz)\) as
\[
\br_t := \bU_{\leq t}^\sT \bz \in \R^{\s_{\leq t}}, \qquad \bz_t := \frac{\bU_{\leq t,\perp}^\sT \bz}{\| \bU_{\leq t,\perp}^\sT \bz\|_2} \in \S^{d_{t+1}-1}, \qquad \by_t := (\by,\br_t) \in \cY_t := \cY \times \R^{\s_{\leq t}}.
\]
Denote by \(\nu_{d,t} := \nu_{d,\bU_{\leq t}}\) the link function of the reduced MIM corresponding to this decomposition.
Conditionally on \(\bz_t\), the response satisfies
\[
\by_t | \bz_t \sim \nu_{d,t} ( \de \by_t | \bW^\sT \bU_{\leq t, \perp} \bz_t ) = \nu_d ( \de \by | \bW^\sT \bz) \tilde{\tau}_{d,\s_{\leq t}} (\de \br_t).
\]

We apply the degree-\(\ell_{t+1}\) harmonic tensor unfolding procedure to this reduced MIM, which has \(\s-\s_{\leq t}\) remaining indices.
Let \(\rnk_{t+1}\) denote the smallest rank satisfying Assumption~\ref{ass:spectral-gap} with gap \(\gamma\)
(in particular, \(c_\mu > \gamma^{-\rnk_{t+1}}\)).
We then define \(\bU_{t+1} \in \Stf_{\s_{t+1}}(\R^{d_{t+1}})\) to be the signal subspace constructed according to~\eqref{eq:def_Z_0_U_0_s_0} for the corresponding one-step procedure. We denote by \(\bZ_{t+1}\) the associated signal tensor, and write $\thr_{t+1} := {\rm rank} ( \Mat_{a,2\ell_{t+1} - a} (\bZ_{t+1}))$. The \(\ell_{t+1}\)-th harmonic coefficient of the reduced MIM is given by
\[
\bxi_{t} (\by_t) := \bxi_{\bU_{\leq t}, \ell_{t+1}} (\by)= \E [ \cH_{d_{t+1},\ell_{t+1}} (\bz_t) | \by_t] \in \TSym_{\s - \s_{\leq t}} (\R^{d_{t+1}}). 
\]
Without loss of generality, we assume that
\(\| \bxi_t \|_{L^2} > 0\) for all \(t=0,\ldots,T-1\), and we set $T$ such that \(\bU_{\leq T} = \bW\) (in particular, $T \leq \s$).
Finally, we assume the existence of a sequence of kernels
\(\{K_t\}_{t=1}^{T}\)
such that each
\(K_{t} : \cY_{t-1} \times \cY_{t-1} \to \R\)
satisfies Assumption~\ref{ass:kernel} for the pair \((\nu_{d,t-1},\ell_{t})\).

\paragraph*{Empirical recovery sequence.} We now describe the empirical multi-step procedure, which iteratively applies harmonic tensor unfolding with parameters chosen according to the population recovery sequence and using fresh samples at each step.

Initialize \(\widehat{\bU}_0 := \emptyset\).
For \(t=0,\ldots,T-1\), suppose that we have obtained estimates
\(\{\widehat{\bU}_k\}_{k=0}^{t}\),
with \(\rank(\widehat{\bU}_k)=\s_k\).
We then apply Algorithm~\ref{alg:tensor_unfold_one_step} to the reduced multi-index model $( \hat \by_t , \hat \bz_t) \sim \hat \nu_{d,t}$ conditioned on $\widehat{\bU}_{\leq t}$:
\[
\hat \nu_{d,t} := \nu_{d,\widehat{\bU}_{\leq t}}, \qquad \hat \by_t := (\by, \br_{\widehat{\bU}_{\leq t}}) , \qquad \hat \bz_t := \bz_{\widehat{\bU}_{\leq t}}.
\]
The algorithm is run at harmonic degree \(\ell_{t+1}\), with the optimal unfolding shape: \((a,b)=(1,0)\) if \(\ell_{t+1}=1\), and
\((a,b)=(\lfloor \ell_{t+1}/2\rfloor,\lceil \ell_{t+1}/2\rceil)\) otherwise.
We use the target ranks \(\thr_{t+1}\) and \(\s_{t+1}\), and the kernel
\begin{equation}\label{eq:symmetrized-kernel}
\overline{K}_{t+1} ((\by,\br),(\by',\br')) = \int_{\cO_{\s_{\leq t}}} K_{t+1} ( (\by,g\cdot \br), (\by',g\cdot \br')) \pi_{\s_{\leq t}} (\de g).
\end{equation}
This symmetrized kernel is used because the algorithm only recovers the span of
\(\bU_{\leq t+1}\), and not necessarily a specific basis. More precisely, the algorithm returns a subspace \(\widetilde{\bU}_{t+1}\subseteq\R^{d_{t+1}}\), which lies in the image of \(\bU_{\leq t,\perp}^{\top}\bW_\ast\). We lift this estimate back to the original \(d\)-dimensional space by taking \(\widehat{\bU}_{t+1}\) to be any orthonormal basis of the image of \(\bU_{\leq t,\perp}\,\widetilde{\bU}_{t+1}\subseteq\R^{d}\). Additional discussion can be found in Appendix \ref{app:multi-step-tensor-unfolding}.

The full multi-step procedure is summarized in
Algorithm~\ref{alg:multi_step_tensor_unfolding}.

\begin{algorithm}[t]
\caption{Multi-step Harmonic Tensor Unfolding for spherical MIMs}
\label{alg:multi_step_tensor_unfolding}
\SetKwFunction{MSTU}{MultiStepTensorUnfolding}
\SetKwProg{Fn}{Function}{}{}
\Fn{\MSTU{\(\{\{(\by^{(t)}_i,\bz^{(t)}_i)\}_{i=1}^{n}\}_{t=0}^{T-1}\), degrees \(\{\ell_t\}_{t= 1}^T\), kernels \(\{\overline{K}_t\}_{t= 1}^T\), ranks \(\{\thr_t\}_{t= 1}^T\), ranks \(\{\s_t\}_{t= 1}^T\)}}{ 
    \tcp{Initialize recovered subspace}
    \(\widehat{\bU}_0 \gets \emptyset\)

    \tcp{Iterate over the steps}
    \For{$t\leftarrow 0$ \KwTo $T-1$}{

        \tcp{Form the residual subspace}
        \(\widehat{\bU}_{\leq t} \gets \bigoplus_{j=0}^{t} \widehat{\bU}_j\) %, \(\hat{\bU}_{\leq t-1,\perp} \gets \text{ONB}(\Span(\hat{\bU}_{\leq t-1})^\perp)\)

        \tcp{Decompose batch $t$ of samples in the residual space}
        \(\{(\hat{\by}_{t,i},\hat{\bz}_{t,i})\}_{i=1}^{n} \gets\) decomposition of \(\{(\by^{(t)}_i,\bz^{(t)}_i)\}_{i=1}^{n}\) conditional on \(\widehat{\bU}_{\leq t}\)

        \tcp{Apply one-step tensor unfolding with $(a,b) = (1,0)$ or $(\lfloor \frac{\ell_{t+1}}{2} \rfloor,\lceil \frac{\ell_{t+1}}{2}\rceil)$}
            \(\widehat{\bU}_{t+1} \gets \TUstep(\{(\hat{\by}_{t,i},\hat{\bz}_{t,i})\}_{i=1}^n, \ell_{t+1},  \overline{K}_{t+1}, \thr_{t+1}, \s_{t+1})\)
    }
    \KwRet{\(\widehat{\bW} = \bigoplus_{j=1}^{T}\widehat{\bU}_j\)}
}
\end{algorithm}

\paragraph*{Recovery guarantees.} To establish that the above procedure indeed recovers the full signal subspace \(\bW\), one would like to apply Theorem~\ref{thm:tensor_unfolding_one_step} at each iteration.
However, this cannot be done directly beyond the first step.
Indeed, for \(t \geq 1\), we only have an approximate estimate of \(\bU_{\leq t}\), and the empirical reduced model
\(\widehat{\nu}_{d,t} := \nu_{d,\widehat{\bU}_{\leq t}}\)
differs from the population model
\(\nu_{d,t} := \nu_{d,\bU_{\leq t}}\).
In particular, since \(\widehat{\bU}_{\leq t}\) is not exactly a subspace of \(\bW\), the model
\(\widehat{\nu}_{d,t}\) almost surely still depends on \(\s\) indices.

To control the resulting propagation of error across iterations, we impose a stability condition on the family of reduced link functions \(\{\nu_{d,t}\}_{t=0}^{T-1}\).
Informally, this condition requires that the harmonic coefficients of the reduced models vary continuously with respect to perturbations of the conditioning subspace.

\begin{assumption}[Stability of reduced link function]\label{ass:stability-conditional-distributions}
    There exists a modulus of continuity $\varphi : [0,1] \to \R_{\geq 0}$, with $\varphi(\eps) \to 0$ as $\eps \to 0$, such that the following hold.
    \begin{itemize}
        \item[(a)] For all $t \in [T]$ and $\bU' \in \Stf_{\s_{\leq t-1}} (\R^d)$, letting $\bU'_\perp$ denote its orthogonal complement,
    \begin{equation}
        \| \bU_{\leq t-1,\perp}^{\otimes \ell_{t}} \bxi_{\bU_{\leq t-1},\ell_{t}} - (\bU'_\perp)^{\otimes \ell_{t}} \bxi_{\bU',\ell_{t}} \|_{L^2} \leq ( \| \bxi_{\bU_{\leq t-1},\ell_{t}} \|_{L^2} + \| \bxi_{\bU',\ell_{t}}\|_{L^2} ) \varphi \big( \dist ( \bU_{\leq t-1}  , \bU') \big).
    \end{equation}
    where $\bU_{\leq t, \perp}$ is the orthogonal complement of $\bU_{\leq t}$ in the population recovery sequence.

    \item[(b)] For all $t = 0, \ldots ,T-1$, and $\bU' \in \Stf_{\s_{\leq t}} (\R^d)$,
    \begin{equation}
        \E\left[ \left( K_{t+1} \left( \by_{t} , \by_{t}' \right) - K_{t+1}\left(\hat{\by}_{t},\hat{\by}_{t}'\right) \right)^2 \right]^{1/2} \leq \varphi ( \| \bU_{\leq t} - \bU' \|_\frob ), 
    \end{equation}
    where $\by_t = (\by, \bU_{\leq t}^\sT \bz)$, $\hat \by_t = ( \by, (\bU')^\sT \bz)$, and $(\by'_t, \hat \by_t ')$ denote independent copies.
    \end{itemize}
\end{assumption}

Under Assumption~\ref{ass:stability-conditional-distributions}.(a), the oracle kernel defined in~\eqref{eq:ideal-kernel} is itself stable with respect to perturbations of the conditioning subspace.
However, since the algorithm does not use this oracle kernel directly, we additionally impose a stability condition on the sequence of kernels \(\{K_t\}_{t=1}^{T}\) (Assumption~\ref{ass:stability-conditional-distributions}.(b)). While Assumption~\ref{ass:stability-conditional-distributions}
is somewhat restrictive, it is satisfied by important examples such as Gaussian multi-index models, as well as by models obtained from Gaussian MIMs after normalizing the input vector to have unit norm.

Under these conditions, we obtain the following recovery guarantee for the multi-step procedure.

\begin{theorem}[Multi-step harmonic tensor unfolding]
\label{thm:tensor_unfolding_mutli_step}
    Let \((\nu_d,\bW)\) be a spherical MIM, and let $(\ell_t)_{t\geq 1}$ be an harmonic degree sequence. Let $\{ \bU_t \}_{t= 1}^T$ be the associated population recovery sequence with spectral gap $\gamma >1$, and let $\widehat{\bW}$ be the output of Algorithm \ref{alg:multi_step_tensor_unfolding}. Assume that the kernel sequence $\{ K_t \}_{t = 1}^T$ satisfies Assumption \ref{ass:kernel}, and that Assumption \ref{ass:stability-conditional-distributions} holds. Then there exist constants $C,C',c,c'>0$ depending only on $\s$, the degree sequence, and the constants in these assumptions, such that for any 
    %$C'\leq d$, 
    $ C' \leq \gamma \leq  \sqrt{d}$, and $n \leq \exp (d^{c'})$,
    \begin{equation}\label{eq:dist_guarantee_multi_step}
        \dist (\widehat{\bW}, \bW) \leq C \left[ \max_{t \in [T]} \left( \frac{d^{\ell_t/2 \vee 1}}{n \| \bxi_{t-1} \|_{L^2}^2}\right)^{1/C}  + \frac{1}{\gamma^{1/C}} \right],
    \end{equation}
    with probability at least $1 - \exp (-d^{c})$.
\end{theorem}

The proof of Theorem \ref{thm:tensor_unfolding_mutli_step} can be found in Appendix \ref{app:multi-step-tensor-unfolding}. This theorem implies that, taking $\gamma$ sufficiently large, the sample complexity to recover the signal subspace $\bW$ with iterative tensor unfolding with degree sequence $\{ \ell_t \}_{t \geq 1}$ is
\[
n \asymp \max_{t \in [T]} \frac{d^{\ell_t/2 \vee 1}}{\| \bxi_{t-1} \|_{L^2}^2}.
\]
Regarding runtime complexity, while \(K_t\) is assumed to be finite rank, the symmetrized kernel \eqref{eq:symmetrized-kernel} does not necessarily preserve this low-rank structure. In Appendix \ref{app:multi-step-tensor-unfolding}, we propose a procedure for constructing a finite-rank approximation. Utilizing this approximation yields a total runtime complexity of
\[
\rt \asymp \max_{t \in [T]} \frac{d^{\ell_t \vee (3/2)  + \frac{1}{2}\delta_{\ell_t \equiv 1[2]} }}{\| \bxi_{t-1} \|_{L^2}^2} \log (d). 
\]

\paragraph*{Optimal multi-step procedures.}

A priori, Algorithm~\ref{alg:multi_step_tensor_unfolding} can be run with any choice of harmonic degree sequence \(\{\ell_t\}_{t=1}^{T}\). In light of the lower bounds established in Theorem~\ref{thm:LB-strong-recovery} and the discussion following it, the harmonic degrees can be chosen to minimize either the sample complexity or the runtime within the LDP and SQ frameworks, or to trade off between these two objectives.

In particular, consider the sequences of sample-optimal and runtime-optimal degrees \(\{\ell_t^{(s)}\}_{t=1}^{T_s}\) and \(\{\ell_t^{(q)}\}_{t=1}^{T_q}\) as defined in \eqref{eq:optimal-sequences}, together with their corresponding population recovery sequences \(\{\bU_t^{(s)}\}_{t=1}^{T_s}\) and \(\{\bU_t^{(q)}\}_{t=1}^{T_q}\).
By Theorem~\ref{thm:tensor_unfolding_mutli_step}, the multi-step harmonic tensor unfolding algorithm equipped with these two degree sequences achieves sample and runtime complexities
\[
    n \asymp  \max_{t \in [T_s]} \inf_{\ell \geq 1} \frac{d^{\ell/2 \vee 1}}{\| \bxi_{\bU^{(s)}_{\leq t-1},\ell} \|_{L^2}^2}, \qquad \quad
    \rt \asymp  \max_{t \in [T_q]} \inf_{\ell \geq 1} \frac{d^{\ell \vee (3/2)  + \frac{1}{2}\delta_{\ell} }}{\| \bxi_{\bU^{(q)}_{\leq t-1},\ell} \|_{L^2}^2} \log (d),
\]
respectively, where $\delta_\ell = 1$ if $\ell$ odd and $0$ otherwise. This matches the sample-leap and query-leap complexities in Theorem~\ref{thm:LB-strong-recovery}, up to possibly some $\widetilde O(\sqrt{d})$ factors. 

\section{Applications}
\label{sec:applications}

To illustrate our framework, we apply our results to several popular classes of multi-index models.

\subsection{Learning Gaussian MIMs}
\label{sec:applications-gaussian-mims}

 Consider a Gaussian MIM with link function $\rho \in \cP(\cY \times \R^\s)$ and support $\bW_\ast \in \Stf_{\s}(\R^d)$:
\begin{equation}\label{eq:gaussian-mim-application}
    (y,\bx) \sim \P_{\rho}^{\bW_*} : \qquad 
    \bx \sim \gamma_d, 
    \qquad 
    y \mid \bx \sim \rho(\,\cdot \mid \bW_*^\sT \bx).
\end{equation}
This setting has been extensively studied; see, e.g., 
\cite{dudeja2018learning,barbier2019optimal,mondelli2018fundamental,troiani2024fundamental,damian2024computational,damian2025generative,diakonikolas2025algorithms,diakonikolas2025robust}
and references therein. In particular, the optimal sample complexity for weak recovery (within the LDP framework) is
$n \asymp d^{k_*/2}$, where $k_*$ is the \emph{generative exponent}, defined as the index of the first non-zero coefficient in the Hermite expansion of $\rho$ \cite{damian2024computational}. As in our framework, recovery of the full subspace $\bW_*$ is then achieved via an iterative procedure that repeatedly conditions on the currently recovered subspace \cite{damian2025generative,diakonikolas2025algorithms,diakonikolas2025robust}.

We revisit this classical setting using our framework. As discussed in Example~\ref{ex:spherically_invariant_MIM}, the model can be rewritten as a spherical MIM using the polar decomposition $\bx = r \bz$, where $r = \|\bx\|_2 \sim \chi_d$ and $\bz = \bx / \|\bx\|_2 \sim \tau_d$ are independent. For clarity, we focus on the complexity of a single learning step; the same analysis extends naturally to the full iterative procedure. Precise statements, proofs, and additional details are deferred to Appendix~\ref{app:gaussian_directional_mims}.

Let 
\begin{equation}\label{eq:hermite-tensor-decompo-coeff}
\bpsi_{\emptyset, k} (y) := \E [ \He_k (\bx) \mid y ] \in \Sym_k (\R^d)
\end{equation}
denote the order-$k$ Hermite coefficient of $\rho$, where $\He_k (\bx)\in \Sym_k (\R^d)$ is the Hermite tensor (see \eqref{eq:Hermite-tensor} in Appendix \ref{app:gaussian_directional_mims}). The generative exponent is 
\begin{equation}
k_* = \argmin_{k \geq 1} \{ k \; : \; \| \bpsi_{\emptyset, k}  \|_{L^2(\rho^Y)}^2 > 0 \},
\end{equation}
where $\rho^Y$ denotes the marginal distribution of $y$. 
In contrast, Theorems~\ref{thm:LB-weak-recovery} and \ref{thm:tensor_unfolding_one_step} show that the computational complexity is governed by the harmonic decomposition, and in particular by
$\| \bxi_{\emptyset,\ell} \|_{L^2}$ for $\ell \geq 1$, where
\begin{equation}\label{eq:harmonic-tensor-decompo-coeff}
\bxi_{\emptyset,\ell} (y,r) := \E [ \cH_{d,\ell} (\bz) \mid y,r] \in \TSym_\ell (\R^d)
\end{equation}
are the harmonic coefficients. To relate harmonic and Hermite coefficients, we establish a decomposition of Hermite tensors into harmonic tensors: for all $\bA \in \Sym_k (\R^d)$,
\begin{equation}\label{eq:hermite-to-harmonic-decompo}
            \< \bA, \He_k(\bx) \>_{\frob} = \sum_{j=0}^{\lfloor k/2\rfloor} \beta^{(d)}_{k,k-2j}(r)\, \< \ptf \left( \tau^j (\bA)\right), \cH_{d,k-2j}(\bz) \>_{\frob},
\end{equation}
where $\| \beta^{(d)}_{\ell+2j,\ell}\|_{L^2}^2 \asymp d^{-j}$ (see Lemma~\ref{lem:Hermite-to-harmonic-tensors} in Appendix~\ref{app:gaussian_directional_mims} for explicit expressions). Under suitable regularity assumptions on $\rho$, this decomposition allows us to relate $\| \bxi_{\emptyset,\ell} \|_{L^2}$ to $\| \bpsi_{\emptyset,\ell} \|_{L^2}$.

First, $\| \bxi_{\emptyset, k_*} \|_{L^2}^2 \asymp \| \bpsi_{\emptyset,k_*} \|_{L^2}^2 \asymp 1$. Thus applying tensor unfolding at degree $\ell = k_*$ yields 
\begin{equation}\label{eq:Gaussian_MIMs-taking-k-star-degree}
    n  \asymp d^{k_\ast/2 \vee 1} , \qquad\qquad \rt \asymp d^{k_\ast \vee (3/2) + \delta_{k_\ast}/2} \log(d),
\end{equation}
matching the optimal sample complexity of \cite{damian2025generative}. The runtime bound follows from tensor unfolding with shape $(a,b) = (\lfloor k_\ast/2 \rfloor, \lceil k_\ast/2 \rceil)$ and improves upon \cite{damian2025generative} whenever $k_\ast > 2$. Indeed, \cite{damian2025generative} relies on a rectangular tensor unfolding that is effectively equivalent to choosing $(a,b) = (1,k_\ast-1)$ in our framework, leading to a runtime $\rt \asymp d^{3k_\ast/2 - 1} \log d$. Thus, while rectangular unfolding suffices for optimal sample complexity, more balanced reshaping is computationally more efficient.

While \eqref{eq:Gaussian_MIMs-taking-k-star-degree} achieves optimal sample complexity, degree $k_*$ need not be runtime-optimal. Indeed, by \eqref{eq:hermite-to-harmonic-decompo}, Hermite degree $k$ contributes to harmonic degrees $\ell < k$. Hence $\bxi_{\emptyset,\ell}$ may be non-zero for $\ell < k_*$ whenever partial traces of Hermite coefficients are non-zero. For example, let $j_*$ be the largest integer $0 \le j < \lceil k_*/2 \rceil$ such that 
$\| \tau^j (\bpsi_{\emptyset, k_*} )\|_{L^2} >0$. Then $\| \bxi_{\emptyset, k_* - 2j_*} \|_{L^2} \asymp d^{-j_*}$, and applying harmonic tensor unfolding at degree $\ell_* = k_* - 2j_*$ (assuming $\ell_* \ge 2$ for simplicity) gives
\[
n \asymp d^{k_\ast/2}, 
\qquad \qquad
\rt \asymp d^{k_\ast -j_* + \frac{1}{2}\delta_{\ell_\ast }} \log(d).
\]
Consider the following two examples introduced in Section \ref{sec:examples}:
\begin{itemize}
    \item For the $\s$-parity function \eqref{eq:parity_function_ex_intro}, one can show that for some $h:\cY \to \R$,
\[
\bpsi_{\emptyset,\s} (y) 
= h(y)\, \psym (\bw_{*,1} \otimes \cdots \otimes \bw_{*,\s}),
\]
and all partial traces vanish. In this case, $\ell=\s$ is both sample-optimal and runtime-optimal.

\item For the Gaussian SIM \eqref{eq:Gaussian-SIM-example} with generative exponent $k_*$,
\[
\bpsi_{\emptyset,k_*}  (y) = h(y)\, \bw_*^{\otimes k_*},
\qquad
\tau^j(\bpsi_{\emptyset,k_*} (y)) 
= h(y)\, \bw_*^{\otimes (k_* - 2j)}.
\]
Contracting as much as possible reduces the degree to $\ell_\ast = 1$ or $2$. The sample complexity remains $ n \asymp d^{k_\ast/2}$, while (with a modification of the algorithm for odd $k_*$ \cite{joshi2025learning})
\[
\rt \asymp d^{k_\star/2 +1 } \log (d).
\]
This recovers the partial-trace estimator of \cite{damian2024computational}.

\end{itemize}

In the above examples, one can attain near-optimal runtime without sacrificing optimal sample complexity. More generally, however, sample-runtime trade-offs may arise from coefficients $\bpsi_{\emptyset,k}$ with $k > k_*$. In such cases, one must accept worse sample complexity to achieve improved runtime, in contrast to Gaussian SIMs where both can always be simultaneously optimal. To illustrate this point, fix an integer $p \ge 1$ and consider a mixture of a $4p$-parity and a Gaussian SIM with generative exponent $6p$, with equal probability. The sample-optimal degree is $\ell = 4p$, thanks to the $4p$-parity component, yielding
\[
n \asymp d^{2p}, 
\qquad \qquad
\rt \asymp d^{4p} \log(d).
\]
The runtime-optimal degree is $\ell = 2$, thanks to the Gaussian SIM component, yielding
\[
n \asymp d^{3p}, 
\qquad \qquad
\rt \asymp d^{3p+1} \log(d).
\]
Appendix~\ref{app:gaussian_MIMs} provides a general characterization of sample- and runtime-optimal degrees and the associated complexities, in terms of the Hermite coefficients of $\rho$.

In summary, for Gaussian MIMs, while optimal sample complexity depends only on the generative exponent, computational complexity can depend on higher-order Hermite coefficients and their partial traces. This is naturally and succinctly captured by our harmonic-analytic framework.

\subsection{Learning directional MIMs}\label{sec:applications-directional-mims}

Consider the same Gaussian MIM \eqref{eq:gaussian-mim-application} as in the previous section, but suppose now that the radial component $r$ is not observed, and that only $(y,\bz)$ (the label and the direction of the input) are available. We refer to this model as a \emph{directional multi-index model}. This setting arises naturally as the common practice in statistics and machine learning to normalize input vectors to have constant norm. While the resulting model is no longer a Gaussian MIM, it remains a spherical MIM (on $\cY \times \S^{d-1}$) and can therefore be analyzed within our framework.

As in the previous section, we use the Hermite-to-harmonic decomposition. However, in the directional setting, the relevant quantity scales as $\E[\beta^{(d)}_{\ell+2j,\ell}(r)]^2 \asymp d^{-2j}$ (while $\E[\beta^{(d)}_{\ell+2j,\ell}(r)^2] \asymp d^{-j}$).
Consequently, choosing a harmonic degree $\ell < k_*$ always increases the sample complexity without improving the runtime. Thus, unlike in the Gaussian MIM setting, no trade-off between sample and runtime complexity arises in directional MIMs. The optimal choice is always $\ell = k_*$, and tensor unfolding at that degree yields
\[
n \asymp d^{k_*/2}, \qquad \qquad \rt \asymp d^{k_* + \frac{1}{2}\delta_{k_*}} \log (d).
\]

In particular, while the optimal sample complexity is unchanged, the runtime can be substantially worse when only directional information is available (i.e., when inputs are normalized). For instance, in Gaussian SIMs, the runtime increases from $d^{k_*/2+1}$ to $d^{k_*}$. This phenomenon was previously observed for Gaussian SIMs in \cite{joshi2025learning}, where it was also shown to have important consequences for gradient-based algorithms and algorithm design.

We refer to Appendix~\ref{app:directional_MIMs} for full statements and proofs.

\subsection{Other applications}

Another natural direction is to study classes of spherical MIMs with structured link functions. Consider inputs $\bx \sim \mu$ that are spherically invariant, with polar decomposition $\bx = r \bz$, where $r \sim \mu_r$ and $\bz \sim \tau_d$. Let $\bW_* \in \Stf_{\s}(\R^d)$ be the support, and consider responses of the form
\[
    y = f(\bW_*^\sT \bx) + \epsilon,
\]
where $\epsilon$ is independent noise and $f:\R^\s \to \R$ belongs to a prescribed function class. Relevant examples include polynomials, intersections of half-spaces, piecewise linear functions (e.g., ReLU networks or multiclass linear classification), multi-layer neural networks, and parity functions. In the Gaussian setting, learning such function classes has been studied in 
\cite{damian2025generative,diakonikolas2025algorithms,diakonikolas2025robust}
and references therein.

Our framework can yield analogous guarantees for general spherically invariant input distributions. As a toy case, one can directly analyze the harmonic decomposition of parity functions $f(\bt) = \sign(t_1 t_2 \cdots t_\s)$ (which does not depend on the radial component $r$), as discussed in Section~\ref{sec:examples}. Intersections of half-spaces, $f(\bt)=2\prod_{j=1}^\s \ind(\langle \ba_j,\bt\rangle \geq b_j)-1,$
can be studied (under suitable regularity assumptions on $\mu_r$) by adapting the arguments of 
\cite{damian2025generative,diakonikolas2025robust}. As another example, consider polynomial link functions
$f(\bt) = \sum_{|\alpha| \leq D} c_\alpha \bt^\alpha,$ for some fixed degree $D \in \N$. In Appendix~\ref{app:learning-polynomials-sphere}, we show that when $\mu_r = \delta_{\sqrt{d}}$, the following holds: for any recovered subspace $\bU$ of dimension $s_\bU < \s$, one has $\|\bxi_{\bU,2}\|_{L^2}^2 \asymp 1.$
Consequently, applying iterative tensor unfolding with degree $\ell=2$ at each step recovers the planted subspace with
\[
    n \asymp d,
    \qquad\qquad
    \rt \asymp d^{2}\log(d).
\]
Unsurprisingly, this matches the complexity in the Gaussian setting, where \cite{damian2025generative} showed that the generative leap complexity is at most $2$.

We leave the exploration of additional function classes to future work.

\section{Discussion}
\label{sec:discussion}

In this paper, we introduced \emph{spherical multi-index models}, a natural equivariant extension of the popular Gaussian multi-index model, and characterized the sample and runtime complexity of support recovery within the LDP and SQ frameworks. We further presented a family of iterative algorithms based on harmonic tensor unfolding that (nearly) match these lower bounds. In particular, these algorithms can realize different trade-offs between sample and runtime complexity by selecting the sequence of harmonic degrees used during recovery.  Our characterization relies on decomposing the learning problem into irreducible representations of $\cO_d$ and exploiting the intertwining isomorphism with traceless symmetric tensors. As an application, we revisit Gaussian MIMs: in addition to recovering the optimal sample complexity established in 
\cite{damian2025generative,diakonikolas2025algorithms}, we obtain improved runtime bounds. In particular, we exhibit algorithms whose query complexity closely tracks the SQ lower bounds, although in some regimes this may come at the expense of optimal sample complexity.

\paragraph*{Limitations.} Our work has several limitations. 
First, for simplicity, we assume that $\nu_d$ is fixed, known. While we expect that this assumption can be relaxed in our algorithm, we leave this extension to future work. 
Second, our lower bounds are stated for detection. In particular, for $\ell=1$, there may exist a detection--recovery gap. Addressing this would require proving a low-degree lower bound directly on estimation; recent progress in this direction includes \cite{schramm2022computational,carpentier2025low}. 
Third, our lower bounds rely on the low-degree conjecture and a heuristic correspondence with query complexity; see, e.g., \cite{wein2025computational,diakonikolas2025ptf,chen2025optimized}.

\paragraph{Future directions.} There are several natural directions for future work. These include:
(i) analyzing gradient-based methods for learning multi-index models (see, e.g., \cite{bietti2023learning,montanari2026phase}); 
(ii) characterizing finer sample--runtime trade-offs within a fixed harmonic subspace using higher-order tensor representations (see Remark~\ref{rmk:additional-trade-offs}); 
(iii) understanding learning when spherical symmetry is broken; and 
(iv) studying regimes where the rank $\s$ grows with $d$ (see, e.g., \cite{oko2024learning}).

Finally, our lower and upper bounds rely on general properties of the orthogonal group $\cO_d$ and its irreducible representations. Consequently, the same analysis applies directly to other $\cO_d$-equivariant models, including single- and multi-spike tensor PCA. In particular, tensor unfolding \cite{montanari2014statistical} and partial-trace estimator \cite{hopkins2016fast} can be interpreted as operating on specific irreducible components $\TSym_{\ell}(\R^d)$, with $\ell = k$ (the tensor order) and $\ell\in \{1,2\}$ respectively. More broadly, in a follow-up paper \cite{joshi2026equivariance}, we extend some of these ideas to general equivariant learning problem and compact group actions. This suggests the possibility of a systematic theory for designing statistically and computationally optimal equivariant learning algorithms based on representation-theoretic principles.

%\printbibliography
\bibliographystyle{amsalpha}
\bibliography{bibliography}

\clearpage

\appendix

\section{Proof of the lower bounds}
\label{app:proof-lower-bounds}

\begin{remark}
The operator norms $\|\bGamma_{\bU,\ell}\|_{\op}$ appearing in the query-alignment and query-leap complexities admit a natural variational interpretation: they characterize the optimal correlation between non-linear transformations of the generalized response $\by_{\bU}$ and degree-$\ell$ spherical harmonics of the input $\bz_{\bU}$,
    \[
    \| \bGamma_{\bU, \ell} \|_\op = \sup_{T \in L^2 (\nu_{d,\bU}^Y), \; \psi \in \sh_{d_\bU,\ell}} \frac{\E_{\P_{\nu_d}} [T(\by_\bU) \psi(\bz_\bU)]^2}{\|T \|_{L^2}^2 \| \psi \|_{L^2}^2}.
    \]
    Similar variational representations appear, for instance, in \cite[Proposition~A.1]{joshi2024complexity} for learning sparse functions, and \cite[Proposition~2.6]{damian2024computational} when learning Gaussian single-index models.
\end{remark}

\begin{proof}[Proof of Theorem \ref{thm:LB-weak-recovery}.(a)] The argument follows a standard second-moment method  \cite{joshi2024complexity,joshi2025learning,misiakiewicz2025short}. Fix $\cA \in \SQ(q,\tol^2)$ and denote $\phi_1, \ldots , \phi_q$ the sequence of queries issued by $\cA$ when it receives responses $v_t = \E_{\P_{\nu_d,\emptyset}} [ \phi_t], t\in [q]$. For these responses, the queries $\phi_t$ are fixed, deterministic, and are independent of $g \cdot \bW_* $ for $g \sim \pi_{d}$. By union bound and Markov's inequality, 
\[
\begin{aligned}
\P_{g \sim \pi_{d}}\left( \exists t \in [q], |\E_{\P_{\nu_d}^{g \cdot \bW_*}} [ \phi_t ] - v_t | >\tol\right) \leq &~ \frac{q}{\tol^2 } \cdot \sup_{t \in [q]} \Var_{g \sim \pi_d} \left\{\E_{\P_{\nu_d}^{g \cdot \bW_*}} [ \phi_t] \right\} \\
\leq&~ \frac{q}{\tol^2 } \cdot  \sup_{\| \phi\|_{L^2 ( \P_{\nu_d,\emptyset})} \leq 1} \Var_{g \sim \pi_d} \left\{\E_{\P_{\nu_d}^{g \cdot \bW_*}} [ \phi] \right\}, 
\end{aligned}
\]
where we used that $\| \phi_t \|_\infty \leq 1$. 
Thus, with positive probability over $g\sim\pi_d$, all responses remain $\tol$-consistent whenever
\begin{equation}\label{eq:LB-SQ-variance}
\frac{q}{\tol^2 } < \left[ \sup_{\| \phi\|_{L^2 ( \P_{\nu_d,\emptyset})} \leq 1} \Var_{g \sim \pi_d} \left\{\E_{\P_{\nu_d}^{g \cdot \bW_*}} [ \phi] \right\} \right]^{-1},
\end{equation}
in which case $\cA$ fails the detection task.

Let's compute the right-hand side of \eqref{eq:LB-SQ-variance}. Define
\[
\Delta_{\phi} (g) := \E_{\P_{\nu_d}^{g \cdot \bW_*}} \left[ \phi (\by,\bz) \right]- \E_{\P_{\nu_d,\emptyset}} \left[  \phi (\by,\bz)\right] = \E_{\P_{\nu_d,\emptyset}} \left[ \left(\rho (g) \cdot \frac{\de \P_{\nu_d}^{\bW_*} }{\de \P_{\nu_d,\emptyset}} (\by,\bz) - 1 \right)  \phi (\by,\bz)\right].
\]
Using the harmonic decomposition \eqref{eq:L2-decompo-product-space} and the equivariance \eqref{eq:equiv-Phi-d-ell}, 
\[
\begin{aligned}
\rho (g) \cdot \frac{\de \P_{\nu_d} }{\de \P_{\nu_d,\emptyset}} (\by,\bz) - 1 =&~ \sum_{\ell =1 }^\infty \< g \cdot \bxi_{\emptyset,\ell} (\by), \cH_{d,\ell} (\bz)\>_\frob ,
\end{aligned}
\]
where we suppressed the dependence on $\bW_*$ (which is fixed arbitrarily here). Hence
\[
\begin{aligned}
    \Delta_{\phi} (g) =  \sum_{\ell = 1}^\infty \E_{\nu_d^Y} \left[ \< g \cdot \bxi_{\emptyset, \ell} (\by) , \balpha (\by) \>_\frob \right], \qquad \balpha (\by) := \E_{\P_{\nu_d,\emptyset}} \left[\phi(\by,\bz) \cH_{d,\ell} (\bz) \mid \by \right].
\end{aligned}
\]
By Schur's orthogonality relations \eqref{eq:schur-orthogonality},
\[
\begin{aligned}
    \E_{g \sim \pi_{d}} \left[ |\Delta_{\phi} (g)|^2 \right] =&~ \sum_{\ell = 1}^\infty \frac{1}{N_{d,\ell}} \E_{\by_1,\by_2 \sim \nu_d^Y}\left[ \< \bxi_{\emptyset, \ell} (\by_1), \bxi_{\emptyset, \ell} (\by_2)\>_\frob  \< \balpha_\ell (\by_1),\balpha_\ell (\by_2) \>_\frob  \right] \\
    =&~ \sum_{\ell = 1}^\infty \frac{\| \E_{\by \sim \nu_d^Y} [ \bxi_{\emptyset,\ell} (\by) \otimes \balpha_\ell (\by)]\|_\frob^2}{N_{d,\ell}}.
\end{aligned}
\]
Therefore,
\[
\begin{aligned}
\sup_{\| \phi\|_{L^2 ( \P_{\nu_d,\emptyset})} \leq 1}  \Var_{g \sim \pi_d} \left\{\E_{\P_{\nu_d}^{g \cdot \bW_*}} [ \phi] \right\} = &~
\sup_{\| \phi\|_{L^2 ( \P_{\nu_d,\emptyset})} \leq 1} \E_{g \sim \pi_{d}} \left[ |\Delta_{\phi} (g)|^2 \right] \\
=&~ \sup_{\ell \geq 1}\; \frac{1}{N_{d,\ell}} \left\{ \sup_{\balpha_\ell : \cY \to \TSym_\ell (\R^d) } \frac{\| \E_{\nu_d^Y} [ \bxi_{\emptyset,\ell} (\by) \otimes \balpha_\ell (\by)]\|_\frob^2}{\E_{\nu_d^Y}[ \| \balpha_\ell (\by) \|_\frob^2] } \right\}.
\end{aligned}
\]
By a standard representer theorem, it suffices to take $\balpha(\by) = \cL [\bxi_{\emptyset,\ell} (\by)]$ for some $\cL \in \Lop_{d,\ell}$. Hence
\[
\begin{aligned}
    \sup_{\balpha_\ell : \cY \to \TSym_\ell (\R^d) }  \frac{\| \E [ \bxi_{\emptyset,\ell} \otimes \balpha_\ell]\|_\frob^2}{\E[ \| \balpha_\ell \|_\frob^2] } =&~ \sup_{\cL \in \Lop_{d,\ell}} \frac{\| \E [ \bxi_{\emptyset,\ell} \otimes \cL[ \bxi_{\emptyset,\ell}]]\|_\frob^2}{\E[ \| \cL[ \bxi_{\emptyset,\ell}] \|_\frob^2] } \\
    =&~ \sup_{\cL \in \Lop_{d,\ell}} \frac{\< \cL \cL^* , \bGamma_{\emptyset,\ell} \otimes_\ell \bGamma_{\emptyset,\ell}\>_\frob }{ \< \cL \cL^*, \bGamma_{\emptyset,\ell} \>_\frob } = \| \bGamma_{\emptyset,\ell} \|_\op,
\end{aligned}
\]
where we used $\bGamma_{\emptyset,\ell} = \E[ \bxi_{\emptyset,\ell} \otimes \bxi_{\emptyset,\ell}]$. Combining this with \eqref{eq:LB-SQ-variance} yields the desired SQ lower bound.
\end{proof}

\begin{proof}[Proof of Theorem \ref{thm:LB-weak-recovery}.(b)] The result follows by adapting the proof of \cite[Theorem 5]{joshi2025learning}. For convenience, set $\sA_\star := \sAlign (\nu_d \| \overline{\nu}_{d,\emptyset})$. Expand the degree-$D$ projection of the likelihood ratio as
    \[
    \begin{aligned}
       \cR_{\leq D} (\{\by_i,\bz_i\}_{i \in [m]} ) = &~ \sum_{\ell_1 + \ldots + \ell_m \leq D} \E_{g \sim \pi_d} \left[ \prod_{i \in [m]} \< \bxi_{\emptyset,\ell_i} (\by_i) , \rho(g) \cdot \cH_{d,\ell_i} (\bz_i) \>_\frob  \right],
    \end{aligned}
    \]
    so that
    \[
\begin{aligned}
       \| \cR_{\leq D}\|_{L^2}^2 =&~1 + \sum_{s =1}^D \binom{m}{s} \sum_{\substack{1 \leq \ell_1, \ldots , \ell_s \leq D\\
    \ell_1 + \ldots + \ell_s \leq D}} \E_{g \sim \pi_d} \left[ \prod_{i \in [s]}  F_{\ell_i} (g) \right],
    \end{aligned}
    \]
where we introduced the matrix-coefficient functions
\[
F_{\ell_i} (g) := \left\< \bGamma_{\emptyset,\ell_i}, \E_{\bz}\Big[ [\rho(g)\cdot \cH_{d,\ell_i} (\bz)] \otimes \cH_{d,\ell_i} (\bz) \Big]\right\>_\frob .
\]
Each $F_\ell$ lies in the subspace $\cM_{d,\ell}$ of degree-$\ell$ matrix coefficients. By hypercontractivity \eqref{eq:hypercontractivity-ineq-O-d} and Hölder’s inequality,
\[
\E_{g \sim \pi_d} \left[ \prod_{i \in [s]}  F_{\ell_i} (g) \right] \leq \prod_{i \in [s]} \| F_{\ell_i} \|_{L^s (\pi_d)} \leq  2^{\frac{s-2}{2}}  \prod_{i \in [s]} (s-1)^{\gamma_d (\ell_i)} \| F_{\ell_i} \|_{L^2}.
\]
Next note that
\[
F_{\ell} (g) = \E_{\by \sim \nu_d^Y}\left[\< g^{-1}  \cdot \bxi_{\emptyset, \ell} ( \by), \bxi_{\emptyset, \ell} ( \by)\>_\frob \right],
\]
so by Schur's orthogonality relations \eqref{eq:schur-orthogonality},
\[
\begin{aligned}
\| F_{\ell} \|_{L^2 (\pi_d)}^2 =&~ \E_{\by_1,\by_2 \sim \nu_d^Y} \left[ \E_g \left[ \< g  \cdot \bxi_{\emptyset, \ell} ( \by_1), \bxi_{\emptyset, \ell} ( \by_1)\>_\frob  \< g  \cdot \bxi_{\emptyset, \ell} ( \by_2), \bxi_{\emptyset, \ell} ( \by_2)\>_\frob \right] \right] \\
=&~ \frac{1}{N_{d,\ell}} \E_{\by_1,\by_2 \sim \nu_d^Y} \left[  \< \bxi_{\emptyset, \ell} ( \by_1), \bxi_{\emptyset, \ell} ( \by_2)\>_\frob ^2 \right] = \frac{\|\bGamma_{\emptyset, \ell} \|_\frob^2 }{N_{d,\ell}}. 
\end{aligned}
\]
Using $\gamma_d (\ell_i) \leq 4 \ell_i$ for $\ell_i \leq D \leq d-2$, we obtain
\[
\begin{aligned}
\|  \cR_{\leq D}\|_{L^2}^2 - 1 \leq&~ \sum_{s =1}^D \binom{m}{s} 2^{s/2} \sum_{\substack{1 \leq \ell_1, \ldots , \ell_s \leq D\\
    \ell_1 + \ldots + \ell_s \leq D}} \prod_{i \in [s]} s^{4 \ell_i} \frac{\|\bGamma_{\emptyset, \ell_i} \|_\frob}{\sqrt{N_{d,\ell_i}}} \leq \sum_{s =1}^D \binom{m}{s} \rho(s,D)^s,
\end{aligned}
\]
where
\[
\rho (s,D) := \sqrt{2} \sum_{\ell = 1}^D s^{4\ell} \frac{\|\bGamma_{\emptyset, \ell} \|_\frob}{\sqrt{N_{d,\ell}}}.
\]

By assumption 
\begin{equation}\label{eq:LDP-proof-bound-1-term}
\frac{\| \bGamma_{\emptyset,\ell} \|_{\frob}}{ \sqrt{N_{d,\ell}}} \leq \sA_\star\leq C d^{p/2}, \qquad \text{ for all $\ell >1$.}
\end{equation}
We also have
\[
\begin{aligned}
 \|  \bGamma_{\emptyset,\ell} \|_\op =&~ \sup_{\bA \in \TSym_{\ell} (\R^d)} \frac{\E[\< \bA,\bxi_{\emptyset,\ell} (\by) \>_\frob ^2 ] }{\| \bA \|_\frob^2} 
 \leq \sup_{\bA \in \TSym_{\ell} (\R^d)} \frac{\E[\< \bA,\cH_{d,\ell} (\bz) \>_\frob ^2 ] }{\| \bA \|_\frob^2}= 1,  
\end{aligned}
\]
where we used $\bxi_{\emptyset,\ell} (\by) = \E[\cH_{d,\ell} (\bz)|\by]$, Jensen's inequality, and the isometry \eqref{eq:Phi-d-ell-isometry} on the last equality. By Lemma~\ref{lem:harmonic_expansion_invariant_functions}, $\bxi_{\emptyset,\ell}$ is supported on a space of dimension at most $\s^\ell$. Thus, $ \| \bGamma_{\emptyset,\ell} \|_\frob \leq s^{\ell/2}$, which gives the alternative bound
\begin{equation}\label{eq:LDP-proof-bound-2-term}
\frac{\| \bGamma_{\emptyset,\ell} \|_{\frob}}{ \sqrt{N_{d,\ell}} }\leq C \frac{s^{\ell/2} d^{p/2}}{\sA_\star \sqrt{N_{d,\ell}}}.
\end{equation}
Splitting the sum between $\ell \leq p $ and $\ell >p$ and using the bounds \eqref{eq:LDP-proof-bound-1-term} and \eqref{eq:LDP-proof-bound-2-term},
\[
\rho(s,D) \leq \sqrt{2} \sum_{\ell =1}^p \frac{s^{4\ell}}{\sA_\star} + C \sum_{\ell =p+1}^D \frac{s^{9\ell/2} d^{p/2}}{\sA_\star \sqrt{N_{d,\ell}}} \leq \frac{s^{9p/2}}{\sA_\star} \left[ \sqrt{2} p + \sum_{\ell =p+1}^D \frac{D^{9(\ell - p)/2}d^{p/2}}{\sqrt{N_{d,\ell}}}  \right].
\]
For $\ell \leq D \leq \sqrt{d}$, we have $N_{d,\ell} \geq c (d/\ell)^\ell$ for some constant $c>0$ and 
\[
\sum_{\ell =p+1}^D \frac{D^{9(\ell - p)/2}d^{p/2}}{\sqrt{N_{d,\ell}}} \leq C \sum_{\ell = p+1}^D \ell^{p/2} (D^9 \ell /d)^{(\ell-p)/2}  \leq C'\frac{D^{p/2 +1}}{\sqrt{d}}.
\]
Assuming  $D = o_d (d^{1/(p + 2)})$, we deduce 
\[
\rho(s,D) \leq 2p \frac{s^{9p/2}}{\sA_\star}.
\]
Thus,
\[
\|  \cR_{\leq D}\|_{L^2}^2 - 1 \leq \sum_{s = 1}^D \binom{n}{s} \frac{s^{9sp/2}}{\sA_\star^s} (2p)^s \leq \sum_{s = 1}^D \left(2ep \frac{n D^{9p/2 -1}}{\sA_\star} \right)^s.
\]
Hence, if $n = o_d (\sA_\star/ D^{9p/2 - 1} )$, then $\|  \cR_{\leq D}\|_{L^2}^2 = 1 + o_d(1)$, completing the proof.
\end{proof}

\clearpage

\section{Analysis of the harmonic tensor unfolding algorithms}
\label{app:analysis_tensor_unfolding}

This section is devoted to the analysis of the iterative tensor unfolding algorithm described in Section~\ref{sec:tensor-unfolding}. We first consider one step of tensor unfolding (Algorithm \ref{alg:tensor_unfold_one_step}) and separate the analysis between the symmetric case $a=b$ and the asymmetric case $a \neq b$. We prove Theorem \ref{thm:tensor_unfolding_one_step} for the asymmetric case in Appendix \ref{app:asymmetric_one_step}, and for the symmetric case in Appendix \ref{app:symmetric_one_step}. We then study in Appendix \ref{app:multi-step-tensor-unfolding} the multi-step procedure, obtained by iteratively applying tensor unfolding to a sequence of reduced spherical MIM (Algorithm \ref{alg:multi_step_tensor_unfolding}), and prove Theorem \ref{thm:tensor_unfolding_mutli_step}. Finally, we analyze the runtime of these algorithms in Appendix \ref{app:runtime-tensor-unfolding}.

Throughout this section, we denote $c,c',C,C' >0$ general constants that only depend on $\s,\ell$, and constants in the assumptions. In particular, these constants are allowed to change from line to line. We will further denote $X \lesssim Y$ if there exists such a constant $C>0$ such that $X \leq C Y$.

\subsection{One-step of tensor unfolding: the asymmetric case}\label{app:asymmetric_one_step}

We first consider Algorithm \ref{alg:tensor_unfold_one_step} with asymmetric unfolding $a \neq b$, that is, $(a,b) = (1,0)$ if $\ell = 1$, or $1 \leq a < b$ with $a + b = \ell$ if $\ell \geq 3$.
Recall that we defined the empirical unfolded matrix
\[
\widehat{\bM} = \frac{1}{n(n-1)} \sum_{1 \leq i \neq j \leq n} K(\by_i,\by_j) \Mat_{a,b} (\cH_{d,\ell} (\bz_i)) \Mat_{a,b} (\cH_{d,\ell} (\bz_j))^\sT \in \R^{d^a \times d^a}.
\]

Our proof proceeds in three steps. We first control the top eigenvectors of $\E\big[\widehat{\bM}\big]$ (Lemma \ref{lem:spectral_bounds_population_U_stat}) and show that their contractions indeed recover the signal subspace $\bU_0 \subseteq \bW$ (Lemma \ref{lem:recovered_subspace_population_U_stat}). We then bound $\big\| \widehat{\bM} - \E\big[\widehat{\bM}\big] \big\|_\op$ (Lemma \ref{lem:concentration_U_statistic_final}). Finally, Theorem \ref{thm:tensor_unfolding_one_step} follows from a standard application of Davis-Kahan theorem. 

We start by analyzing the expectation $\E\big[\widehat{\bM}\big]$.
Let $\{ \mu_j\}_{j \in [D_{\s,\ell}]} $ and $\{\bV_j \}_{j \in [D_{\s,\ell}]}$ be the eigenvalues and eigenfunctions of $\bUpsilon_{\emptyset, \ell} := \E[\bzeta_{\emptyset,\ell}(\by) \otimes \bzeta_{\emptyset,\ell}(\by)]$ as defined in \eqref{eq:eigendecomposition-Upsilon-ell}. Let $\rnk \in [D_{\s,\ell}]$ be the rank in Assumption \ref{ass:spectral-gap}. Define the tensors
\begin{equation}\label{eq:decomposition_Z0_Zplus}
    \bZ_0 = \sum_{j=1}^\rnk  (\bW^{\otimes \ell}\bV_j) \otimes (\bW^{\otimes \ell}\bV_j)\quad \text{ and }\quad \bZ_+ = \sum_{j=\rnk+1}^{D_{\s,\ell}}  (\bW^{\otimes \ell}\bV_j) \otimes (\bW^{\otimes \ell}\bV_j),
\end{equation}
both belonging to \((\R^d)^{\otimes 2\ell}\). Let \(\bPi_0^{(a)}\) and \(\bPi_+^{(a)}\) denote the orthogonal projectors in $\R^{d^a}$ onto the ranges of \(\Mat_{a,2\ell-a}(\bZ_0)\) and \(\Mat_{a,2\ell-a}(\bZ_+)\), respectively. The next lemma controls the top eigenspaces of $\E\big[\widehat{\bM}\big]$ in terms of the projectors \(\bPi_0^{(a)}\) and \(\bPi_+^{(a)}\).

\begin{lemma}\label{lem:spectral_bounds_population_U_stat}
    Under Assumptions \ref{ass:spectral-gap}, \ref{ass:kernel}.\ref{ass:kernel-boundedness} and \ref{ass:kernel}.\ref{ass:kernel-positive-definite}, there exist constants \(c,C,C'>0\) that only depend on $\s$, $\ell$, and the constants in these assumptions, such that for all \(d\geq C'\) and spectral gap $\gamma >1$, \begin{equation}
    c\|\bxi_{\emptyset,\ell}\|_{L^2}^2 \bPi_0^{(a)} \preceq \E\big[\widehat{\bM}\big] \preceq C \|\bxi_{\emptyset,\ell} \|_{L^2}^2 \left(\bPi_0^{(a)} +  \frac{1}{\gamma}\bPi_+^{(a)} + \frac{1}{\sqrt{d}} \bI_{d^a}\right).
    \end{equation}
\end{lemma}

For large enough \(d\), this lemma shows that the $\thr := {\rm rank} ( \Mat_{a,2\ell - a} (\bZ_0))$ top eigenvectors of \(\E\big[\widehat{\bM}\big]\) span approximately the column span of $\Mat_{a,2\ell - a} (\bZ_0)$. To collapse this subspace back to the original space \(\R^d\), we contract each eigenvector with itself to form a $d \times d$ matrix. The next lemma shows that this procedure starting from $\Span (\Mat_{a,2\ell - a} (\bZ_0))$ indeed recover $\bU_0$, which is defined as the column span of $\Mat_{1,2\ell -1} (\bZ_0)$.

\begin{lemma}\label{lem:recovered_subspace_population_U_stat}
    Denote $\thr = {\rm rank} ( \Mat_{a,2\ell - a} (\bZ_0))$. Let $\bV = [\bv_1, \ldots , \bv_{\thr} ] \in \R^{d^a \times \thr}$ be an arbitrary orthonormal basis of the column span of $\Mat_{a,2\ell - a} (\bZ_0) $. Then,
    \begin{equation}
        \Span\left(\sum_{i=1}^{\thr}\Mat_{1,a-1}(\bv_i)\Mat_{1,a-1}(\bv_i)^\sT\right) = \Span(\Mat_{1,2\ell-1}(\bZ_0)) = \Span (\bU_0).
    \end{equation}
\end{lemma}

The proof of the above two lemmas can be found in Section \ref{app:proof_expectation_asymmetric_one_step} below.

\begin{lemma}[Concentration of $\widehat{\bM} - \E \widehat{\bM}$ with $a \neq b$]\label{lem:concentration_U_statistic_final}
    Under Assumption \ref{ass:kernel}.\ref{ass:kernel-boundedness} and \ref{ass:kernel}.\ref{ass:finite_rank}, there exist constants $c,c',C,C' >0$ that only depend on $\ell$, $\s$, and the constants in these assumptions such that for any \(n,d\) with $d \geq C'$ and \(d^{\ell/2}\leq n \leq \exp(d^{c'})\), we have  with probability at least \(1 - e^{-d^{c}}\),
    \begin{equation}
        \big\|\widehat{\bM} - \E\big[\widehat{\bM} \big] \big\|_\op \leq C \left[   \frac{d^{\ell/2 \vee 1}}{n} + \frac{d^{\ell/4 \vee 1/2}\|\bxi_{\emptyset,\ell}\|_{L^2}}{\sqrt{n}} \right].
    \end{equation}
\end{lemma}

The proof of this lemma can be found in Section \ref{app:proof-concentration-asym-one-step}. Note that a slightly lengthier argument allows to remove the finite rank Assumption \ref{ass:kernel}.\ref{ass:finite_rank} in Lemma \ref{lem:concentration_U_statistic_final} (and thus in Theorem \ref{thm:tensor_unfolding_one_step}). However, the above lemma will be sufficient for the purpose of this paper.

We are now ready to prove Theorem \ref{thm:tensor_unfolding_one_step} in the asymmetric case.

\begin{proof}[Proof of Theorem \ref{thm:tensor_unfolding_one_step} with $a \neq b$] Recall that \(\thr=\mathrm{rank}(\Mat_{a,2\ell-a}(\bZ_0))\). Let \(\hat{\bV}\in\R^{d^a\times\thr}\) be the top \(\thr\) eigenvectors of \(\widehat{\bM}\), and \(\bV\in\R^{d^a\times\thr}\) be an orthonormal basis of \(\Span(\bPi_0^{(a)})=\Span(\Mat_{a,2\ell-a}(\bZ_0))\). For any matrix \(\bA=[\ba_1,\ldots,\ba_\thr] \in \R^{d^a \times \thr}\), introduce 
    \[
    F(\bA):= \bigl[\Mat_{1,a-1}(\ba_1),\ldots,\Mat_{1,a-1}(\ba_\thr)\bigr] \in \R^{d \times (d^{a-1}\thr)}.
    \]
    By Lemma \ref{lem:recovered_subspace_population_U_stat}, \(\Span(F(\bV)F(\bV)^\sT)=\Span(\bU_0)\). Note that $F$ is a linear operator and an isometry with respect to the Frobenius norm. Thus, for any orthogonal  matrix \(\bQ\in\R^{\thr\times\thr}\),
    \[
    \begin{aligned}
    F(\bV)F(\bV)^\sT-F(\hat{\bV})F(\hat{\bV})^\sT =&~ F(\bV)F(\bV)^\sT-F(\hat{\bV} \bQ)F(\hat{\bV}\bQ)^\sT  \\
    =&~F(\bV - \hat{\bV}\bQ)F(\bV)^\sT + F(\hat{\bV}\bQ)F(\bV-\hat{\bV}\bQ)^\sT,
    \end{aligned}
    \]
    and since \(\|\bV\|_\frob=\|\hat{\bV}\|_\frob=\sqrt{\thr}\),
    \[
    \big\|F(\bV)F(\bV)^\sT-F(\hat{\bV})F(\hat{\bV})^\sT\big\|_\frob\leq 2\sqrt{\thr}\,\|\bV-\hat{\bV}\bQ\|_\frob.
    \]
    Taking the infinum over \(\bQ\in\cO_\thr\) and using the standard relation
    between Procrustes and projector distances,
    \[
    \inf_{\bQ\in\cO_\thr}\|\bV-\hat{\bV}\bQ\|_\frob \leq \sqrt{2\thr}\,\|\bV\bV^\sT-\hat{\bV}\hat{\bV}^\sT\|_\op,
    \]
    we conclude that
    \[
    \|F(\bV)F(\bV)^\sT-F(\hat{\bV})F(\hat{\bV})^\sT\|_\frob \leq 2\sqrt{2}\,\thr\,\|\bPi_0^{(a)}-\hat{\bV}\hat{\bV}^\sT\|_\op.
    \]
    
    Decompose \(\bPi_0^{(a)}-\hat{\bV}\hat{\bV}^\sT = \bDelta_1+\bDelta_2\) where
    \[
    \bDelta_1:=\bPi_0^{(a)}-\bPi_{\mathrm{pop}}^{(a)},
    \qquad
    \bDelta_2:=\bPi_{\mathrm{pop}}^{(a)}-\hat{\bV}\hat{\bV}^\sT,
    \]
    and \(\bPi_{\mathrm{pop}}^{(a)}\) is the projector onto the top
    \(\thr\) eigenspace of \(\E[\widehat{\bM}]\). We can write
    \[
    \|  \bDelta_1 \|_\op = \| ( \bI - \bPi_0^{(a)}) \bPi_{\mathrm{pop}}^{(a)} \|_\op, \qquad \|  \bDelta_2 \|_\op = \| ( \bI - \bPi_{\mathrm{pop}}^{(a)})  \hat{\bV} \|_\op,
    \]
    which we bound using Davis-Kahan theorem (e.g., see Wedin's sin theta theorem in~\cite[Theorem~VII.5.9]{Bhatia_1997}). For $\bDelta_1$, we compare the top $\thr$-eigenspaces of $\E[\widehat{\bM}]$ and
    \[
    \bB = \bPi_0^{(a)}\E[\widehat{\bM}]\bPi_0^{(a)} +\bPi_{0,\perp}^{(a)}\E[\widehat{\bM}]\bPi_{0,\perp}^{(a)},
    \]
    where $\bPi_{0,\perp}^{(a)} = \bI - \bPi_{0}^{(a)}$. 
    By Lemma \ref{lem:spectral_bounds_population_U_stat}, there exist $c,C>0$, such that, for all $d,\gamma > C$, the top $\thr$-eigenspaces of $\bB$ are exactly $\bPi_0^{(a)}$, and the top $\thr$ eigenvalues are separated to the rest of the eigenvalues by $c \| \bxi_{\emptyset,\ell} \|_{L^2}^2$. Furthermore,
    \[
    \begin{aligned}
          \|\E[\widehat{\bM}]-\bB\|_\op  \leq 2 \| \bPi_0^{(a)} \E[\widehat{\bM}] \bPi_{0,\perp}^{(a)} \|_\op \leq 2 \| \E[\widehat{\bM}] \|_\op^{1/2} \|\bPi_{0,\perp}^{(a)} \E[\widehat{\bM}] \bPi_{0,\perp}^{(a)}\|_\op^{1/2}
         \leq C  \frac{\|\bxi_{\emptyset,\ell}\|_{L^2}^2}{\sqrt{\gamma}}.
    \end{aligned}
    \]
    We deduce by Davis-Kahan theorem that
    \begin{equation}\label{eq:bound_Delta_1}
        \|\bDelta_1\|_\op  \leq C \frac{\|\E[\widehat{\bM}]-\bB\|_\op}{\|\bxi_{\emptyset,\ell}\|_{L^2}^2} \leq C' \frac{1}{\sqrt{\gamma}}.
    \end{equation}
   For $\bDelta_2$, by Lemma \ref{lem:concentration_U_statistic_final}, there exist $c,c',C,C'>0$ such that with probability at least $1 - e^{-d^{c}}$,
   \begin{equation}\label{eq:bound_Delta_2}
       \|\bDelta_2\|_\op \leq C \frac{\|\widehat{\bM}-\E[\widehat{\bM}]\|_\op}{\|\bxi_{\emptyset,\ell}\|_{L^2}^2} \leq \sqrt{\delta},
   \end{equation}
   where we denoted $\delta = C' d^{\ell/2 \vee 1}/ (n \| \bxi_{\emptyset,\ell} \|_{L^2}^2)$ and assumed $d \geq C'$, $n \leq \exp(d^{c'})$, and $\delta \leq c'$.

   Combining \eqref{eq:bound_Delta_1} and \eqref{eq:bound_Delta_2} yields
    \[
    \|F(\bV)F(\bV)^\sT-F(\hat{\bV})F(\hat{\bV})^\sT\|_\frob \leq C \left(\sqrt{\delta} + \frac{1}{\sqrt{\gamma}}\right).
    \]
    From \eqref{eq:eigenvalues_unfolded_Z_a} applied at $a = 1$, the non-zero eigenvalues of \(\sum_{i=1}^{\thr}\Mat_{1,a-1}(\bv_i)\Mat_{1,a-1}(\bv_i)^\sT\) are bounded away from zero, with constants that only depend on $\s,\ell$, and constants in the assumptions. Hence, another application of Davis-Kahan, with $\gamma \geq C$ and $\delta \leq 1/C$, implies that
    \[
        \dist\big(\widehat{\bU}_0,\bU_0\big) \leq C' \|F(\bV)F(\bV)^\sT-F(\hat{\bV})F(\hat{\bV})^\sT\|_\frob \leq C''\left(\sqrt{\delta} + \frac{1}{\sqrt{\gamma}}\right),
    \]
    where we used that $\Span (F(\bV)) = \Span (\bU_0)$ and $\widehat{\bU}_0$ are the top $\s_0$ eigenvectors of $F(\widehat{\bV})F(\widehat{\bV})^\sT$.   
\end{proof}

\subsubsection{Proof of Lemma \ref{lem:spectral_bounds_population_U_stat} and Lemma \ref{lem:recovered_subspace_population_U_stat}}\label{app:proof_expectation_asymmetric_one_step}

Before proving Lemma \ref{lem:spectral_bounds_population_U_stat}, we first show the following low-rank approximation of $\E\big[\widehat{\bM}\big]$:

\begin{lemma}\label{lem:population_unfolded_U_stat}
    Under Assumption \ref{ass:kernel}.\ref{ass:kernel-boundedness}, if \(d\geq 4\ell^4\s\), then
    \[
        \E\big[\widehat \bM\big] = (\bW^{\otimes a}) \E\left[K(\by,\by') \Mat_{a,b}\left(\bzeta_{\emptyset,\ell}(\by) \right) \Mat_{a,b}\left( \bzeta_{\emptyset,\ell}(\by')\right)^\sT\right](\bW^{\otimes a})^\sT + \bDelta_{d,\ell}
    \]
    where $\bW^{\otimes a} \in \R^{d^a \times \s^a}$ is the orthonormal $\s^a$-frame as defined in \eqref{eq:matrix_W_tensorialize_a_times} and
    \[
        \|\bDelta_{d,\ell}\|_\frob \leq  B_K \| \bxi_{\emptyset,\ell}\|_{L^2}^2 \frac{8\ell^2\sqrt{\s}}{\sqrt{d}}.
    \]
\end{lemma}

\begin{proof}
Since the samples \(\{(\by_i,\bz_i)\}_{i=1}^n\) are i.i.d.~and recalling the definition $\bxi_{\emptyset,\ell} (\by) = \E[\cH_{d,\ell} (\bz)|\by]$,
\begin{equation}\label{eq:expectation_Mhat_asym_PSD}
\begin{aligned}
    \E\big[\widehat \bM\big] =&~ \E\left[K(\by,\by') \Mat_{a,b} \big(\cH_{d,\ell}(\bz)\big) \Mat_{a,b} \big(\cH_{d,\ell}(\bz')\big)^\sT\right] \\
    =&~ \E\left[ K(\by,\by') \Mat_{a,b}\left(\bxi_{\emptyset,\ell}(\by) \right) \Mat_{a,b}\left(\bxi_{\emptyset,\ell}(\by') \right)^\sT \right],
\end{aligned}
\end{equation}
where \((\by,\bz),(\by',\bz')\sim \P_{\nu_d}\) independently. 
Using \(\bxi_{\emptyset,\ell}(\by) = \ptf(\bW^{\otimes \ell}\bzeta_{\emptyset,\ell}(\by))\) and Lemma~\ref{lem:approximation_traceless_projection}, we get whenever \(d\geq 4\ell^2\s\),
\begin{equation}\label{eq:approximation_xi_zeta}
    \big\|\bxi_{\emptyset,\ell}(\by) - \bW^{\otimes \ell}\bzeta_{\emptyset,\ell}(\by)\big\|_\frob \leq \|\bxi_{\emptyset,\ell}(\by)\|_\frob \frac{2\ell^2\sqrt{\s}}{\sqrt{d}}
\end{equation}
which implies that \(\|\bxi_{\emptyset,\ell} - \bW^{\otimes \ell}\bzeta_{\emptyset,\ell}\|_{L^2} \leq \|\bxi_{\emptyset,\ell}\|_{L^2} 2\ell^2\sqrt{\s/d}\) and, by the reverse triangle inequality,
\begin{equation}\label{eq:bound_wzeta_xi}
    \|\bW^{\otimes \ell}\bzeta_{\emptyset,\ell}\|_{L^2} \leq \|\bxi_{\emptyset,\ell}\|_{L^2} \left(1 + \frac{2\ell^2\sqrt{\s}}{\sqrt{d}}\right).
\end{equation}
Thus,
\begin{align*}
    \E\big[\widehat \bM\big] & = \E \left[K(\by,\by') \Mat_{a,b}\big(\bW^{\otimes \ell}\bzeta_{\emptyset,\ell}(\by) \big) \Mat_{a,b}\big(\bW^{\otimes \ell}\bzeta_{\emptyset,\ell}(\by') \big)^\sT\right] + \bDelta_{d,\ell},
\end{align*}
where by \eqref{eq:approximation_xi_zeta} and \eqref{eq:bound_wzeta_xi}, 
\[
    \|\bDelta_{d,\ell}\|_\frob \leq \|K\|_{\infty} \| \bxi_{\emptyset,\ell}\|_{L^2}^2 \frac{4\ell^2\sqrt{\s}}{\sqrt{d}} \left(1 + \frac{2\ell^2\sqrt{\s}}{\sqrt{d}}\right) \leq B_K  \| \bxi_{\emptyset,\ell}\|_{L^2}^2 \frac{8 \ell^2 \sqrt{\s}}{\sqrt{d}}.
\]
It only remains to observe that
\[
    \Mat_{a,b}\left(\bW^{\otimes \ell}\bzeta_{\emptyset,\ell}(\by)\right) = \bW^{\otimes a}\Mat_{a,b}\left(\bzeta_{\emptyset,\ell}(\by)\right)(\bW^{\otimes b})^\sT,
\]
and $(\bW^{\otimes b})^\sT (\bW^{\otimes b}) = \bI_{\s^b}$.
This concludes the proof.
\end{proof}

\begin{proof}[Proof of Lemma \ref{lem:spectral_bounds_population_U_stat}] First note that $\E[\widehat{\bM}]$ is PSD by \eqref{eq:expectation_Mhat_asym_PSD}. By Lemma \ref{lem:population_unfolded_U_stat}, for any $\bv \in \R^{d^a}$,
\begin{equation}\label{eq:quadratic_form_population_U_stat}
    \bv^\sT \E\big[\widehat \bM\big] \bv  = \bv_{\bW}^\sT \E\left[K(\by,\by') \Mat_{a,b}\left(\bzeta_{\emptyset,\ell}(\by) \right) \Mat_{a,b}\left( \bzeta_{\emptyset,\ell}(\by')\right)^\sT\right] \bv_{\bW} + \bv^\sT \bDelta_{d,\ell}\bv,
\end{equation}
where \(\bv_{\bW} := (\bW^{\otimes a})^\sT \bv\in \R^{\s^a}\) verifies $\| \bv_{\bW} \|_2 \leq \| \bv \|_2$, using that \(\bW^{\otimes a} \in \R^{d^a \times \s^a}\) is an orthonormal \(\s^a\)-frame, and
\begin{align*}
    |\bv^\sT \bDelta_{d,\ell}\bv| & \leq \|\bDelta_{d,\ell}\|_\frob \|\bv\|_2^2 \leq B_K \|\bxi_{\emptyset,\ell}\|_{L^2}^2 \frac{8\ell^2\sqrt{\s}}{\sqrt{d}} \| \bv \|_2^2.
\end{align*}
Let \(\bzeta_{\emptyset,\ell}(\by)[\bv_\bW]\) denote the contraction of the tensor \(\bzeta_{\emptyset,\ell}(\by)\) with the vector \(\bv_\bW \in \R^{\s^a}\) along the firsts \(a\) indices such that 
\[
\bv_{\bW}^\sT \E\left[K(\by,\by') \Mat_{a,b}\left(\bzeta_{\emptyset,\ell}(\by) \right) \Mat_{a,b}\left( \bzeta_{\emptyset,\ell}(\by')\right)^\sT\right] \bv_{\bW} = \E[K(\by,\by') \<\bzeta_{\emptyset,\ell}(\by)[\bv_{\bW}], \bzeta_{\emptyset,\ell}(\by')[\bv_{\bW}]\>_\frob].
\]
By Cauchy-Schwarz inequality and Assumption \ref{ass:kernel}.\ref{ass:kernel-boundedness},
\begin{equation}\label{eq:first_upper_bound_zeta}
    |\E[K(\by,\by') \<\bzeta_{\emptyset,\ell}(\by)[\bv_{\bW}], \bzeta_{\emptyset,\ell}(\by')[\bv_{\bW}]\>_\frob] | \leq B_K \|\bzeta_{\emptyset,\ell}[\bv_{\bW}]\|_{L^2}^2.
\end{equation}
On the other hand, by Assumption \ref{ass:kernel}.\ref{ass:kernel-positive-definite}, there exists \(c_K>0\) such that
\begin{equation}\label{eq:first_lower_bound_zeta}
    \E[K(\by,\by') \<\bzeta_{\emptyset,\ell}(\by)[\bv_{\bW}], \bzeta_{\emptyset,\ell}(\by')[\bv_{\bW}]\>_\frob] \geq c_K \|\bzeta_{\emptyset,\ell}[\bv_{\bW}]\|_{L^2}^2.
\end{equation}
It remains to characterize the subspace of \(\R^{\s^a}\) where the quantity \(\|\bzeta_{\emptyset,\ell}[\bv_{\bW}]\|_{L^2}^2\) is large.

First, note that we can decompose
\begin{equation}\label{eq:decomposition_zeta_v_W}
\|\bzeta_{\emptyset,\ell}[\bv_{\bW}]\|_{L^2}^2  = \sum_{j=1}^{D_{\s,\ell}} \mu_j \< \bv_{\bW}, \Mat_{a,b}(\bV_j) \Mat_{a,b}(\bV_j)^\sT \bv_{\bW} \>.
\end{equation}
Let \(\bZ_0\) and \(\bZ_+\) be defined as in~\eqref{eq:decomposition_Z0_Zplus}. We may write
\begin{align*}
    \bZ_0 \otimes_{2\ell-a} \bZ_0 & = \sum_{i,j=1}^{\rnk} \mu_i \mu_j \left( (\bW^{\otimes \ell}\bV_i) \otimes (\bW^{\otimes \ell}\bV_i) \right) \otimes_{2\ell - a} \left( (\bW^{\otimes \ell}\bV_j) \otimes (\bW^{\otimes \ell}\bV_j) \right).
\end{align*}
Since \(2\ell - a \geq \ell\), the above is contracting at least the \(\ell\) indices from $\bW^{\otimes \ell}\bV_i$ and $\bW^{\otimes \ell}\bV_j$. Given that the tensors \(\{\bW^{\otimes \ell}\bV_j\}_{j\in [D_{\s,\ell}]}\) form an orthonormal set in \(\Sym_{\ell}(\R^d)\), we get
\begin{align*}
    \bZ_0\otimes_{2\ell-a}\bZ_0 & = \sum_{i=1}^{\rnk} \mu_i^2  \left((\bW^{\otimes \ell}\bV_i) \otimes_{\ell-a} (\bW^{\otimes \ell}\bV_i)\right),
\end{align*}
and similarly for \(\bZ_+\) by summing \(j=\rnk+1,\ldots,D_{\s,\ell}\) instead. Introduce the matrices 
\[
\bP^{(a)}_0 := \sum_{j=1}^{\rnk}  \Mat_{a,b}(\bV_j) \Mat_{a,b}(\bV_j)^\sT, \qquad  \bP^{(a)}_+ := \sum_{j=\rnk+1}^{D_{\s,\ell}}  \Mat_{a,b}(\bV_j) \Mat_{a,b}(\bV_j)^\sT. 
\]
Note that the tensors $\{ \bV_i\}_{i \in [D_{\s,\ell}]}$ form an orthonormal set in $\Sym_\ell (\R^\s)$. By compactness, there exist constants $\lambda_{0,1}, \lambda_{0,2}, \lambda_{+,1}, \lambda_{+,2}>0$ that only depends on $(\s,\ell,\rnk)$, such that all non-zero eigenvalues of $\bP^{(a)}_0$ and $\bP^{(a)}_+$ are in $[\lambda_{0,1}, \lambda_{0,2}]$ and $[\lambda_{+,1}, \lambda_{+,2}]$ respectively. Thus, we deduce that
\begin{equation}\label{eq:eigenvalues_unfolded_Z_a}
\begin{aligned}
    \lambda_{0,1} \bPi_0^{(a)} \preceq &~\Mat_{a,2\ell-a}(\bZ_0)\Mat_{a,2\ell-a}(\bZ_0)^\sT = (\bW^{\otimes a}) \bP^{(a)}_0 (\bW^{\otimes a})^\sT \preceq  \lambda_{0,2} \bPi_0^{(a)}, \\
    \lambda_{+,1} \bPi_+^{(a)} \preceq &~\Mat_{a,2\ell-a}(\bZ_+)\Mat_{a,2\ell-a}(\bZ_+)^\sT = (\bW^{\otimes a}) \bP^{(a)}_+ (\bW^{\otimes a})^\sT \preceq  \lambda_{+,2} \bPi_+^{(a)}.
\end{aligned}
\end{equation}

From \eqref{eq:decomposition_zeta_v_W}, we deduce that
\begin{align*}
    \|\bzeta_{\emptyset,\ell}[\bv_{\bW}]\|_{L^2}^2  \leq&~ \mu_1 \< \bv, (\bW^{\otimes a}) \bP^{(a)}_0 (\bW^{\otimes a})^\sT  \bv \> + \mu_{\rnk +1} \< \bv, (\bW^{\otimes a}) \bP^{(a)}_+ (\bW^{\otimes a})^\sT  \bv \> \\
    \leq&~ \left\< \bv , \left[  \mu_1 \lambda_{0,2} \bPi_0^{(a)} + \frac{c_\mu  \mu_1 \lambda_{+,2}}{\gamma} \bPi_+^{(a)}\right] \bv \right\>,
\end{align*}
and
\begin{align*}
     \|\bzeta_{\emptyset,\ell}[\bv_{\bW}]\|_{L^2}^2   \geq \mu_\rnk \< \bv, (\bW^{\otimes a}) \bP^{(a)}_0 (\bW^{\otimes a})^\sT  \bv \>  \geq \lambda_{0,1}c_\mu \mu_1 \< \bv, \bPi_0^{(a)} \bv \>,
\end{align*}
where we used Assumption \ref{ass:spectral-gap}. Combining these bounds with \eqref{eq:first_upper_bound_zeta},~\eqref{eq:first_lower_bound_zeta}, and \eqref{eq:quadratic_form_population_U_stat},  and the observation that \(\mu_1 \leq \|\E[\bzeta_{\emptyset,\ell}(\by) \otimes \bzeta_{\emptyset,\ell}(\by)]\|_\frob\leq \s^\ell \mu_1\) with \eqref{eq:approximation_xi_zeta}, so that
\begin{equation}\label{eq:bounds_mu_1}
    s^{-\ell}\left(1-\frac{6\ell^2 \sqrt{\s}}{\sqrt{d}}\right)\|\bxi_{\emptyset,\ell}\|_{L^2}^2 \leq \mu_1 \leq \left(1+\frac{6\ell^2 \sqrt{\s}}{\sqrt{d}}\right)\| \bxi_{\emptyset,\ell}\|_{L^2}^2,
\end{equation}
yields the desired result. 
\end{proof}

Lemma \ref{lem:recovered_subspace_population_U_stat} follows from the following general property of the span of unfolded tensors.

\begin{lemma}\label{lem:partial_trace_unfolding}
Let \(\bT \in (\R^d)^{\otimes \ell}\) and \(\bX^{(a)} = \Mat_{a,\ell-a}(\bT)\) for $2 \leq a \leq \ell$. Denote $\thr_a = {\rm rank} (\bX^{(a)})$ and let \(\bY = [\by_1, \ldots, \by_{\thr_a}] \in \R^{d^a \times \thr_a}\) be any orthonormal basis of \(\Span(\bX^{(a)})\). For \(1 \leq c < a\),
\[
    \Span\left(\sum_{j = 1}^{\thr_a} \Mat_{c,a-c}(\by_j)\Mat_{c,a-c}(\by_j)^\sT\right) = \Span(\bX^{(c)}).
\]
\end{lemma}

\begin{proof} 
    Since each matrix \(\Mat_{c,a-c}(\by_j)\Mat_{c,a-c}(\by_j)^\sT\) is positive semidefinite, we have
    \[
        \Span\left(\sum_{j=1}^{\thr_a}\Mat_{c,a-c}(\by_j)\Mat_{c,a-c}(\by_j)^\sT\right) = \sum_{j=1}^{\thr_a} \Span\left(\Mat_{c,a-c}(\by_j)\right),
    \]
    where the right-hand side denotes the sum of subspaces. Let \(\{\bx_j\}_{j=1}^{d^{\ell-a}}\) denote the columns of \(\bX^{(a)}\).
    Since \(\{\by_j\}_{j=1}^{\thr_a}\) is a basis of \(\Span(\bX^{(a)})\),
    the families \(\{\by_j\}_{j=1}^{\thr_a}\) and \(\{\bx_j\}_{j=1}^{d^{\ell-a}}\) generate the same subspace. Consequently,
    \[
        \sum_{j=1}^{\thr_a} \Span\left(\Mat_{c,a-c}(\by_j)\right) = \sum_{j=1}^{d^{\ell-a}} \Span\left(\Mat_{c,a-c}(\bx_j)\right).
    \]
    Finally, by construction,
    \[
    \bX^{(c)} = \bigl[\Mat_{c,a-c}(\bx_1)\mid \cdots \mid \Mat_{c,a-c}(\bx_{d^{\ell-a}})\bigr],
    \]
    after a reordering of columns, which implies
    \[
    \sum_{j=1}^{d^{\ell-a}} \Span\left(\Mat_{c,a-c}(\bx_j)\right) = \Span(\bX^{(c)}). \qedhere
    \]
\end{proof}

\subsubsection{Proof of Lemma \ref{lem:concentration_U_statistic_final}}
\label{app:proof-concentration-asym-one-step}

Under Assumption \ref{ass:kernel}.\ref{ass:finite_rank}, there exists \(\sm\in \naturals\) such that we can decompose
\[
    \widehat{\bM} = \sum_{k = 1}^\sm \widehat{\bM}_k,
\]
where
\[
    \widehat{\bM}_k := \frac{1}{n(n-1)} \sum_{1 \leq i\neq j\leq n} \cT_k (\by_i) \cT_k (\by_j) \Mat_{a,b}\big(\cH_{d,\ell} (\bz_i) \big)\Mat_{a,b}\big(\cH_{d,\ell} (\bz_j) \big)^\sT.
\]
Furthermore, because \(\sum_{k=1}^{\sm}\cT_k(\by)^2=K(\by,\by) \leq B_K\) for all \(y\in \cY\) by Assumption \ref{ass:kernel}.\ref{ass:kernel-boundedness}, it follows that \(\|\cT_k\|_\infty \leq B_K^{1/2}\). By union bound over $k \in [\sm]$, it is sufficient to show Lemma \ref{lem:concentration_U_statistic_final} for a rank-$1$ kernel $K(\by,\by') = \cT(\by) \cT(\by')$ with $\| \cT \|_\infty \leq B_K^{1/2}$. Note that one can modify the argument to directly prove Lemma \ref{lem:concentration_U_statistic_final} for arbitrary bounded kernel (without the finite rank assumption). For simplicity, we ignore this lengthier argument.

Below, we prove Lemma \ref{lem:concentration_U_statistic_final} for rank-$1$ kernels and separate the analysis between \(1\leq a < b\) and \((a,b)=(1,0)\).

\begin{proof}[Proof of Lemma \ref{lem:concentration_U_statistic_final}: case \(1\leq a<b\)]
\noindent\textbf{Step 1: Decoupling.} Using a now standard decoupling argument due to~\cite{delapeña1999decouplinginequalitiestailprobabilities}, we can reduce the problem of bounding \( \| \widehat{\bM} - \E[\widehat{\bM}] \|_\op \) to bounding a decoupled analogue \(\| \widetilde{\bM} - \E[\widetilde{\bM}] \|_\op \) defined by
\[
\widetilde{\bM} = \frac{1}{n (n-1)} \sum_{1 \leq i \neq j \leq n} \cT(y_i) \cT (\tilde y_j) \Mat_{a,b}(\cH_{d,\ell}(\bz_i))\Mat_{a,b}(\cH_{d,\ell}(\tilde{\bz}_j))^\sT =: \sum_{i = 1}^n \bA_{i}(\tilde{\bB}_{i})^\sT,
\]
where \(\tilde{\cD}:=\{(\tilde{\by}_i,\tilde{\bz}_i)\}_{i=1}^{n}\) are i.i.d.~copies of \(\cD:=\{(\by_i,\bz_i)\}_{i=1}^{n}\) and
\[
    \bA_i = \frac{1}{n} \cT(\by_i) \Mat_{a,b}(\cH_{d,\ell}(\bz_i)), \quad \text{and}\quad \tilde{\bB}_i = \frac{1}{n-1} \sum_{j\neq i} \cT(\tilde{\by}_j) \Mat_{a,b}(\cH_{d,\ell}(\tilde{\bz}_j)).
\]
Let \(\bE = \E[\cT(\by)\Mat_{a,b}(\cH_{d,\ell}(\bz))]\) and note that $\E[\widetilde{\bM} ] = \bE \bE^\sT$. We decompose 
\begin{equation}\label{eq:decompo_M_U_stat}
\widetilde{\bM} - \bE\bE^\sT = \bDelta_1 + \bDelta_2+\bDelta_3,
\end{equation}
where
\begin{gather*}
    \bDelta_1 := \sum_{i=1}^{n}(\bA_{i}-\bE)(\tilde{\bB}_{i}-\bE)^\sT,\qquad\bDelta_2 := \sum_{i=1}^{n}(\bA_{i} - \bE)\bE^\sT, \qquad  \bDelta_3 := \sum_{i=1}^{n}\bE(\tilde{\bB}_{i} - \bE)^\sT.
\end{gather*}
Note that \(\bDelta_1\), \(\bDelta_2\) and \(\bDelta_3\) are sums of independent centered random matrices. We will apply Lemma~\ref{lem:matrix_concentration} to these matrices, as well as to \(\tilde{\bB}_{i}-\bE\), to show concentration in operator norm of \eqref{eq:decompo_M_U_stat}.

\medskip
\noindent\textbf{Step 2: Concentration of \(\Tilde \bB_i-\bE\).} We first bound the parameters defined in~\eqref{eq:matrix_concentration_parameters}. We have
\begin{align*}
    \sigma_\ast(\Tilde \bB_i-\bE)^2 & = \sup_{\|\bu\|_2 = \|\bv\|_2 = 1} \E\left[\<\bu, (\Tilde \bB_i-\bE)\bv\>^2\right] 
    \\ & \leq \frac{1}{(n-1)^2}\sum_{j\in [n]\setminus \{i\}} \sup_{\|\bu\|_2 = \|\bv\|_2 = 1} \E\left[\<\bu, \cT(\by_j)\Mat_{a,b}(\cH_{d,\ell}(\bz_j))\bv\>^2\right].
\end{align*}
For every \(j\in [n]\setminus \{i\}\) and any \(\bX \in \R^{d^a \times d^b}\), it follows from Assumption \ref{ass:kernel}.\ref{ass:kernel-boundedness}, Lemma~\ref{lem:isometry_spherical_harmonics} and the fact that \(\Mat_{a,b}\) is an isometry with respect to the Frobenius norm that
\begin{equation}\label{eq:quad_forms_Bi}
    \E\left[\<\bX, \cT(\by_j)\Mat_{a,b}(\cH_{d,\ell}(\bz_j))\>^2_\frob\right] \leq B_K\left\|\ptf\left(\Mat_{a,b}^{-1}(\bX)\right)\right\|_\frob^2 \leq B_K\|\bX\|_\frob^2.
\end{equation}
Hence, applying~\eqref{eq:quad_forms_Bi} with \(\bX=\bu\bv^\sT\), we obtain \(\sigma_\ast(\tilde \bB_i-\bE)^2 \leq \frac{B_K}{n-1}\) and, by~\eqref{eq:sigma_vs_sigma_ast}, \(\sigma(\tilde \bB_i-\bE)^2 \leq \frac{B_K d^b}{n-1}\). Furthermore, by~\eqref{eq:quad_forms_Bi},
\begin{align*}
    v(\tilde \bB_i-\bE)^2 \leq \frac{1}{n-1}\sup_{\|\bX\|_\frob \leq 1}\E\left[\<\bX,\cT(\by_j) \Mat_{a,b}(\cH_{d,\ell}(\bz_j))\>^2_\frob\right] \leq \frac{B_K}{n-1}.
\end{align*}
Additionally, note that the summands \(\cT(\by_j) \Mat_{a,b}(\cH_{d,\ell}(\bz_j))/(n-1)\) are deterministically bounded in operator norm by \(\sqrt{B_K N_{d,\ell}}/(n-1)\) thanks to Assumption \ref{ass:kernel}.\ref{ass:kernel-boundedness} along with the definition of the harmonic tensor in~\eqref{eq:harmonic_tensor_definition} (see~\eqref{eq:definition_kappa}). Thus, we directly obtain
\begin{align*}
    \bar{R}(\Tilde \bB_i-\bE)^2 & = \E\left[\max_{j\in [n]\setminus \{i\}} \left\|\frac{1}{n-1}\cT(\by_j) \Mat_{a,b}(\cH_{d,\ell}(\bz_j))-\bE\right\|_\op^2\right]
    \\ & \leq \frac{4}{(n-1)^2}\E\left[\max_{j\in [n]\setminus \{i\}} \left\|\cT_k(\by_j) \Mat_{a,b}(\cH_{d,\ell}(\bz_j))\right\|_\op^2\right] \leq \frac{4 B_k N_{d,\ell}}{(n-1)^2}.
\end{align*}
Applying Lemma~\ref{lem:matrix_concentration} with the bounds above, we get that there exists a constant \(C>0\) such that for all \(t\geq 0\) and \(C\leq d\leq n^{2/\ell}\),
\[
    \|\Tilde \bB_i-\bE\|_\op \leq C\sqrt{B_K}\sqrt{\frac{d^b}{n}}\left(1+\frac{t^{1/2}}{d^{b/2}} + \left(\frac{d^a}{n}\right)^{1/12}t^{2/3} + \left(\frac{d^{a}}{n}\right)^{1/4}t\right),
\] 
with probability at least \(1 - (d^a+d^b+1)e^{-t}\). Choose \(0<c<3(b-a)/48\) and set \(t=d^{c}\) such that
$\|\bB_i-\bE\|_\op \leq C'\sqrt{B_K d^b/n}$ 
with probability at least \(1 - (d^a+d^b+1)e^{-d^{c}}\) for all \(C\leq d\leq n^{2/\ell}\). Since this bound holds uniformly over \(i\in [n]\), a union bound yields that
\begin{equation}\label{eq:concentration_max_Bi}
    \max_{i\in [n]}\|\Tilde \bB_i-\bE\|_\op \leq C'\sqrt{B_K}\sqrt{\frac{d^b}{n}},
\end{equation}
with probability at least \(1 - n(d^a+d^b+1)e^{-d^{c}}\) for all \(C\leq d\leq n^{2/\ell}\).

\medskip
\noindent\textbf{Step 3: Concentration of \(\bDelta_1\).} Let \(\bC_i := (\bA_{i}-\bE)(\tilde{\bB}_{i}-\bE)^\sT\) denote the summands of \(\bDelta_1\) for \(i\in [n]\). We apply Lemma~\ref{lem:matrix_concentration} to \(\bDelta_1\) conditionally on \(\tilde{\cD}\), where we place ourselves on the \(\tilde{\cD}\)-measurable event \eqref{eq:concentration_max_Bi}, which occurs with probability at least \(1 - n(d^a+d^b+1)e^{-d^{c}}\). By~\eqref{eq:quad_forms_Bi}, we have for every \(i\in [n]\),
\begin{align*}
    \E\left[\<\bX, \bC_{i}\>^2_\frob \mid \tilde{\cD}\right] & \leq \frac{1}{n^2} \E\left[\<\bX(\tilde{\bB}_{i}-\bE), \cT(\by_i)\Mat_{a,b}(\cH_{d,\ell}(\bz_i))\>^2_\frob \mid \tilde{\cD}\right]
        \\ & \leq \frac{B_K^2}{n^2} \|\bX(\tilde{\bB}_{i}-\bE)\|_\frob^2 \leq C^2\frac{B_K^2 d^b}{n^3}\|\bX\|_\frob^2,
\end{align*}
where we used \(\|\bM\bX\|_\frob \leq \|\bM\|_\op \|\bX\|_\frob\) and \eqref{eq:concentration_max_Bi}. Hence, if we denote by \(\sigma_\ast(\cdot\mid \tilde{\cD})\), \(v(\cdot\mid \tilde{\cD})\) and \(\bar{R}(\cdot\mid \tilde{\cD})\) the parameters defined in~\eqref{eq:matrix_concentration_parameters} but for conditional expectations given \(\tilde{\cD}\), it follows from a similar calculation as before that
\[
    \sigma_\ast(\bDelta_1 \mid \tilde{\cD})^2 \leq C^2 \frac{B_K^2 d^b}{n^2}, \quad \sigma(\bDelta_1 \mid \tilde{\cD})^2 \leq C^2 \frac{B_K^2 d^{\ell}}{n^2}, \quad \text{and}\quad v(\bDelta_1\mid \tilde{\cD})^2 \leq C^2 \frac{B_K^2 d^b}{n^2}.
\]
To bound the summands' operator norm, let \(p\geq 1\) be an integer and note that by monotonicity of \(L^p\)-norms and Minkowski's inequality,
\begin{align*}
    \left\|\max_{i\in [n]} \|\bC_i\|_\op\right\|^2_{L^{2p}(\mid \tilde{\cD})} & \leq \sum_{\alpha,\beta=1}^{d^a} \left\|\max_{i\in [n]} |\<\be_\alpha, \bC_i \be_\beta\>| \right\|_{L^{2p}(\mid \tilde{\cD})}^2
\end{align*}
where \(\{\be_\alpha\}_{\alpha=1}^{d^a}\) is the canonical basis of \(\R^{d^a}\). Since \(\max_{i\in [n]} |X_i| \leq (\sum_{i=1}^n |X_i|^{2p})^{1/(2p)}\) for any random variables \(\{X_i\}_{i=1}^n\),
\begin{align*}
     \left\|\max_{i\in [n]} \|\bC_i\|_\op\right\|_{L^{2p}(\mid \tilde{\cD})}^2 & \leq \sum_{\alpha,\beta=1}^{d^a} \left(\sum_{i=1}^n \E[|\<\be_\alpha, \bC_i \be_\beta\>|^{2p} \mid \tilde{\cD}] \right)^{1/p}.
\end{align*}
By Assumption \ref{ass:kernel}.\ref{ass:kernel-boundedness} and hypercontractivity~\eqref{eq:hypercontractivity-ineq-sphere}, for every \(i\in [n]\) and any \(\alpha,\beta\in [d^a]\),
\begin{align*}
    \E[|\<\be_\alpha, \bC_i \be_\beta\>|^{2p} \mid \tilde{\cD}] & \leq \frac{B_K^{p}}{n^{2p}}\E\left[\left|\<\be_\alpha, \Mat_{a,b}(\cH_{d,\ell}(\bz_i))(\tilde{\bB}_i-\bE)^\sT \be_\beta\>\right|^{2p} \mid \tilde{\cD}\right]
    \\ & \leq (2p-1)^{\ell p}\frac{B_K^{p}}{n^{2p}}\E\left[\left|\<\be_\alpha, \Mat_{a,b}(\cH_{d,\ell}(\bz_i))(\tilde{\bB}_i-\bE)^\sT \be_\beta\>\right|^{2} \mid \tilde{\cD}\right]^{p} .
\end{align*}
In particular, by~\eqref{eq:quad_forms_Bi},
\[
    \E[|\<\be_\alpha, \bC_i \be_\beta\>|^{2p} \mid \tilde{\cD}]  \leq (2p-1)^{\ell p}\frac{B_K^{2p}C^{2p}d^{pb}}{n^{3p}},
\]
and consequently,
\[
     \left\|\max_{i\in [n]} \|\bC_i\|_\op\right\|_{L^{2p}(\mid \tilde{\cD})} \leq \frac{CB_K(2p-1)^{\ell/2}d^{a/2}d^{\ell/2}}{n^{3/2-1/(2p)}}.
\]
Let \(c\in (0,1-2a/\ell)\) and \(p= \lfloor d^{c}\rfloor\). Then,
\[
     \left\|\max_{i\in [n]} \|\bC_i\|_\op\right\|_{L^{2p}(\mid \tilde{\cD})} \leq \frac{CB_K2^{\ell/2}d^{3\ell/4}}{d^{\epsilon}n^{3/2}}\exp\left(\frac{\log(n)}{2\lfloor d^c\rfloor}\right)
\]
where \(\epsilon = ((1-2c)\ell/2 - a)/2>0\). Monotonicity of \(L^p\) norms in \(p\) yields
\[
    \bar{R}(\bDelta_1\mid \tilde{\cD}) \leq  \left\|\max_{i\in [n]} \|\bC_i\|_\op\right\|_{L^{2p}(\mid \tilde{\cD})}\leq \frac{CB_K2^{\ell/2}d^{3\ell/4}}{d^{\epsilon}n^{3/2}}\exp\left(\frac{\log(n)}{2\lfloor d^c\rfloor}\right)
\]
while Markov's inequality gives
\begin{align*}
    \P\left(\max_{i\in [n]} \|\bC_i\|_\op\geq \frac{B_K d^{3\ell/4}}{d^{\epsilon/2}n^{3/2}}\mid \tilde{\cD}\right) &  \leq \left(\frac{C 2^{\ell/2}}{d^{\epsilon/2}}\exp\left(\frac{\log(n)}{2\lfloor d^c\rfloor}\right)\right)^{2\lfloor d^{c}\rfloor}.
\end{align*}
Suppose that \(d^{\ell/2}\leq n \leq e^{d^{c'}}\) for a suitably small constant \(c'>0\). Applying Lemma~\ref{lem:matrix_concentration} conditionally on \(\tilde{\cD}\) and using~\eqref{eq:concentration_max_Bi}, we get that there exists a constant \(C'>0\) such that for all \(d\geq C'\) and \(t\geq 0\),
\[
    \|\bDelta_1\|_\op \leq \frac{C'B_K d^{\ell/2}}{n}\left(1+\frac{\log(d^b)^{3/4}}{d^{a/4}} + \frac{t^{1/2}}{d^{a/2}} + \frac{t^{2/3}}{d^{\epsilon/12}}+\frac{t}{d^{\epsilon/4}}\right),
\]
with probability at least
\[
    1-\left(\frac{CB_K^{1/2}2^{\ell/2}}{d^{\epsilon/2}}\exp\left(\frac{\log(n)}{2\lfloor d^c\rfloor}\right)\right)^{2\lfloor d^{c}\rfloor} - (1+2d^a)e^{-t}-n(d^a+d^b+1)e^{-d^{c}}.
\]
In particular, we can take \(c''>0\) small enough and \(t=d^{c''}\) such that
\begin{equation}\label{eq:concentration_delta1}
    \|\bDelta_1\|_\op \leq \frac{C''B_K d^{\ell/2}}{n},
\end{equation}
with probability at least \(1-e^{-d^{c''}}\).

\medskip
\noindent\textbf{Step 4: Concentration of \(\bDelta_2\) and \(\bDelta_3\).} Note that
\[
  \bE = \E[\cT(\by)\Mat_{a,b}(\cH_{d,\ell}(\bz))] = \E[\cT(\by)\Mat_{a,b}(\bxi_{\emptyset,\ell}(\by))],  
\]
so that we can use Assumption \ref{ass:kernel}.\ref{ass:kernel-boundedness} to obtain $\|\bE\|_\frob \leq B_K^{1/2}\|\bxi_{\emptyset,\ell}\|_{L^2}$.
Applying a similar argument as the one leading to~\eqref{eq:concentration_delta1}, replacing the bound on \( \max_{i\in [n]}\|\bB_i-\bE\|_\op\) by the bound on \( \|\bE\|_\frob\) above, we get that there exist constants \(C,c,c'>0\) such that for \(d\geq C\) and \(d^{\ell/2}\leq n\leq e^{d^{c'}}\),
\begin{equation}\label{eq:concentration_delta2}
    \|\bDelta_2\|_\op,  \|\bDelta_3\|_\op \leq \frac{CB_Kd^{\ell/4}\|\bxi_{\emptyset,\ell}\|_{L^2}}{\sqrt{n}},
\end{equation}
with probability at least \(1-e^{-d^c}\). We omit these repetitive details. 

\medskip
\noindent\textbf{Step 5: Concluding.} Let \(t\geq 0\). By the decoupling inequality of~\cite{delapeña1999decouplinginequalitiestailprobabilities}, there exists a universal constant \(C>0\) such that
\begin{align*}
    \P\left(\|\widehat{\bM}-\E[\widehat{\bM}]\|_\op\geq t\right) &\leq C\P\left(\|\widetilde{\bM}-\bE\bE^{\sT}\|_\op\geq \frac{t}{C}\right) \leq C \sum_{j \in [3]} \P\left(\|\bDelta_j\|_\op\geq \frac{t}{3C}\right).
\end{align*}
Combining~\eqref{eq:concentration_delta1} and~\eqref{eq:concentration_delta2}, it follows that there exist constants \(C',c,c'>0\) that depends only on $\s,\ell$ and the constants in the assumptions such that for all \(d\geq C'\) and \(d^{\ell/2}\leq n\leq e^{d^{c'}}\),
\[
    \|\widehat{\bM}-\E[\widehat{\bM}]\|_\op \leq C'  \left(\frac{d^{\ell/2}}{n}+\frac{d^{\ell/4}\|\bxi_{\emptyset,\ell}\|_{L^2}}{\sqrt{n}}\right),
\]
with probability at least \(1- e^{-d^{c}}\). This concludes the proof.
\end{proof}

\begin{proof}[Proof of Lemma \ref{lem:concentration_U_statistic_final}: case \((a,b)=(1,0)\)] For $\ell = 1$ (space of vectors), we have simply \(\cH_{d,1}(\bz)=\sqrt{d}\kappa_{d,1}\bz\). Using the same decoupling argument as in the case \(1\leq a < b\), we introduce 
\[
    \widetilde{\bM} = \frac{d \kappa_{d,1}^2}{n(n-1)} \sum_{1 \leq i \neq j \leq n} \cT(\by_i) \cT(\tilde \by_j) \bz_i \Tilde \bz_j^\sT =: \sum_{i=1}^{n}\ba_{i}\tilde{\bb}_{i}^\sT,
\]
where \(\tilde{\cD}:=\{(\tilde{\by}_i,\tilde{\bz}_i)\}_{i=1}^{n}\) are i.i.d.~copies of \(\cD:=\{(\by_i,\bz_i)\}_{i=1}^{n}\) and
\[
    \ba_i = \frac{\sqrt{d}\kappa_{d,1}}{n} \cT(\by_i) \bz_i, \qquad  \tilde{\bb}_i = \frac{\sqrt{d}\kappa_{d,1}}{n-1} \sum_{j\neq i} \cT(\tilde{\by}_j) \tilde{\bz}_j.
\]
Let \(\be := \sqrt{d}\kappa_{d,1}\E[\cT(\by)\bz]\) and decompose \(\widetilde{\bM} - \be\be^\sT = \bDelta_1 + \bDelta_2+\bDelta_3\) as in the previous case \(1\leq a <b \), where
\begin{gather*}
    \bDelta_1 := \sum_{i=1}^{n}(\ba_{i}-\be)(\tilde{\bb}_{i}-\be)^\sT,\qquad\bDelta_2 := \sum_{i=1}^{n}(\ba_{i} - \be)\be^\sT, \qquad  \bDelta_3 := \sum_{i=1}^{n}\be(\tilde{\bb}_{i} - \be)^\sT.
\end{gather*}
Let \(\bu,\bv\in \S^{d-1}\) be arbitrary. By Assumption \ref{ass:kernel}.\ref{ass:kernel-boundedness} and the fact that \(\|\<\bz,\bu\>\|_{\psi_2}\lesssim 1/\sqrt{d}\), there exists a universal constant \(C_1>0\) such that $ \|\<\ba_i,\bu\>\|_{\psi_2} \leq C_1\kappa_{d,1}B_{K}^{1/2}/ n$.
Conditionally on \(\tilde{\cD}\), \(\bu^\sT\bDelta_1\bv\) is the sum of independent centered subgaussian random variables. Applying the subgaussian Hoeffding inequality conditionally on \(\tilde{\cD}\), there exists a constant \(C_2>0\) such that
\[
    \P\left(|\bu^\sT\bDelta_{1}\bv|\geq t \mid \tilde{\cD}\right) \leq 2\exp\left(-\frac{C_2 n t^2}{\kappa_{d,1}^2B_{K}\max_{i\in [n]}\|\bb_i-\be\|^2_2}\right),
\]
for any \(t\geq 0\). On the other hand, another application of Hoeffding's inequality yields $ \|\<\bb_i,\bv\>\|_{\psi_2}\leq  C\kappa_{d,1}B_{K}^{1/2}/\sqrt{n-1}$. A standard \(\epsilon\)-net argument with \(\epsilon=1/4\) gives
\[
    \P\left(\|\bb_i-\be\|_2 \geq \frac{C_3B_{K}^{1/2}\kappa_{d,1}\sqrt{d}}{\sqrt{n-1}}\right) \leq 2e^{-d},
\]
for some sufficiently large \(C_3>0\). Taking a union bound, we obtain
\[
    \P\left(|\bu^\sT\bDelta_{1}\bv|\geq t\right) \leq 2\exp\left(-\frac{C_4 n(n-1) t^2}{\kappa_{d,1}^{4}B_{K}^2d}\right) + 2ne^{-d}
\]
for any \(t\geq 0\), where \(C_4>0\) is a constant. Choosing \(t=C_5\kappa_{d,1}^2B_Kd/n\) and using a standard \(\epsilon\)-net argument, we get that there exist constants \(C,c,c'>0\) such that
\[
    \|\bDelta_1\|_\op \leq \frac{C\kappa_{d,1}^2B_Kd}{n},
\]
with probability at least \(1-e^{-d^{c}}\) for any \(C\leq d \leq n \leq e^{d^{c'}}\). Applying similar arguments, using \(\|\be\|_2 \leq B_k^{1/2}\|\bxi_{\emptyset,1}\|_{L^2}\), we obtain
\[
    \|\bDelta_2\|_\op \leq \frac{CB_K\sqrt{d}\|\bxi_{\emptyset,\ell}\|_{L^2}}{\sqrt{n}},\qquad \|\bDelta_3\|_\op\leq \frac{CB_K\sqrt{d}\|\bxi_{\emptyset,\ell}\|_{L^2}}{\sqrt{n}},
\]
with probability at least \(1-e^{-d^c}\). The result follows similarly as in the case $1 \leq a < b$.
\end{proof}

\subsection{One-step of tensor unfolding: the symmetric case}\label{app:symmetric_one_step}

We now consider Algorithm \ref{alg:tensor_unfold_one_step} with symmetric unfolding $a = b = \ell/2$ ($\ell$ even). In that case, the empirical unfolded matrix is given by
\[
\widehat{\bM} = \frac{1}{n^2} \sum_{1 \leq i , j \leq n} K(\by_i,\by_j) \Mat_{a,a} (\cH_{d,\ell} (\bz_i)) \Mat_{a,a} (\cH_{d,\ell} (\bz_j))^\sT \in \R^{d^a \times d^a}.
\]
Under the finite-rank Assumption \ref{ass:kernel}.\ref{ass:finite_rank}, we can rewrite this matrix as
\begin{align}
  \widehat{\bM} = \sum_{r = 1}^\m \widehat{\bS}_r  \widehat{\bS}_r^\sT, \qquad\qquad \widehat{\bS}_r:= \frac{1}{n} \sum_{1 \leq i  \leq n} \cT_r (\by_i) \Mat_{a,a} (\cH_{d,\ell}(\bz_i) ) \in \R^{d^a \times d^a}.
\end{align}
We proceed similarly as in the proof of the asymmetric case, but now comparing to the population $\E\big[ \widehat{\bS}_r\big]\E\big[ \widehat{\bS}_r\big]^\sT$ instead of $\E\big[\widehat{\bM}\big]^\sT$. The following two lemmas are analogous to Lemmas \ref{lem:spectral_bounds_population_U_stat} and \ref{lem:concentration_U_statistic_final}.

\begin{lemma}\label{lem:spectral_bounds_population_symmetric}
    Under Assumptions \ref{ass:spectral-gap} and Assumption \ref{ass:kernel}, there exist constants \(c,C,C'>0\) that only depend on $\s$, $\ell$, and the constants in these assumptions, such that for all \(d\geq C'\) and $\gamma >1$, 
    \begin{equation}
    c\|\bxi_{\emptyset,\ell}\|_{L^2}^2 \bPi_0^{(a)}\preceq \sum_{r =1}^\m \E\big[ \widehat{\bS}_r\big]\E\big[ \widehat{\bS}_r\big]^\sT \preceq C \|\bxi_{\emptyset,\ell} \|_{L^2}^2 \left(\bPi_0^{(a)} +  \frac{1}{\gamma}\bPi_+^{(a)} + \frac{1}{\sqrt{d}} \bI_{d^a}\right).
    \end{equation}
\end{lemma}

This lemma follows directly from Lemma \ref{lem:spectral_bounds_population_U_stat} by noting that
\[
\sum_{r =1}^\m \E\big[ \widehat{\bS}_r\big]\E\big[ \widehat{\bS}_r\big]^\sT = \E\left[K(\by,\by') \Mat_{a,b} \big(\cH_{d,\ell}(\bz)\big) \Mat_{a,b} \big(\cH_{d,\ell}(\bz')\big)^\sT\right].
\]

The concentration part proceeds differently. We expand \(\widehat{\bS}_r\) using the semisimple decomposition of the tensor space \(\TSym_{\ell}(\R^d)\otimes \TSym_{\ell}(\R^d)\) (established in Lemma~\ref{lem:product-harmonic-tensors-decomposition}). The leading component corresponds to a sample covariance matrix, which we bound via a Bai-Yin-type inequality (for heavy-tailed random vectors). The other components are lower-dimensional and are bounded using the same matrix concentration bound (Lemma~\ref{lem:matrix_concentration}) as in the asymmetric case. 

% The leading component in the decomposition is controlled using a Bai-Yin-type inequality for sample covariance matrices with heavy-tailed entries.
% the semisimple decomposition of \(\TSym_{a}(\R^d)\otimes \TSym_{a}(\R^d)\) to expand $\Mat_{a,a} (\cH_{d,\ell}(\bz))$ as $\Mat_{a,0}(\cH_{d,a} (\bz)) \Mat_{a,0}(\cH_{d,a} (\bz))^\sT$ (a covariance part) plus remainder terms. We then show concentration of $\bS_r - \E[\bS_r]$ using a Bai-Yin type inequality for the sample covariance matrix part (with heavy-tailed entries), and Lemma~\ref{lem:matrix_concentration} for the remainder terms.

\begin{lemma}[Concentration: symmetric case $a = b$]\label{lem:concentration_aggregated_symmetric_unfolding}
Under Assumption \ref{ass:kernel}.\ref{ass:kernel-boundedness} and \ref{ass:kernel}.\ref{ass:finite_rank}, there exist constants $c,C,C' >0$ that only depend on $\ell$, $\s$, and the constants in these assumptions such that for any \(n,d\) with $d \geq C'$ and \( n \geq d^{\ell/2}\), we have  with probability at least \(1 - e^{-d^{c}}\),
    \begin{equation}
        \left\|\widehat{\bM}  - \sum_{r=1}^{\m}\E\big[ \widehat{\bS}_r\big]\E\big[ \widehat{\bS}_r\big]^\sT \right\|_\op \leq C \left[ \frac{d^{\ell/2}}{n} + \frac{d^{\ell/4} \|\bxi_{\emptyset,\ell}\|_{L^2}}{\sqrt{n}} \right].
    \end{equation}
\end{lemma}

The proofs of this lemma can be found in Appendix \ref{app:proof_concentration-symmetric-one-step} below. Using these two lemmas, the proof of Theorem \ref{thm:tensor_unfolding_one_step} with symmetric unfolding is identical to the proof with asymmetric unfolding.

\subsubsection{Proof of Lemma \ref{lem:concentration_aggregated_symmetric_unfolding}}\label{app:proof_concentration-symmetric-one-step}

\begin{proof}[Proof of Lemma \ref{lem:concentration_aggregated_symmetric_unfolding}] Again, by union bound, it is sufficient to prove this lemma for $\sm = 1$. In the following, we assume that $K(\by,\by') = \cT(\by) \cT(\by')$ with $\| \cT \|_{\infty} \leq B_K^{1/2}$ and write $\widehat{\bS} := \widehat{\bS}_1$.

    \medskip
    \noindent\textbf{Step 1: Decomposition of \(\widehat{\bS}\).} Let \(\bu \in \S^{d^a -1}\) be arbitrary and denote by \(\bU=\Mat^{-1}_{a,0}(\bu)\in (\R^d)^{\otimes a}\) its tensor representation. By Lemma~\ref{lem:product-harmonic-tensors-decomposition}, we can decompose
    \begin{multline*}
        b^{(d)}_{a,a,0}\<\cH_{d,\ell}(\bz), \ptf(\bU) \otimes \ptf(\bU)\>_\frob = \<\cH_{d,a}(\bz)\otimes \cH_{d,a}(\bz),\ptf(\bU)\otimes \ptf(\bU)\>_\frob
        \\ - \sum_{j = 1}^{\ell/2} b^{(d)}_{a,a,j}\<\cH_{d,\ell - 2j} (\bz), \ptf(\bU) \diamond_j \ptf(\bU)\>_\frob,
    \end{multline*}
    where the coefficients \(b^{(d)}_{a,a,j} = \Theta_d(1)\) are explicitly given in Lemma~\ref{lem:product-harmonic-tensors-decomposition}. For convenience, denote $\bh_{d,a}(\bz) := \Mat_{a,0} (\cH_{d,a} (\bz)) \in \R^{d^a- 1}$. Then,
    \begin{align*}
        &\|\widehat{\bS} - \E[\widehat{\bS}]\|_\op  \leq \frac{1}{b^{(d)}_{a,a,0}}  \left\|\frac{1}{n}\sum_{i=1}^{n} \cT(\by_i) \bh_{d,a}(\bz_i)\bh_{d,a}(\bz_i)^\sT - \E\left[\cT(\by)\bh_{d,a}(\bz)\bh_{d,a}(\bz)^\sT\right]\right\|_\op
        \\ & \quad + \sum_{j=1}^{\ell/2} \sup_{\bu\in \S^{d^a-1}} \left|\frac{b^{(d)}_{a,a,j}}{b^{(d)}_{a,a,0}} \left\<\frac{1}{n}\sum_{i=1}^{n} \cT(\by_i) \cH_{d,\ell - 2j}(\bz_i)- \E[\cT(\by)\cH_{d,\ell - 2j}(\bz)], \ptf(\bU) \diamond_j \ptf(\bU)\right\>_\frob\right|.
    \end{align*}
    Denote the first term on the right-hand side by \(\bDelta_0\) and the \(j\)-th term in the sum by \(\bDelta_j\) for \(j\in [\ell/2]\). For any \(j\in [\ell/2]\), it follows from the definition of the contracted product \(\diamond_j\) that
    \begin{align*}
        &~\left|\frac{b^{(d)}_{a,a,j}}{b^{(d)}_{a,a,0}} \left\<\frac{1}{n}\sum_{i=1}^{n} \cT(\by_i) \cH_{d,\ell - 2j}(\bz_i)- \E[\cT(\by)\cH_{d,\ell - 2j}(\bz)], \ptf(\bU) \diamond_j \ptf(\bU)\right\>_\frob\right| 
        \\ \lesssim &~ \left|\left\<\frac{1}{n}\sum_{i=1}^{n} \cT(\by_i) \cH_{d,\ell - 2j}(\bz_i)- \E[\cT(\by)\cH_{d,\ell - 2j}(\bz)], \ptf(\bU) \otimes_j \ptf(\bU)\right\>_\frob\right|
        \\ =&~\left|\left\<\Mat_{a-j,j}(\ptf(\bU))^\sT \left(\widehat{\bS}^{(\ell - 2j)} - \E[\widehat{\bS}^{(\ell - 2j)}]\right) \Mat_{a-j,j}(\ptf(\bU))\right\>_\frob\right|
    \end{align*}
    where we defined
    \[
        \widehat{\bS}^{(\ell-2j)} := \frac{1}{n}\sum_{i=1}^{n} \cT(\by_i) \Mat_{a-j,a-j}(\cH_{d,\ell - 2j}(\bz_i)) \in \R^{d^{a-j}\times d^{a-j}}.
    \]
    Since \(\|\Mat_{a-j,j}(\ptf(\bU))\|_\frob \leq \|\bU\|_\frob = 1\), we get that for any \(j\in \{1,\ldots,\ell/2\}\),
    \[
        \bDelta_j \lesssim \left\|\widehat{\bS}^{(\ell - 2j)} - \E[\widehat{\bS}^{(\ell - 2j)}]\right\|_\op.
    \]

    \medskip
    \noindent\textbf{Step 2: Concentration of \(\bDelta_j\) for \(j\in \{1,\ldots,\ell/2\}\).} Fix \(j\in \{1,\ldots,\ell/2\}\). Let us bound the parameters appearing in~\eqref{eq:matrix_concentration_parameters}. By Assumption \ref{ass:kernel}.\ref{ass:kernel-boundedness} and~\eqref{eq:definition_kappa},
    \[
        \left\|\cT(\by)\Mat_{a-j,a-j}\left(\cH_{d,\ell-2j}(\bz)\right)\right\|_\frob \leq B_K^{1/2}\sqrt{N_{d,\ell-2j}},
    \]
    and thus
    \[
        \bar{R}\left(\widehat{\bS}^{(\ell-2j)}-\E[\widehat{\bS}^{(\ell-2j)}]\right)^2 \leq\frac{4B_K N_{d,\ell-2j}}{n^2}.
    \]
    Furthermore, we can use Assumption \ref{ass:kernel}.\ref{ass:kernel-boundedness} along with Lemma~\ref{lem:isometry_spherical_harmonics} to get
    \begin{align*}
        \sigma_\ast\left(\widehat{\bS}^{(\ell-2j)}-\E[\widehat{\bS}^{(\ell-2j)}]\right)^2 & \leq \frac{1}{n^2}\sum_{i=1}^{n}\sup_{\|\bu\|_2=\|\bv\|_2=1}\E\left[\left\<\cT(\by_i)\cH_{d,\ell-2j}(\bz_i),\Mat_{\ell/2-j,\ell/2-j}^{-1}(\bu\bv^\sT)\right\>^2_\frob\right]
        \\ & \leq \frac{B_K}{n},
    \end{align*}
    and, by a similar argument,
    \[
        v\left(\widehat{\bS}^{(\ell-2j)}-\E[\widehat{\bS}^{(\ell-2j)}]\right)^2 \leq \frac{B_K}{n}.
    \]
    Therefore, applying Lemma~\ref{lem:matrix_concentration}, we get that there exists constants \(C,C'>0\) such that, for all \(t\geq 0\) and \(C'\leq d\leq n^{2/\ell}\),
    \[
        \left\|\widehat{\bS}^{(\ell-2j)}-\E[\widehat{\bS}^{(\ell-2j)}]\right\|_\op \leq \frac{Cd^{\frac{\ell}{4}-\frac{j}{2}}}{\sqrt{n}}\left(1+\frac{t^{1/2}}{n} + \left(\frac{d^{\frac{\ell}{2}-j}}{n}\right)^{1/12}t^{2/3} + \left(\frac{d^{\frac{\ell}{2}-j}}{n}\right)^{1/4} t\right),
    \]
    with probability at least \(1 - 3d^{\frac{\ell}{2}-j}e^{-t}\). In particular, we can choose \(t=d^{c}\) for some sufficiently small constant \(c>0\) such that
    \[
        \max_{j\in \{1,\ldots, \ell/2\}}\|\bDelta_j\|_\op \leq \frac{C'd^{\frac{\ell}{4}-\frac{1}{2}}}{\sqrt{n}} \lesssim \frac{1}{\sqrt{d}},
    \]
    with probability at least \(1 -  e^{-d^{c}}\).

    \medskip
    \noindent\textbf{Step 3: Concentration of \(\bDelta_0\).} We now control the term \(\bDelta_0\) using a Bai-Yin-type inequality for sample covariance matrices with heavy-tailed entries. Decompose \(\cT(\by)\) as \(\cT(\by) = \cT_{+}(\by) - \cT_{-}(\by)\) where \(\cT_{+}(\by) = \max\{\cT(\by),0\}\) and \(\cT_{-}(\by) = \max\{-\cT(\by),0\}\). By the triangle inequality and the fact that \(|b_{a,a,0}^{(d)}| = \Theta_d(1)\), it suffices to control the concentration of the two matrices
    \[
       \bDelta_{0,\pm} :=  \frac{1}{n}\sum_{i=1}^n \cT_{\pm}(\by_i)\bh_{d,\ell/2}(\bz_i) \bh_{d,\ell/2}(\bz_i)^\sT - \E[\cT_{\pm}(\by)\bh_{d,\ell/2}(\bz) \bh_{d,\ell/2}(\bz)^\sT].
    \]
    We focus on the matrix \(\bDelta_{0,+}\); \(\bDelta_{0,-}\) can be treated similarly. For every deterministic unit vector \(\bu\in \R^{d^{\ell/2}}\), we have by hypercontractivity on the sphere~\eqref{eq:hypercontractivity-ineq-sphere} and Assumption~\ref{ass:kernel}.\ref{ass:kernel-boundedness},
    \[
        \E[\<\cT_{+}^{1/2}(\by_i)\bh_{d,\ell/2}(\bz),\bu\>^{2p}] \leq B_{K}^{p/2}(2p-1)^{p\ell/2},
    \]
    for every integers \(p\geq 1\). In particular, this means that \(\<\cT_{+}^{1/2}(\by)\bh_{d,\ell/2}(\bz),\bu\>\) is a sub-Weibull random variable, and therefore for every \(0<\alpha < 4/\ell\leq 2\), there exists a constant \(C>0\) independent of \(n,d\) such that for every \(t>0\),
    \[
        \P\left(|\<\cT_{+}^{1/2}(\by)\bh_{d,\ell/2}(\bz),\bu\>| \geq t\right) \leq C\exp\left(-t^{\alpha}\right).
    \]
    By~\cite[Theorem 4.7]{Guédon_Litvak_Pajor_Tomczak-Jaegermann_2017},
    \[
        \|\bDelta_{0,+}\|_\op \lesssim \frac{1}{n}\max_{i\in[n]}\|\cT_{+}^{1/2}(\by_i)\bh_{d,\ell/2}(\bz_i)\|^2_\op + \sqrt{\frac{d^{\ell/2}}{n}}
    \]
    with probability at least\footnote{At the time of writing, there is a minor typo in the statement of~\cite[Theorem 4.7]{Guédon_Litvak_Pajor_Tomczak-Jaegermann_2017}, where there is a missing minus sign in the exponent of the second term in the probability. Nonetheless, we have confirmed with the authors that the expression above is correct.}
    \[
        1 - 8e^{-n} + \frac{1}{(10n)^{4}}\exp\left(-\frac{4 d^{\frac{\alpha \ell}{4}}}{(3.5\ln(2d^{\ell/2}))^{2\alpha}}\right) + \frac{n^2}{2\exp\left((2nd^{\ell/2})^{\alpha/4}\right)}.
    \]
    Note that by Assumption~\ref{ass:kernel}.\ref{ass:kernel-boundedness} and the definition of \(\bh_{d,\ell/2}(\bz)\), we have $\|\cT_{+}^{1/2}(\by_i)\bh_{d,\ell/2}(\bz_i)\|^2_\op \leq \|\cH_{d,\ell/2}(\bz_i)\|_\frob^2 \lesssim d^{\ell/2}$ deterministically. Therefore, if \(n\geq d^{\ell/2}\), we have shown that there exists constants \(C,c>0\) such that for \(C\leq d\leq n^{2/\ell}\),
    \begin{align*}
        \|\bDelta_{0,+}\|_\op \leq C\sqrt{\frac{d^{\ell/2}}{n}},
    \end{align*}
    with probability at least \(1-e^{-d^c}\). A similar bound holds for the second matrix involving \(\cT_{-}(\by)\), and therefore it follows that
    \[
        \|\bDelta_{0}\|_\op \leq C\sqrt{\frac{d^{\ell/2}}{n}},
    \]
    with probability at least \(1-e^{-d^c}\).

    \medskip\noindent\textbf{Step 4: Concluding.} To conclude the proof and aggregate the estimates obtained above, decompose
    \begin{equation*}
        \widehat{\bM}- \E[\widehat{\bS}]  \E[\widehat{\bS}]^\sT = \left(\widehat{\bS} - \E[\widehat{\bS}]\right)\left(\widehat{\bS} - \E[\widehat{\bS}]\right)^\sT + \left(\widehat{\bS} - \E[\widehat{\bS}]\right)\E[\widehat{\bS}]^\sT  +  \E[\widehat{\bS}]\left(\widehat{\bS} - \E[\widehat{\bS}]\right)^\sT
    \end{equation*}
    such that
    \[
        \left\|\widehat{\bM}-\E[\widehat{\bS}]  \E[\widehat{\bS}]^\sT\right\|_\op \leq \|\widehat{\bS} - \E[\widehat{\bS}]\|_\op^2 + 2\|\E[\widehat{\bS}]\|_\op \|\widehat{\bS} - \E[\widehat{\bS}]\|_\op.
    \]
    By the bounds obtained in Steps 2 and 3, along with the decomposition in Step 1, there exist constants \(C,C',c>0\) such that for all \(C'\leq d\leq n^{2/\ell}\),
    \[
        \|\widehat{\bS} - \E[\widehat{\bS}]\|_\op \leq C\sqrt{\frac{d^{\ell/2}}{n}}
    \]
    with probability at least \(1 - e^{-d^{c}}\). Additionally, by Lemma~\ref{lem:isometry_spherical_harmonics} and Assumption \ref{ass:kernel}.\ref{ass:kernel-boundedness}, $ \|\E[\widehat{\bS}]\|_\frob \leq B_{K}^{1/2}\|\bxi_{\emptyset,\ell}\|_{L^2}$. Combining those estimates concludes the proof.
\end{proof}

\subsection{Multi-step procedure: iterative tensor unfolding}\label{app:multi-step-tensor-unfolding}

The proof of Theorem~\ref{thm:tensor_unfolding_mutli_step} follows from the next lemma, which controls a single intermediate step of the multi-step procedure in Algorithm~\ref{alg:multi_step_tensor_unfolding}.

\begin{lemma}
    Assume the setting of Theorem \ref{thm:tensor_unfolding_mutli_step}. Fix $t \in [T-1]$ and suppose  that at step \(t\) we have recovered a subspace $\widehat{\bU}_{\leq t} \in \Stf_{\s_{\leq t}} (\R^d)$. Let $\varphi$ be the modulus of continuity in Assumption \ref{ass:stability-conditional-distributions}. Define
    \[
    \delta := \varphi (\dist ( \widehat{\bU}_{\leq t}, \bU_{\leq t} ) ) \vee \dist ( \widehat{\bU}_{\leq t}, \bU_{\leq t} ).
    \]
    Let $\widehat{\bU}_{\leq t+1}$ be the subspace returned at the next iteration by Algorithm \ref{alg:multi_step_tensor_unfolding}. Then there exist constants $C,C',c,c' >0$ depending only on $\{\ell_t\}_{t \in [T]}$, $\s$, and the constants in the assumptions, such that for all $C'\leq d$, $ C' \leq \gamma \leq  \sqrt{d}$, $\delta \leq 1/C'$, and $n \leq \exp (d^{c'})$,
    \begin{equation}\label{eq:intermediary-guarantee}
        \dist ( \widehat{\bU}_{\leq t+1}, \bU_{\leq t+1} ) \leq C  \left[ \frac{d^{\ell_{t+1}/2 \vee 1}}{n \| \bxi_t \|_{L^2}^2} + \frac{1}{\sqrt{\gamma}} + \sqrt{\delta} \right],
    \end{equation}
    with probability at least $1 - \exp ( - d^c)$.
\end{lemma}

\begin{proof} For notational simplicity, set \(\ell := \ell_{t+1}\). Recall that the unfolding shape is taken to be \((a,b)=(1,0)\) if \(\ell=1\), and
\(a=\lfloor \ell/2 \rfloor\), \(b=\ell-a\) otherwise.
Let \(\bU_{\leq t,\perp}\) and \(\widehat{\bU}_{\leq t,\perp}\) be orthogonal complements of \(\bU_{\leq t}\) and \(\widehat{\bU}_{\leq t}\), respectively.

Let \(\widehat{\bM}\) denote the empirical \(d_{t+1}^a \times d_{t+1}^a\) matrix constructed at step \(t+1\) when conditioning on \(\widehat{\bU}_{\leq t}\), and let \(\widetilde{\bM}\) denote the analogous empirical matrix constructed when conditioning on the population subspace \(\bU_{\leq t}\) (using the corresponding complements \(\widehat{\bU}_{\leq t,\perp}\) and \(\bU_{\leq t,\perp}\)).

\medskip
\noindent{\bf Step 1: Population matrix.} We first verify that \(\E[\widetilde{\bM}]\) satisfies the bounds of Lemma~\ref{lem:spectral_bounds_population_U_stat}.
It suffices to check Assumption~\ref{ass:kernel}.\ref{ass:kernel-positive-definite} (restricted to the top signal directions) for the symmetrized kernel \(\overline{K}_{t+1}\) defined in~\eqref{eq:symmetrized-kernel}.
Fix \(\bV \in \TSym_\ell(\R^{\s_{t+1}})\). By definition,
    \[
    \begin{aligned}
     &~ \E\left[\overline{K}_{t+1} (\by_t,\by'_t)\< \bV,\bzeta_{\bU_{\leq t},\ell}(\by_t)\>_\frob \< \bzeta_{\bU_{\leq t},\ell}(\by_t'), \bV\>_\frob  \right]\\
     = &~  \int_{\cO_{\s_{\leq t}}} \E\left[K_{t+1} (g\cdot \by_t, g \cdot \by'_t)\< \bV,\bzeta_{\bU_{\leq t},\ell}(\by_t)\>_\frob \< \bzeta_{\bU_{\leq t},\ell}(\by_t'), \bV\>_\frob  \right] \de \pi_{\s_{\leq t}} (g),
    \end{aligned}
    \]
    where $g \cdot \by_t = (\by, g \cdot \bU_{\leq t}^\sT \bz)$ and $g \cdot \by_t' = (\by', g \cdot \bU_{\leq t}^\sT \bz')$.

    Let \(O \subseteq \cO_{\s_{\leq t}}\) be a neighborhood of the identity such that
\(\varphi(\|\bU_{\leq t}-\bU_{\leq t}\cdot g\|_\frob) \leq \eps\) for all \(g \in O\).
Restricting the integral to \(O\) and using Assumption~\ref{ass:kernel}.\ref{ass:kernel-positive-definite} for \(K_{t+1}\), we obtain
    \begin{equation}\label{eq:lowerbound_Kbar}
    \begin{aligned}
        &~ \E\left[\overline{K}_{t+1} (\by_t,\by'_t)\< \bV,\bzeta_{\bU_{\leq t},\ell}(\by_t)\> \< \bzeta_{\bU_{\leq t},\ell}(\by_t'), \bV\>  \right] \\
        \geq&~ \pi_{\s_{\leq t}} (O) \E\left[\< \bV,\bzeta_{\bU_{\leq t},\ell}(\by_t)\>^2  \right] \\
    &~    - \| \bxi_{\bU_{\leq t},\ell} \|_{L^2}^2 \int_{O}   \E \left[ (K_{t+1} (\by_t,\by'_t) - K_{t+1} ((\by,g \cdot \bU_{\leq t}^\sT \bz),(\by',g \cdot \bU_{\leq t}^\sT \bz')))^2  \right]^{1/2} \de \pi_{\s_{\leq t}} (g) .
    \end{aligned}
    \end{equation}
    Using Assumption~\ref{ass:stability-conditional-distributions}.(b), we obtain
    \[
    \begin{aligned}
       \E\left[\overline{K}_{t+1} (\by_t,\by'_t)\< \bV,\bzeta_{\bU_{\leq t},\ell}(\by_t)\>_\frob \< \bzeta_{\bU_{\leq t},\ell}(\by_t'), \bV\>_\frob  \right]  \geq &~ c \E\left[\< \bV,\bzeta_{\bU_{\leq t},\ell}(\by_t)\>^2_\frob  \right] - C \| \bxi_{\bU_{\leq t},\ell} \|_{L^2}^2 \eps. 
    \end{aligned}
    \]
    Choosing \(\eps\) such that the second term is small compared to the \(\rnk_{t}\)-th eigenvalue (e.g., \(C\|\bxi_{\bU_{\leq t},\ell}\|_{L^2}^2\eps \leq c \mu_{\rnk_{t}}/2\)),
Lemma~\ref{lem:spectral_bounds_population_U_stat} applies, yielding constants \(c,C>0\) such that for all \(d\) large enough,
    \begin{equation}\label{eq:intermediary-population-idealized}
            c\|\bxi_{\bU_{\leq t},\ell}\|_{L^2}^2 \bPi_0^{(a)} \preceq \E\big[\widehat{\bM}\big] \preceq C \|\bxi_{\bU_{\leq t},\ell} \|_{L^2}^2 \left(\bPi_0^{(a)} +  \frac{1}{\gamma}\bPi_+^{(a)} + \frac{1}{\sqrt{d}} \bI_{d^a}\right),
    \end{equation}
    where $\bPi_0^{(a)},\bPi_+^{(a)}$ are defined as in \eqref{eq:decomposition_Z0_Zplus} for the reduced model $\nu_{d,t}$ associated to the population $\bU_{\leq t}$.

\medskip
\noindent
{\bf Step 2: Lifting to \(\R^d\) and stability of expectations.} Define the lifted matrices
    \[
    \widehat{\bM}^{(L)} := [\widehat{\bU}_{\leq t,\perp}^{\otimes \ell}] \widehat{\bM}[\widehat{\bU}_{\leq t,\perp}^{\otimes \ell}]^\sT \in \R^{d^a \times d^a}, \qquad  \widetilde{\bM}^{(L)} := [\bU_{\leq t,\perp}^{\otimes \ell}] \widetilde{\bM}[\bU_{\leq t,\perp}^{\otimes \ell}]^\sT \in \R^{d^a \times d^a}.
    \]
    Using the definition of \(\widehat{\bM}\), we may write
    \[
     \E \left[  \widehat{\bM}^{(L)}\right] = \E\left[\overline{K}_{t+1} ( \hat \by_t, \hat \by_t')  \Mat_{a,b} (\widehat{\bU}_{\leq t,\perp}^{\otimes \ell} \bxi_{\widehat{\bU}_{\leq t}, \ell} (\hat \by_t)) \Mat_{a,b} (\widehat{\bU}_{\leq t,\perp}^{\otimes \ell} \bxi_{\widehat{\bU}_{\leq t}, \ell} (\hat \by_t))^\sT \right],
    \]
    where \(\widehat{\by}_t=(\by,\widehat{\bU}_{\leq t}^\sT\bz)\) and \(\widehat{\by}_t'\) is an independent copy; an analogous identity holds for \(\E[\widetilde{\bM}^{(L)}]\).
    
    We now bound the difference between these expectations. Using \(\|\overline{K}_{t+1}\|_\infty \leq B_K\) (Assumption~\ref{ass:kernel}.\ref{ass:kernel-boundedness}), the triangle inequality, and Cauchy--Schwarz,
    \[
    \begin{aligned}
        &~\left\|  \E \big[  \widehat{\bM}^{(L)}\big] - \E \big[ \widetilde{\bM}^{(L)} \big] \right\|_\frob \\
        \leq&~  \E\left[ \left| \overline{K}_{t+1} ( \hat \by_t, \hat \by_t') - \overline{K}_{t+1} ( \by_t, \by_t') \right| \| \bxi_{\widehat{\bU}_{\leq t}, \ell} (\hat \by_t) \|_\frob \| \bxi_{\widehat{\bU}_{\leq t}, \ell} (\hat \by_t') \|_\frob \right] \\
        &~+ B_{K} \E\left[ ( \| \bxi_{\widehat{\bU}_{\leq t}, \ell} (\hat \by_t) \|_\frob + \| \bxi_{\bU_{\leq t}, \ell} ( \by_t) \|_\frob ) \| \widehat{\bU}_{\leq t,\perp}^{\otimes \ell} \bxi_{\widehat{\bU}_{\leq t}, \ell} (\hat \by_t') - \bU_{\leq t,\perp}^{\otimes \ell} \bxi_{\bU_{\leq t}, \ell} (\by_t') \|_\frob \right].
    \end{aligned}
    \]
    By Assumption~\ref{ass:stability-conditional-distributions}, if \(\delta \leq 1/2\) then
\(\|\bxi_{\widehat{\bU}_{\leq t},\ell}\|_{L^2} \leq 3\|\bxi_{\bU_{\leq t},\ell}\|_{L^2}\),
and using Cauchy--Schwarz, both terms are controlled by \(C\|\bxi_{\bU_{\leq t},\ell}\|_{L^2}^2\delta\).
Consequently,
    \[
  \left\|  \E \big[  \widehat{\bM}^{(L)}\big] - \E \big[ \widetilde{\bM}^{(L)} \big] \right\|_\frob  \leq C \|\bxi_{\bU_{\leq t}, \ell}\|_{L^2}^2 \delta. 
    \]
    For \(\delta\) sufficiently small, this implies that the leading eigenspaces of
\(\E[\widehat{\bM}^{(L)}]\) are close to those of \(\E[\widetilde{\bM}^{(L)}]\), which by~\eqref{eq:intermediary-population-idealized} correspond to
\(\big[\bU_{\leq t,\perp}^{\otimes \ell}\big]\bPi_0^{(a)}\big[\bU_{\leq t,\perp}^{\otimes \ell}\big]^\sT\).

\medskip
\noindent
{\bf Step 3: Concluding via Davis--Kahan.}
The bound~\eqref{eq:intermediary-guarantee} then follows by repeating the argument used in the proof of Theorem~\ref{thm:tensor_unfolding_one_step}, with two applications of the Davis--Kahan \(\sin\Theta\) theorem, where we compare the empirical matrix \(\widehat{\bM}\) to $ \E \big[ \widetilde{\bM}^{(L)} \big]$. We obtain
\[
\dist(\bU_{\leq t,\perp} \bU_{t+1}, \widehat{\bU}_{\leq t,\perp} \widehat{\bU}_{t+1} ) \leq  C  \left[ \frac{d^{\ell_{t+1}/2 \vee 1}}{n \| \bxi_t \|_{L^2}^2} + \frac{1}{\sqrt{\gamma}} + \sqrt{\delta} \right]. 
\]
Using the triangle inequality
\[
\dist ( \bU_{\leq t+1} , \widehat{\bU}_{\leq t+1}) \leq \dist ( \bU_{\leq t} , \widehat{\bU}_{\leq t}) + \dist(\bU_{\leq t,\perp} \bU_{t+1}, \widehat{\bU}_{\leq t,\perp} \widehat{\bU}_{t+1} )
\]
concludes the proof.
\end{proof}

\begin{remark}[Finite rank approximation of $\overline{K}_t$] To match the SQ lower bound and to avoid summing over all pairs \((i,j)\in[n]^2\) when forming \(\widehat{\bM}\), we replace \(\overline{K}_t\) by a finite-rank approximation while preserving the population bounds in~\eqref{eq:intermediary-population-idealized}.
Specifically, we approximate the Haar integral defining \(\overline{K}_t\) by an average over an \(\eta\)-net of \(\cO_{\s_{\leq t}}\), denoted \(\{g_j\}_{j\in[N_\eta]}\), and define
\[
\overline{K}^{(\eta)}_{t} = \frac{1}{N_\eta} \sum_{j\in [N_\eta]} K_t ( g_j \cdot \by_t, g_j \cdot \by_t').
\]
Since \(K_t\) satisfies Assumption~\ref{ass:kernel}.\ref{ass:finite_rank}, the kernel \(\overline{K}^{(\eta)}_t\) has rank at most \(\m N_\eta \leq \m (1+2/\eta)^{\s_{\leq t}^2} \) where we used  a standard covering-number bound. Finally, the same argument as in~\eqref{eq:lowerbound_Kbar} shows that \(\eta\) can be chosen sufficiently small---depending only on the constants of the problem---so that the key population bound~\eqref{eq:intermediary-population-idealized} continues to hold with \(\overline{K}_t\) replaced by \(\overline{K}^{(\eta)}_t\).
\end{remark}

\subsection{Runtime of the algorithm}
\label{app:runtime-tensor-unfolding}

The dominant computational cost in Algorithm~\ref{alg:tensor_unfold_one_step} arises from computing the leading \(\thr\) eigenvectors of the empirical matrix \(\widehat{\bM} \in \R^{d^a \times d^a}\).
We compute these eigenvectors using subspace power iteration.
Specifically, an \(\eps\)-approximation to the top \(\thr\)-dimensional eigenspace can be obtained in
\(O(\log(d^a/\eps)/g)\) iterations, where
\[
g := 1 - \hat{\lambda}_{\thr+1}/\hat{\lambda}_{\thr},
\]
and \(\hat{\lambda}_j\) denotes the \(j\)-th largest eigenvalue of \(\widehat{\bM}\).

Let \(\rt_{\mathrm{mult}}\) denote the time required to compute a single matrix--vector product
\(\widehat{\bM}\bv\), with \(\bv \in \R^{d^a}\).
Each iteration of subspace power iteration involves \(\thr\) such products, together with an orthonormalization step of cost \(O(\thr^3)\).
Therefore, the total runtime of this first eigen-decomposition step is
\[
O\!\left(
(\thr\,\rt_{\mathrm{mult}} + \thr^3)\,
\frac{\log(d/\eps)}{g}
\right).
\]

The second step of the algorithm computes the top \(\s_0\) eigenvectors of the matrix
\[
\sum_{s=1}^{\thr}
\Mat_{1,a-1}(\widehat{\bv}_s)\,
\Mat_{1,a-1}(\widehat{\bv}_s)^\sT,
\]
where \(\{\widehat{\bv}_s\}_{s=1}^{\thr}\) are the leading eigenvectors of \(\widehat{\bM}\).
Applying the same subspace iteration argument, this step requires
\[
O\!\left(
\thr\, d^a\, \frac{\log(d/\eps)}{h}
\right)
\]
operations, where
\[
h := 1 - \widetilde{\mu}_{\s_0+1}/\widetilde{\mu}_{\s_0},
\]
and \(\widetilde{\mu}_j\) denotes the \(j\)-th largest eigenvalue of the above matrix.
Under the assumptions of Theorem~\ref{thm:tensor_unfolding_one_step}, we have \(h \geq c\) with high probability.
Consequently, this second step is asymptotically cheaper than the first whenever \(\rt_{\mathrm{mult}} = \Omega (d^a)\) and does not affect the overall runtime.

\paragraph{Matrix--vector multiplication.}
We now bound the cost \(\rt_{\mathrm{mult}}\) of computing a matrix--vector product with \(\widehat{\bM}\).
Under the finite-rank kernel assumption
(Assumption~\ref{ass:kernel}.\ref{ass:finite_rank}),
the matrix \(\widehat{\bM}\) admits the decomposition
\[
\widehat{\bM}
=
\left(1 + \frac{\delta_{a \neq b}}{n-1}\right)
\sum_{r=1}^{\m}
\bigl[
\bS_r \bS_r^\sT
-
\delta_{a \neq b}\,\bD_r
\bigr],
\]
where
\[
\bS_r
:=
\frac{1}{n}
\sum_{i \in [n]}
\cT_r(\by_i)\,
\Mat_{a,b}\!\big(\cH_{d,\ell}(\bz_i)\big),
\qquad
\bD_r
:=
\frac{1}{n^2}
\sum_{i \in [n]}
\cT_r(\by_i)^2\,
\Mat_{a,b}\!\big(\cH_{d,\ell}(\bz_i)\big)
\Mat_{a,b}\!\big(\cH_{d,\ell}(\bz_i)\big)^\sT .
\]
As a result, computing \(\widehat{\bM}\bv\) reduces to evaluating
\(O(\m n)\) matrix--vector products involving the harmonic tensor
\(\Mat_{a,b}(\cH_{d,\ell}(\bz_i))\).
By~\cite[Lemma~15]{joshi2025learning}, each such product can be computed in
\(O(d^a + d^b)\) elementary operations, without explicitly forming the full matrix.
Therefore,
\[
\rt_{\mathrm{mult}} = O\!\big(\m n (d^a + d^b)\big).
\]

\clearpage

\section{Details and proofs for the applications}

This section makes precise the applications outlined in Section~\ref{sec:applications} and supplies the necessary details and proofs. In particular, we show how Gaussian and directional MIMs fit within the spherically invariant framework and characterize the complexity of learning spherical MIMs with polynomial relations between input and output.

\subsection{Specialization to Gaussian and directional MIMs}\label{app:gaussian_directional_mims}

In this section, we explain how our framework applies to Gaussian and directional MIMs, as discussed in Section~\ref{sec:applications-gaussian-mims} and Section~\ref{sec:applications-directional-mims} respectively, and provide the missing proofs. Our approach follows the strategy developed in~\cite{joshi2025learning} and relies on several properties of Hermite tensors, as well as on the harmonic decomposition of Hermite tensors introduced therein. We recall these properties before proceeding with the proofs.

\subsubsection{Hermite tensors and their harmonic decomposition}

Let $\gamma_d (\bx) = (2\pi)^{-\frac{d}{2}} e^{- \| \bx\|_2^2/2} $ denote the density of a standard Gaussian vector in $\R^d$. For any \(d\geq 1\) and $k \geq 0$, define $\He_{k} : \R^d \to \Sym_k (\R^d)$ the (normalized) degree-$k$ Hermite tensor in $\R^d$ by
\begin{equation}\label{eq:Hermite-tensor}
\He_{k} (\bx) := \frac{(-1)^k}{\sqrt{k!}} \frac{\nabla^k \gamma_d (\bx)}{\gamma_d (\bx)},
\end{equation}
where $\nabla^k \gamma_d(\bx)$ denotes the $k$-th derivative of $\gamma_d$, viewed as a symmetric $k$-tensor. In particular, when $d=1$, the collection $\{\He_k\}_{k\ge0}$ reduces to the classical orthonormal basis of Hermite polynomials in $L^2(\R,\gamma_1)$. 

The Hermite tensors realize the classical isometry between $\Sym_k (\R^d)$ and the $k$-th Wiener chaos:  for all $\bA \in \Sym_k (\R^d)$ and $ \bB \in \Sym_j (\R^d)$,
\begin{equation}
    \E_{\gamma_d} [ \< \bA,\He_k (\bx) \>_\frob\< \bB,\He_j (\bx) \>_\frob ] = \delta_{jk} \< \bA, \bB\>_\frob.
\end{equation}
Hence any $f\in L^2(\gamma_d)$ admits the Wiener chaos expansion
\begin{equation*}
    f(\bx)  = \sum_{k = 0}^\infty \< \bA_k, \He_k (\bx) \>_\frob , \qquad \bA_k := \E_{\gamma_d} [ f(\bx) \He_k (\bx)] \in \Sym_k (\R^d),
\end{equation*}
with convergence in $L^2(\gamma_d)$. Analogously to the second part of Lemma~\ref{lem:harmonic_expansion_invariant_functions}, if in addition \(f\) is \(\cO_d^\bW\) invariant for some \(\bW \in \Stf_\s(\R^d)\), then one may write \(\bA_k = \bW^{\otimes k} \bB_k\) for some \(\bB_k \in \Sym_k (\R^\s)\). We also recall the identity, valid for all $\bw\in\S^{d-1}$,
\begin{equation}\label{eq:hermite_gegenbauer}
    \He_k ( \< \bw, \bx\>) = \< \bw^{\otimes k}, \He_k (\bx) \>_\frob.
\end{equation}
More generally,~\eqref{eq:hermite_gegenbauer} together with the (scalar) Hermite generating function implies
\[
    \exp\left(\<\bx,\by\> - \frac{\|\by\|_2^2}{2}\right) = \sum_{k=0}^\infty \frac{\< \He_k(\bx), \by^{\otimes k}\>_\frob}{\sqrt{k!}},
\]
and hence, for any  \(\bW \in \Stf_\s(\R^d)\) and \(\bx\in\R^d\),
\begin{equation}\label{eq:hermite_tensor_projection}
    (\bW^\sT)^{\otimes k} \He_k(\bx) = \He_k(\bW^\sT \bx).
\end{equation}

To make the algebraic structure of Hermite tensors explicit, it is convenient to recall the
classical closed-form expression of one-dimensional Hermite polynomials. Namely, for
$u\in\R$ and $k\ge 0$,
\begin{equation}\label{eq:hermite_1d_explicit}
    \He_k(u)
    =
    \frac{1}{\sqrt{k!}}
    \sum_{j=0}^{\lfloor k/2\rfloor}
    \frac{(-1)^j k!}{2^j j!(k-2j)!}\, u^{k-2j}.
\end{equation}
Combined with the identity~\eqref{eq:hermite_gegenbauer}, this yields an explicit tensorial
representation of multivariate Hermite tensors:
\begin{equation}\label{eq:hermite_tensor_explicit}
    \He_k(\bx)
    =
    \frac{1}{\sqrt{k!}}
    \sum_{j=0}^{\lfloor k/2\rfloor}
     \frac{(-1)^jk!}{2^j j!(k-2j)!}\;
    \psym\!\big(\bx^{\otimes(k-2j)}\otimes \bI_d^{\otimes j}\big).
\end{equation}

The next lemma provides an explicit decomposition of Hermite tensors into harmonic tensors. A related decomposition appears in \cite[Proposition 2 and Lemma 3]{joshi2025learning}. For completeness, we include a proof here, which follows directly from the tensorial framework developed in this paper and yields a short derivation.

\begin{lemma}[Hermite-to-harmonic decomposition]\label{lem:Hermite-to-harmonic-tensors}
    For any integers \(d\geq 1\) and \(k\geq 0\), and any \(\bx \in \R^d\), we have the decomposition
    \[
        \He_k(\bx) = \sum_{j=0}^{\lfloor k/2\rfloor} \beta^{(d)}_{k,k-2j}(r)\, \psym\left(\cH_{d,k-2j}(\bz) \otimes \bI_d^{\otimes j}\right),
    \]
    where \(\bx = r \bz\) with \(r = \|\bx\|_2\) and \(\bz = \bx/\|\bx\|_2\), and
    \[
        \beta^{(d)}_{k,k-2j}(r) = \frac{\sqrt{k!}}{\kappa_{d,k-2j} \sqrt{N_{d,k-2j}}}\sum_{i=0}^{j} \frac{(-1)^i r^{k-2i}}{2^i i!(k-2i)!} f^{(d)}_{k-2i,j-i}
    \]
    with \(\kappa_{d,\ell},N_{d,\ell}\) defined in \eqref{eq:harmonic_tensor_definition} and \(f^{(d)}_{\ell,j}\) defined in Lemma~\ref{lem:fischer_decomposition}.

    Furthermore, for any fixed integers \(\ell,j\geq 0\),
    \[
        \E_{r\sim\chi_d}[\beta^{(d)}_{\ell+2j,\ell}(r)]^2 = \Theta_d\left(d^{-2j}\right), \quad \E_{r\sim\chi_d}[\beta^{(d)}_{\ell+2j,\ell}(r)^2] = \Theta_d\left(d^{-j}\right).
    \]
    Also, there exists a constant \(C>0\) such that for all \(d,j\geq C\) and any fixed integer \(\ell\geq 0\),
    \[
        \frac{\E_{r\sim\chi_d}[\beta^{(d)}_{\ell+2(j+1),\ell}(r)]^2}{\E_{r\sim\chi_d}[\beta^{(d)}_{\ell+2j,\ell}(r)]^2} \leq 1, \quad\text{ and }\quad \frac{\E_{r\sim\chi_d}[\beta^{(d)}_{\ell+2(j+1),\ell}(r)^2]}{\E_{r\sim\chi_d}[\beta^{(d)}_{\ell+2j,\ell}(r)^2]} \leq 1.
    \]
\end{lemma}

\begin{proof}
    By~\eqref{eq:hermite_tensor_explicit} and linearity of the symmetrization operator,
    \[
        \He_k(\bx) = \frac{1}{\sqrt{k!}}\sum_{j=0}^{\lfloor k/2\rfloor}\frac{(-1)^jk! r^{k-2j}}{2^j j!(k-2j)!}\psym\left(\bz^{\otimes (k-2j)}\otimes \bI_d^{\otimes j}\right),
    \]
    where we wrote $\bx = r \bz$ with $r = \|\bx\|_2$ and $\bz = \bx/\|\bx\|_2$. By the Fischer decomposition~\eqref{eq:fischer_decomp_rank_one}, we further decompose
    \[
        \He_k(\bx) = \frac{1}{\sqrt{k!}}\sum_{j=0}^{\lfloor k/2\rfloor}\frac{(-1)^jk! r^{k-2j}}{2^j j!(k-2j)!}\sum_{i=0}^{\lfloor (k-2j)/2\rfloor} f^{(d)}_{k-2j,i} \psym\left(\ptf(\bz^{\otimes (k-2j-2i)}) \otimes \bI_d^{\otimes (j+i)}\right).
    \]
    Re-indexing the sums over $j$ and $i$ by letting $m = j+i$, we obtain
    \[
        \He_k(\bx) = \sum_{m=0}^{\lfloor k/2\rfloor} \left[\frac{1}{\sqrt{k!}}\sum_{j=0}^{m} \frac{(-1)^jk! r^{k-2j}}{2^j j!(k-2j)!} f^{(d)}_{k-2j,m-j}\right] \psym\left(\ptf(\bz^{\otimes (k-2m)}) \otimes \bI_d^{\otimes m}\right).
    \]
    Using the definition of $\cH_{d,k-2m}(\bz)$ from \eqref{eq:harmonic_tensor_definition}, we can rewrite this as
    \[
        \He_k(\bx) = \sum_{m=0}^{\lfloor k/2\rfloor} \left[\frac{\sqrt{k!}}{\kappa_{d,k-2m} \sqrt{N_{d,k-2m}}}\sum_{j=0}^{m} \frac{(-1)^j r^{k-2j}}{2^j j!(k-2j)!} f^{(d)}_{k-2j,m-j}\right]  \psym\left(\cH_{d,k-2m}(\bz) \otimes \bI_d^{\otimes m}\right).
    \]
    This concludes the proof of the first part of the lemma.

    For the second part, we specialize the general expansion to $k=\ell+2j$ and $m=j$. Using the explicit expression of the Fischer coefficients from Lemma~\ref{lem:fischer_decomposition},
    \[
        f^{(d)}_{\ell+2j-2i,j-i} = \frac{(\ell+2j-2i)!}{2^{2(j-i)}(j-i)!\ell!(d/2+\ell)_{j-i}},
    \]
    we obtain
    \begin{align*}
    \beta^{(d)}_{\ell+2j,\ell}(r)
    & =
    \frac{\sqrt{(\ell+2j)!}}{\kappa_{d,\ell}\sqrt{N_{d,\ell}}}
    \sum_{i=0}^{j}
    \frac{(-1)^i r^{\ell+2j-2i}}{2^i i!(\ell+2j-2i)!}
    \frac{(\ell+2j-2i)!}{2^{2(j-i)}(j-i)!\,\ell!\,(d/2+\ell)_{j-i}}
    \\ & = \frac{(-1)^j\sqrt{(\ell+2j)!}\,r^\ell}{\kappa_{d,\ell}\sqrt{N_{d,\ell}}\,2^{j}\ell!}
    \sum_{t=0}^j
    \frac{(-1)^t}{t!(j-t)!}\,
    \frac{(r^2/2)^t}{(d/2+\ell)_t},
    \end{align*}
    where we changed variables $t=j-i$ in the last step.
    Using the standard identity for generalized Laguerre polynomials,
    \[
    L_j^{(\alpha)}(x)
    =
    \frac{(\alpha+1)_j}{j!}
    \sum_{t=0}^j
    \binom{j}{t}\frac{(-x)^t}{(\alpha+1)_t},
    \qquad \alpha>-1,
    \]
    with $\alpha=d/2+\ell-1$, we arrive at
    \[
        \beta^{(d)}_{\ell+2j,\ell}(r)
        =
        \frac{(-1)^j\sqrt{(\ell+2j)!}}{2^{j}\kappa_{d,\ell}\sqrt{N_{d,\ell}}\,\ell!\,(d/2+\ell)_j}\;
        r^\ell\,
        L_j^{(d/2+\ell-1)}\!\left(\frac{r^2}{2}\right).
    \]

    Let \(r\sim \chi_d\) denote a chi-distributed random variable with \(d\) degrees of freedom. Note that \(y = r^2/2\) follows a Gamma distribution with shape parameter \(d/2\) and scale parameter \(1\), i.e., it has density \(y^{d/2-1}e^{-y}/\Gamma(d/2)\) for \(y>0\). Using the integral identity \(\int_0^\infty y^{\beta-1}e^{-y}L_n^{(\alpha)}(y)\d y = \frac{\Gamma(\beta)(\alpha+1-\beta)_n}{n!}\) (see e.g.~\cite[p.~119, Eq.~\(4\beta\)]{Buchholz_1969}) and the orthogonality of Laguerre polynomials (see e.g.~\cite[p.~136, Eq.~9]{Buchholz_1969}), we obtain
    \[
        \E\left[r^\ell L_j^{(d/2+\ell-1)}\left(\frac{r^2}{2}\right)\right] = 2^{\ell/2}\frac{\Gamma(d/2+\ell/2)}{\Gamma(d/2)}\frac{(\ell/2)_j}{j!}, \quad \E\left[r^{2\ell}\Big(L_j^{(d/2+\ell-1)}\left(\frac{r^2}{2}\right)\Big)^2\right] = 2^\ell \frac{(d/2)_{\ell+j}}{j!}.
    \]
    Plugging these formulas into the expression for \(\beta^{(d)}_{\ell+2j,\ell}(r)\) yields the explicit formulas
    \[
        \E[\beta^{(d)}_{\ell+2j,\ell}(r)] = \frac{(-1)^j\,2^{-j+\ell/2}\sqrt{(\ell+2j)!}}{\kappa_{d,\ell}\sqrt{N_{d,\ell}}\,\ell!\,(d/2+\ell)_j}\;
        \frac{\Gamma(d/2+\ell/2)}{\Gamma(d/2)}\;
        \frac{(\ell/2)_j}{j!}
    \]
    and
    \[
        \E[\beta^{(d)}_{\ell+2j,\ell}(r)^2] = \frac{2^{\ell-2j}(\ell+2j)!}{\kappa_{d,\ell}^2\,N_{d,\ell}\,(\ell!)^2}\frac{(d/2)_\ell}{(d/2+\ell)_j}\;\frac{1}{j!}.
    \]
    Since \(N_{d,\ell} = \Theta_d(d^\ell)\) and \(\kappa_{d,\ell} = \Theta_d(1)\) as \(d\to\infty\) with \(\ell,j\) fixed, it follows that
    \[
        \E[\beta^{(d)}_{\ell+2j,\ell}(r)]^2 = \Theta_d\left(d^{-2j}\right), \quad \E[\beta^{(d)}_{\ell+2j,\ell}(r)^2] = \Theta_d\left(d^{-j}\right).
    \]
    To prove the last part of the lemma, we compute the ratios
    \begin{align*}
        \frac{\E[\beta^{(d)}_{\ell+2(j+1),\ell}(r)]^2}{\E[\beta^{(d)}_{\ell+2j,\ell}(r)]^2} = \frac{(\ell+2j+2)(\ell+2j+1)(\ell/2+j)^2}{4(d/2+\ell+j)^2 (j+1)^2}
    \end{align*}
    and
    \begin{align*}
        \frac{\E[\beta^{(d)}_{\ell+2(j+1),\ell}(r)^2]}{\E[\beta^{(d)}_{\ell+2j,\ell}(r)^2]} = \frac{(\ell+2j+2)(\ell+2j+1)}{4(d/2+\ell+j)(j+1)}.
    \end{align*}
    For \(d\) large enough (depending on \(\ell\)), both ratios eventually approach \(1\) from below as \(j\to\infty\), which concludes the proof.
\end{proof}

Lemma \ref{lem:Hermite-to-harmonic-tensors} yields the following chaos-harmonic decomposition: for any $\bA \in \Sym_k (\R^d)$,
\begin{equation}\label{eq:chaos_harmonic_decomposition}
    \< \bA, \He_{k} (\bx) \>_\frob  = \sum_{j = 0}^{\lfloor k/2 \rfloor} \beta^{(d)}_{k,k-2j} (r) \< \ptf ( \tau^j (\bA) ), \cH_{d,k - 2j} (\bz) \>_\frob.
\end{equation}

\subsubsection{Gaussian MIMs}\label{app:gaussian_MIMs}

Consider a Gaussian MIM with link function $\rho \in \cP(\cY \times \R^\s)$, that is,
\[
    (y,\bx) \sim \P_{\rho}^{\bW_*} : \qquad 
    \bx \sim \gamma_d, 
    \qquad 
    y \mid \bx \sim \rho(\,\cdot \mid \bW_*^\sT \bx),
\]
and let $\bW_\ast \in \Stf_{\s}(\R^d)$ denote the planted subspace. For the sake of conciseness, we will omit the dependence on $\bW_\ast$ in the following.

This model is spherically invariant. Consequently, as discussed in Example~\ref{ex:spherically_invariant_MIM}, using the polar decomposition $\bx = r \bz$ of a Gaussian vector, where $r = \|\bx\|_2 \sim \nu_d^R=\chi_d$ and $\bz = \bx / \|\bx\|_2 \sim \tau_d$ are independent, we may equivalently view this model as a spherical MIM by defining $\by = (y,r)$ and
\[
    \bz \sim \tau_d, 
    \qquad 
    \by = (y,r) \mid \bz 
    \sim 
    \nu_d(\d \by \mid \bW_*^\sT \bz)
    :=
    \rho(\d y \mid r\, \bW_*^\sT \bz)\, \chi_d(\d r).
\]

There are two competing notions of null distribution in this model, leading to two different definitions of the generative exponent. First, in this paper as in~\cite{joshi2025learning}, we consider the null distribution in which the label and the radial component of the input are decoupled from its direction. Specifically, let $\nu_d^{(Y,R)}$ denote the marginal distribution of $(y,r)$ under $\P_{\nu_d}$, and define the null distribution
\[
    \P_{\nu_d,\emptyset} := \nu_d^{(Y,R)} \otimes \tau_d .
\]
Under this null model, the likelihood ratio admits the harmonic decomposition
\begin{equation}\label{eq:likelihood_ratio_gaussian_MIM_harmonic_decomposition}
    \frac{\d \P_{\nu_d}}{\d \P_{\nu_d,\emptyset}} (y,\bx) = \sum_{\ell = 0}^\infty \< \bxi_{\emptyset,\ell}(\by), \cH_{d,\ell}(\bz)\>_\frob, \quad \bxi_{\emptyset,\ell}(\by) = \E[\cH_{d,\ell}(\bz)\mid y,r].
\end{equation}

Alternatively, in much of the recent literature on Gaussian MIMs~\cite{damian2024computational,damian2025generative,joshi2025learning}, the null distribution is taken to fully decouple the label from the input. Specifically, one defines \(\tilde{\P}_{\nu_d,\emptyset}:=\nu_d^{Y} \otimes \chi_d \otimes \tau_d\), where $\nu_d^{Y}$ denotes the marginal distribution of $y$. Under this null distribution, the likelihood ratio admits the Wiener chaos expansion
\[
    \frac{\d \P_{\nu_d}}{\d \tilde{\P}_{\nu_d,\emptyset}}(y,\bx) = \sum_{k=0}^{\infty}\<\bpsi_{\emptyset,k}(y),\He_k(\bx)\>_\frob, \quad \bpsi_{\emptyset,k}(y) = \E[\He_k(\bx)\mid y].
\]
The \textit{generative exponent} is defined as
\[
    k_\ast := \argmin_{k \geq 1} \Big\{\, k : \|\bpsi_{\emptyset,k}(y)\|_{L^2(\nu_d^Y)}^2 > 0 \,\Big\},
\]
so that $\bpsi_{\emptyset,k}(y) \equiv 0$ for all $k < k_\ast$. Note that $\|\bpsi_{\emptyset,k}(y)\|_{L^2(\nu_d^Y)}$ depends only on the link function $\rho$ and is independent of the ambient dimension $d$. Indeed, by invariance of the conditional law of $\bx\mid y$ under the stabilizer $\cO_d^{\bW_\ast}$, the tensor $\bpsi_{\emptyset,k}(y)=\E[\He_k(\bx)\mid y]$ must lie in the span of $(\bW_\ast)^{\otimes k}$. Then, we use~\eqref{eq:hermite_tensor_projection} to write
\[
    \bpsi_{\emptyset,k}(y) = \E[(\bW_\ast^\sT)^{\otimes k}\He_k(\bx)\mid y] = \E[\He_k(\bW_\ast^\sT \bx)\mid y]
\]
with \(\bW_\ast^\sT \bx \sim \cN(0,\bI_\s)\) independent of \(d\).

Using the chaos-harmonic decomposition~\eqref{eq:chaos_harmonic_decomposition}, we can rewrite this likelihood ratio in the harmonic basis as
\begin{equation}\label{eq:likelihood_ratio_gaussian_MIM_harmonic_decomposition_alternate}
    \frac{\d \P_{\nu_d}}{\d \tilde{\P}_{\nu_d,\emptyset}}(y,\bx)  =\sum_{\ell = 0}^\infty \sum_{j = 0}^\infty \beta^{(d)}_{\ell + 2j,\ell} (r) \< \tau^j( \bpsi_{\emptyset,\ell + 2j} (y) ) , \cH_{d,\ell} (\bz) \>_\frob = \sum_{\ell = 0}^\infty \< \bar \bxi_{\emptyset, \ell} (\by), \cH_{d,\ell} (\bz)\>_\frob,
\end{equation}
where
\begin{equation}\label{eq:definition_xi_bar}
    \bar \bxi_{\emptyset, \ell} (\by) = \sum_{j \geq 0} \beta^{(d)}_{\ell + 2j,\ell} (r) \ptf (\tau^j( \bpsi_{\emptyset,\ell + 2j} (y))).
\end{equation}
Writing
\[
    \frac{\d \P_{\nu_d}}{\d \tilde{\P}_{\nu_d,\emptyset}}(y,\bx) = \frac{\d \P_{\nu_d}}{\d \P_{\nu_d,\emptyset}} (y,\bx) \frac{\d \P_{\nu_d,\emptyset}}{\d \tilde{\P}_{\nu_d,\emptyset}}(y,\bx) = \frac{\d \P_{\nu_d}}{\d \P_{\nu_d,\emptyset}} (y,\bx) \frac{\d \nu_d^{(Y,R)}}{\d \nu_d^{Y} \otimes \nu_d^{R}}(y,r)
\]
and comparing the harmonic decompositions in \eqref{eq:likelihood_ratio_gaussian_MIM_harmonic_decomposition} and \eqref{eq:likelihood_ratio_gaussian_MIM_harmonic_decomposition_alternate}, we obtain the relation
\begin{equation}\label{eq:relation_xi_bar_xi}
    \bar \bxi_{\emptyset, \ell} (\by) = \bxi_{\emptyset,\ell}(\by) \frac{\d \nu_d^{(Y,R)}}{\d \nu_d^{Y} \otimes \nu_d^{R}}(y,r).
\end{equation}
% Since the likelihood ratios are \(\cO^{\bW_*}_d\)-invariant, there exists \(\bphi_{\emptyset,k}(y) \in \Sym_k (\R^\s)\) such that \(\bpsi_{\emptyset,k}(y) = \bW_*^{\otimes k} \bphi_{\emptyset,k}(y)\). Then, since
Because \(\bxi_{\emptyset,0}(\by) = 1\), \(\ptf(c)=c\) for any \(c \in \R\),
%and \(\tau^j(\bW^{\otimes k}\bA)=\bW^{\otimes (k-j)}\tau^j(\bA)\) for every \(j\leq k\) and orthonormal frame \(\bW\), 
we have from \eqref{eq:definition_xi_bar} and \eqref{eq:relation_xi_bar_xi} that
\begin{equation}\label{eq:relation_xi0_bar_xi0}
    \frac{\d \nu_d^{(Y,R)}}{\d \nu_d^{Y} \otimes \nu_d^{R}}(y,r) = \bar \bxi_{\emptyset, 0} (\by) = \sum_{j \geq 0} \beta^{(d)}_{2j,0} (r) \tau^j( \bpsi_{\emptyset,2j} (y)) = \sum_{j \geq 0} \beta^{(d)}_{2j,0} (r) \tau^j(\bphi_{\emptyset,2j} (y)).
\end{equation}

In fact, the coefficients \(\bar \bxi_{\emptyset, \ell} (\by)\) and \(\bxi_{\emptyset,\ell}(\by)\) are close in \(L^2\) norm under mild assumptions. This is the content of the next lemma.

\begin{lemma}\label{lem:stability_xi_bar_xi_gaussian_MIM}
    Let \(\ell,\s \geq 1\). Suppose that there exists constants \(C',L,\eta>0\) such that
    \[
        \|\bxi_{\emptyset,\ell}(\by)\|_{L^{2+\eta}(\nu_d^{(Y,R)})} \leq L\|\bxi_{\emptyset,\ell}(\by)\|_{L^{2}(\nu_d^{(Y,R)})}
    \]
    for all \(d\geq C'\). Furthermore, suppose that there exists constants \(C>\sqrt{\s}\vee \sqrt{\s}(\lceil 2(2+\eta)/\eta\rceil - 1)\) and \(J>0\) such that \(\|\bpsi_{\emptyset,j}(y)\|_{L^{2}(\nu_d^{Y})} \leq C^{-j/2}\) for all \(j \geq J\) and all \(d\geq C'\). Then, there exists constants \(C''>0\) such that for all \(d\geq C''\),
    \[
        \frac{1}{2}\|\bar{\bxi}_{\emptyset,\ell}\|_{L^2(\nu_d^{Y}\otimes \nu_d^{R})}^2 \leq \|\bxi_{\emptyset,\ell}(\by)\|_{L^2(\nu_d^{(Y,R)})}^2 \leq 2 \|\bar{\bxi}_{\emptyset,\ell}\|_{L^2(\nu_d^{Y}\otimes \nu_d^{R})}^2.
    \]
\end{lemma}

\begin{proof}
    We suppose throughout that \(d\) is large enough so that all the assumptions of the lemma are satisfied. We also suppose that \(J=0\) for simplicity, as the general case follows by adjusting the constants. By~\eqref{eq:relation_xi_bar_xi},
    \[
         \|\bxi_{\emptyset,\ell}(\by)\|_{L^2(\nu_d^{(Y,R)})}^2 = \E_{\nu_d^{Y}\otimes \nu_d^{R}}\left[\<\bxi_{\emptyset,\ell}(\by),\bar{\bxi}_{\emptyset,\ell}(\by)\>_\frob\right].
    \]
    Mirroring the approach in \cite[Appendix G]{joshi2025learning}, let
    \[
        S = \sum_{j \geq 1} \beta^{(d)}_{2j,0} (r) \tau^j(\bpsi_{\emptyset,2j} (y)) \in \R
    \]
    such that \(1+S= \frac{\d \nu_d^{(Y,R)}}{\d \nu_d^{Y} \otimes \nu_d^{R}}(y,r)\) by \eqref{eq:relation_xi0_bar_xi0} and decompose
    \begin{equation}\label{eq:decomposition_xi_xi_bar_gaussian_MIM}
        \E_{\nu_d^{Y}\otimes \nu_d^{R}}\left[\<\bxi_{\emptyset,\ell},\bar{\bxi}_{\emptyset,\ell}\>_\frob\right] = \|\bar{\bxi}_{\emptyset,\ell}\|_{L^2(\nu_d^{Y}\otimes \nu_d^{R})}^2 -  \E_{\nu_d^{Y}\otimes \nu_d^{R}}\left[S\<\bxi_{\emptyset,\ell},\bar{\bxi}_{\emptyset,\ell}\>_\frob\right].
    \end{equation}
    
    By Hölder's inequality, for any \(p,q,r \geq 1\) such that \(1/p + 1/q + 1/r = 1\),
    \[
        |\E_{\nu_d^{Y}\otimes \nu_d^{R}}\left[S\<\bxi_{\emptyset,\ell},\bar{\bxi}_{\emptyset,\ell}\>_\frob\right]| \leq \|S\|_{L^p(\nu_d^{Y}\otimes \nu_d^{R})} \|\bar{\bxi}_{\emptyset,\ell}^{1-1/r}\|_{L^q(\nu_d^{Y}\otimes \nu_d^{R})} \|\bxi_{\emptyset,\ell}\|_{L^{2r}(\nu_d^{(Y,R)})}^2.
    \]
    Set \(p = q = 2(2+\eta)/\eta\) and \(r = (2+\eta)/2\). Then, by assumption, \(\|\bxi_{\emptyset,\ell}\|_{L^{2r}(\nu_d^{(Y,R)})}^2 \leq L \|\bxi_{\emptyset,\ell}\|_{L^{2}(\nu_d^{(Y,R)})}^2\) for some constant \(L>0\) independent of \(d\). Furthermore,
    \[
        \|\bar{\bxi}_{\emptyset,\ell}^{1-1/r}\|_{L^q(\nu_d^{Y}\otimes \nu_d^{R})}  = \|\bar{\bxi}_{\emptyset,\ell}\|_{L^2(\nu_d^{Y}\otimes \nu_d^{R})}^{\frac{\eta}{2+\eta}} \leq \left\|\frac{\de \P_{\nu_d}}{\de \tilde{\P}_{\nu_d,\emptyset}}\right\|_{L^2(\tilde{\P}_{\nu_d,\emptyset})}^{\frac{\eta}{2+\eta}}
    \]
    with
    \[
         \left\|\frac{\de \P_{\nu_d}}{\de \tilde{\P}_{\nu_d,\emptyset}}\right\|_{L^2(\tilde{\P}_{\nu_d,\emptyset})} = \sum_{k=0}^\infty \|\bpsi_{\emptyset,k}(y)\|_{L^2(\nu_d^{Y})}^2 \leq \sum_{k\geq 0}^\infty C^{-2k}<\infty
    \]
    independently of \(d\) by assumption. Therefore, to conclude the proof, it suffices to show that \(\|S\|_{L^p(\nu_d^{Y}\otimes \nu_d^{R})}\) can be made arbitrarily small for all \(d\) large enough. Let \(p' = \lceil 2(2+\eta)/\eta\rceil \) be an integer. By Hölder's inequality,
    \[
        \|S\|_{L^{p'}(\nu_d^{Y}\otimes \nu_d^{R})} \leq \sum_{j \geq \lceil k_*/2 \rceil} \|\beta^{(d)}_{2j,0} (r)\|_{L^{p'}(\nu_d^{R})} \|\tau^j(\bpsi_{\emptyset,2j} (y))\|_{L^{p'}(\nu_d^{Y})}.
    \]
    Momentarily fixing \(j\), we bound the two terms in the right-hand side separately. First, it follows from Lemma~\ref{lem:Hermite-to-harmonic-tensors} that \(\beta_{2j,0}^{(d)} (r)\) is a polynomial of degree \(2j\) in \(r\) with only even degree term. Since \(r=\|\bx\|_2\) with \(\bx \sim \cN(0,\bI_d)\), we can think of \(\beta_{2j,0}^{(d)} (r)\) as a degree-\(2j\) polynomial of a standard Gaussian vector in \(\R^d\). Hence, by Gaussian hypercontractivity (see for instance~\cite[Corollary 5.21]{Boucheron_Lugosi_Massart_2013}),
    \[
        \|\beta^{(d)}_{2j,0} (r)\|_{L^{p'}(\nu_d^{R})} \leq (p'-1)^{j} \|\beta^{(d)}_{2j,0} (r)\|_{L^{2}(\nu_d^{R})}.
    \]
    Next, since \(\bpsi_{2j} (y)\) is supported on a finite number of chaos levels, it follows again by Gaussian hypercontractivity that
    \[
        \|\tau^j(\bpsi_{\emptyset,2j} (y))\|_{L^{p'}(\nu_d^{Y})} \leq \s^{j/2} \|\bpsi_{\emptyset,2j} (y)\|_{L^{p'}(\nu_d^{Y})} \leq [\sqrt{\s}(p'-1)]^{j}\|\bpsi_{\emptyset,2j} (y)\|_{L^{2}(\nu_d^{Y})}
    \]
    Combining the above two bounds, we obtain
    \[
        \|S\|_{L^{p'}(\nu_d^Y \otimes \nu_d^R)} \leq \sum_{j \geq \lceil k_*/2 \rceil} [\sqrt{\s}(p'-1)^2]^{j}\|\beta^{(d)}_{2j,0} (r)\|_{L^{2}(\nu_d^{R})}\|\bpsi_{\emptyset,2j} (y)\|_{L^{2}(\nu_d^{Y})}.
    \]
    By assumption, there exists a constant \(\alpha\in (0,1)\) such that \([\sqrt{\s}(p'-1)^2]^{j}\|\bpsi_{\emptyset,2j} (y)\|_{L^{2}(\nu_d^{Y})} \leq \alpha^{j}\) for all \(j\). Hence, using the estimate on the ratios of \(\E[\beta^{(d)}_{2j,0}(r)^2]\) from Lemma~\ref{lem:Hermite-to-harmonic-tensors}, we have
    \[
        \frac{[\sqrt{\s}(p'-1)^2]^{j+1}\|\beta^{(d)}_{2j+2,0} (r)\|_{L^{2}(\nu_d^{R})}\|\bpsi_{\emptyset,2j+2} (y)\|_{L^{2}(\nu_d^{Y})}}{[\sqrt{\s}(p'-1)^2]^{j}\|\beta^{(d)}_{2j,0} (r)\|_{L^{2}(\nu_d^{R})}\|\bpsi_{\emptyset,2j} (y)\|_{L^{2}(\nu_d^{Y})}} \leq \alpha < 1.
    \]
    Therefore, it follows that
    \[
        \|S\|_{L^{p'}(\nu_d^Y \otimes \nu_d^R)} \leq  \frac{\alpha^{\lceil k_*/2 \rceil} }{1-\alpha} \|\beta^{(d)}_{2\lceil k_*/2 \rceil,0} (r)\|_{L^{2}(\nu_d^{R})} \lesssim d^{-\lceil k_*/2 \rceil/2},
    \]
    where the last inequality follows from the asymptotic of \(\E[\beta^{(d)}_{2j,0}(r)^2]\) in Lemma~\ref{lem:Hermite-to-harmonic-tensors}. Hence, for all \(d\) large enough, \(\|S\|_{L^{p'}(\nu_d^Y \otimes \nu_d^R)} \leq 1/2\).
\end{proof}

We can now prove the main lemma of this section, which relates the magnitudes of the harmonic coefficients \(\bxi_{\emptyset,\ell}(\by)\) to the generative exponent \(k_\ast\). To do so, we introduce the \emph{contraction index}: for \(1 \leq \ell \leq k_\ast\),
\begin{equation}\label{eq:definition_contraction_index}
    j_\ast(\ell) = \argmin_{j \geq 0} \left\{\left\|\ptf\left(\tau^{\lceil \frac{k_\ast-\ell}{2}\rceil+j}(\bpsi_{\emptyset,\ell+2\lceil (k_\ast-\ell)/2\rceil + 2j}(y))\right)\right\|_{L^{2}(\nu_d^{Y})}^2 > 0\right\}.
\end{equation}
In words, \(j_\ast(\ell)\) is the smallest number of contractions required so that the projection onto harmonic tensors of order \(\ell\) of the contracted tensor \(\bpsi_{\emptyset,\ell+2\lceil (k_\ast-\ell)/2\rceil + 2j}(y)\) is non-zero. Note that \(j_\ast(k_\ast) = 0\) for all \(d\) larger than some constant depending on \(k_\ast,\s\) by Lemma~\ref{lem:approximation_traceless_projection}.

\begin{lemma}\label{lem:Gaussian-MIM-harmonic-coefficients}
    Let \(1\leq \ell\leq k_\ast\) and suppose that the conditions of Lemma~\ref{lem:stability_xi_bar_xi_gaussian_MIM} are satisfied. Then,
    \[
        \|\bxi_{\emptyset,\ell}(\by)\|_{L^2(\nu_d^{(Y,R)})}^2 \asymp d^{- \lceil (k_\ast - \ell)/2 \rceil - j_\ast(\ell)}.
    \]
\end{lemma}

\begin{proof}[{Proof of Lemma~\ref{lem:Gaussian-MIM-harmonic-coefficients}}]
    Again, we suppose throughout that \(d\) is large enough so that all the assumptions of the lemma are satisfied and that the condition in Lemma~\ref{lem:stability_xi_bar_xi_gaussian_MIM} holds with \(J=0\) for simplicity. By Lemma~\ref{lem:stability_xi_bar_xi_gaussian_MIM}, \(\|\bar{\bxi}_{\emptyset,\ell}\|^2_{L^2(\nu_d^{Y}\otimes \nu_d^{R})}\asymp \|\bxi_{\emptyset,\ell}(\by)\|^2_{L^2(\nu_d^{(Y,R)})}\). Therefore, it suffices to prove the desired asymptotic for \(\|\bar{\bxi}_{\emptyset,\ell}\|^2_{L^2(\nu_d^{Y}\otimes \nu_d^{R})}\). Let \(m_\ast = \lceil (k_\ast - \ell)/2 \rceil + j_\ast(\ell)\). Plugging the definition of \(\bar{\bxi}_{\emptyset,\ell}\) from \eqref{eq:definition_xi_bar} and expanding the square,
    \begin{align*}
        \|\bar{\bxi}_{\emptyset,\ell}\|_{L^2(\nu_d^{Y}\otimes \nu_d^{R})}^2 & = \sum_{j,j' \geq 0} \E_{\nu_d^{Y}}\left[\<\ptf(\tau^j( \bpsi_{\ell + 2j} (y))), \ptf(\tau^{j'}( \bpsi_{\ell + 2j'} (y)))\>_\frob\right] \E_{\nu_d^{R}}\left[\beta^{(d)}_{\ell + 2j,\ell} (r) \beta^{(d)}_{\ell + 2j',\ell} (r)\right]
        \\ & = \sum_{j \geq m_\ast} \E_{\nu_d^{Y}}\left[\|\ptf(\tau^j( \bpsi_{\ell + 2j} (y)))\|_{\frob}^2\right] \E_{\nu_d^{R}}\left[\beta^{(d)}_{\ell + 2j,\ell} (r)^2\right]
        \\ & \quad + 2 \sum_{j' > j \geq  m_\ast} \E_{\nu_d^{Y}}\left[\<\ptf(\tau^j( \bpsi_{\ell + 2j} (y))), \ptf(\tau^{j'}( \bpsi_{\ell + 2j'} (y)))\>_\frob\right] \E_{\nu_d^{R}}\left[\beta^{(d)}_{\ell + 2j,\ell} (r) \beta^{(d)}_{\ell + 2j',\ell} (r)\right].
    \end{align*}

    Let us first analyze the diagonal terms. Since \(\bpsi_{\ell + 2j} (y)= \bW_*^{\otimes (\ell + 2j)} \bphi_{\ell + 2j} (y)\) for some \(\bphi_{\ell + 2j} (y) \in \Sym_{\ell + 2j} (\R^\s)\) and \(\ptf\) is an orthogonal projection,
    \[
        \|\ptf(\tau^j\left(\bpsi_{\ell+2j}(y)\right))\|_\frob^2\leq \|\tau^j\left(\bpsi_{\ell+2j}(y)\right)\|_\frob^2 \leq \s^{j}\|\bpsi_{\ell+2j}(y)\|_\frob^2.
    \]
    Let \(\alpha \in (0,1)\) such that, by assumption, \(\s^{j}\|\bpsi_{\ell+2j}(y)\|_{L^2(\nu_d^{Y})}^2 \leq \alpha^{j}\) for all \(j\). Then,
    \[
        \sum_{j \geq m_\ast} \E_{\nu_d^{Y}}\left[\|\ptf(\tau^j( \bpsi_{\ell + 2j} (y)))\|_{\frob}^2\right] \E_{\nu_d^{R}}\left[\beta^{(d)}_{\ell + 2j,\ell} (r)^2\right] \leq \sum_{j \geq m_\ast} \alpha^j \E_{\nu_d^{R}}\left[\beta^{(d)}_{\ell + 2j,\ell} (r)^2\right].
    \]
    Indeed, since \(\E[\beta^{(d)}_{\ell + 2j+2,\ell} (r)^2] / \E[\beta^{(d)}_{\ell + 2j,\ell} (r)^2] \leq 1\) by Lemma~\ref{lem:Hermite-to-harmonic-tensors}, the series on the right-hand side converges and
    \[
        \sum_{j \geq m_\ast} \alpha^j \E_{\nu_d^{R}}\left[\beta^{(d)}_{\ell + 2j,\ell} (r)^2\right] \leq \frac{ \alpha^{m_\ast}}{1-\alpha} \E_{\nu_d^{R}}\left[\beta^{(d)}_{\ell + 2m_\ast,\ell} (r)^2\right] \lesssim d^{-m_\ast}.
    \]
    The last inequality follows from the asymptotic equation for \(\E[\beta^{(d)}_{\ell + 2j,\ell} (r)^2]\) in Lemma~\ref{lem:Hermite-to-harmonic-tensors}. On the other hand, it is straightforward to see that the diagonal terms are lower bounded by any individual term in the sum. In particular,
    \[
        \sum_{j \geq m_\ast} \E_{\nu_d^{Y}}\left[\|\ptf(\tau^j( \bpsi_{\ell + 2j} (y)))\|_{\frob}^2\right] \E_{\nu_d^{R}}\left[\beta^{(d)}_{\ell + 2j,\ell} (r)^2\right] \gtrsim  d^{-m_\ast}\E_{\nu_d^{Y}}\left[\|\ptf(\tau^{m_\ast}( \bpsi_{\ell + 2m_\ast} (y)))\|_{\frob}^2\right],
    \]
    where we have used the fact that \(\E[\beta^{(d)}_{\ell + 2j,\ell} (r)^2] = \Theta_d(d^{-j})\) from Lemma~\ref{lem:Hermite-to-harmonic-tensors} whenever \(j\) is fixed.

    We now turn to the cross-terms. By Cauchy-Schwarz inequality, they can be bounded as
    \begin{align*}
         \left|\sum_{j' > j \geq  m_\ast} \E_{\nu_d^{Y}}\left[\<\ptf(\tau^j( \bpsi_{\ell + 2j} (y))), \ptf(\tau^{j'}( \bpsi_{\ell + 2j'} (y)))\>_\frob\right] \E_{\nu_d^{R}}\left[\beta^{(d)}_{\ell + 2j,\ell} (r) \beta^{(d)}_{\ell + 2j',\ell} (r)\right]\right|
         \leq S_{m_\ast+1} S_{ m_\ast},
    \end{align*}
    where we have defined 
    \[
        S_m = \sum_{j \geq m} \sqrt{\E_{\nu_d^{Y}}\left[\|\ptf(\tau^j( \bpsi_{\ell + 2j} (y)))\|_{\frob}^2\right] \E_{\nu_d^{R}}\left[\beta^{(d)}_{\ell + 2j,\ell} (r)^2\right]}.
    \]
    It follows from the argument above that the series defining \(S_m\) is convergent for all \(m\). Indeed, using the same notation as above, we have for all sufficiently large \(d\),
    \[
        S_m \leq \sum_{j \geq m} \sqrt{\alpha^j \E_{\nu_d^{R}}\left[\beta^{(d)}_{\ell + 2j,\ell} (r)^2\right]}.
    \]
    In particular,
    \[
         \left|\sum_{j' > j \geq  m_\ast} \E_{\nu_d^{Y}}\left[\<\ptf(\tau^j( \bpsi_{\ell + 2j} (y))), \ptf(\tau^{j'}( \bpsi_{\ell + 2j'} (y)))\>_\frob\right] \E_{\nu_d^{R}}\left[\beta^{(d)}_{\ell + 2j,\ell} (r) \beta^{(d)}_{\ell + 2j',\ell} (r)\right]\right| \lesssim d^{-m_\ast -\frac{1}{2}}.
    \]
    This is of lower order compared to the diagonal terms.
\end{proof}

\subsubsection{Directional MIMs}\label{app:directional_MIMs}

We now consider a directional (Gaussian) MIMs, which is a related observation model in which the radial component of the input is unobserved, and only the label and the direction of the input are available. Consider the Gaussian MIM defined via the link distribution \(\rho \in \cP(\cY \times \R^\s)\) as in Appendix~\ref{app:gaussian_MIMs} and let the generative exponent \(k_\ast\) be defined as before.

Suppose we only observe samples of the form \((y,\bz)\in \cY\times \S^{d-1}\), where \(\bz = \bx/\|\bx\|_2\) is the direction of the input. The conditional distribution of \(y\) given \(\bz\) is therefore obtained by marginalizing over the unobserved radial component \(r = \|\bx\|_2\). To embed this model into our general framework of spherical MIMs, we introduce an auxiliary radial variable \(r\) with deterministic distribution \(\delta_1\), independent of \((y,\bz)\). We thus define a lifted distribution \(\P_{\tilde\nu_d}\) on \(\cY\times \R_{\ge 0}\times \S^{d-1}\) by
\[
    (y,r,\bz)\sim \P_{\tilde{\nu}_d}, \quad \bz \sim \tau_d, \quad (y,r) \mid \bz \sim \E_{R\sim \chi_d}[\rho(\d y \mid R \bW_*^\sT \bz)]\delta_1(\d r).
\]
This lifted representation leaves the joint law of \((y,\bz)\) unchanged, but allows us to apply the harmonic analysis developed for spherical MIMs.

In this setting, the null distribution associated with the directional MIM coincides with the null distribution \(\tilde{\P}_{\tilde{\nu}_d,\emptyset}\) which arises from a fully decoupled model as in Appendix~\ref{app:gaussian_MIMs}, namely
\[
    \P_{\tilde{\nu}_d,\emptyset}
    =
    \nu_d^Y \otimes \delta_1 \otimes \tau_d
    =
    \tilde{\P}_{\tilde{\nu}_d,\emptyset}.
\]
By~\eqref{eq:likelihood_ratio_gaussian_MIM_harmonic_decomposition_alternate}, we have
\[
    \frac{\d \P_{\tilde{\nu}_d}}{\d \P_{\tilde{\nu}_d,\emptyset}}(y,\bz) = \E_{R\sim \chi_d}\left[ \frac{\d \P_{\nu_d}}{\d \tilde{\P}_{\nu_d,\emptyset}}(y,R\bz )\right] = \sum_{\ell = 0}^\infty \< \E_{R\sim \chi_d}[\bar \bxi_{\emptyset, \ell} (y,R)], \cH_{d,\ell} (\bz)\>_\frob,
\]
where 
\[
    \E_{R\sim\chi_d}[\bar \bxi_{\emptyset, \ell} (y,R)] = \sum_{j \geq 0} \E_{r\sim\chi_d}[\beta^{(d)}_{\ell + 2j,\ell} (r)] \ptf (\tau^j( \bpsi_{\emptyset,\ell + 2j} (y))).
\]
On the other hand, we can also directly perform a harmonic decomposition of the likelihood ratio \(\d \P_{\tilde{\nu}_d} / \d \P_{\tilde{\nu}_d,\emptyset}\), yielding
\[
    \frac{\d \P_{\tilde{\nu}_d}}{\d \P_{\tilde{\nu}_d,\emptyset}}(y,\bz) = \sum_{\ell = 0}^\infty \< \tilde{\bxi}_{\emptyset, \ell} (y), \cH_{d,\ell} (\bz)\>_\frob.
\]
Here, we use the notation \(\tilde{\bxi}_{\emptyset, \ell} (y)\) to distinguish these coefficients from those of the original Gaussian MIM. By comparing the two harmonic decompositions, we obtain the relation
\begin{equation}\label{eq:definition_xi_tilde_gaussian_MIM}
    \tilde{\bxi}_{\emptyset, \ell} (y) = \E_{R\sim \chi_d}[\bar \bxi_{\emptyset, \ell} (y,R)] = \sum_{j \geq 0} \E_{R\sim \chi_d}[\beta^{(d)}_{\ell + 2j,\ell} (R)] \ptf (\tau^j( \bpsi_{\emptyset,\ell + 2j} (y))).
\end{equation}

We use a similar argument as in Lemma~\ref{lem:Gaussian-MIM-harmonic-coefficients} to characterize the magnitude of the harmonic coefficients \(\tilde{\bxi}_{\emptyset, \ell} (\by)\) in terms of the generative exponent \(k_\ast\) and the contraction index \(j_\ast(\ell)\) defined in \eqref{eq:definition_contraction_index}.

\begin{lemma}\label{lem:Directional-Gaussian-MIM-harmonic-coefficients}
    Let \(1\leq \ell\leq k_\ast\). Suppose that there exists constants \(C',J>0\) and \(C>\sqrt{\s}\) such that \(\|\bpsi_{\emptyset,j}(y)\|_{L^{2}(\nu_d^{Y})} \leq C^{-j/2}\) for all \(j \geq J\) and \(d\geq C'\). Then,
    \[
        \|\tilde{\bxi}_{\emptyset,\ell}\|_{L^2(\nu_d^{Y})}^2 \asymp d^{- 2\lceil (k_\ast - \ell)/2 \rceil - 2j_\ast(\ell)}.
    \]
\end{lemma}

\begin{proof}
    Let \(m_\ast = \lceil (k_\ast - \ell)/2 \rceil + j_\ast(\ell)\), where \(j_\ast(\ell)\) is defined as in \eqref{eq:definition_contraction_index}. Expanding \(\|\tilde{\bxi}_{\emptyset,\ell}\|_{L^2(\nu_d^{Y})}^2\) using the definition in \eqref{eq:definition_xi_tilde_gaussian_MIM}, we have
    \begin{align*}
        \|\tilde{\bxi}_{\emptyset,\ell}\|_{L^2(\nu_d^{Y})}^2 &  = \sum_{j \geq m_\ast} \E_{\nu_d^{Y}}\left[\|\ptf(\tau^j( \bpsi_{\ell + 2j} (y)))\|_{\frob}^2\right] \E_{\nu_d^{R}}\left[\beta^{(d)}_{\ell + 2j,\ell} (R)\right]^2
        \\ &  + 2 \sum_{j' > j \geq  m_\ast} \E_{\nu_d^{Y}}\left[\<\ptf(\tau^j( \bpsi_{\ell + 2j} (y))), \ptf(\tau^{j'}( \bpsi_{\ell + 2j'} (y)))\>_\frob\right] \E_{\nu_d^{R}}\left[\beta^{(d)}_{\ell + 2j,\ell} (R)\right] \E_{\nu_d^{R}}\left[\beta^{(d)}_{\ell + 2j',\ell} (R)\right].
    \end{align*}
    The proof then follows the same steps as in Lemma~\ref{lem:Gaussian-MIM-harmonic-coefficients}, replacing the second moment of \(\beta^{(d)}_{\ell + 2j,\ell} (r)\) by the square of its first moment and using the estimates from Lemma~\ref{lem:Hermite-to-harmonic-tensors}.
\end{proof}

\subsection{Learning polynomials on the sphere}\label{app:learning-polynomials-sphere}

In this section, we study spherically invariant MIMs in which the input \(\bx\in\R^d\) admits the decomposition \(\bx=r\bz\), with \(r=\sqrt{d}\) and \(\bz\sim\tau_d\), and where \(\bW_*=[\bw_{\ast,1},\ldots,\bw_{\ast,\s}]\in\Stf_\s(\R^d)\) denotes the planted subspace. We observe responses of the form
\[
    y\mid \bx = \nu_d(\d y\mid\bW_*^\sT \bx):=  f(\bW_*^\sT \bx) + \epsilon,
\]
where \(\epsilon\) is independent additive noise and \(f:\R^\s\to\R\) is a non-constant polynomial of degree at most \(D\). Writing \(f\) in symmetric tensor form, there exist tensors \(\{\bC_j\}_{j=1}^D\) with \(\bC_j\in\Sym_j(\R^\s)\) and a scalar \(C_0\in\R\) such that
\[
    f(\bt) = C_0 + \sum_{j=1}^D \langle \bC_j, \bt^{\otimes j} \rangle_{\frob}.
\]
The normalization \(r=\sqrt{d}\) ensures that the coordinates of \(\bx\) are order one as \(d\to\infty\).

In this appendix we show, analogously to the Gaussian case~\cite{damian2025generative}, that the computational complexity of learning polynomial MIMs is governed by the second harmonic component of the likelihood ratio. Our argument follows the strategy of~\cite{chen2020learning} and parallels~\cite[Proposition~4]{damian2025generative}. A key distinction from Gaussian MIMs is that the Hermite coefficients of the likelihood ratio are dimension-free (see Appendix~\ref{app:gaussian_directional_mims}), whereas in spherical MIMs the harmonic coefficients depend on \(d\) in a non-trivial way. Consequently, in the spherical setting it is not enough to show that a given harmonic component is non-zero, we must also control its magnitude as a function of \(d\), which requires additional work.

Since the leap complexity is defined in terms of the harmonic structure of reduced models, it suffices to analyze the likelihood ratio after conditioning on previously recovered subspaces (see Section~\ref{sec:leap_complexity}). Let \(\bU\in\Stf_{<}(\R^\s)\) be a recovered subspace of dimension \(s_\bU<\s\), identified with its image \(\bU:=\bW_* \bU\in\Stf_{s_\bU}(\R^d)\). Set \(d_\bU=d-s_\bU\) and let \(\bU_\perp\in\Stf_{d_\bU}(\R^d)\) be an orthogonal complement. As in~\eqref{eq:decomposition-input}, we decompose
\[
\bz = \bU \br_\bU + \sqrt{1-\|\br_\bU\|_2^2}\,\bU_\perp \bz_\bU,
\qquad
\br_\bU := \bU^\sT \bz,
\qquad
\bz_\bU := \frac{\bU_\perp^\sT \bz}{\sqrt{1-\|\br_\bU\|_2^2}} \in \S^{d_\bU-1}.
\]
Conditioning on \(\br_\bU\) yields the reduced spherical MIM \((\by_\bU,\bz_\bU)\sim\P_{\nu_{d},\bU}\) defined in~\eqref{eq:decomposition-input}, with augmented response \(\by_\bU=(y,\br_\bU)\). By the characterization of leap complexity in Section~\ref{sec:leap_complexity}, controlling the harmonic coefficients of this reduced model is sufficient to determine the sample and runtime complexity of the full learning procedure. For the problem to be well posed, we assume that \(\sLeap(\nu_d)<\infty\).  This means that for every strict subframe \(\bU\subsetneq\bW_*\), the reduced spherical MIM \((\by_\bU,\bz_\bU)\sim\P_{\nu_d,\bU}\) retains nontrivial dependence on the residual direction \(\bz_\bU\); otherwise the reduced model would contain no further information and no additional directions could be learned.

It will be convenient to derive a variational characterization of the magnitude of the second harmonic component of the likelihood ratio in the reduced model. This is the content of the following lemma.

\begin{lemma}\label{lem:variational-characterization-second-harmonic-reduced-model}
    In the reduced spherical MIM \((\by_\bU,\bz_\bU) \sim \P_{\nu_{d},\bU}\) defined above, let \(\bW\in \Stf_{\s - s_\bU}(\R^{d_\bU})\) be an orthonormal basis for the image of \(\bU_\perp^\sT \bW_\ast\). Then, the second harmonic component of the likelihood ratio admits the variational characterization
    \begin{align*}
        \|\bxi_{\bU,2}(\by)\|_{L^2}^2 & = \frac{(d_\bU+2)d_\bU}{2}\sup_{G:\R \times \R^{\s_\bU} \mapsto \Sym_2(\R^{\s-\s_\bU}) } \frac{\E\left[\left\<\left(\bW^\sT \bz_\bU \bz_\bU^\sT \bW  - \frac{1}{d_\bU}\bI_{\s - s_\bU}\right), \bG(\by_\bU)\right\>_\frob\right]^2}{\|G(\by_\bU)\|_{L^2}^2}
        \\ & + \frac{(d_\bU+2)d_\bU(d_\bU-\s+\s_\bU)}{2}\sup_{g:\R \times \R^{\s_\bU} \mapsto \R} \frac{\E\left[\left(\frac{1-\|\bW^\sT \bz_\bU\|_2^2}{d_\bU - (\s - s_\bU)} - \frac{1}{d_\bU}\right) g(\by_\bU)\right]^2}{\|g(\by_\bU)\|_{L^2}^2},
    \end{align*}
    where the supremums are taken over square-integrable functions that are not almost surely zero.
\end{lemma}

\begin{proof}
    We start by observing that, by definition of the harmonic tensor \(\cH_{d,2}(\bz)\) (see~\eqref{eq:harmonic_tensor_definition},~\eqref{eq:coefficients_projection_traceless} and~\eqref{eq:projection_traceless_rank_one}),
    \[
        \cH_{d,2}(\bz) = \sqrt{\frac{(d+2)d}{2}}\left(\bz\bz^\sT - \frac{1}{d}\bI_d\right)
    \]
    and hence
    \[
        \bxi_{\bU,2}(\by) = \sqrt{\frac{(d_\bU+2)d_\bU}{2}}\E\left[\bz_\bU\bz_\bU^\sT - \frac{1}{d_\bU}\bI_{d_\bU} \mid \by_\bU\right].
    \]
    Note that \(\by_\bU\) only depends on \(\bz_\bU\) through \(\bW_\ast^\sT \bU_\perp \bz_\bU\). Hence, if we decompose \(\bU_\perp^\sT \bW_\ast = \bW\bV^\sT\) for some matrix \(\bV \in \R^{\s\times (\s-\s_\bU)}\) and some \(\bW \in \Stf_{\s-s_\bU}(\R^{d_\bU})\), then by rotational invariance of \(\bz_\bU\) it follows that \(\bxi_{\bU,2}(\by_\bU)\) is invariant under the action of the stabilizer subgroup \(\cO_{d_\bU}^\bW\). In particular, by Lemma~\ref{lem:decomposition_invariant_tensors}, there exists \(\bC_1(\by_\bU)\in \Sym_{2}(\R^{\s-s_\bU})\) and \(c_2(\by_\bU)\in \R\) such that
    \[
        \E[\bz_\bU \bz_\bU^\sT \mid \by_\bU] = \bW \bC_1(\by_\bU) \bW^\sT + c_2(\by_\bU)\left(\bI_{d_\bU} - \bW\bW^\sT\right).
    \]
    Taking the trace on both sides yields
    \[
        1 = \Tr(\bC_1(\by_\bU)) + c_2(\by_\bU)(d_\bU - (\s - s_\bU)),
    \]
    which gives \(c_2(\by_\bU) = \frac{1 - \Tr(\bC_1(\by_\bU))}{d_\bU - (\s - s_\bU)}\). In addition, \(\bC_1(\by_\bU) = \E[\bW^\sT \bz_\bU \bz_\bU^\sT \bW \mid \by_\bU]\), and hence
    \begin{align*}
        \bxi_{\bU,2}(\by)  & = \sqrt{\frac{(d_\bU+2)d_\bU}{2}}\bW\left(\E[\bW^\sT \bz_\bU \bz_\bU^\sT \bW \mid \by_\bU] - \frac{1}{d_\bU}\bI_{\s - s_\bU}\right)\bW^\sT
        \\ & + \sqrt{\frac{(d_\bU+2)d_\bU}{2}}\left(\frac{1-\E[\|\bW^\sT \bz_\bU\|_2^2 \mid \by_\bU]}{d_\bU - (\s - s_\bU)} - \frac{1}{d_\bU}\right)(\bI_{d_\bU} - \bW\bW^\sT).
    \end{align*}

    Using orthogonality of the two components above, we get
    \begin{align*}
        \|\bxi_{\bU,2}(\by)\|_\frob^2  & = \frac{(d_\bU+2)d_\bU}{2}\|\E[\bW^\sT \bz_\bU \bz_\bU^\sT \bW \mid \by_\bU] - \frac{1}{d_\bU}\bI_{\s - s_\bU}\|_\frob^2
        \\ & + \frac{(d_\bU+2)d_\bU (d_\bU - \s + \s_\bU)}{2}\left(\frac{1-\E[\|\bW^\sT \bz_\bU\|_2^2 \mid \by_\bU]}{d_\bU - (\s - s_\bU)} - \frac{1}{d_\bU}\right)^2.
    \end{align*}
    The result follows using the standard characterization of conditional expectations as \(L^2\)-projections.
\end{proof}

We may now state our main result on the computational complexity of learning polynomial spherical MIMs.

\begin{lemma}\label{lem:polynomial-second-harmonic}
    In the reduced spherical MIM \((\by_\bU,\bz_\bU) \sim \P_{\nu_{d},\bU}\) defined above, \(\|\bxi_{\bU,2}(\by)\|_{L^2}^2 \asymp 1\).
\end{lemma}

\begin{proof}
    An application of the triangle inequality and Cauchy-Schwarz gives the deterministic bound
    \[
        |f(\bW_*^\sT \bx)| \leq |C_0| + \sum_{j=1}^D \|\bC_j\|_{\frob} d^{j/2} \|\bW_*^\sT \bz\|_2^j.
    \]
    Fix \(\eta > |C_0|\) and define
    \[
        \rho_\eta := \min_{1 \le j \le D}\left( \frac{\eta - |C_0|}{D \|\bC_j\|_{\frob}} \right)^{1/j} > 0.
    \]
    Then \(|f(\bW_*^\sT \bx)| \leq \eta\) whenever \(\|\bW_*^\sT \bz\|_2 \leq \rho_\eta /\sqrt{d}\). Equivalently,
    \[
        \cE_f(\eta) := \{|f(\bW_*^\sT \bx)| > \eta\} \subseteq \left\{ \|\bW_*^\sT \bz\|_2^2 > \rho_\eta^2 /d \right\}.
    \]

    Let \(\bW\in \Stf_{\s - s_\bU}(\R^{d_\bU})\) be an orthonormal basis for the image of \(\bU_\perp^\sT \bW_\ast\), and consider the decomposition
    \[
        \bW_\ast^\sT \bz = \bU\br_\bU + \sqrt{1-\|\br_\bU\|_2^2}\bV \bW^\sT \bz_\bU,
    \]
    where \(\bV \in \R^{\s\times (\s-\s_\bU)}\) such that \(\bU_\perp^\sT \bW_\ast = \bW \bV^\sT\). Then,
    \[
        \|\bW_*^\sT \bz\|_2^2 \leq \|\br_\bU\|_2^2 + (1-\|\br_\bU\|_2^2)\|\bV\|_{\op}^2 \|\bW^\sT \bz_\bU\|_2^2.
    \]
    Fix \(b \in (0,\rho_\eta^2)\) and define \(\cE_{\br_\bU}(b) := \{\|\br_\bU\|_2^2 \le b /d\}\). On the event \(\cE_f(\eta) \cap \cE_{\br_\bU}(b)\) we therefore have
    \[
        \|\bW^\sT \bz_\bU\|_2^2 \geq \frac{\rho_\eta^2 - b}{(d-b)\|\bV\|_{\op}^2},
    \]
    which is bounded below by a strictly positive constant independent of \(d\).

    Since \(f(\bW_*^\sT \bx)\) is not observed directly, define \(\cE_y(\eta) := \{|y| > \eta\}\) and \(\cE_\epsilon(\eta) := \{|\epsilon| \leq \eta\}\). Then \(\cE_f(2\eta) \cap \cE_\epsilon(\eta) \subseteq \cE_y(\eta)\). Using Lemma~\ref{lem:variational-characterization-second-harmonic-reduced-model} with the test function
    \[
        G(\by_\bU) = \bI_{\s-s_\bU}\ind_{\cE_y(\eta)\cap \cE_{\br_\bU}(b)},
    \]
    we obtain
    \[
        \|\bxi_{\bU,2}(\by)\|_{L^2}^2 \geq \frac{(d_\bU+2)d_\bU}{2(\s-\s_\bU)}\frac{\E\left[\left(\|\bW^\sT \bz_\bU\|_2^2 - \frac{\s-\s_\bU}{d_\bU}\right)\ind_{\cE_y(\eta)\cap \cE_{\br_\bU}(b)}\right]^2}{\P(\cE_y(\eta)\cap \cE_{\br_\bU}(b))}.
    \]
    Using independence of \(\epsilon\) and \(\bx\), together with the lower bound above, yields
    \[
        \|\bxi_{\bU,2}(\by)\|_{L^2}^2 \geq \frac{(d_\bU+2)d_\bU}{2(\s-\s_\bU)}\left(\frac{\rho_\eta^2 - b}{(d-b)\|\bV\|_{\op}^2} - \frac{\s-\s_\bU}{d_\bU}\right)^2\P(\cE_f(2\eta)\cap \cE_{\br_\bU}(b))^2\P(\cE_\epsilon(\eta))^2.
    \]
    Here, we choose \(\eta\) large enough so that \(\P(\cE_\epsilon(\eta)) > 0\) and so that the lower bound above is positive. 

    On the other hand, to lower bound \(\P(\cE_f(2\eta)\cap \cE_{\br_\bU}(b))\), choose \(\bu\in\S^{\s-1}\) such that \(\bV^\sT \bu\neq 0\) (i.e., \(\bu\) has a nontrivial component in the unrecovered subspace) and such that \(t\mapsto f(t\bu)\) is non-constant. Such a vector exists since we assume that the response \(y\) depends on the residual direction \(\bz_\bU\). Define \(\ba:=\bW\bV^\sT \bu\in\R^{d_\bU}\), so that \(\ba\neq 0\). By rotational invariance of \(\bz_\bU\sim\tau_{d_\bU}\),
    \[
        \langle \bu,\bV\bW^\sT\bz_\bU\rangle
        =\langle \ba,\bz_\bU\rangle
        \stackrel{d}{=}\|\ba\|_2 z,
        \qquad z\sim\tilde \tau_{d_\bU,1}.
    \]
    In particular, for every \(c>0\) and all sufficiently large \(d\) (depending on \(\|\ba\|_2\) and \(c\)), the event  
    \(\cE_\bu(c):=\{\langle \bu,\bV\bW^\sT\bz_\bU\rangle \ge c\}\)  
    has probability bounded below by a constant independent of \(d\). On \(\cE_\bu(c)\cap \cE_{\br_\bU}(b)\) we have
    \[
        \langle \bu,\bW_*^\sT\bx\rangle
        \ge c(1+o(1))-\sqrt b .
    \]
    Choosing \(c\) large enough and using that \(|f(t\bu)|\to\infty\) as \(t\to\infty\), it follows that  
    \(\cE_\bu(c)\cap \cE_{\br_\bU}(b)\subseteq \cE_f(2\eta)\) for all large \(d\). Consequently,
    \[
        \P(\cE_f(2\eta)\cap \cE_{\br_\bU}(b))
        \ge \P(\cE_\bu(c)\cap \cE_{\br_\bU}(b))>0
    \]
    uniformly for all sufficiently large \(d\). The last inequality notably uses the independence of \(\br_\bU\) and \(\bz_\bU\).
\end{proof}

\clearpage

\section{Additional technical background}\label{app:tech-bg}

In this appendix, we collect technical background needed to state and prove our main results. For completeness, we provide proofs for less standard results on symmetric traceless tensors and their relations to spherical harmonics.  

% In this appendix, we collect some additional technical background on symmetric traceless tensors and their relation to spherical harmonics.

\subsection{Symmetric and traceless tensors}
\label{app:symmetric-traceless-tensor}

%This section is a self-contained introduction to symmetric and traceless tensors, with an emphasis on the explicit constructions of the projections onto these subspaces and some decompositions that are useful in the sequel.

%Due to their relation with spherical harmonics, we are interested in properties of traceless symmetric tensors. 

For \(\ell,d\in \N\), recall that we denote by \((\R^d)^{\otimes \ell}\) the space of order \(\ell\) tensors over \(\R^d\) endowed with the Frobenius inner product. \(\Sym_\ell(\R^d)\subseteq (\R^d)^{\otimes \ell}\) is the subspace of symmetric tensors, and \(\psym: (\R^d)^{\otimes \ell} \to \Sym_\ell(\R^d)\) is the orthogonal projection onto this subspace defined by~\eqref{eq:projection_symmetric_tensors}. Note that \(\psym\) is self-adjoint with respect to the Frobenius inner product.

The subspace \(\TSym_{\ell}(\R^d)\subseteq \Sym_\ell(\R^d)\) denotes the space of traceless symmetric tensors for which the partial trace (see~\eqref{eq:partial_trace_operator}) is zero. The orthogonal projection onto this subspace is explicitly given by
\begin{equation}\label{eq:projection_traceless_symmetric_tensors}
    \ptf(\bA) = \sum_{j=0}^{\lfloor \ell/2\rfloor} h^{(d)}_{\ell,j} \psym(\tau^{j}(\bA) \otimes \bI_d^{\otimes j}),
\end{equation}
where \(\bI_d\) is the identity matrix in \(\R^{d\times d}\), \(\tau^j\) is the \(j\)-th composition of \(\tau\) with itself, and the coefficients \(h^{(d)}_{\ell,j}\) are defined recursively by
\begin{equation}\label{eq:coefficients_projection_traceless_recursive}
    h^{(d)}_{\ell,0} = 1, \qquad h^{(d)}_{\ell,j} = -\frac{(\ell - 2j + 2)(\ell - 2j + 1)}{2j(d + 2\ell - 2j - 2)} h^{(d)}_{\ell,j-1}, \quad j\geq 1.
\end{equation}
Unrolling this recursion, we get the closed-form expression
\begin{equation}\label{eq:coefficients_projection_traceless}
    h^{(d)}_{\ell,j} = (-1)^j \frac{(d + 2\ell - 2j - 4)!!}{(d + 2\ell - 4)!!} \frac{\ell!}{2^j j! (\ell - 2j)!}
\end{equation}
for any \(j=0,1,\ldots, \lfloor \ell/2\rfloor\). 

%It is not obvious from the definition that \(\ptf\) is indeed a projection onto \(\TSym_\ell(\R^d)\), as well as self-adjoint with respect to the Frobenius inner product. The following lemma shows that \(\ptf\) is indeed a projection onto \(\TSym_\ell(\R^d)\).

\begin{lemma}\label{lem:projection_traceless_definition}
    For any \(\ell,d\in \N\), the operator \(\ptf: \Sym_\ell(\R^d) \to \TSym_\ell(\R^d)\) defined by~\eqref{eq:projection_traceless_symmetric_tensors} is an orthogonal projection onto \(\TSym_\ell(\R^d)\) with respect to the Frobenius inner-product.
\end{lemma}
\begin{proof} First, $\ptf$ is self-adjoint: for any $\bA,\bB \in \Sym_\ell (\R^d)$,
\[
\begin{aligned}
    \< \ptf (\bA), \bB \>_\frob  =&~ \sum_{j=0}^{\lfloor \ell/2\rfloor} h^{(d)}_{\ell,j} \< \tau^{j}(\bA) \otimes \bI_d^{\otimes j}, \bB \>_\frob  \\
    =&~ \sum_{j=0}^{\lfloor \ell/2\rfloor} h^{(d)}_{\ell,j} \< \tau^{j}(\bA), \tau^j (\bB)\>_\frob  
    = \sum_{j=0}^{\lfloor \ell/2\rfloor} h^{(d)}_{\ell,j} \< \bA, \tau^j (\bB) \otimes \bI_d^{\otimes j} \>_\frob  = \< \bA, \ptf(\bB)\>_\frob .
\end{aligned}
\]

Let us next show that $\ptf(\bA) \in \TSym_\ell (\R^d)$ for all $\bA \in \Sym_\ell (\R^d)$. For \(\ell=0,1\), we have \(\ptf(\bA) = \bA\) and the result is trivial. Suppose that \(\ell\geq 2\) and define the homogeneous polynomial
    \[
        p_{\ptf(\bA)}(\bx) = \<\ptf(\bA), \bx^{\otimes \ell}\>_\frob .
    \]
    Since \(\bx^{\otimes \ell}\) is symmetric,
    \begin{align*}
        p_{\ptf(\bA)}(\bx) &= \sum_{j=0}^{\lfloor \ell/2\rfloor} h^{(d)}_{\ell,j} \<\tau^j(\bA) \otimes \bI_d^{\otimes j}, \bx^{\otimes \ell}\>_\frob   
        = \sum_{j=0}^{\lfloor \ell/2\rfloor} h^{(d)}_{\ell,j} \<\tau^j(\bA), \bx^{\otimes (\ell - 2j)}\>_\frob  \|\bx\|_2^{2j}.
    \end{align*}
    Taking the Laplacian of this polynomial, \begin{equation}\label{eq:laplacian_projection_traceless}
        \Delta p_{\ptf(\bA)}(\bx) = \sum_{j=0}^{\lfloor \ell/2\rfloor} h^{(d)}_{\ell,j} \Delta\left(\<\tau^j(\bA), \bx^{\otimes (\ell - 2j)}\>_\frob  \|\bx\|_2^{2j}\right)
    \end{equation}
    with
    \begin{align*}
        &~\Delta\left(p_{\tau^j(\bA)}(\bx)\|\bx\|_2^{2j}\right)  \\
        =&~ \sum_{i=1}^d \partial_{x_i}^2 \left(p_{\tau^j(\bA)}(\bx)\|\bx\|_2^{2j}\right)
        \\  =&~ \Delta p_{\tau^j(\bA)}(\bx) \|\bx\|_2^{2j} + 2 \sum_{i=1}^d \partial_{x_i} p_{\tau^j(\bA)}(\bx) \partial_{x_i} \|\bx\|_2^{2j} + p_{\tau^j(\bA)}(\bx) \Delta \|\bx\|_2^{2j}
        \\  =&~ \Delta p_{\tau^j(\bA)}(\bx) \|\bx\|_2^{2j} + 4j \<\nabla p_{\tau^j(\bA)}(\bx), \bx\> \|\bx\|_2^{2(j-1)} + 2j(d+2(j-1)) p_{\tau^j(\bA)}(\bx) \|\bx\|_2^{2(j-1)},
    \end{align*}
    where we denoted \(p_{\tau^j(\bA)}(\bx) = \<\tau^j(\bA), \bx^{\otimes (\ell - 2j)}\>_\frob \). Since \(p_{\tau^j(\bA)}\) is a homogeneous polynomial of degree \(\ell - 2j\), it follows from Euler's homogeneous function theorem that
    \[
        \<\nabla p_{\tau^j(\bA)}(\bx), \bx\> = (\ell - 2j) p_{\tau^j(\bA)}(\bx).
    \]
    Substituting this into~\eqref{eq:laplacian_projection_traceless}, we get
    \begin{align*}
       &~ \Delta p_{\ptf(\bA)}(\bx) \\
       =&~ \sum_{j=0}^{\lfloor \ell/2\rfloor} h^{(d)}_{\ell,j} \left(\Delta p_{\tau^j(\bA)}(\bx) \|\bx\|_2^{2j} + 4j (\ell - 2j) p_{\tau^j(\bA)}(\bx) \|\bx\|_2^{2(j-1)} + 2j(d+2(j-1)) p_{\tau^j(\bA)}(\bx) \|\bx\|_2^{2(j-1)}\right)
        \\  =&~ \sum_{j=0}^{\lfloor \ell/2\rfloor} h^{(d)}_{\ell,j} \Delta p_{\tau^j(\bA)}(\bx) \|\bx\|_2^{2j} + \sum_{j=1}^{\lfloor \ell/2\rfloor} 2j h^{(d)}_{\ell,j} (2(\ell - 2j) + d + 2(j-1)) p_{\tau^j(\bA)}(\bx) \|\bx\|_2^{2(j-1)}.
    \end{align*}
    We can relate \(\Delta p_{\tau^j(\bA)}\) to \(p_{\tau^{j+1}(\bA)}\) by noting that
    \begin{align*}
        \frac{\partial^2}{\partial x_i^2} p_{\tau^j(\bA)}(\bx) & = \frac{\partial^2}{\partial x_i^2} \<\tau^j(\bA), \bx^{\otimes (\ell - 2j)}\>_\frob 
        \\ & = \frac{\partial}{\partial x_i} \<\tau^j(\bA), (\ell - 2j) \bx^{\otimes (\ell - 2j - 1)} \otimes \be_i\>_\frob 
        \\ & = (\ell - 2j)(\ell - 2j - 1) \<\tau^j(\bA), \bx^{\otimes (\ell - 2j - 2)} \otimes \be_i \otimes \be_i\>_\frob .
    \end{align*}
    Taking the sum over \(i\in [d]\) and noticing that \(\sum_{i=1}^d \be_i \otimes \be_i = \bI_d\), we get
    \begin{equation}\label{eq:laplacian_polynomial_contracts}
        \Delta p_{\tau^j(\bA)}(\bx) = (\ell - 2j)(\ell - 2j - 1) \<\tau^j(\bA), \bx^{\otimes (\ell - 2j - 2)} \otimes \bI_d\>_\frob  = (\ell - 2j)(\ell - 2j - 1) p_{\tau^{j+1}(\bA)}(\bx).
    \end{equation}
    Substituting this into the previous expression, we obtain
    \begin{align*}
       &~ \Delta p_{\ptf(\bA)}(\bx) \\
       & = \sum_{j=0}^{\lfloor \ell/2\rfloor} h^{(d)}_{\ell,j} (\ell - 2j)(\ell - 2j - 1) p_{\tau^{j+1}(\bA)}(\bx) \|\bx\|_2^{2j} + 2j h^{(d)}_{\ell,j} (2(\ell - 2j) + d + 2(j-1)) p_{\tau^j(\bA)}(\bx) \|\bx\|_2^{2(j-1)}
        \\ & = \sum_{j=1}^{\lfloor \ell/2\rfloor} \left(h^{(d)}_{\ell,j-1} (\ell - 2(j-1))(\ell - 2(j-1) - 1) + 2j h^{(d)}_{\ell,j} (2(\ell - 2j) + d + 2(j-1))\right) p_{\tau^j(\bA)}(\bx) \|\bx\|_2^{2(j-1)}.
    \end{align*}
    The coefficients \(h^{(d)}_{\ell,j}\) defined by~\eqref{eq:coefficients_projection_traceless_recursive} satisfy the recurrence relation
    \[
        h^{(d)}_{\ell,j-1} (\ell - 2(j-1))(\ell - 2(j-1) - 1) + 2j h^{(d)}_{\ell,j} (2(\ell - 2j) + d + 2(j-1)) = 0
    \]
    for any \(j\in \{1,\ldots,\lfloor \ell/2\rfloor\}\). Therefore, \(\Delta p_{\ptf(\bA)}(\bx) = 0\) for all \(\bx\in \R^d\). By~\eqref{eq:laplacian_polynomial_contracts}, this in particular implies that
    \[
        0 = \Delta p_{\ptf(\bA)}(\bx) = \ell (\ell - 1) \<\tau(\ptf(\bA)), \bx^{\otimes (\ell - 2)}\>_\frob 
    \]
    for all \(\bx\in \R^d\). Since rank one tensors span the space of symmetric tensors, it follows that \(\tau(\ptf(\bA)) = 0\), and hence \(\ptf(\bA)\in \TSym_\ell(\R^d)\).

    Furthermore, if \(\bA\in \TSym_\ell(\R^d)\), then \(\tau^j(\bA) = 0\) for all \(j\geq 1\), and hence \(\ptf(\bA) = h^{(d)}_{\ell,0} \bA = \bA\). This concludes the proof that \(\ptf\) is an orthogonal projection onto \(\TSym_\ell(\R^d)\).
\end{proof}

It will be useful to extend the definition of \(\ptf\) to non-symmetric tensors. In that case, we will write \(\ptf(\bA) = \ptf(\psym(\bA))\) for any \(\bA \in (\R^d)^{\otimes \ell}\). This is still an orthogonal projection onto \(\TSym_\ell(\R^d)\) as the composition of two nested orthogonal projections. 

Symmetric tensors can be decomposed into a direct sum of traceless symmetric tensors, which corresponds to the semisimple decomposition of $\Sym_\ell (\R^d)$---seen as a $\cO_d$-representation---into irreducible representations:
\begin{equation}\label{eq:decomposition-tensor-irreducibles}
    \Sym_\ell (\R^d) \cong \bigoplus_{j = 0}^{\lfloor \ell/2 \rfloor} \TSym_{\ell - 2j} (\R^d).
\end{equation}

\begin{lemma}[Fischer decomposition]\label{lem:fischer_decomposition}
    For any \(\ell,d\in \N\) and \(\bA\in \Sym_\ell(\R^d)\), we have the decomposition
    \begin{equation}
        \bA
        = \sum_{j=0}^{\lfloor \ell/2\rfloor}
            f^{(d)}_{\ell,j}\psym\left(\ptf(\tau^j(\bA)) \otimes \bI_d^{\otimes j}\right),
    \end{equation}
    where
    \begin{equation}
        f^{(d)}_{\ell,j} = \frac{\ell!}{2^{2j}j!(\ell-2j)!(d/2+\ell-2j)_j}.
    \end{equation}
\end{lemma}

\begin{proof}
    By induction on \(\ell\), it is straightforward to show that there exist constants \(\{f^{(d)}_{\ell,j}\}_{j=0}^{\lfloor \ell/2\rfloor}\) such that
    \[
        \bA
        = \sum_{j=0}^{\lfloor \ell/2\rfloor}
            f^{(d)}_{\ell,j}\psym\left(\ptf(\tau^j(\bA)) \otimes \bI_d^{\otimes j}\right)
    \]
    holds for all \(\bA\in \Sym_\ell(\R^d)\). Plugging this into~\eqref{eq:projection_traceless_symmetric_tensors}, we obtain
    \begin{align*}
        \ptf(\bA)  = \sum_{j=0}^{\lfloor \ell/2\rfloor} h^{(d)}_{\ell,j} \psym\left(\tau^j(\bA) \otimes \bI_d^{\otimes j}\right)
        & = \sum_{j=0}^{\lfloor \ell/2\rfloor} h^{(d)}_{\ell,j} \sum_{i=0}^{\lfloor \ell/2-j\rfloor}f_{\ell-2j,i}^{(d)} \psym\left(\ptf(\tau^{j+i}(\bA)) \otimes \bI_d^{\otimes (j+i)}\right)
        \\ & = \sum_{k=0}^{\lfloor \ell/2\rfloor}\left(\sum_{j=0}^{k}h^{(d)}_{\ell,j}f_{\ell-2j,k-j}^{(d)}\right)\psym\left(\ptf(\tau^{k}(\bA)) \otimes \bI_d^{\otimes k}\right).
    \end{align*}
    Therefore, if we match the coefficients with the trivial equation \(\ptf(\bA) = h_{\ell,0}^{(d)}\psym(\ptf(\bA))\), we get that
    \[
        \sum_{j=0}^{k}h^{(d)}_{\ell,j}f_{\ell-2j,k-j}^{(d)} = 
        \begin{cases*}
            1 & if $k = 0$ \\
            0 & if $k \geq 1$.
        \end{cases*}
    \]
    This gives the recurrence relation
    \[
    f_{\ell,0}^{(d)} = 1,\qquad f_{\ell,k}^{(d)}=-\sum_{j=1}^{k}h^{(d)}_{\ell,j}f_{\ell-2j,k-j}^{(d)}
    \]
    for the coefficients \(\{f^{(d)}_{\ell,j}\}_{j=0}^{\lfloor \ell/2\rfloor}\) in terms of \(\{h^{(d)}_{\ell,j}\}_{j=0}^{\lfloor \ell/2\rfloor}\), from which we obtain the closed form expression stated in the lemma.
\end{proof}

The classical Fischer decomposition, named after~\cite{fischer1918differentiationsprozesse}, provides a decomposition of homogeneous polynomials in terms of harmonic polynomials and radial components. In view of the isomorphism between harmonic polynomials and traceless symmetric tensors (see Lemma~\ref{lem:isometry_spherical_harmonics}), Lemma~\ref{lem:fischer_decomposition} is simply the tensor analogue of the classical Fischer decomposition.

The projection \(\ptf\) simplifies on rank-one tensors $\bx^{\otimes \ell}$, where \(\bx \in \R^d\) is an arbitrary vector. From the definition of \(\ptf\) in~\eqref{eq:projection_traceless_symmetric_tensors} and using that \(\tau^j(\bx^{\otimes \ell}) = \|\bx\|_2^{2j} \bx^{\otimes (\ell - 2j)}\),
\begin{equation}\label{eq:projection_traceless_rank_one}
    \ptf(\bx^{\otimes \ell}) = \sum_{j=0}^{\lfloor \ell/2\rfloor} h^{(d)}_{\ell,j} \|\bx\|_2^{2j} \psym(\bx^{\otimes (\ell - 2j)} \otimes \bI_d^{\otimes j}).
\end{equation}
Since \(\ptf\) is an orthogonal projection, we have \( \<\ptf(\bx^{\otimes \ell}), \bx^{\otimes \ell}\>_\frob = \|\ptf(\bx^{\otimes \ell})\|_\frob^2\). Using~\eqref{eq:projection_traceless_rank_one} and \( \<\bI_d^{\otimes j}, \bx^{\otimes 2j}\>_\frob = \<\bI_d, \bx^{\otimes 2}\>^j_\frob = \|\bx\|_2^{2j}\), we can express this quantity as
\begin{equation}\label{eq:norm_projection_traceless_rank_one}
    \|\ptf(\bx^{\otimes \ell})\|_\frob^2  = \sum_{j=0}^{\lfloor \ell/2\rfloor} h^{(d)}_{\ell,j} \| \bx\|_2^{2j} \<\bx^{\otimes (\ell - 2j)} \otimes \bI_d^{\otimes j}, \bx^{\otimes \ell}\>_\frob
    =  \sum_{j=0}^{\lfloor \ell/2\rfloor} h^{(d)}_{\ell,j} \|\bx\|_2^{2\ell }.
\end{equation}
In particular, if \(\|\bx\|_2 = 1\), then
\begin{equation}\label{eq:definition_kappa}
    \|\ptf(\bx^{\otimes \ell})\|_\frob^2 = \sum_{j=0}^{\lfloor \ell/2\rfloor} h^{(d)}_{\ell,j}=1/\kappa_{d,\ell}^2
\end{equation}
Finally, the Fischer decomposition in Lemma~\ref{lem:fischer_decomposition} applied to rank-one tensors gives
\begin{equation}\label{eq:fischer_decomp_rank_one}
    \bx^{\otimes \ell}
    = \sum_{j=0}^{\lfloor \ell/2\rfloor}
        f^{(d)}_{\ell,j}\|\bx\|_2^{2j}
        \psym\left(
            \ptf(\bx^{\otimes (\ell - 2j)}) \otimes \bI_d^{\otimes j}
        \right),
\end{equation}
where the coefficients \(f^{(d)}_{\ell,j}\) are defined as in Lemma~\ref{lem:fischer_decomposition}.

\subsection{Orthogonal group action on tensors}

The projection operators \(\psym\) and \(\ptf\) defined in~\eqref{eq:projection_symmetric_tensors} and~\eqref{eq:projection_traceless_symmetric_tensors}, as well as the partial trace operator \(\tau\) defined in~\eqref{eq:partial_trace_operator}, interact nicely with the action of the orthogonal group \(\cO_d\) on tensors.

\begin{lemma}\label{lem:equivariant_operators}
    The partial trace \(\tau: \Sym_\ell(\R^d) \to \Sym_{\ell-2}(\R^d)\) and projection operators \(\psym: (\R^d)^{\otimes \ell} \to \Sym_\ell(\R^d)\) and \(\ptf: \Sym_\ell(\R^d) \to \TSym_\ell(\R^d)\) are equivariant with respect to the action of \(\cO_d\).
\end{lemma} 

\begin{proof}
    The action of \(\cO_d\) on tensors commutes with permutations of indices and with contractions. Hence \(\psym\), being the averaging operator over all permutations,
    satisfies \(\psym(g\cdot \bA) = g\cdot \psym(\bA)\), and \(\tau\), being a
    contraction with the invariant tensor \(\bI_d\), satisfies \(\tau(g\cdot \bA)
    = g\cdot \tau(\bA)\) for all \(g\in\cO_d\). Finally, \(\ptf\) is a finite linear combination of operators of the form \(\bA \mapsto \psym(\tau^j(\bA)\otimes \bI_d^{\otimes j})\), each of which is equivariant since \(\bI_d\) is \(\cO_d\)-invariant and \(\tau\), \(\psym\) are equivariant. Therefore \(\ptf(g\cdot \bA) = g\cdot \ptf(\bA)\) for all \(g\in\cO_d\), completing the proof.
\end{proof}

For $\bW \in \Stf_\s (\R^d)$, define \(\Sym^{\bW}_\ell(\R^d)\subseteq \Sym_\ell(\R^d)\) as the subspace of symmetric tensors that are fixed by the stabilizer subgroup \(\cO_d^\bW\) defined in~\eqref{eq:stabilizer_subgroup}, that is,
\[
    \Sym^\bW_\ell(\R^d) := \{\bA \in \Sym_\ell(\R^d) : g\cdot \bA = \bA \text{ for all } g\in \cO_d^\bW\}.
\]
The analogous subspace of traceless symmetric tensors is denoted \(\TSym^\bW_\ell(\R^d)\subseteq \TSym_\ell(\R^d)\). For \(\s=0\), \(\Sym_\ell^\bW(\R^d)\) is the the subspace of rotationally invariant symmetric tensors, which by~\cite[Theorem 5.3.3]{Goodman_Wallach_2009} or Lemma~\ref{lem:decomposition_invariant_tensors} below is either the trivial subspace $\{0\}$ (if $\ell$ is odd) or the subspace spanned by $\psym (\bI_d^{\otimes p})$ (if $\ell = 2p$ is even).

The subspace \(\Sym^\bW_\ell(\R^d)\) admits a convenient decomposition that separates the components of the tensor according to how they interact with the subspace spanned by the frame \(\bW\) and its orthogonal complement. The following lemma makes this precise.

\begin{lemma}\label{lem:decomposition_invariant_tensors}
    Let \(\s\geq 0\), \(d\geq \s\), and \(\ell\geq 0\) be integers. Let \(\bW\in \Stf_{\s}(\R^d)\) be an orthonormal \(\s\)-frame and let \(\bA\in \Sym^\bW_\ell(\R^d)\) be an arbitrary symmetric tensor invariant under the action of the stabilizer subgroup \(\cO_d^\bW\). Then, there exist tensors \(\bB_{\ell - 2j}\in \Sym_{\ell - 2j}(\R^d)\), \(j=0,\ldots,\lfloor \ell/2\rfloor\) such that
    \[
        \bA = \sum_{j=0}^{\lfloor \ell/2\rfloor} \psym(((\bW\bW^\sT)^{\otimes (\ell-2j)} \bB_{\ell-2j}) \otimes (\bI_d - \bW\bW^\sT)^{\otimes j}).
    \]
\end{lemma}

\begin{proof}
    Let us assume without loss of generality that \(\bW = [\be_1,\ldots,\be_s]\) where \(\{\be_i\}_{i=1}^d\) is the canonical basis of \(\R^d\). Decompose \(\R^d = U \oplus V\) where \(U = \Span(\be_1,\ldots,\be_\s)\) and \(V = \Span(\be_{\s+1},\ldots,\be_d)\).

    For every \(k\in \{0,\ldots,\ell\}\), let \(\bA^{(k)}\) be the component of \(\bA\) that has exactly \(\ell-k\) indices in \(U\) and \(k\) indices in \(V\), i.e.,
    \[
        \bA^{(k)}_{i_1,\ldots,i_\ell} = \begin{cases}
            \bA_{i_1,\ldots,i_\ell} & \text{if exactly } \ell-k \text{ of the indices } i_1,\ldots,i_\ell \text{ are in } [\s], \\
            0 & \text{otherwise},
        \end{cases}
    \]
    and write \(\bA = \sum_{k=0}^{\ell}\bA^{(k)}\). Because \(\bA\in \Sym^\bW_\ell(\R^d)\), the collection of tensors \(\{\bA^{(k)}\}_{k=0}^\ell\) are supported on disjoint sets of indices and the action of \(\cO_d^\bW\) preserves the number of indices in \(U\) and \(V\), it follows that \(\bA^{(k)} \in \Sym^\bW_\ell(\R^d)\) for each \(k\in \{0,\ldots,\ell\}\).

    Let us fix \(k\in \{0,\ldots,\ell\}\) and consider the component \(\bA^{(k)}\). Fix \(\{\bu_{i}\}_{i\in [\ell-k]}\subseteq U\) and consider the multilinear form
    \[
        F_{\bu_{1:\ell-k}}: \bv_{1:k} \in V^{k} \mapsto \<\bA^{(k)}, \bu_1 \otimes \cdots \otimes \bu_{\ell-k} \otimes \bv_1 \otimes \cdots \otimes \bv_{k}\>_\frob,
    \]
    where \(\bu_{1:\ell-k} = (\bu_1,\ldots,\bu_{\ell-k})\). By construction, \( F_{\bu_{1:\ell-k}}\) is symmetric in \(\bv_{1:k}\). Furthermore, for any \(g\in O(V)\) we can extend \(g\) to an element of \(\cO_d^\bW\) by letting it act as trivially on \(U\). Since \(\bA^{(k)}\) is invariant under the action of \(\cO_d^\bW\), it follows that \( F_{\bu_{1:\ell-k}}\) is \(O(V)\)-invariant. By~\cite[Theorem 5.3.3]{Goodman_Wallach_2009}, we have the following characterization of \(F_{\bu_{1:\ell-k}}\) which depends on the parity of \(k\):
    \begin{enumerate}[label=(\roman*)]
        \item If \(k\) is odd, then \( F_{\bu_{1:\ell-k}} \equiv 0\) for all choices of \(\bu_{1:\ell-k}\). This implies that \(\bA^{(k)} \equiv 0\).
        \item If \(k\) is even, then there exists a scalar \(C\equiv C(\bu_{1:\ell-k})\) such that
        \[
            F_{\bu_{1:\ell-k}} =C\psym\left((\bI_d - \bW\bW^\sT)^{\otimes (k/2)}\right).
        \]
    \end{enumerate}
    Suppose that \(k\) is even. The map
    \[
        C:\bu_{1:\ell-k}\in U^{\ell-k}\mapsto C(\bu_{1:\ell-k})
    \]
    is a symmetric multilinear form on \(U\). Hence, there exists a unique tensor \(\bC_{\ell - k}\in \Sym_{\ell - k}(U)\) such that
    \[
        C(\bu_{1:\ell-k}) = \<\bC_{\ell - k}, \bu_1 \otimes \cdots \otimes \bu_{\ell-k}\>_\frob.
    \]
    We extend \(\bC_{\ell - k}\) to a tensor in \(\Sym_{\ell - k}(\R^d)\) supported on indices in \([\s]\) by setting
    \[
        \bC_{\ell - k} = (\bW\bW^\sT)^{\otimes (\ell - k)} \bC_{\ell - k}.
    \] 
    Now define
    \[
        \tilde{\bA}^{(k)} := \psym\left(\left((\bW\bW^\sT)^{\otimes (\ell - k)} \bC_{\ell - k}\right)\otimes (\bI_d - \bW\bW^\sT)^{\otimes (k/2)}\right).
    \]
    For any \(\{\bu_i\}_{i\in [\ell-k]}\subseteq U\) and \(\{\bv_i\}_{i\in [k]}\subseteq V\), we have
    \begin{align*}
        \<\tilde{\bA}^{(k)}, \bu_{1:\ell-k} \otimes \bv_{1:k}\>_\frob 
        & = c_{\ell,k}\<\bC_{\ell - k}, \bu_1 \otimes \cdots \otimes \bu_{\ell-k}\>_\frob \<\psym\left((\bI_d - \bW\bW^\sT)^{\otimes (k/2)}\right), \bv_{1:k}\>_\frob
        \\ & = c_{\ell,k}F_{\bu_{1:\ell-k}}(\bv_{1:k})
    \end{align*}
    for some constant \(c_{\ell,k}\) that depends only on \(\ell\) and \(k\). The constant \(c_{\ell,k}\) is a combinatorial factor that links permutations of the indices \([\ell]\) to permutations that fix the subsets of indices in \(U\) and \(V\). Since both \(\bA^{(k)}\) and \(\tilde{\bA}^{(k)}\) are symmetric tensors, the previous equation implies that \( c_{\ell,k}\bA^{(k)} = \tilde{\bA}^{(k)}\). This concludes the proof.
\end{proof}

\subsection{Isomorphism between traceless symmetric tensors and spherical harmonics}
\label{app:isomorphism}

It is convenient to work with the canonical representation of spherical harmonics in terms of traceless symmetric tensors. The next lemma shows that the evaluation map in~\eqref{eq:isometry_spherical_harmonics} defines an \(\cO_d\)-equivariant isometric isomorphism between \(\TSym_\ell(\R^d)\) equipped with the Frobenius inner product and \(\sh_{d,\ell}\) equipped with the \(L^2(\tau_d)\) inner product.

\begin{lemma}\label{lem:isometry_spherical_harmonics}
    For any \(d,\ell\in \N\), the map \(\Phi_{d,\ell}\) defined by~\eqref{eq:isometry_spherical_harmonics} is an isometric isomorphism between \(\TSym_\ell(\R^d)\) and \(\sh_{d,\ell}\). Furthermore, \(\Phi_{d,\ell}\) is equivariant under the action of \(\cO_d\), i.e., for any \(g\in \cO_d\) and \(\bA\in \TSym_\ell(\R^d)\), \(\Phi_{d,\ell}(g\cdot \bA) = g\cdot \Phi_{d,\ell}(\bA)\).
\end{lemma}

\begin{proof}
    Let \(\bA\in \TSym_\ell(\R^d)\) be arbitrary. Since \(\ptf\) is self-adjoint, we can write
    \[
        \Phi(\bA)(\bx) = \kappa_{d,\ell}\sqrt{N_{d,\ell}}\<\bA,\bx^{\otimes \ell}\>_\frob
    \]
    for all \(\bx\in \R^{d}\). We will first show that \(\Phi_{d,\ell}\) is an isomorphism. To show that \(\Phi_{d,\ell}\) is well-defined, note that for any \(\bx\in \R^d\),
    \begin{align*}
        \Delta \Phi_{d,\ell}(\bA)(\bx) & = \kappa_{d,\ell}\sqrt{N_{d,\ell}}\Delta\<\bA,\bx^{\otimes \ell}\>_\frob = \kappa_{d,\ell}\sqrt{N_{d,\ell}}\ell(\ell-1)\<\tau(\bA),\bx^{\otimes (\ell-2)}\>_\frob = 0
    \end{align*}
    by~\eqref{eq:laplacian_polynomial_contracts} using that \(\bA\in \TSym_\ell(\R^d)\). Therefore, \(\Phi_{d,\ell}(\bA)\) is a harmonic polynomial and hence \(\Phi_{d,\ell}(\bA)\in \sh_{d,\ell}\) when restricted to the sphere.

    Injectivity of \(\Phi_{d,\ell}\) follows directly from the fact that \(\Sym_\ell(\R^d)\) is spanned by tensors of the form \(\bx^{\otimes \ell}\). Next, we show that \(\Phi_{d,\ell}\) is surjective. To this end, let \(p\) be any \(\ell\)-homogeneous harmonic polynomial in \(\R^d\) and define \(\bA \in \Sym_\ell(\R^d)\) entry-wise by
    \[
        \bA_{i_1,\ldots,i_\ell} = \frac{1}{\ell!}\frac{\partial^\ell p}{\partial x_{i_1} \cdots \partial x_{i_\ell}}(\boldsymbol{0}).
    \]
    Note that \(\bA\) is symmetric since partial derivatives of a polynomial commute. Furthermore,
    \[
        p(\bx) = \frac{1}{\ell!} \sum_{i_1,\ldots,i_\ell=1}^d \frac{\partial^\ell p}{\partial x_{i_1} \cdots \partial x_{i_\ell}}(\boldsymbol{0}) x_{i_1} \cdots x_{i_\ell} = \<\bA, \bx^{\otimes \ell}\>_\frob
    \]
    since \(p\) is \(\ell\)-homogeneous. Finally, using the fact that \(p\) is harmonic, it follows from~\eqref{eq:laplacian_polynomial_contracts} that
    \[
        0 = \Delta p(\bx) = \ell(\ell-1)\<\tau(\bA),\bx^{\otimes (\ell-2)}\>_\frob
    \] 
    for all \(\bx\in \R^d\), and hence \(\bA\in \TSym_\ell(\R^d)\). By construction, \(\Phi_{d,\ell}(\bA) = p\) and hence \(\Phi_{d,\ell}\) is surjective.

    Finally, we show that \(\Phi_{d,\ell}\) is an isometry. To this end, let \(\bA,\bB\in \TSym_\ell(\R^d)\) be arbitrary. Then,
    \begin{align*}
        \<\Phi_{d,\ell}(\bA), \Phi_{d,\ell}(\bB)\>_{L^2(\tau_d)} & = \kappa^2_{d,\ell}N_{d,\ell} \left\<\E_{\bx\sim \tau_d}[\bx^{\otimes 2\ell}], \bA \otimes \bB\right\>_\frob.
    \end{align*}
    By the rotational invariance of \(\tau_d\), for every \(\bQ\in \cO_d\),
    \[
        \bQ\cdot \E_{\bx\sim \tau_d}[\bx^{\otimes 2\ell}] = \E_{\bx\sim \tau_d}[(\bQ\bx)^{\otimes 2\ell}] = \E_{\bx\sim \tau_d}[\bx^{\otimes 2\ell}],
    \] 
    where we used \(\bQ^{\otimes \ell}\bx^{\otimes 2\ell} = (\bQ\bx)^{\otimes \ell}\). Hence, \(\E_{\bx\sim \tau_d}[\bx^{\otimes 2\ell}]\in \Sym_{2\ell}(\R^d)\) is invariant under the action of \(\cO_d\).
    
    By Lemma~\ref{lem:decomposition_invariant_tensors}, there exists a constant \(c\) such that $\E_{\bx\sim \tau_d}[\bx^{\otimes 2\ell}] = c \psym(\bI_d^{\otimes \ell})$, and hence
    \[
       \<\Phi_{d,\ell}(\bA), \Phi_{d,\ell}(\bB)\>_{L^2(\tau_d)}   = c\kappa^2_{d,\ell}N_{d,\ell} \left\<\psym(\bI_d^{\otimes \ell}), \bA \otimes \bB\right\>_\frob.
    \]
    Expanding inner product,
    \begin{align*}
        \left\<\psym(\bI_d^{\otimes \ell}), \bA \otimes \bB\right\>_\frob &  = \frac{1}{\ell!}\sum_{\sigma \in \perm_{2\ell}}\sum_{i_1,\ldots, i_{2\ell}}\prod_{j=1}^{\ell}\delta_{i_{\sigma(j)},i_{\sigma(\ell+j)}}
        \bA_{i_1,\ldots, i_\ell}\bB_{i_{\ell+1},\ldots, i_{2\ell}}.
    \end{align*}
    If \(\sigma\in \perm_{2\ell}\) is such that there exists \(j\in [\ell]\) such that \(\{\sigma(j),\sigma(\ell+j)\}\subseteq [\ell]\) or \(\{\sigma(j),\sigma(\ell+j)\}\subseteq [2\ell]\setminus [\ell]\), then then the inner product \(\<(\bI_d^{\otimes \ell})^{\sigma}, \bA \otimes \bB\>_\frob\) will contract two indices of either \(\bA\) or \(\bB\). Since both \(\bA,\bB\in \TSym_{\ell}(\R^d)\), those terms will not contribute to the inner product. Consequently, the only permutations \(\sigma \in \perm_{2\ell}\) that contribute to the above are those that follow the natural partition of \([2\ell]\) into \([\ell]\) and \([2\ell]\setminus [\ell]\). In particular, because \(\bA,\bB\) are symmetric, this means that
    \[
        \left\<\psym(\bI_d^{\otimes \ell}), \bA \otimes \bB\right\>_\frob  = c\<\bA,\bB\>_\frob
    \]
    for a (possibly different) constant \(c>0\). Combining everything, there exists a constant \(c>0\) such that, for all \(\bA,\bB\in \TSym_\ell(\R^d)\),
    \[
         \<\Phi_{d,\ell}(\bA), \Phi_{d,\ell}(\bB)\>_{L^2(\tau_d)}   = c\kappa^2_{d,\ell}N_{d,\ell} \<\bA, \bB\>_\frob.
    \]
    This means that \(\E_{\bx\sim \tau_d}[\ptf(\bx^{\otimes \ell})\otimes \ptf(\bx^{\otimes \ell})]\in \Lop_{d,\ell}\) admits a spectral decomposition
    \[
        \E_{\bx\sim \tau_d}[\ptf(\bx^{\otimes \ell})\otimes \ptf(\bx^{\otimes \ell})] = c\sum_{j\in [N_{d,\ell}]}\bV_j\otimes \bV_j
    \]
    with \(\{\bV_j\}_{j\in [N_{d,\ell}]}\) an orthonormal basis of \(\TSym_\ell(\R^d)\). Taking a trace on both sides and using~\eqref{eq:norm_projection_traceless_rank_one},
    \[
        c N_{d,\ell} = \E_{\bx\sim \tau_d}[\|\ptf(\bx^{\otimes \ell})\|_\frob^2] = \sum_{j=0}^{\lfloor \ell/2\rfloor} h^{(d)}_{\ell,j} = \frac{1}{\kappa^2_{d,\ell}}
    \]
    where the coefficients \(h^{(d)}_{\ell,j}\) are defined in~\eqref{eq:coefficients_projection_traceless}. Thus,
    \[
        \<\Phi_{d,\ell}(\bA), \Phi_{d,\ell}(\bB)\>_{L^2(\tau_d)} =  \<\bA, \bB\>_\frob.
    \]
    This shows that \(\Phi_{d,\ell}\) is a linear bijection that preserves inner products. The equivariance property follows from the definitions of \(\Phi_{d,\ell}\) and the equivariance of \(\ptf\) (see Lemma~\ref{lem:equivariant_operators}).
\end{proof}

For any \(\s\in \N\) and orthonormal \(\s\)-frame \(\bW\in \R^{d\times \s}\), let \(\sh^{W}_{d,\ell}\subseteq \sh_{d,\ell}\) be the subspace of spherical harmonics that are invariant under the action of \(\cO_d^\bW\). Since \(\cO_d^\bW\) is a stabilizer subgroup of \(\cO_d\), it follows from Lemma~\ref{lem:isometry_spherical_harmonics} that \(\Phi_{d,\ell}\) restricts to an isometric isomorphism between \(\TSym_\ell^\bW(\R^d)\) and \(\sh_{d,\ell}^\bW\). We obtain a decomposition of \(L^2(\S^{d-1},\tau_d)\) in terms of harmonic tensors.

\begin{lemma}\label{lem:harmonic_expansion_invariant_functions}
    Let \(f\in L^2(\S^{d-1},\tau_d)\) be arbitrary. Then, we may write
    \[
        f(\bx) = \sum_{\ell\geq 0}\<\bA_{d,\ell},\cH_{d,\ell}(\bx^{\otimes \ell})\>_\frob
    \]
    for all \(\bx\in \S^{d-1}\), where the tensor coefficients \(\bA_{d,\ell}\in \TSym_\ell(\R^d)\) are uniquely determined by
    \[
        \bA_{d,\ell} = \E_{\bx\sim \tau_d}[f(\bx)\cH_{d,\ell}(\bx^{\otimes \ell})]
    \]
    for all \(\ell\geq 0\).
    
    Moreover, if \(f\) is invariant under the stabilizer subgroup \(\cO_d^\bW\) for some orthonormal \(\s\)-frame \(\bW\in \R^{d\times \s}\), then \(\bA_{d,\ell}\in \TSym_\ell^\bW(\R^d)\) for all \(\ell\geq 0\), and there exist tensors \(\bB_{\s,\ell}\in \Sym_\ell(\R^\s)\) such that \(\bA_{d,\ell} = \ptf(\bW^{\otimes \ell} \bB_{\s,\ell})\) for all \(\ell\geq 0\), where \(\bW^{\otimes \ell}\bB_{\s,\ell}\) denote the canonical embedding of \(\bB_{\s,\ell}\) into \(\Sym_\ell(\R^d)\) via \(\bW\).
\end{lemma}

\begin{proof}
    Let \(\{\psi_{j}\}_{j=1}^{n_\ell}\) be an orthonormal basis of \(\sh_{d,\ell}\) and write the expansion in \(L^2(\S^{d-1},\tau_d)\)
    \[
        f(\bx) = \sum_{\ell\geq 0} \sum_{j=1}^{n_\ell} \<f,\psi_j\>_{L^2(\tau_d)} \psi_j(\bx)
    \]
    for all \(\bx\in \S^{d-1}\). By the isometric isomorphism of Lemma~\ref{lem:isometry_spherical_harmonics}, the \(L^2(\tau_d)\)-orthonormal projection \(f_\ell\) of \(f\) onto \(\sh_{d,\ell}\) can be identified with a unique tensor \(\bA_{d,\ell}\in \TSym_\ell(\R^d)\) such that
    \[
        f_\ell(\bx) = \<\bA_{d,\ell}, \cH_{d,\ell}(\bx^{\otimes \ell})\>_\frob.
    \]
    Furthermore, by the isometry property of \(\Phi_{d,\ell}\), we have
    \begin{align*}
        \<\bA_{d,\ell},\bB\>_\frob = \<f_\ell, \Phi_{d,\ell}(\ptf(\bB))\>_{L^2(\tau_d)} 
         = \E[f(\bx)\Phi_{d,\ell}(\bB)(\bx)] = \left\<\E[f(\bx)\cH_{d,\ell}(\bx^{\otimes \ell})],\bB\right\>_\frob
    \end{align*}
    for every \(\bB\in \Sym_\ell(\R^d)\), where the expectation is taken with respect to \(\bx\sim \tau_d\). Here, we used the fact that since \(\ptf\) is an orthogonal projection, \(\Phi(\bB) = \Phi(\ptf(\bB))\) for all \(\bB\in \Sym_\ell(\R^d)\). It follows that \(\bA_{d,\ell} = \E[f(\bx)\cH_{d,\ell}(\bx^{\otimes \ell})] \in \TSym_\ell(\R^d)\). Summing over all \(\ell\geq 0\), we get the stated \(L^2\) expansion of \(f\).

    Now assume that \(f\) is \(\cO_d^{\bW}\)-invariant for some orthonormal \(\s\)-frame \(\bW\in \R^{d\times \s}\). Using the \(\cO_d\)-equivariance of \(\cH_{d,\ell}\), we get
    \begin{align*}
        \bQ\cdot \bA_{d,\ell} = \E[f(\bx)\cH_{d,\ell}((\bQ\bx)^{\otimes \ell})]  = \E[f(\bQ^\sT \bx)\cH_{d,\ell}(\bx^{\otimes \ell})] 
         = \E[f(\bx)\cH_{d,\ell}(\bx^{\otimes \ell})]
         = \bA_{d,\ell},
    \end{align*}
    so that \(\bA_{d,\ell}\in \TSym_\ell^\bW(\R^d)\) for all \(\ell\geq 0\). By Lemma~\ref{lem:decomposition_invariant_tensors}, there exist tensors \(\bC_{d,\ell}\in \Sym_\ell(\R^d)\) such that
    \[
        \bA_{d,\ell} = \sum_{j=0}^{\lfloor \ell/2\rfloor} \psym(((\bW\bW^\sT)^{\otimes (\ell-2j)} \bC_{d,\ell-2j}) \otimes (\bI_d - \bW\bW^\sT)^{\otimes j}).
    \]
    Expanding \((\bI_d - \bW\bW^\sT)^{\otimes j}\) as the product of components of the form \(\bW\bW^\sT\) and \(\bI_d\), we can define new tensors \(\tilde{\bC}_{d,\ell - 2m}\in \Sym_{\ell - 2m}(\R^d)\), \(m=0,\ldots,\lfloor \ell/2\rfloor\) such that
    \[
        \bA_{d,\ell} = \sum_{j=0}^{\lfloor \ell/2\rfloor}\psym(((\bW\bW^\sT)^{\otimes (\ell-2j)} \tilde{\bC}_{d,\ell-2j}) \otimes \bI_d^{\otimes j}).
    \]
    Taking the inner product with \(\cH_{d,\ell}(\bx^{\otimes \ell})\) for any \(\bx\in \S^{d-1}\) and using the fact that \(\cH_{d,\ell}(\bx^{\otimes \ell})\) is traceless, we get
    \[
        \<\bA_{d,\ell},\cH_{d,\ell}(\bx^{\otimes \ell})\>_\frob = \<(\bW\bW^\sT)^{\otimes \ell}\tilde{\bC}_{\ell},\cH_{d,\ell}(\bx^{\otimes \ell})\>_\frob.
    \]
    Since this is true for all \(\ell\geq 0\),
    \[
        f(\bx) = \sum_{\ell\geq 0}\<\ptf((\bW\bW^\sT)^{\otimes \ell}\tilde{\bC}_\ell),\cH_{d,\ell}(\bx^{\otimes \ell})\>_\frob
    \]
    where, by the above, \(\ptf((\bW\bW^\sT)^{\otimes \ell}\tilde{\bC}_\ell) = \bA_{d,\ell}\). In particular, the final part of the lemma holds with \(\bB_{\s,\ell} = (\bW^\sT)^{\otimes \ell}\tilde{\bC}_{d,\ell} \in \Sym_\ell(\R^\s)\).
\end{proof}

Lemma~\ref{lem:harmonic_expansion_invariant_functions} indicates that any \(\cO_d^\bW\)-invariant function \(f\in L^2(\S^{d-1},\tau_d)\) can be expanded as
\[
    f(\bx) = \sum_{\ell\geq 0}\<\ptf(\bW^{\otimes \ell}\bB_{\s,\ell}), \cH_{d,\ell}(\bx^{\otimes \ell})\>_\frob
\]
for some \(\bB_{\s,\ell}\in \Sym_\ell(\R^\s)\). Since the tensors \(\{\bB_{\s,\ell}\}_{\ell\geq 0}\) are finite dimensional (dimension does not depend on \(d\)), this expansion is more amenable to analysis in the high-dimensional limit \(d\to \infty\). For instance, we can show that \(\ptf(\bW^{\otimes \ell}\bB_{\s,\ell})\) is well approximated by \(\bW^{\otimes \ell}\bB_{\s,\ell}\) when \(d\) is large.

\begin{lemma}\label{lem:approximation_traceless_projection}
    Let \(d,\s,\ell\in \N\) and \(\bW\in \R^{d\times \s}\) be an orthonormal \(\s\)-frame. If \(d\geq 4\ell^4\s\), then for any \(\bB_{\s,\ell}\in \Sym_\ell(\R^\s)\),
    \[
        \left\|\ptf\left(\bW^{\otimes \ell}\bB_{\s,\ell}\right) - \bW^{\otimes \ell}\bB_{\s,\ell}\right\|_\frob \leq \left\|\ptf\left(\bW^{\otimes \ell}\bB_{\s,\ell}\right)\right\|_\frob \frac{2\ell^2\sqrt{\s}}{\sqrt{d}}.
    \]
\end{lemma}

\begin{proof}
    By~\eqref{eq:projection_traceless_symmetric_tensors},
    \[
        \ptf\left(\bW^{\otimes \ell}\bB_{\s,\ell}\right) - \bW^{\otimes \ell}\bB_{\s,\ell} =  \sum_{j=1}^{\lfloor \ell/2\rfloor} h^{(d)}_{\ell,j} \psym\left(\tau^{j}\left(\bW^{\otimes \ell}\bB_{\s,\ell}\right) \otimes \bI_d^{\otimes j}\right)
    \]
    where the coefficients \(h^{(d)}_{\ell,j}\) are defined explicitly in~\eqref{eq:coefficients_projection_traceless}. Taking the Frobenius norm and using the triangle inequality as well as the fact that \(\psym\) is non-expansive as an orthogonal projection,
    \begin{align*}
        \left\|\ptf\left(\bW^{\otimes \ell}\bB_{\s,\ell}\right) - \bW^{\otimes \ell}\bB_{\s,\ell}\right\|_\frob & \leq \sum_{j=1}^{\lfloor \ell/2\rfloor} |h^{(d)}_{\ell,j}| \|\tau^{j}(\bW^{\otimes \ell}\bB_{\s,\ell})\otimes \bI_d^{\otimes j}\|_\frob
         = \sum_{j=1}^{\lfloor \ell/2\rfloor} |h^{(d)}_{\ell,j}| d^{j/2} \|\tau^{j}(\bW^{\otimes \ell}\bB_{\s,\ell})\|_\frob.
    \end{align*}
    Using the definition of the partial trace operator \(\tau\) in~\eqref{eq:partial_trace_operator}, one can check that
    \[
        \tau\left(\bW^{\otimes \ell} \bB_{\s,\ell}\right) = \bW^{\otimes (\ell-2)}\tau(\bB_{\s,\ell}).
    \]
    Iterating this \(j\) times yields
    \[
        \tau^{j}\left(\bW^{\otimes \ell}\bB_{\s,\ell}\right) = \bW^{\otimes (\ell-2j)} \tau^j(\bB_{\s,\ell})
    \]
    for all \(j=1,\ldots,\lfloor \ell/2\rfloor\). Using the crude bound \(\|\tau^j(\bB_{\s,\ell})\|_\frob \leq s^{j/2} \|\bB_{\s,\ell}\|_\frob\),
    \begin{align*}
        \left\|\ptf\left(\bW^{\otimes \ell}\bB_{\s,\ell}\right) - \bW^{\otimes \ell}\bB_{\s,\ell}\right\|_\frob & \leq \|\bB_{\s,\ell}\|_\frob \sum_{j=1}^{\lfloor \ell/2\rfloor} |h^{(d)}_{\ell,j}| (ds)^{j/2}.
    \end{align*}
    Using the explicit expression for the coefficients \(h^{(d)}_{\ell,j}\) in~\eqref{eq:coefficients_projection_traceless}, we have $|h^{(d)}_{\ell,j}|  \leq \left(\ell^2/(2d)\right)^j$ for all \(j=1,\ldots,\lfloor \ell/2\rfloor\), and thus
    \begin{equation}\label{eq:bound-finite-d-W-Od}
         \left\|\ptf\left(\bW^{\otimes \ell}\bB_{\s,\ell}\right) - \bW^{\otimes \ell}\bB_{\s,\ell}\right\|_\frob \leq \|\bB_{\s,\ell}\|_\frob \sum_{j=1}^{\lfloor \ell/2\rfloor} \left(\frac{\ell^2\sqrt{\s}}{2\sqrt{d}}\right)^j
    \end{equation}
    When \(d \geq 4\ell^4 \s\), the factor on the right-hand side satisfies
    \[
        \sum_{j=1}^{\lfloor \ell/2\rfloor} \left(\frac{\ell^2\sqrt{\s}}{2\sqrt{d}}\right)^j \leq \frac{\ell^2\sqrt{\s}}{\sqrt{d}} \leq 2^{-1}.
    \]
    Therefore, from \eqref{eq:bound-finite-d-W-Od} and the triangle inequality,
    \[
        \frac{1}{2}\|\bW^{\otimes \ell}\bB_{\s,\ell}\|_\frob \leq \|\bW^{\otimes \ell}\bB_{\s,\ell}\|_\frob - \left\|\ptf\left(\bW^{\otimes \ell}\bB_{\s,\ell}\right) - \bW^{\otimes \ell}\bB_{\s,\ell}\right\|_\frob \leq \|\ptf(\bW^{\otimes \ell}\bB_{\s,\ell})\|_\frob,
    \]
    which implies that \(\|\bW^{\otimes \ell}\bB_{\s,\ell}\|_\frob \leq 2 \|\ptf(\bW^{\otimes \ell}\bB_{\s,\ell})\|_\frob\).
\end{proof}

In the regime of interest where \(d\to \infty\) with fixed \(\s,\ell\), Lemma~\ref{lem:approximation_traceless_projection} shows that \(\ptf(\bW^{\otimes \ell}\bB_{\s,\ell})\) is well approximated by \(\bW^{\otimes \ell}\bB_{\s,\ell}\) at a multiplicative rate that vanishes as \(d\to \infty\). 

\subsection{Tensor product representation}

We consider the tensor product representation \(\TSym_{p} (\R^d) \otimes \TSym_{q} (\R^d)\) under the diagonal action \(g\cdot (\bA\otimes \bB) = (g \cdot \bA)\otimes (g\cdot \bB)\), which admits the irreducible decomposition given in~\eqref{eq:semisimple-decomposition-tensor-representation}. The next lemma gives an explicit decomposition for tensors $\bA \otimes \bB$. The general case $\bC \in \TSym_{p} (\R^d) \otimes \TSym_{q} (\R^d)$ follows by linearity.

\begin{lemma}\label{lem:product-harmonic-tensors-decomposition}
    Let \(p,q \geq 0\). Then, for all \(\bA \in \TSym_{p} (\R^d)\) and \(\bB \in \TSym_{q} (\R^d)\), we have
    \[
        \<\bA \otimes \bB, \cH_{d,p}(\bz) \otimes \cH_{d,q} (\bz)\>_\frob = \sum_{j = 0}^{p\wedge q} b^{(d)}_{p,q,j}\< \bA \diamond_j \bB,\cH_{d,p+q - 2j} (\bz)\>_\frob ,
    \]
    where \(\diamond_j\) denotes the bilinear operator defined in~\eqref{eq:definition_diamond_operator}, and the scalars \(\{b^{(d)}_{p,q,j}\}_{j=0}^{p\wedge q}\) are defined as
    \[
        b^{(d)}_{p,q,j} = 
            \frac{f^{(d)}_{p+q,j}\kappa_{d,p}\kappa_{d,q}\sqrt{N_{d,p}N_{d,q}}}{\kappa_{d,p+q-2j}\sqrt{N_{d,p+q-2j}}} \frac{2^j (p!)(q!)(p+q-2j)!}{(p+q)!(p-j)!(q-j)!}
    \]
\end{lemma}

\begin{proof}
    By~\eqref{eq:projection_traceless_rank_one} and~\eqref{eq:fischer_decomp_rank_one},
    \begin{multline*}
        \<\cH_{d,p}(\bz) \otimes \cH_{d,q} (\bz),\bA \otimes \bB\>_\frob  = \kappa_{d,p}\kappa_{d,q}\sqrt{N_{d,p}N_{d,q}}\<\bz^{\otimes (p+q)},\bA \otimes \bB\>_\frob
        \\  = \sum_{j=0}^{\lfloor (p+q)/2\rfloor}\frac{f^{(d)}_{p+q,j}\kappa_{d,p}\kappa_{d,q}\sqrt{N_{d,p}N_{d,q}}}{\kappa_{d,p+q-2j}\sqrt{N_{d,p+q-2j}}}\left\<\psym\left(\cH_{d,p+q-2j}(\bz)\otimes \bI_d^{\otimes j}\right),\bA \otimes \bB\right\>_\frob.
    \end{multline*}
    Momentarily fix \(j\in [\lfloor (p+q)/2\rfloor]\) and consider the term \(\<\psym(\cH_{d,p+q-2j}(\bz)\otimes \bI_d^{\otimes j}),\bA \otimes \bB\>_\frob\). Since \(\psym\) is an orthogonal projection, we may transfer it to the second argument, yielding
    \[
        \<\psym(\cH_{d,p+q-2j}(\bz)\otimes \bI_d^{\otimes j}),\bA \otimes \bB\>_\frob=\<\cH_{d,p+q-2j}(\bz),\psym(\bA \otimes \bB)[\bI_d^{\otimes j}]\>_\frob
    \]
    where \(\psym(\bA \otimes \bB)[\bI_d^{\otimes j}]\in \Sym_{p+q-2j}(\R^d)\) denotes contraction along the last \(2j\) components of \(\psym(\bA \otimes \bB)\). Let \(G\subseteq \perm_{p+q}\) denote the set of cross-pairing permutations, that is the set of permutations \(\pi\in \perm_{p+q}\) such that for all \(k\in \{p+q-2j+1,\ldots, p+q-j\}\), we have \(|\{\pi(k), \pi(k+j)\} \cap [p]| = 1\). If \(\pi \notin G\), then there exists some \(k\in \{p+q-2j+1,\ldots, p+q-j\}\) such that both \(\pi(k), \pi(k+j)\leq p\) or both \(\pi(k), \pi(k+j) > p\), which when summed over the indices \(i_{p+q-2j+1},\ldots, i_{p+q}\) yields zero since both \(\bA\) and \(\bB\) are traceless. Therefore,
    \[
        \psym(\bA \otimes \bB)[\bI_d^{\otimes j}] = \frac{1}{(p+q)!}\sum_{\pi\in G} (\bA \otimes \bB)^{\pi}[\bI_d^{\otimes j}].
    \]
    Every \(\pi\in G\) induces a matching between \(j\) slots of \(\bA\) and \(j\) slots of \(\bB\). Since \(\bA,\bB\) are symmetric tensors, we can permute the slots of \(\bA\) and \(\bB\) independently without changing the value of the summand. In particular, there exists a permutation \(\sigma_\pi\in \perm_{p+q-2j}\) such that
    \[
        (\bA \otimes \bB)^{\pi}[\bI_d^{\otimes j}] = (\bA \otimes_j \bB)^{\sigma_\pi},
    \]
    and consequently,
    \[
        \psym\left(\psym(\bA \otimes \bB)[\bI_d^{\otimes j}]\right) = \frac{|G|}{(p+q)!} \psym(\bA \otimes_j \bB).
    \]
    It only remains to compute \(|G|\), which is done using a combinatorial argument to obtain
    \[
        |G| = \frac{(p!)(q!)2^j (p+q-2j)!}{(p-j)!(q-j)!}
    \]
    whenever \(j \leq \min(p,q)\), and \(|G|=0\) otherwise. Combining the above observations and the definition of the bilinear operator \(\diamond_j\) in~\eqref{eq:definition_diamond_operator} concludes the proof.
\end{proof}

\subsection{Hypercontractivity of matrix coefficients}
\label{app:hypercontractivity}

Below, we prove hypercontractivity for the subspace of degree-$k$ matrix coefficients of the orthogonal group $\cO_d$ acting on functions on the unit sphere. The proof relies on classical log-Sobolev and hypercontractivity results for the heat semigroup on the connected component $\cSO_d$, which we briefly recall for completeness.

The orthogonal group admits a semidirect product decomposition $\cO_d = \cSO_d \rtimes \cR$, where $\cR = \{e, \mathsf{R}\}$ with $\mathsf{R}$ any fixed reflection (i.e., $\mathsf{R}^2 = e$ and $\det \mathsf{R} = -1$ where $e$ is the identity element). Since \(\cSO_d\) is a normal subgroup of \(\cO_d\), the action representation of \(\cO_d\) on \(L^2(\S^{d-1})\) restricts to an action representation of \(\cSO_d\) on \(L^2(\S^{d-1})\) with the same semisimple decomposition in terms of spherical harmonic subspaces. In particular, each spherical harmonic subspace \(\sh_{d,\ell}\) is an irreducible representation of both \(\cO_d\) and \(\cSO_d\).

Fix \(\ell \in \naturals\). For \(f,h \in \sh_{d,\ell}\) and \(\rho\) the action representation of \(\cO_d\) on \(L^2(\S^{d-1})\), define the \emph{matrix coefficient}
\[
    \rho_{fh} (g) := \<\rho(g)\cdot f,h \>_{L^2}, \qquad g \in \cO_d,
\]
which we can think as elements of \(L^2 (\cSO_d)\) or \(L^2 (\cO_d)\). Let 
\[
\cM_{d,\ell} = {\rm span} \left\{ \rho_{fh} \; : \; f,h \in \sh_{d,\ell} \right\}
\]
be the isotypic subspace of $L^2 (\cSO_d)$ associated with the irreducible representation $\sh_{d,\ell}$.

Let $\Delta$ denote the Laplace-Beltrami operator on $\cSO_d$ associated with the standard bi-invariant Riemannian metric. Since $\cSO_d$ is a compact Lie group, the Laplace-Beltrami operator coincides with the Casimir operator acting in the left regular representation (see, e.g.,~\cite[Chapter 12]{faraut2008analysis}). As a consequence, $\Delta$ acts by a scalar on each isotypic component of $L^2(\cSO_d)$~\cite[Corollary 6.7.2]{faraut2008analysis},~\cite[Lemma 3.3.8]{Goodman_Wallach_2009}. In particular, every $F \in \cM_{d,\ell}$ is an eigenfunction of $\Delta$ with eigenvalue equal to minus the Casimir eigenvalue of $\sh_{d,\ell}$~\cite[Proposition 8.2.1]{faraut2008analysis}. Since the highest weight of $\sh_{d,\ell}$ is $\lambda = (\ell, 0, \dots, 0)$, it follows from standard formulas for Casimir eigenvalues~\cite[Proposition 12.1.2]{faraut2008analysis} that
\begin{equation}\label{eq:casimir_eigenvalues}
\Delta F = - \lambda_{d,\ell} F, \qquad F \in \cM_{d,\ell}, \qquad \lambda_{d,\ell} = \ell (\ell+ d -2),
\end{equation}

Let \((P_t)_{t \geq 0} = (e^{t\Delta})_{t \geq 0}\) denote the associated heat semigroup. A classical result states that \(\cSO_d\) satisfies a logarithmic Sobolev inequality (LSI) with constant \( c_{\mathrm{LS}}(\cSO_d) = \frac{4}{d-2}\) for \(d\geq 3\). This follows from the Bakry–Émery curvature–dimension criterion~\cite[Theorem 6.8.1]{bakry2013analysis}. Indeed, with respect to the metric induced by the negative of the Killing form, the Ricci curvature of \(\cSO_d\) is lower bounded by \(1/4\)~\cite[Theorem 3]{Rothaus_1986}. Since the Killing metric differs from the standard bi-invariant metric on \(\cSO_d\) by a constant scaling factor \(2(d-2)\), the claimed LSI constant follows by the scaling behavior of logarithmic Sobolev inequalities. By Gross' theorem~\cite[Theorem 5.2.3]{bakry2013analysis}, this implies hypercontractivity of \((P_t)\): for \(1 < p \leq q < \infty\) and \(t \geq 0\),
\begin{equation}\label{eq:hc-semigroup}
\|P_t F\|_{L^q(\cSO_d)} \le \|F\|_{L^p(\cSO_d)} 
\quad\text{whenever}\quad q-1 \le e^{2t/c_{\mathrm{LS}}(\cSO_d)}(p-1).
\end{equation}
Combining \eqref{eq:casimir_eigenvalues} and \eqref{eq:hc-semigroup} yields the hypercontractivity inequality on $\cSO(d)$ for each degree-$\ell$ block: 
\begin{equation}\label{eq:SOdHC}
\|F\|_{L^q(\cSO_d)} \;\le\;
\Bigl(\tfrac{q-1}{p-1}\Bigr)^{\gamma_d(\ell)}\|F\|_{L^p(\cSO_d)}, \qquad \gamma_d(\ell) = \frac{2 \ell (\ell + d - 2)}{d-2}.
\end{equation}
The hypercontractivity result for $\cO_d$ follows by writing $\cO_d = \cSO_d \sqcup \mathsf{R} \cSO_d$ and applying the above hypercontractivity to both components. 

\begin{lemma}[Hypercontractivity of matrix coefficients]\label{lem:hypercontractivity_matrix_coefficients} Let $d \geq 3$ and $\ell \geq 1$.
For all $F \in \cM_{d,\ell}$ and $1 < p \leq q < \infty$
\[
\| F \|_{L^q (\cO_d)} \leq 2^{\frac{1}{p} - \frac{1}{q}}\left( \frac{q-1}{p-1}\right)^{\gamma_{d} (\ell)} \| F\|_{L^p (\cO_d)}, \qquad \gamma_d(\ell) = \frac{2 \ell (\ell + d - 2)}{d-2}.
\]
\end{lemma}

\begin{proof}
    Let ${\sf R} \in \cO_d$ be any fixed reflection and decompose $\cO_d$ as $\cO_d = \cSO_d \sqcup \cSO_d {\sf R}$. For $F \in \cM_{d,\ell}$, note that $F\circ {\sf R} \in \cM_{d,\ell}$ and also satisfies \eqref{eq:SOdHC}. Thus,
    \[
    \begin{aligned}
    \| F \|_{L^q (\cO_d)}^q =&~ \frac{1}{2} \| F \|_{L^q (\cSO_d)}^q + \frac{1}{2} \| F \circ {\sf  R } \|_{L^q (\cSO_d)}^q \\
    \leq&~ \frac{1}{2} \left( \frac{q-1}{p-1}\right)^{q\gamma_{d} (\ell)}  \left\{  \| F \|_{L^p (\cSO_d)}^q +\| F \circ {\sf  R } \|_{L^p (\cSO_d)}^q\right\}  \\
    \leq&~ \frac{1}{2} \left( \frac{q-1}{p-1}\right)^{q\gamma_{d} (\ell)}  \left\{  \| F \|_{L^p (\cSO_d)}^p +\| F \circ {\sf  R } \|_{L^p (\cSO_d)}^p\right\}^{q/p} =
    2^{\frac{q}{p}-1}\left( \frac{q-1}{p-1}\right)^{q\gamma_{d} (\ell)}   \| F \|_{L^p (\cO_d)}^{q},
    \end{aligned}
    \]
    where we have used the fact that \(a^r+b^r\leq (a+b)^r\) for \(a,b\geq 0\) and \(r\geq 1\). Taking $q$-th roots concludes the proof.
\end{proof}

\subsection{Matrix concentration inequality}\label{app:prob-tools}

Below, we present a matrix concentration inequality for sums of independent random matrices, that follows from the line of work~\cite{brailovskaya2024universality,bandeira2023matrix}. Similar versions have been used in analogous context, see for instance~\cite{damian2025generative,damian2024computational,joshi2025learning}. 

Given a random matrix \(\bX = \sum_{i\in [n]}\bX_i\in \R^{p \times q}\), define
\begin{equation}\label{eq:matrix_concentration_parameters}
    \begin{gathered}
        \sigma(\bX)^2 = \max\left\{\|\E[\bX\bX^\sT]\|_\op, \|\E[\bX^\sT\bX]\|_\op\right\},
    \\ \sigma_\ast(\bX)^2 = \sup_{\|\bu\|_2=\|\bv\|_2=1} \E[\<\bu,\bX\bv\>^2],\quad \bar{R}(\bX)^2 = \E\left[\max_{1\leq i \leq p} \|\bX_i\|_\op^2\right],\\
    v(\bX)^2 = \|\Cov(\bX)\|_\op.
    \end{gathered}
\end{equation}
Here, \(\|\Cov(\bX)\|_\op\) denotes the operator norm of the covariance operator of the random matrix \(\bX\) viewed as a vector in \(\R^{pq}\).
Note that
\begin{align*}
    \|\E[\bX\bX^\sT]\|_\op & = \sup_{\|\bu\|_2=1} \E[\|\bX^\sT \bu\|_2^2]
    \\ & = \sup_{\|\bu\|_2=1}\sum_{j=1}^p \E[\<\be_j,\bX^\sT \bu\>^2] \leq p\sigma_\ast(\bX)^2.
\end{align*}
Similarly, \(\|\E[\bX^\sT \bX]\|_\op \leq q \sigma_\ast(\bX)^2\) so that
\begin{equation}\label{eq:sigma_vs_sigma_ast}
    \sigma(\bX)^2 \leq (p \vee q) \sigma_\ast(\bX)^2.
\end{equation}

\begin{lemma}[Matrix concentration inequality]\label{lem:matrix_concentration}
    Let \(\{\bX_i\}_{i=1}^n \subseteq \R^{p\times q}\) be independent centered random matrices and \(\bX = \sum_{i=1}^n \bX_i\). Then, there exists a universal constant \(C>0\) such that, for all \(t>0\) and \(R\geq \bar{R}(\bX)^{1/2}\sigma(\bX)^{1/2}+2^{1/2}\bar{R}(\bX)\) such that \(\P(\max_{i\in [n]}\|\bX_i\|_\op > R)\leq \delta\),
    \begin{align*}
        \|\bX\|_\op & \leq C\left(2\sigma(\bX) + \sigma(\bX)^{1/2}v(\bX)^{1/2}\log(p+q)^{3/4} + \sigma_\ast(\bX)t^{1/2} + R^{1/3}\sigma(\bX)^{2/3}t^{2/3}+Rt\right),
    \end{align*}
    with probability at least \(1-\delta - (p+q+1)e^{-t}\).
\end{lemma}

\begin{proof}
    For every \(i\in [n]\), define the symmetric dilation of \(\bX_i\) as
    \[
        \breve{\bX}_i = 
        \begin{bmatrix}
            \bzero & \bX_i 
            \\ \bX_i^\sT & \bzero
        \end{bmatrix} \in \R^{(p+q)\times (p+q)}.
    \]
    Indeed, \(\{\breve{\bX}_i\}_{i=1}^n\) are independent self-adjoint random matrices. By~\cite[Lemma 4.10]{bandeira2023matrix}, \( \sigma(\breve{\bX}) = \sigma(\bX)\) and \(\sigma_\ast(\breve{\bX}) = \sigma_\ast(\bX)\). Furthermore,
    \begin{align*}
        \|\breve{\bX}_i\|_\op^2 & = \|\breve{\bX}_i^2\|_\op
         = \max\{\|\bX_i\bX_i^\sT\|_\op, \|\bX_i^\sT\bX_i\|_\op\} 
         = \|\bX_i\|_\op^2,
    \end{align*}
    so that \(\bar{R}(\breve{\bX}) = \bar{R}(\bX)\).
    By~\cite[Theorem 2.8]{brailovskaya2024universality}, and the fact that \(\|\breve{\bX}\|_\op = \|\bX\|_\op\), there exists a universal constant \(C>0\) such that,
    \begin{align*}
        \P\left(\|\bX\|_\op- \|\bG\|_\op\geq C\epsilon_R(t), \max_{1\leq i \leq n} \|\bX_i\|_\op \leq R\right) & \leq (p+q)e^{-t},
    \end{align*}
    for all \(t>0\) and \(R\geq \bar{R}(\bX)^{1/2}\sigma(\bX)^{1/2}+2^{1/2}\bar{R}(\bX)\) where
    \[
        \epsilon_R(t) = \sigma_\ast(\bX)t^{1/2} + R^{1/3}\sigma(\bX)^{2/3}t^{2/3}+Rt,
    \]
    and \(\bG = \sum_{i=1}^n \bG_i\) with \(\{\bG_i\}_{i=1}^n\) independent centered Gaussian random matrices satisfying \(\E[\bG_i \otimes \bG_i] = \E[\bX_i \otimes \bX_i]\).

    Next, by~\cite[Corollary 2.2 and Lemma 2.5]{bandeira2023matrix}, there exists a universal constant \(C'>0\) such that, for all \(t>0\),
    \[
        \P\left(\|\bG\|_\op \geq 2\sigma(\bX) + C'\sigma(\bX)^{1/2}\|\Cov(\bX)\|^{1/4}\log(p+q)^{3/4}+C'\sigma_\ast(\bX)t\right)\leq e^{-t^2},
    \]
    where \(\Cov(\bX)\) is the covariance of the random matrix \(\bX\) viewed as a vector in \(\R^{pq}\). Combining the above two inequalities, we get that there exists a universal constant \(C''>0\) such that, for all \(t>0\) and \(R\geq \bar{R}(\bX)^{1/2}\sigma(\bX)^{1/2}+2^{1/2}\bar{R}(\bX)\) such that
    \[
        \P\left(\max_{i\in [n]}\|\bX_i\|_\op > R\right)\leq \delta,
    \]
    then
    \begin{align*}
        \|\bX\|_\op & \leq C''\left(\sigma_\ast(\bX)t^{1/2} + R^{1/3}\sigma(\bX)^{2/3}t^{2/3}+Rt+\sigma(\bX)^{1/2}\|\Cov(\bX)\|^{1/4}\log(p+q)^{3/4} + 2\sigma(\bX)\right),
    \end{align*}
    with probability at least \(1-(p+q+1)e^{-t}-\delta\).
\end{proof}

\end{document}